\documentclass{ws-m3as}
\usepackage[a4paper]{geometry}
\usepackage{epstopdf}
\usepackage{amsmath,amssymb,amsfonts}
\usepackage{mathrsfs}
\usepackage{xcolor}
\usepackage[linesnumbered,ruled]{algorithm2e}
\usepackage{algpseudocode}
\usepackage{listings}
\usepackage[colorlinks,
            linkcolor=blue!88!black,
            anchorcolor=green!55!black,
            citecolor=green!66!black,
            ]{hyperref}
\usepackage{units}
\usepackage{multirow}
\usepackage{makecell}
\usepackage{bookmark}
\usepackage{booktabs}
\usepackage{caption}
\usepackage{threeparttable}
\usepackage{tabularx}
\usepackage[section]{placeins}
\usepackage{comment}
\usepackage{tikz}
\usepackage{mathtools}
\usetikzlibrary{arrows.meta, positioning, shadows.blur}

\allowdisplaybreaks[4]

\newenvironment{MyProof}[1][\proofname]{%
  \par\noindent\textbf{#1.}\enspace
}{%
  \hfill\ensuremath{\square}\par
}

\newcommand{\mblue}[1]{\textcolor{blue}{#1}}

\begin{document}

\markboth{F.G. Ma, Z.M. Ouyang, W.Y. Tian \& L. Wu}{Dynamics of chemotactic gliding-aggregation in Myxobacteria on rough bounded domains}

%
\catchline{}{}{}{}{}
%

\title{Dynamics of Chemotactic Gliding-Aggregation in Myxobacteria on Bounded Domains: Stochastic Modeling, Analysis, and Deep Neural Network Simulations}

\author{Fugui Ma$^{\dag,}$\footnote{The first author is the corresponding author.}~~ {and}~~ Zhimeng Ouyang$^\ddag$}

\address{School of Mathematical Sciences, Peking University, 5 Yiheyuan Road, \\
Haidian District, Beijing 100871, China\\
$^\dag$mafugui@math.pku.edu.cn\\
$^\ddag$ouyangzm@math.pku.edu.cn
}

\author{Wenyi Tian}

\address{Center for Applied Mathematics and KL-AAGDM, Tianjin University, 92 Weijin Road,\\
Nankai District, Tianjin 300072, China\\
twymath@gmail.com
}

\author{Lei Wu}

\address{Department of Mathematics, Lehigh University, Bethlehem, 17 Memorial Drive East,\\
Bethlehem, PA 18015, USA\\
lew218@lehigh.edu
}
\maketitle

\begin{history}
\received{(Day Month Year)}
\revised{(Day Month Year)}
\comby{(xxxxxxxxxx)}
\end{history}

\begin{abstract}
Bacterial chemotactic movement and collective aggregation have long attracted substantial interest in mathematical biology and applied modeling. Classical Keller--Segel-type systems, however, are typically formulated under idealized laboratory assumptions, such as smooth agar substrates, and thus cannot adequately capture the gliding dynamics of myxobacteria in naturally rough environments like soil. In this paper, we propose a unified framework that integrates stochastic modeling, rigorous analysis, and deep neural network-based simulation of chemotactic gliding--diffusion and aggregation processes on bounded domains. Starting from a lattice-based discrete agent description and a subordinated Langevin equation driven by an inverse stable subordinator at the microscopic level, we characterize anomalous gliding dynamics on rough surfaces and derive a macroscopic time-nonlocal Keller--Segel-type chemotaxis model with logarithmic sensitivity. We then establish a comprehensive solution theory for the resulting model, covering mass conservation, novel regularity results, local well-posedness in any spatial dimension, and global well-posedness in two and three. The analysis relies on several newly developed ingredients, including a fractional Lyapunov functional, a variational inequality adapted to the time-nonlocal structure, logarithmic Sobolev-type estimates, Bregman distance techniques, and a weighted bootstrap mechanism adapted to the singular sensitivity and time-nonlocal memory. Finally, we design a mesh-free, positivity-preserving, multi-objective, time-marching physics-informed neural network method with separate architectures and tailored variable transformations. Numerical experiments on complex geometries, including a butterfly-shaped domain, demonstrate the robustness, accuracy, and flexibility of the proposed computational framework across a range of Keller--Segel-type systems.
\end{abstract}



\section{Introduction}
Unraveling the mechanisms of bacterial motility is essential for advancing our understanding of infection pathogenesis, microbial ecology, and collective cellular behavior \cite{Cremer2019,Keegstra2022,Uar2025}. Driven by these biological imperatives, the mathematical modeling of bacterial movement has evolved into a rigorous quantitative discipline, built upon fundamental studies (see, e.g., Refs.~\refcite{Adler1966,BERG1972,Budrene1995,Jain2025,Mittal2003,Scheidweiler2024}). At the heart of this endeavor lies bacteria chemotaxis, the directed movement of cells along gradients of chemoattractants. A cornerstone model for this phenomenon, the Keller-Segel (KS) system (first introduced in Refs.~\refcite{Keller70} and \refcite{Keller1971}), which has been extensively refined and applied across disciplines to describe diverse biological phenomena (see, e.g., Refs.~\refcite{Arumugam21,Bellomo22,Chen20,Dai23,Estrada18,Hillen13,Keller1971,Li14}). Despite these advancements, a critical limitation remains, most existing models are based on simplified experimental conditions, typically assuming movement on smooth surfaces like agar or within uniform solutions.

In natural ecosystems, bacterial chemotaxis rarely occurs within the pristine, uniform conditions of the laboratory; rather, it unfolds on host surfaces marked by geometric roughness and structural complexity. This physical complexity directly affects how bacteria navigate and search. A prototypical example is soil-dwelling myxobacteria. Unlike the smooth, continuous expansion observed in standard laboratory assays (e.g., E.~coli swimming in liquid or on soft agar\cite{Cremer2019,Kearns2010}), myxobacterial navigation in soil is profoundly constrained by surface roughness and microscale granular obstacles (see, e.g., Refs.~\refcite{Malla2025,Mauriello2010,Wolgemuth2002}), which impose significant physical limits on their `gliding' motility\cite{Dawid2000}.  As cells glide across these rough terrains, they must continually circumvent obstacles and dynamically chart new paths for migration. These constraints drive cells into a discontinuous `stop-and-go' pattern with frequent re-orientations\cite{Islam2015,Nan2011}, an adaptive search strategy that breaks the Markovian assumption of classical diffusion. The accumulated effect of these interactions manifests as anomalous, memory-driven transport, wherein the population's current state remains intrinsically coupled to its historical trajectory\cite{Felix13}. Consequently, existing chemotaxis models lack the capacity to reliably capture such historically dependent motility. To describe these history‑guided dynamics, we introduce a new chemotaxis-diffusion model, a time-nonlocal KS system \eqref{Pro:FKS}, that incorporates a time non-local operator.

\subsection{The Model}
Let $\Omega\subset\mathbb{R}^{d}$ ($d\geq2$) be an open bounded domain with smooth boundary $\partial\Omega$. We model the chemotaxis gliding–diffusion–aggregation behavior of myxobacteria on rough surfaces using the following time-nonlocal KS system with singular sensitivity
\begin{subequations}\label{Pro:FKS}
\begin{align}
&{_{0}^{C}\mathfrak{D}^{\alpha}_t}n(\mathbf{x},t)=\mathcal{D}\Delta n(\mathbf{x},t)-\mathcal{D}\chi\nabla\cdot\Big(\frac{n(\mathbf{x},t)}{c(\mathbf{x},t)}\nabla c(\mathbf{x},t)\Big),
             &(\mathbf{x},t)\in\Omega\times(0,T],\label{Pro:KS1}\\
&{_{0}^{C}\mathfrak{D}^{\alpha}_t}c(\mathbf{x},t)= \mathcal{D}\Delta c(\mathbf{x},t)-\gamma c(\mathbf{x},t)+ n(\mathbf{x},t),
             &(\mathbf{x},t)\in\Omega\times(0,T],\label{Pro:KS2}\\
&n(\mathbf{x},0)= n_{0}(\mathbf{x}),~c(\mathbf{x},0)= c_{0}(\mathbf{x}),
             &\mathbf{x}\in\Omega,\label{Pro:KS3}\\
&\frac{\partial n(\mathbf{x},t)}{\partial \mathbf{\nu}}=0,~~~~~\frac{\partial c(\mathbf{x},t)}{\partial \mathbf{\nu}}=0,
             &(\mathbf{x},t)\in\partial\Omega\times(0,T], \label{Pro:NS4}
\end{align}
\end{subequations}
where $n(\mathbf{x},t)$ and $c(\mathbf{x},t)$ denote the density of myxobacteria and the concentration of slime trail, respectively. With $\mathcal{D}>0$ denoting the diffusion coefficient, $\chi>0$ the chemotactic sensitivity, and $\gamma>0$ the decay rate of the slime, the final time is set as $T>0$.  Initial conditions are given by positive functions $n_{0}(\mathbf{x})$ and $c_{0}(\mathbf{x})$, and $\mathbf{\nu}$ denotes the unit outward normal vector on $\partial\Omega$. Time evolution is described by the Caputo fractional derivative ${_{0}^{C}\mathfrak{D}^{\alpha}_t}$ of order $\alpha\in(0,1)$, which is defined as
\begin{equation}
{_{0}^{C}\mathfrak{D}^{\alpha}_t}u(\mathbf{x},t)
\triangleq\frac{1}{\Gamma(1-\alpha)}\int_{0}^{t}(t-s)^{-\alpha}\frac{\partial u(\mathbf{x},s)}{\partial s}\,\mathrm{d}s.
\end{equation}
This operator mathematically captures the anomalous diffusion and memory effects observed during myxobacterial gliding across rough surfaces. Importantly, when $\alpha\rightarrow 1$, System \eqref{Pro:FKS} simplifies to the classical KS model with singular (logarithmic) sensitivity,  which characterizes chemotaxis under Markovian diffusion conditions.

The analysis of System \eqref{Pro:FKS} is rendered particularly challenging by two core mathematical features: ($i$) the non-local nature of the Caputo fractional derivative operator ${_{0}^{C}\mathfrak{D}^{\alpha}_t}$; and ($ii$) the singular chemotactic sensitivity $\frac{1}{c(\mathbf{x},t)}\nabla c(\mathbf{x},t)=\nabla\log(c(\mathbf{x},t))$, which becomes unbounded as $c(\mathbf{x},t)\to0$. For $\alpha\in(0,1)$, the system no longer possesses a standard local time‑derivative structure; instead, the time non-local derivative operator introduces memory effects that break the Markovian property, while the singular sensitivity couples to this nonlocality, collectively yielding analytical difficulties well beyond those encountered in classical Keller–Segel models.

Despite these challenges, non-local models of this type possess substantial theoretical value, as they inherently capture the memory effects and anomalous transport characteristic of complex biological systems. We will show that the proposed model offers a more faithful description of the adaptive searching behavior of myxobacteria in natural, heterogeneous habitats, thereby providing greater ecological relevance and predictive capacity than conventional models formulated under homogeneous assumptions.

To this end, we systematically investigate the modeling validity of System \eqref{Pro:FKS} and establish its theoretical foundation via rigorous PDE analysis. Moreover, capitalizing on the mesh-free advantage of deep learning algorithms, we conduct numerical simulations to visualize the solution dynamics and verify the global existence of solutions in two spatial dimensions.

\subsection{Connections with previous work and further motivations}
\label{subsec:Prew}
Although bacterial chemotaxis is ubiquitous in nature,  our understanding of it remains largely derived from laboratory studies of model organisms, most notably \emph{E.~coli}\cite{Kearns2010,Keegstra2022}. Research on \emph{E.~coli} has been instrumental in bridging microscopic molecular signaling and macroscopic population behavior\cite{Adler1966,Cremer2019,Scheidweiler2024}. On smooth agar surfaces, \emph{E.~coli} detects temporal chemical gradients via methyl-accepting proteins and modulates its motility through a phosphotransferase signaling cascade. This mechanism results in the characteristic `run-and-tumble' motion directed toward attractants (see, e.g., Refs.~\refcite{Cremer2019,Keegstra2022,Livne2024,Scheidweiler2024}). This behavior is fundamentally Markovian at the microscopic scale.

To mathematically describe such population-level chemotactic behaviors, the Keller-Segel framework is a standard paradigm. In its general form, the system reads\cite{Bellomo22}
\begin{subequations}\label{Pro:cFKS}
\begin{align}
&\partial_t n(\mathbf{x},t)=-\nabla\cdot \mathbf{\mathcal{J}}\big(n(\mathbf{x},t),c(\mathbf{x},t)\big)
+f\big(n(\mathbf{x},t),c(\mathbf{x},t)\big),
&\text{in}~~\Omega\times(0,T],\label{Pro:cKS1}   \\
&\epsilon\,\partial_t c(\mathbf{x},t)=\mathcal{D}\nabla^2 c(\mathbf{x},t)+g\big(c(\mathbf{x},t),n(\mathbf{x},t)\big),
&\text{in}~~\Omega\times(0,T].\label{Pro:cKS2}
\end{align}
\end{subequations}
Here, equation \eqref{Pro:cKS1} governs the evolution of the bacterial density $n(\mathbf{x},t)$.
The total flux $\mathbf{\mathcal{J}}(n,c)$ comprises two main contributions: ($i$) Fickian diffusion driven by the bacterial concentration gradient; and ($ii$) chemotactic drift, i.e., directed movement along chemical gradients (e.g., toward lower attractant or higher repellent concentrations). Equation \eqref{Pro:cKS2} describes the dynamics of the chemical signal $c(\mathbf{x},t)$ (e.g., protons or nutrients), which diffuses with coefficient $\mathcal{D}$. The functions $f(n,c)$ and $g(n,c)$  represent source/sink terms associated with biological processes. $f(n,c)$ typically models cellular proliferation or death, while $g(n,c)$ accounts  chemical production or consumption. By choosing specific constitutive forms for $\mathbf{\mathcal{J}}$ in \eqref{Pro:cFKS}, one recovers a variety of classical and phenomenological KS models (see Table \ref{Tab:ref} for the case $f(n,c)=0$). For a broader survey of KS‑type models, we refer the reader to review articles such as Refs.~\refcite{Arumugam21,Hillen09}.

\begin{table}[htbp]\footnotesize{
    \centering
    \renewcommand{\arraystretch}{1.2}
    \captionsetup{skip=8pt}
    \begin{threeparttable}
    \caption{Several phenomenological KS models satisfying \eqref{Pro:cFKS} with $f(n,c)=0$ and $\chi>0$.}
    \label{Tab:ref}
    \begin{tabular}{p{0.35\textwidth}p{0.35\textwidth}c p{0.20\textwidth}}
      \toprule
      Model Name $^{\mathrm{Refs.}}$ & Determined Flux $\mathbf{\mathcal{J}}(\cdot,\cdot)$ & Source/Sink $g(\cdot,\cdot)$ \\
      \midrule
        Patlak-Keller-Segel model\cite{Chen20,Kavallaris09,Keller70} & $-\mathcal{D}\big(\nabla n+\chi n\nabla c\big)$ & $-c+n$, $\epsilon=1$\\
        \emph{Logistic} model\cite{Hou18,Keller1971,Li14,Winkler16,Winkler22}
        & $-\mathcal{D}\big(\nabla n+\chi\, \frac{n}{c}\nabla c\big)$ & $-cn$, $\epsilon=1$\\
        Modified KS\cite{Livne2024,Livne2025}
        & $-\mathcal{D}\big(\nabla n+\chi\,\frac{n}{c+\beta}\nabla (c+\beta)\big)$  \tnote{\mblue{$\dag$}}  & $-\gamma nc$ or $-\gamma c^{3/2}n$\\
        \emph{Volume-filling} model I\cite{Hillen09}
        & $-\mathcal{D}\big(\nabla n+\chi\, n\big(1-\frac{n}{\Upsilon}\big)\nabla c\big)$  \tnote{\mblue{$\ddag$}}  & $-\gamma c+n$, $\epsilon=1$\\
        \emph{Volume-filling} model II\cite{Velazquez04a,Velazquez04b}
        & $-\mathcal{D}\big(\nabla n+\chi\, \frac{n}{1+\varepsilon n}\nabla c\big)$, $\varepsilon>0$  & $n$, $\epsilon=0$\\
        \emph{Receptor} model\cite{Bouvard2022,Meyer2014,Winkler10}
        & $-\mathcal{D}\big(\nabla n+\chi\, \frac{n}{(1+\bar{\varepsilon} c)^2}\nabla c\big)$, $\bar{\varepsilon}>0$  & $n$, $\epsilon=0$\\
        Nonlinear-diffusion model\cite{Arumugam21,Hillen09}
        & $-\mathcal{D}\big(n^{\kappa}\nabla n+\chi\, n\nabla c\big)$, $\kappa\in \mathbb{N}_{+}\cup\{0\}$\tnote{\mblue{$\P$}} & $-\gamma c+n$, $\epsilon=1$\\
        Nonlocal-diffusion model\cite{Estrada18} & $-\mathcal{D}\big(\nabla^{s-1} n+\chi\, n\nabla c\big)$, $1<s<2$  & $-\gamma c+n$, $\epsilon=1$\\
      \bottomrule
    \end{tabular}
    \begin{tablenotes}
     \item[\mblue{$\dag$}] The model reduces to the KS with logarithmic sensitivity as $\beta\rightarrow0$; \mblue{$^\ddag$} For $\Upsilon, \varepsilon>0$, the limit of $\Upsilon\rightarrow\infty$ or $\varepsilon\rightarrow0$ leads to the KS; \mblue{$^\P$} The limit of $\kappa\rightarrow0$ results in the KS.
     \end{tablenotes}
     \end{threeparttable}}
\end{table}

The KS model with logarithmic (singular) sensitivity has been widely adopted due to its empirical success in capturing bacterial aggregation and pattern formation (see, e.g., Refs.~\refcite{Bouvard2022,Livne2024,Livne2025,Meyer2014,Salek2019,Uar2025}). Mathematically, this system often referred to as the logarithmic KS system, which exhibits rich behavior under Neumann boundary conditions, particularly concerning global solvability and long-time dynamics. Below we outline key theoretical milestones, focusing on the case where $g(n,c)=-\gamma c+n$, $f(n,c)\equiv0$, $\mathcal{D}>0$, and the initial data are sufficiently regular and positive. The global existence of solutions depends critically on the spatial dimension $d$ and the chemotactic sensitivity coefficient $\chi$. In two dimensions ($d=2$), Ref.~\refcite{Aida05} establishes the global existence of classical solutions and characterizes their asymptotic behaviour. For genral $d\geq2$, Ref.~\refcite{Winkler11} proves global classical solvability under the condition $\chi<\sqrt{2/d}$. Moreover, by means of a key inequality (Lemma 2.3 in Ref. \refcite{Winkler11}), the same work also demonstrates the existence of global weak solutions for all $d\geq2$ whenever $0<\chi<\sqrt{(d+2)/(3d+4)}$. A subsequent refinement in Ref.~\refcite{Winkler22} clarifies that the condition $\chi<\sqrt{2/d}$ essentially requires the ratio $\chi^2/\mathcal{D}>0$ to be sufficiently small; under this smallness assumption, classical solutions converge exponentially to the homogeneous steady state $(\bar{n}_0,\bar{c}_0)$, where $\bar{u}_0:=\frac{1}{|\Omega|}\int_{\Omega}u(\mathbf{x},0)\,\mathrm{d}\mathbf{x}$.

The frontier was subsequently pushed to larger values of $\chi$. Through sophisticated analytical techniques, Ref.~\refcite{Stinner11} establishes the existence of global weak solutions for arbitrary $\chi>0$.  Ref.~\refcite{Lankeit17} further demonstrates global solvability for $d=2$ with any finite $\chi$, for $d=3$ with $\chi<\sqrt{8}$, and for $d\geq4$ with $\chi<\frac{d}{d-2}$. These landmark results fundamentally rely on the classical parabolic PDE theory, including the comparison principle, Moser-type estimates, and energy functional methods\cite{Aida05,Lankeit17,Stinner11}, all of which are intrinsically tied to the local-in-time derivative structure of the equations. This foundational framework, however, breaks down when the standard time derivative is replaced by a nonlocal operator. Consequently, the analytical tools developed for classical KS systems are no longer directly applicable, necessitating a fundamentally different approach, which is precisely the focus of the present work.

Recently, inspired by the success of non‑ergodic anomalous diffusion models in capturing diffusion phenomena in non‑equilibrium, heterogeneous, and porous media\cite{Klafter15}, as well as by the efficacy of the KS model in describing chemotaxis‑diffusion dynamics in homogeneous environments\cite{Langlands10,Meyer2014}, researchers are establishing the connection between non‑ergodic anomalous diffusion and bacterial chemotaxis‑diffusion\cite{Henry2010,Felix13}. By organically merging the two, the aim is to model and analyze more realistic and complex bacterial chemotaxis‑diffusion dynamics that better reflect natural habitats.

Most current efforts in this direction focus on modifications of the Patlak–Keller–Segel (PKS) model. For instance, Ref.~\refcite{Estrada18} incorporates  super-diffusion into the bacterial chemotaxis‑diffusion framework to characterize long‑step bacterial movement in nutrient‑scarce environments, achieved by replacing the spatial Laplace operator with a fractional Laplacian. Ref.~\refcite{Ma25} combines sub-diffusion with myxobacterial chemotaxis to describe the chemotactic aggregation dynamics in porous media, where the temporal local derivative is replaced by the non‑local Caputo fractional derivative.The well‑posedness and solution theory for these two classes of modified models, under various domains and generalized settings, have been systematically developed in Refs.~\refcite{Bezerra22,Costa23,Jiang25,Liu18} and related woderivativerks.

It is worth emphasizing, however, that merging anomalous diffusion with bacterial chemotaxis–diffusion–aggregation dynamics requires more than mere operator substitutions; rigorous biophysical justification and careful model derivation are equally essential. Looking ahead, further refinements to generalized KS models, combined with experimental validation, are poised to accelerate progress in this emerging direction.

The transition to the time-nonlocal KS system represents far more than a technical generalization; it marks a fundamental conceptual shift. Specifically, the historical memory inherent in the time non-local derivative operator shatters the core analytical pillars, most notably the conventional Lyapunov functionals and comparison principles, upon which the established theory for logarithmic models has traditionally rested. This invalidation renders classical methods inapplicable, as the hereditary nature of the system prevents the direct translation of prior global existence results. Consequently, a profound theoretical void emerges. The central challenge of the present analysis, therefore, is to establish a novel analytical framework, designed to address the complex interdependencies and inherent nonlocal characteristics that fundamentally define the system.

Since analytical solutions to KS models are generally unavailable, numerical methods are indispensable for visualizing the system dynamics. For a broad class of simplified KS models, classical discretization  techniques are widely employed, including finite element methods (see, e.g., Refs.~\refcite{Saito07,Sulman19,Wang25}), finite difference schemes (e.g., Refs.~\refcite{Epshteyn19,Hu23,Wang22}), finite volume approaches (e.g., Refs.~\refcite{Chertock08,Filbet06,Zhou17}), and Local Galerkin methods (e.g., Refs.~\refcite{Epshteyn08,Guo19,Qiu21}). These methods rest on a solid theoretical foundation, providing well‑established convergence and stability analyses, rigorous error estimates, and predictable computational costs. However, when applied to strongly nonlinear and coupled systems such as \eqref{Pro:FKS}, these classical methods exhibit pronounced limitations. Prominent among them are high implementation complexity, the need for carefully tailored iterative solvers, cumbersome mesh generation on irregular domains, and the curse of dimensionality. To our knowledge, no existing numerical scheme addresses the logarithmic KS model, highlighting the urgent demand for methods tailored to this nonlocal, singular system. This motivates our adoption of a mesh‑free deep learning strategy, which bypasses many of these difficulties and is naturally suited to the nonlocal structure of the system.

Physics-Informed Neural Networks (PINNs) offer a promising mesh-free approach, demonstrating strong adaptability to complex geometries and high-dimensional problems\cite{Raissi19}. Their implementation leverages automatic differentiation for computing complex derivatives within a unified framework\cite{Karniadakis2021}, and they exhibit an inherent capacity for refinement during training\cite{McClenny23}. Motivated by recent advances in hybrid methodologies that combine classical numerical methods with PINNs\cite{Fang23,Guo22}, we propose a hybrid strategy to solve the strongly nonlinear, coupled time-nonlocal system \eqref{Pro:FKS}, aiming to bridge the gap between biophysical realism and computational tractability.

\subsection{Main contributions and novelties of the paper}
The collective aggregation of myxobacteria serves as a paradigmatic example of self‑organization across scales, from individual stochastic motion to population‑level pattern formation. To capture this multiscale phenomenology, we develop an integrated framework that ties together biophysically grounded modeling, a rigorous existence theory for the resulting time-nonlocal system, and a robust hybrid numerical strategy. In particular, our model is designed to more faithfully describe the chemotactic gliding, diffusion, and aggregation of myxobacteria on rough surfaces, where surface roughness and physical obstacles profoundly modulate population dynamics. In doing so, we seek to forge a coherent narrative that bridges biological realism and analytical tractability. The main contributions of this work are summarized below.
\begin{itemize}
  \item \textbf{Multiscale mathematical modeling (micro-to-macro)}: We establish a novel mathematical framework that more faithfully characterize the chemotactic gliding, diffusion and aggregation of myxobacteria on rough surfaces, where geometric irregularities and granular obstacles fundamentally modulate population behavior. By formulating a subordinated Langevin equation driven by an inverse $\alpha$-stable subordinator, we provide a rigorous bridge between microscopic stochastic trajectories and the macroscopic system. This derivation substantiates both the physical rationality and mathematical integrity of the proposed model.

  \item \textbf{Advanced solution theory and analytical innovation}: We establish a comprehensive well-posedness theory for the proposed time-nonlocal KS system \eqref{Pro:FKS} with logarithmic sensitivity. Our results include local well-posedness of mild solution in arbitrary spatial dimensions, mass conservation, positivity preservation, and regularity in generalized Sobolev spaces, as well as global well-posedness in dimensions two and three. The analysis rests on a suite of novel tools: a fractional Lyapunov functional, Caputo-type variational inequalities, fractional convexity estimates, logarithmic Sobolev-type inequalities, Bregman distance techniques, and a weighted bootstrap mechanism specifically adapted to the singular sensitivity and time-nonlocal memory.

  \item \textbf{Robust positivity-preserving time-marching PINNs}:  We propose a multi-objective, positivity/non-negativity-preserving time-marching PINN algorithm. A distinctive feature of this approach is the use of independent neural networks for coupled variables, integrated with the $L_1$ temporal semidiscrete scheme. To strictly enforce physical constraints, we introduce the transformations $n = \rho^2 \ge 0$ and $c = \exp(v) > 0$ directly into the continuous system, ensuring the algorithm's generalizability across a wide class of KS-type models.

  \item \textbf{Numerical validation in complex geometries}: We design original numerical examples to verify the robustness of our theoretical and algorithmic frameworks. By constructing exact solutions with temporal H\"{o}lder continuity and performing simulations on a complex `butterfly-shaped' domain, we validate the global existence of solutions and demonstrate the algorithm's superior ability to handle irregular geometries and capture long-term evolutionary dynamics.
\end{itemize}

Overall, the proposed framework constitutes a promising and versatile approach for modeling biophysical dynamics in complex, irregular environments. Not only does it provide a powerful computational tool for capturing chemotactic dynamics within irregular geometries, but its intrinsic connection to non-equilibrium statistical physics also unlocks broad prospects for future theoretical and numerical explorations.

\subsection{Organization of the paper}

The remainder of this paper is organized as follows. Section~\ref{sec:secmF}
presents the modeling framework. We  begin by discussing the motility mechanisms of myxobacteria and constructing a discrete agent-based model to describe their trajectories on rough surfaces. Building on this microscopic description, we derive the macroscopic System~\eqref{Pro:FKS} via a subordinated Langevin equation. Section~\ref{sec:ST} states the main theoretical results for the resulting system, including mass conservation, positivity preservation, local mild well-posedness in arbitrary spatial dimensions, and global well-posedness in dimensions two and three. Detailed technical proofs, along with further regularity estimates, are provided in Sections~\ref{sec:fLST}, \ref{sec:LEMS}, and \ref{Sec:CLS}. Section~\ref{sec:SPINN} then introduces a mesh-free, multi-objective, time-marching PINNs algorithm designed to preserve the non-negativity and positivity structures of the model. Numerical experiments on complex geometries are presented to demonstrate the accuracy, robustness, and flexibility of the proposed computational framework. Finally, Section~\ref{sec:conclusions} summarizes the main findings and outlines possible directions for future research.

\section{Stochastic Modeling: From Micro to Macro}
\label{sec:secmF}
\subsection{The biology mechanism}
\label{sec:bioMechanism}
Myxobacteria are Gram-negative, rod-shaped bacteria that lack flagella and are thus incapable of swimming. Instead, they exhibit gliding motility to move across surfaces, including over the surfaces of sibling cells\cite{Wolgemuth2002}. These bacteria exhibit social behaviors, such as cooperative feeding, coordinated movement, and social development\cite{Velicer2003}. Gliding occurs on solid surfaces as well as at the water-air interface, and is characterized by cell bending and slime secretion. Their gliding speed ranges between 10 and 60 $\mu$m/min, depending on temperature, nutrient availability, and initial cell density\cite{Dawid2000}. Myxobacteria cells are typically rod-shaped, with dimensions that varys by species; they generally measure between $0.6$--$0.9$ $\mu$m in wide and $3$--$8$ $\mu$m in length\cite{Dawid2000}. Vegetative cells commonly exhibit one of two morphological types: either slender, flexible rods with tapered ends, or cylindrical, rigid rods with rounded ends\cite{Saggu2023}.

\begin{figure}[!ht]
 \centering
 \centerline{\includegraphics[width=0.86\textwidth]{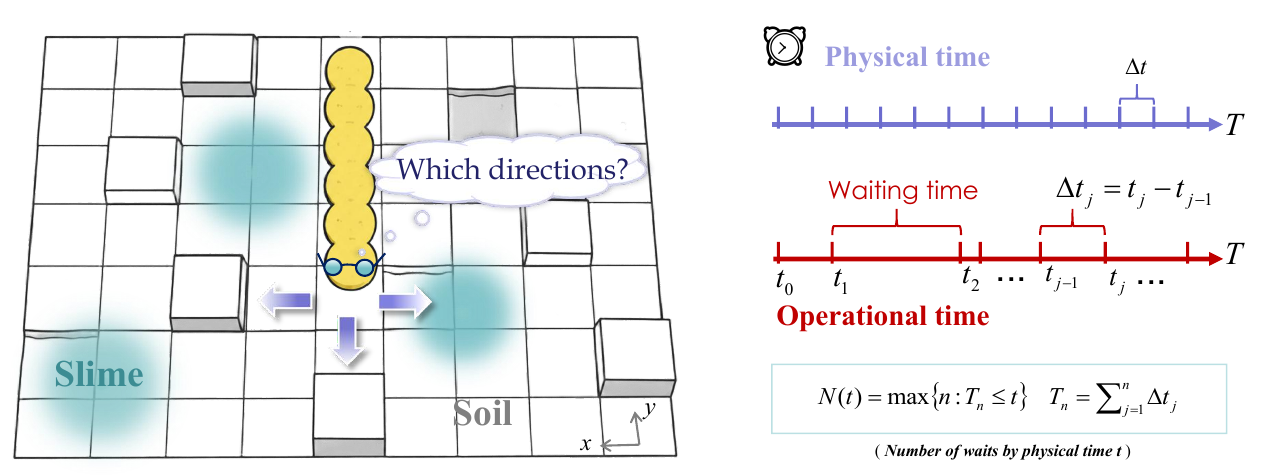}}
\caption{Schematic illustration of myxobacterial gliding on a rough surface and the corresponding random-walk timelines. $(a)$ Myxobacteria gliding on  a rough surface: In the discrete agent model, each bacterium occupies connected grid cells. During each update, the bacterium moves forward by extending its head to one of the four nearest-neighbor sites that is not occupied by its own body. Blue shaded areas indicate regions of bacterial aggregation and high slime concentration. $(b)$ Continuous time random walk timelines: The upper (blue) timeline represents physical time, sampled at regular intervals $\Delta t$, which corresponds to the waiting-time distribution $\psi(t)=\delta(t-\Delta t)$. The lower (red) timeline represents the internal (operational) time of the random-walk process, where each segment corresponds to a waiting time drawn from the distribution \eqref{eq:Dwaitingtime} for an individual movement event.}
\label{Fig:example0}
\end{figure}

Soil constitutes a primary natural habitat for myxobacteria\cite{Dawid2000}. These facultatively multicellular microorganisms are commonly found in topsoil environments and exhibit gliding motility, characterized by a smooth, non-rotational movement along the long axis of the cell\cite{McBride93}. During gliding, cells frequently pause and reverse direction, manifesting as `\emph{stop-and-go}' patternsunder electron microscopy (see, e.g., Refs.~\refcite{Kuhlwen71,Kuhlwen68,Reichenbach68}). As they move, myxobacteria deposit a slime trail behind them on the substrate. Each cell produces a distinct slime trail, visible as a phase-bright line, with clusters forming wider trails and single cells leaving narrower ones. When a cell encounters an existing trail, it tends to turn (probably through the acute angle of intersection), to follow the pre‑deposited trail. This turning behavior is thought to arise from the alignment maximizing cohesive interactions between newly secreted slime filaments and those already present in the trail. For further biological details, see, e.g., Refs.~\refcite{Kiskowski2004,Mauriello2010,Sozinova2005}.

Under starvation conditions, myxobacteria aggregate into large colonies and form fruiting bodies, ranging from 10 to 100 $\mu$m in size\cite{Wolgemuth2002}, through directed cell movement. It has been experimentally demonstrated that Myxococcus can detect and glide directly toward nearby colonies of potential prey to feed on them\cite{Kaiser2003}. Fruiting body formation is induced by nutritional deficiency and is regulated by factors such as nutrient concentration, pH, cation availability, and temperature, see Ref.~\refcite{Dawid2000} for further details.

Building on this biological background, the following sections analyze the characteristic motility patterns of microorganisms, with particular emphasis on the gliding behavior of myxobacteria and the associated diffusion and dissipation dynamics of slime trails.

\subsection{The discrete agent model for micro patterns}
\label{sec:lattice}
Building on the biological mechanisms detailed in Section \ref{sec:bioMechanism}, this section explores the chemotactic diffusion patterns of myxobacteria and the spatial evolution of slime on a $100\times100$ square-grid domain subject to Neumann boundary conditions. These analyses lay the theoretical groundwork for the stochastic modeling presented hereafter.

Following established agent-based frameworks (see, e.g., Refs. \refcite{Othmer97,Stevens00}), we treat each cell as an autonomous agent governed by the simple rules introduced in Sec.~\ref{sec:bioMechanism}, together with the following specific assumptions (see Fig. \ref{Fig:example0} for an illustration):
\begin{description}
  \item[A.1.]Each myxobacteria is represented as an $8\times1$ filament, occupying eight contiguous grid cells.
  \item[A.2.] At each step, a myxobacterium glides to one of the four cells adjacent to its marked pole, which we refer to as the `head' for descriptive purposes, though it carries no biological meaning.
  \item[A.3.] The complex structure of soil acts as a formidable obstacle to myxobacteria, resulting in frequent and prolonged trapping events. This behavior is captured by a heavy-tailed power-law  waiting‑time distribution(cf. e.g., Ref.~\refcite{Klafter15}):
      \begin{equation}\label{eq:Dwaitingtime}
        \psi(t)\sim \frac{\tau^{\alpha}}{t^{1+\alpha}},\quad  0 <\alpha<1,
      \end{equation}
      where $\tau$ is the characteristic time scale of the distribution.
\end{description}
Based on these assumptions, we construct a discrete lattice‑based agent model as follows.

Let $\mathcal{P}=\{1,2,\ldots,100\}$ and define the set of lattice points as $\mathcal{P}^2$. We establish the discrete lattice model on $\mathcal{P}^2$. For each site $x\in \mathcal{P}^2$, let $N_x$ be the set of its four nearest neighbors (up, down, left, right), as specified in Assumption \textbf{A.2.} and illustrated in Fig. \ref{Fig:example0} (a). Consistent with Assumption \textbf{A.1}, the sites occupied by myxobacteria $k$ at time $t_j$ are given by $\mathcal{B}_{k}(t_j)=\{(x_1,\ldots,x_8):x_{i}\in \mathcal{P}^{2}, i=1,2,\ldots,8 \}$, with head orientation $h_k(t_j)\in\{(1,0),(0,1),(-1,0),(0,-1)\}$. Let $\mathcal{I}_x(t_j)=\sharp\{y\in \mathcal{B}_{k}(t_j),k=1,2,\ldots,m:x=y\}$ denote the number of myxobacteria segments covering lattice point $x$ at time $t_j$. The bacterial density $\mathcal{N}(x,t_j):\mathcal{P}^{2}\times\mathbb{N}\rightarrow \mathbb{R}_{+}$ and slim concentration $\mathcal{C}(x,t_j):\mathcal{P}^{2}\times\mathbb{N}\rightarrow \mathbb{R}_{+}$ evolve according to production rates $\kappa_n$, $\kappa_c\geq0$ and decay rates $\lambda_n$, $\lambda_s\geq0$, along with a slime diffusion coefficient $D_c$.

The dynamics of slime production are specified as follows. The myxobacterial density evolves according to the update rule in Ref.~\refcite{Stevens00}, given by
\begin{equation}\label{eq:bacden}
\mathcal{N}(x,t_{j+1})=\big(1-\lambda_{n}\big)\mathcal{N}(x,t_j)+\kappa_n\mathcal{I}_{x}(t_j), \quad \mathcal{N}(x,0)=\kappa_n \mathcal{I}_{x}(0).
\end{equation}
Slime dynamics are governed by a threshold-dependent production mechanism. Production commences at $t_0(x)$, defined as the first instant when $\mathcal{I}_x(t_j)\geq M_1$ for some $M_1\in\mathbb{N}_{+}$, with initial slime concentration set to $\mathcal{C}(x,t_0)=\kappa_c \mathcal{I}_x(t_0)$. Subsequently, slime is produced if either the occupancy exceeds $M_1$ or the local slime concentration surpasses the sensing threshold $M_2$. The complete evolution, integrating diffusion, decay, and production (adapted from Ref.~\refcite{Stevens00}), is
\begin{equation}\label{eq:blmden}
\mathcal{C}(x,t_{j+1})=\underbrace{(1-\lambda_{c})}_{\textit{Decay}}\Big(\underbrace{\big(1-D_c\big)\mathcal{C}(x,t_j)}_{\textit{Retention}}+\underbrace{\sum\nolimits_{y\in N_x}\frac{D_c}{4}\mathcal{C}(y,t_j)}_{\textit{Diffusion Inflow}}\Big) + \underbrace{\kappa_c\mathcal{I}_{x}(t_j)}_{\textit{Secretion}}.
\end{equation}
If neither production condition is satisfied, the same update rule applies with the production term set to zero (i.e., $\kappa_c \mathcal{I}_x(t_j) = 0$). During the update process of slime concentration $\mathcal{C}(\cdot,t_j)$ and myxobacteria density $\mathcal{N}(\cdot,t_j)$, the transition from the current time step $t_j$ to the next $t_{j+1}$ must follow the waiting-time distribution $\psi(t)$ specified in \eqref{eq:Dwaitingtime}, as prescribed by Assumption \textbf{A.3.} (see also Fig. \ref{Fig:example0} (b)).

In model \eqref{eq:blmden}, the term $\kappa_c\mathcal{I}_{x}(t_j)$ acts as a `source term', representing localized slime secretion by myxobacteria at a rate proportional to the local occupancy count $\mathcal{I}_x$. Spatial spreading of the chemical is governed by an isotropic diffusion scheme with coefficient $D_c$; the term $(1-D_c)\mathcal{C}(x,t_j)$ accounts for mass retention at the current site, while the summation $\sum\nolimits_{y\in N_x}\frac{D_c}{4}\mathcal{C}(y,t_j)$ captures the diffusive flux from the four nearest neighbors ($N_x$),  each contributing one quarter of its diffusing mass. The prefactor $(1-\lambda_c)$ models linear degradation or environmental evaporation of slime over each discrete time step.

\begin{figure}[!ht]
\centering
\hspace*{-28pt}
\begin{minipage}[c]{0.45\textwidth}
 \centering
 \centerline{\includegraphics[width=.75\textwidth]{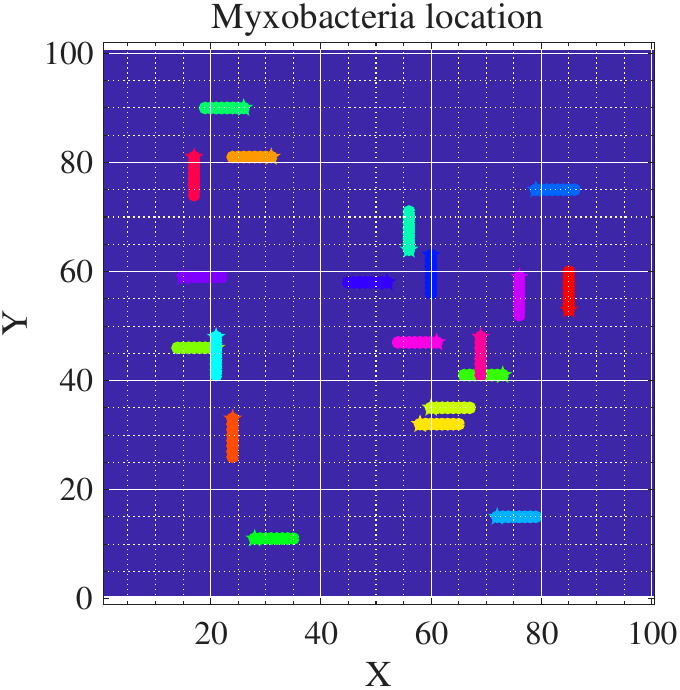}}
\end{minipage}
\begin{minipage}[c]{0.45\textwidth}
 \centering
 \centerline{\includegraphics[width=1.1\textwidth]{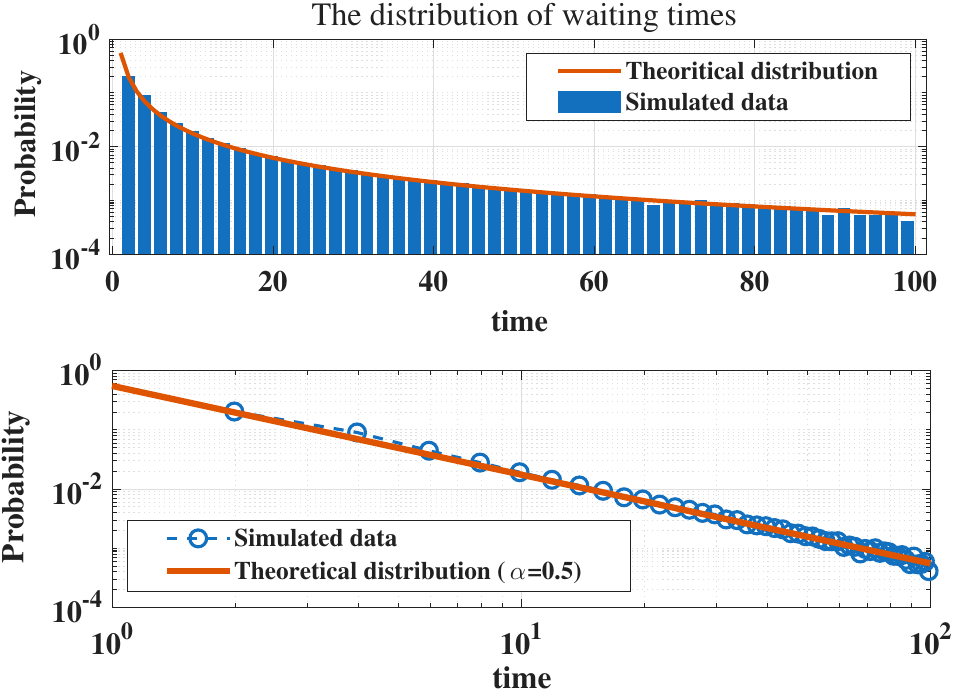}}
\end{minipage}
\caption{(Color online) Microscopic simulation setup and statistical validation of the heavy-tailed waiting time distribution. (\textit{Left}) Initial configuration: $20$ myxobacteria are randomly distributed across a $100\times100$ lattice. Each agent occupies $8$ contiguous cells. Movement is restricted to four cardinal directions: up, down, left, and right, but are not allowed to reverse their direction abruptly. For instance, if an individual is currently moving upward, its next move may only proceed to the left, right, or continuing upward by one grid cell. Overlapping between different myxobacteria is permitted to simulate high-density aggregation. (\textit{Right}) Comparative analysis of theoretical predictions and simulated waiting time distributions. The data are obtained from Fig. \ref{Fig:example2}, based on statistics collected from $1000$ agents over $1000$ discrete steps ($t_{min}=1$, $t_{max}=100$). The top-right inset presents a frequency histogram with a fitted curve, visually highlighting the heavy-tailed characteristic of the waiting time distribution. The bottom-right plot demonstrates the robust agreement between the theoretical model and simulated data under the exponent $\alpha=0.50$, substantiating the use of time-nonlocal operators to capture anomalous delay effects.}
\label{Fig:example1}
\end{figure}

To incorporate the observed directional persistence in myxobacterial motion, we introduce a directional weight factor $d(\cdot,\cdot,\cdot)$ that biases the bacterium's `head' toward grid points aligned with its current orientation in the next time step. Following  Ref.~\refcite{Stevens00}, we define
\begin{displaymath}
d\big(x_0,y,h_k(t_j)\big)=
\left\{
\begin{aligned}
&10, &&\quad \text{if}~y~\text{aligns with}~h_k(t_j);\\
&1,  &&\quad \text{otherwise}.
\end{aligned}
\right.
\end{displaymath}
This weight ensures that, in the absence of chemotactic cues (i.e., $\mathcal{N}=0$ and $\mathcal{C}=0$), the bacterium maintains its trajectory over a distance roughly comparable to its body length, thereby preventing non-physical, erratic reorientations.

Let $w_n, w_c\geq0$ denote the weight factors for myxobacteria density and slime concentration, respectively. The probability $P_{k,x_0}(x,t+1)$ for the head of the $k_{th}$ myxobacterium, initially located at $x_0$ at time $t_j$, moves to a neighboring site $x\in N_{x_0}$ with $x\notin \mathcal{B}_k(t_j)$ at time $t_{j+1}$ is define as\cite{Stevens00}
\begin{equation}\label{eq:Probbl}
P_{k,x_0}(x,t_{j+1})
=\frac{w_c\,\mathcal{C}(x,t_j)+\big(w_n\,\mathcal{N}(x,t_j)+\kappa_n\big)d\big(x_0,y,h_k(t_j)\big)}
{\sum_{y\in N_{x_0},y\notin\,\mathcal{B}_{k}(t_j)}\big(w_c\,\mathcal{C}(y,t_j)
+\big(w_n\,\mathcal{N}(y,t_j)+\kappa_n\big)d\big(x_0,y,h_k(t_j)\big)\big)},
\end{equation}
with $P_{k,x_0}(x,t_{j+1})=0$ for $x\in \mathcal{B}_{k}(t_j)$. This transition probability is shaped by three principal biological mechanisms:
\begin{itemize}
    \item \textbf{Chemotaxis ($w_c\mathcal{C}$)}: biases the agent toward higher chemoattractant (slime) concentration; the weight $w_c$ modulates the sensitivity to chemical gradients.
    \item \textbf{Contact Guidance ($w_n\mathcal{N} + \kappa_n$)}: encodes ``trail-following" behavior, where $w_n\mathcal{N}$ favors movement toward areas of higher historical bacterial density (i.e., existing slime trails), while the baseline constant $\kappa_n$ ensures mobility in previously unexplored terrain.
    \item \textbf{Directional persistence $d(\cdot,\cdot,\cdot)$}: the multiplication of the contact guidance term by $d(x_0, y, h_k(t_j))$ prioritizes alignment with the current heading, reflecting the mechanical difficulty that a rod-shaped bacterium faces when attempting sharp turns along a slime trail.
    \item \textbf{Normalization}: The denominator sums these weights over all neighboring sites not occupied by the bacterium's own body ($\mathcal{B}_k$), ensuring that $\sum_{x} P_{k,x_0}(x,t_{j+1})=1$.
\end{itemize}

Equation \eqref{eq:blmden} governs the spatio-temporal evolution of slime concentration by integrating local secretion, isotropic grid-based diffusion, and linear decay. The stochastic motion of myxobacteria is then captured by the transition probability in \eqref{eq:Probbl}. Together, these rules define a biased random walk whose movement direction is determined by a linearly weighted combination of three biological drivers: chemoattractant toward slime gradients, contact guidance along existing trails, and directional persistence along the current heading.

\begin{figure}[!ht]
\centering
\begin{minipage}[c]{0.24\textwidth}
 \centering
 \centerline{\includegraphics[width=0.95\textwidth]{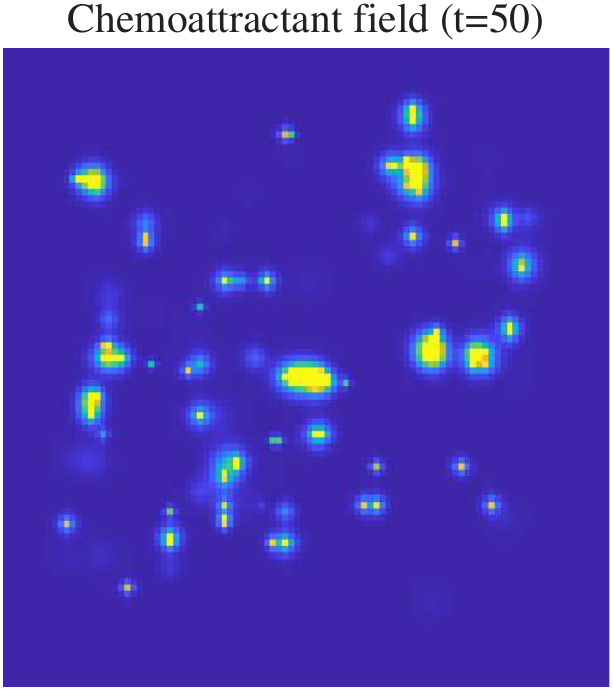}}
\end{minipage}
\begin{minipage}[c]{0.24\textwidth}
 \centering
 \centerline{\includegraphics[width=0.95\textwidth]{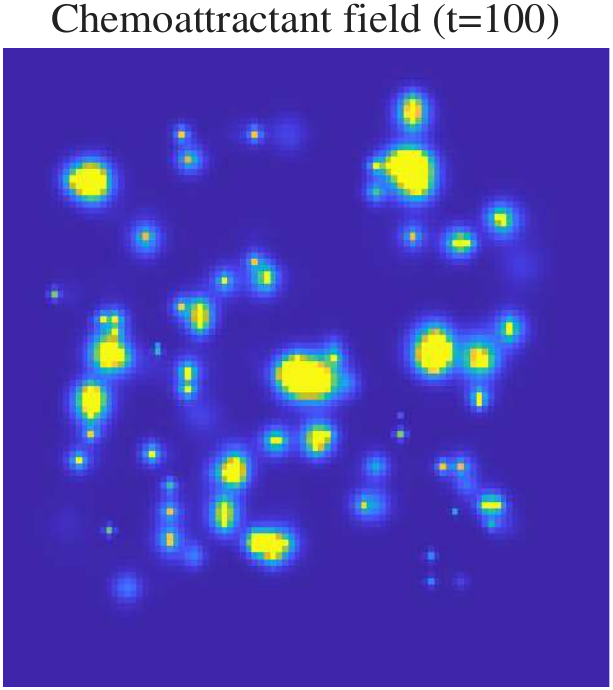}}
\end{minipage}
\begin{minipage}[c]{0.24\textwidth}
 \centering
 \centerline{\includegraphics[width=0.95\textwidth]{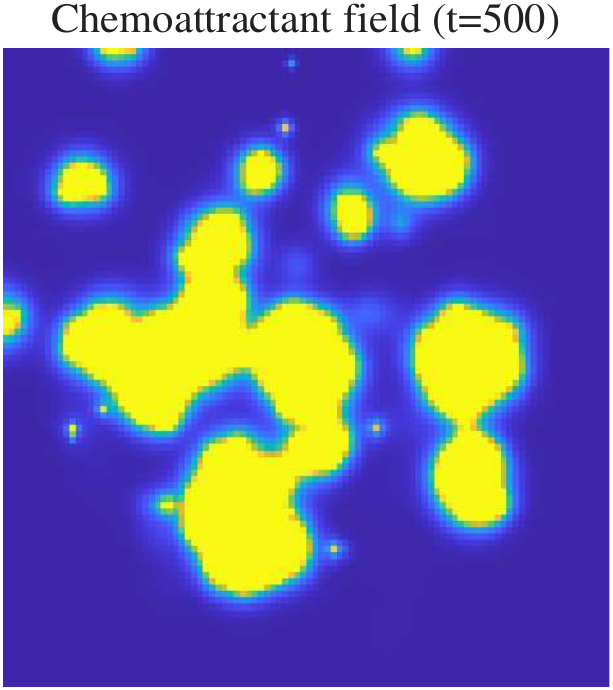}}
\end{minipage}
\begin{minipage}[c]{0.24\textwidth}
 \centering
 \centerline{\includegraphics[width=0.95\textwidth]{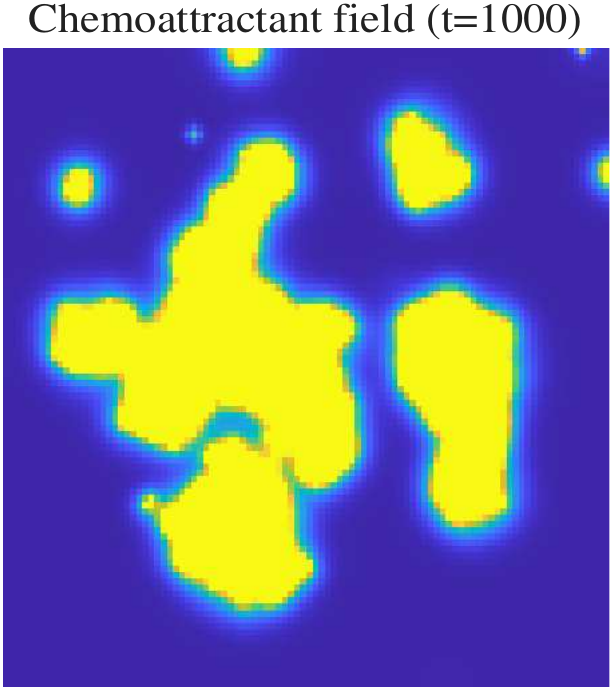}}
\end{minipage}
\caption{(Color online)  Spatiotemporal evolution of slime concentration $c$ and the emergence of fruiting body patterns.This simulation tracks $1000$ myxobacteria on a $100\times100$ lattice,  capturing the transition from stochastic gliding to collective aggregation.
The numerical results demonstrates that the gliding motility of the bacteria, along with the diffusion and degradation of slime, largely reproduces the dynamical behaviors observed under laboratory microscopy. Specifically, the myxobacteria exhibit directed gliding toward regions of higher slime concentration and aggregate there. Upon reaching a critical density threshold ($M_1=5$), these clusters secrete additional slime to facilitate social growth and structural reinforcement. The model effectively resolves the coarsening dynamics of the population, where adjacent spores merge into large-scale fruiting bodies while transient, smaller structures undergo apoptosis-like dissolution. Simulation parameters are set as: $\Lambda_n=0.10$, $\Lambda_c=0.02$,   $\kappa_n=0.20$, $\kappa_c=0.30$, $D_c=1.00$, $w_n=11.00$, $w_c=111.00$, with a fractional exponent $\alpha=0.50$ over $1000$ time steps.}
\label{Fig:example2}
\end{figure}

All the above simulations performed under the parameter regime $d < w_n \ll w_c$ are illustrated in Figs. \ref{Fig:example1} and \ref{Fig:example2}. These results demonstrate a robust qualitative agreement with biological observations documented in foundational cinematographic studies (e.g., \refcite{Kuhlwen71,Kuhlwen68,Reichenbach68}). Such alignment with both classical experimental recordings and recent scholarly reviews \refcite{Islam2015,Mauriello2010} corroborates our modeling framework and validates the hypotheses \textbf{A.1}–\textbf{A.3}.

\subsection{Macroscopic dynamics via stochastic subordination}
The microscopic stochastic simulations in the preceding section confirm the validity of our model assumptions (\textbf{A.1}–\textbf{A.3}) and corroborate the underlying biophysical mechanisms. Building on these physically plausible microscopic assumptions and the associated biological rationale, we now derive the macroscopic system \eqref{Pro:FKS} rigorously through a multiscale stochastic modeling approach. This derivation not only endows the model with a sound biophysical interpretation from the micro‑scale up, but also highlights its accuracy and indispensability in capturing realistic chemotactic dynamics.

Following the standard theory of stochastic differential equations (see, e.g., Ref. \refcite{Risken89}), the classical  Fokker-Planck equation is statistically equivalent to the Langevin dynamics $\dot{x}(t)=A+\sqrt{2\mathcal{D}}\dot{\mathcal{W}}(t)$, where $\mathcal{W}(t)$ is a standard Wiener process. Combining the biological mechanisms of myxobacteria described in Sec. \ref{sec:bioMechanism} with the lattice-based agent model introduced in Sec. \ref{sec:lattice}, we model the directed gliding motion toward aggregates via the subordinated Langevin equation
\begin{equation}\label{eq:Langevin}
\dot{x}(t) = A\big(x\big(\mathcal{S}(t)\big),\mathcal{S}(t)\big)\,\dot{\mathcal{S}}(t)
        + \sqrt{2\mathcal{D}}\,\dot{\mathcal{W}}_{\mathcal{S}(t)},
\end{equation}
where $\mathcal{S}(t):=\inf\big\{\bar{\tau}\geq0:T(\bar{\tau})>t\big\}$ is the inverse $\alpha$-stable subordinator (see e.g., Ref. \refcite{Applebaum09}). The drift term $A(x,t)$, which characterizes the directional bias of gliding, is defined by
\begin{equation}\label{eq:drift}
A(x,t):=\underbrace{\chi\,\delta_x^2/(2\tau^\alpha)}_{(\textit{coefficient})}
 \times \underbrace{v(x,t)^{-1}\cdot\partial_x v(x,t)}_{(\textit{effective force field})}.
\end{equation}
The quantity $\delta_x^2/(2\tau^\alpha)$ represents an effective mobility associated with the migration rate, with $\delta_x$ the fixed gliding step and $\tau$ the characteristic timescale in the waiting-time distribution \eqref{eq:Dwaitingtime}. The logarithmic gradient $\partial_x\ln v(x,t)=v(x,t)^{-1}\partial_x v(x,t)$ acts as an effective driving force generated by spatial variations in the concentration potential; this form follows from the asymptotic expansion of transition probabilities (see \eqref{eq:prpl-def} in Sec. \ref{Der:drift_coefficient} for details) and implies that bacteria preferentially move toward regions with higher potential. In the spatially homogeneous case $\partial_x v=0$, the drift vanishes and the dynamics reduce to symmetric subdiffusion.

The process $\dot{\mathcal{S}}(t)$ denotes increments in operational time\cite{Magdziarz07}, which maps discrete gliding events onto physical time $t$. Biologically, this captures the ``stop-and-go" motility of myxobacteria: each myxobacteria glides and then enters a trapped state for a random duration governed by a power-law distribution. During trapping intervals, $\dot{\mathcal{S}}(t)=0$, so both deterministic forcing and stochastic fluctuations are effectively suspended; dynamical updates occur only upon escape, when $\dot{\mathcal{S}}(t)>0$. Consequently, the particle experiences fluctuations solely during active displacement phases, leading to transport that is slower than classical Brownian diffusion.

The stochastic equation \eqref{eq:Langevin} is formally equivalent to the following system of coupled stochastic differential equations in the overdamped limit
\begin{equation}\label{Eq:Langevin_system}
\left\{\begin{aligned}
x(t)
&=y\big(\mathcal{S}(t)\big),\quad \\
\dot{y}(\bar{\tau})
&=A\big(y(\bar{\tau}),\bar{\tau}\big)
+\sqrt{2\mathcal{D}}\,\dot{\mathcal{W}}(\bar{\tau}).
\end{aligned}\right.
\end{equation}
The process $T(\bar{\tau})$ asts as a ``random clock'' capturing the heavy-tailed waiting times and is characterized by the Laplace exponent $\mathbb{E}[e^{-zT(\bar{\tau})}]=e^{-\bar{\tau}z^{\alpha}}$, $z>0$, see Refs. \refcite{Ma26,Ma23,Magdziarz07}. Here, $\bar{\tau}$ denotes the  operational time. Unlike the physical time variable $t$, the operational time evolves continuously during active motion, while the random clock $\mathcal{S}(t)$ introduces intermittent trapping events and thereby produces anomalous temporal scaling. The noise term $\mathcal{W}(\bar{\tau})$ is assumed to be Gaussian white noise, independent of the subordinator, with $\langle \mathcal{W}(\bar{\tau})\rangle=0$ and $\langle \mathcal{W}(\bar{\tau}_1)\mathcal{W}(\bar{\tau}_2)\rangle=\delta(\bar{\tau}_1-\bar{\tau}_2)$. In \eqref{Eq:Langevin_system}, inertial effects have been neglected. This overdamped approximation is justified because, in strongly dissipative media, the momentum relaxation time is much shorter than the characteristic waiting time associated with the power-law distribution \eqref{eq:Dwaitingtime}. Hence, over the timescales relevant to aggregation dynamics, the motion is effectively governed by force balance rather than inertia.

Model \eqref{Eq:Langevin_system} is thus constructed directly from the underlying biological mechanisms. On the operational timescale $T(\bar{\tau})$, the bacterium undergoes overdamped stochastic motion driven by the external concentration field governed by \eqref{eq:blmden}, while the mapping from operational time to physical time is determined by the heavy-tailed waiting-time distribution \eqref{eq:Dwaitingtime}. As a result, the model consistently captures both chemotactic gliding and anomalous diffusive transport in a unified framework.

Starting from either the drift structure \eqref{eq:drift} or the stochastic system \eqref{Eq:Langevin_system}, we can derive the macroscopic equation \eqref{Pro:KS1}. Let $\mathcal{G}(x,\bar{\tau})$ denote the probability density function (PDF) of the process $y(\bar{\tau})$. Then $\mathcal{G}$ satisfies the classical Fokker-Planck equation (see, e.g., Refs.~\refcite{Deng20,Henry2010,Magdziarz2008,Magdziarz07}),
\begin{equation}
\frac{\partial}{\partial\bar{\tau}}\mathcal{G}(x,\bar{\tau})
=-\frac{\partial}{\partial x}\big(A(x,\bar{\tau})\mathcal{G}(x,\bar{\tau})\big)
 +\mathcal{D}\frac{\partial^2}{\partial x^2}\mathcal{G}(x,\bar{\tau}), \quad \mathcal{G}(x,0)=n(x,0).
\end{equation}
Define the operator $\mathcal{L}_{x}:=-\frac{\partial}{\partial x}\big(A(x,\bar{\tau})\cdot\big)+\mathcal{D}\frac{\partial^2}{\partial x^2}$. Taking the Laplace transform with respect to the operation time $\bar{\tau}$ gives
\begin{equation}
s\,\widehat{\mathcal{G}}(x,s)-\mathcal{G}(x,0)
=\mathcal{L}_{x}\,\widehat{\mathcal{G}}(x,s),\quad\forall~ s\in\mathbb{C},
\end{equation}
and hence $\widehat{\mathcal{G}}(x,s)=(s-\mathcal{L}_{x})^{-1}\mathcal{G}(x,0)$. Next, let $\mathcal{H}(\bar{\tau},t)$ denote the PDF of the inverse $\alpha$-stable subordinator $S(t)$. Its Laplace transform with respect to $t$ satisfies (see (6)-(7) in Ref.~\refcite{Magdziarz07})
\begin{equation}
\mathfrak{L}_{t\rightarrow z}\big\{\mathcal{H}(\bar{\tau},t)\big\}
=-\partial_{\bar{\tau}}\big(e^{-\bar{\tau} z^{\alpha}}\big)z^{-1}=z^{\alpha-1}e^{-\bar{\tau}z^{\alpha}},\quad \forall~ z\in\mathbb{C}.
\end{equation}
Now let $n(x,t)$ denote the PDF of the process $y(\mathcal{S}(t))$ (equivalently, $x(t)$). By the total probability formula and the independence of $y(\bar{\tau})$ and $\mathcal{S}(t)$, we get that the PDF  $n(x,t)=\int_{0}^{\infty}\mathcal{G}(x,\bar{\tau})\mathcal{H}(\bar{\tau},t)\mathrm{d}\bar{\tau}$ (see Ref.~\refcite{Meerschaert13}). Taking the Laplace transform with respect to $t$ yields $\widehat{n}(x,z)=z^{\alpha-1}\widehat{\mathcal{G}}(x,z^\alpha)$ (cf. (9) in Ref.~\refcite{Magdziarz07}). Substituting $s=z^{\alpha}$ into the resolvent representation above gives \ref{Der:drift_coefficient}
\begin{equation}
 \widehat{\mathcal{G}}(x,z^{\alpha})=\big(z^{\alpha}-\mathcal{L}_{x}\big)^{-1}\mathcal{G}(x,0),
\end{equation}
which implies
\begin{equation}
z\,\widehat{n}(x,z)-n(x,0)=z^{1-\alpha}\,\mathcal{L}_{x}\widehat{n}(x,z).
\end{equation}
For $\alpha\in(0,1)$, applying the inverse Laplace transform together with the standard identities for Riemann-Liouville fractional operators, namely $\mathfrak{L}_{t\rightarrow z}\{_{0}D^{1-\alpha}_{t}f(t)\}=z^{1-\alpha}\widehat{f}(z)$ and $_{0}I^{1-\alpha}_{t}\big(\frac{\mathrm{d}}{\mathrm{d}t}n(x,t)\big)={_{0}^{C}\mathfrak{D}^{\alpha}_t}n(x,t)$, yields the fractional Fokker-Planck equation
\begin{equation}\label{eq:FFPE}
{_{0}^{C}\mathfrak{D}^{\alpha}_t}n(x,t)
=-\chi \mathcal{D}\frac{\partial}{\partial x}
  \left(\frac{n(x,t)}{v(x,t)}\frac{\partial}{\partial x}v(x,t)\right)
  +\mathcal{D}\frac{\partial^2n(x,t)}{\partial x^2}, \quad \forall~~ t>0.
\end{equation}
Here,
$${_{0}I^{1-\alpha}_{t}}u(t):=\frac{1}{\Gamma(\alpha)}\int_{0}^{t}(t-s)^{\alpha-1}u(s)\mathrm{d}s$$
and
$${_{0}D^{1-\alpha}_{t}f(t)}u(t):=\frac{1}{\Gamma(1-\alpha)}\frac{\mathrm{d}}{\mathrm{d}t}\int_{0}^{t}(t-s)^{-\alpha}u(s)\mathrm{d}s$$
denote the Riemiann-Liouville fractional integral and derivative, respectively (see, e.g., Ref. \refcite{Podlubny99}).

Finally, extending the single-particle dynamics \eqref{eq:Langevin} to a system of interacting particles leads to the macroscopic equation \eqref{Pro:KS1}; see, for instance, Ref.~\refcite{Stevens00b}). Specifically, the trajectory of the $i_{th}$ particle is governed by
\begin{equation}\label{eq:Langevin-ext}
\dot{x}_{i}(t)
=A\left(x_{i}(\mathcal{S}_{i}(t)),\mathcal{S}_{i}(t)\right)\,\dot{\mathcal{S}}_{i}(t)
 +\sqrt{2\mathcal{D}}\,\dot{\mathcal{W}}^{(i)}_{\mathcal{S}_{i}(t)},
\end{equation}
where $\mathcal{W}^{(i)}$ are mutually independent Wiener processes and $\{\mathcal{S}_i(t)\}_{i\ge1}$ are independent inverse $\alpha$-stable subordinators. Thus, each particle possesses its own random waiting clock, corresponding to an independent continuous-time random walk.  Since independent superposition preserves the one-particle statistical law, the collective particle density still satisfies the fractional Fokker–Planck equation \eqref{eq:FFPE}. Setting $v(x,t)=c(x,t)$, where $c(x,t)$ denotes the slime concentration (as described later), \eqref{eq:FFPE} directly reduces to \eqref{Pro:KS1}, which is the target macroscopic equation.
\begin{remark}
The quantity $\delta_x^2/(2\tau^\alpha)$, which has physical dimension [$\mathrm{length}^2/\mathrm{time}^{\alpha}$], serves as a generalized diffusion coefficient and characterizes the effective transport rate of particles in the medium. Here, $\delta_x^2$ denotes the characteristic jump intensity, while $\tau^\alpha$ captures the temporal memory effect induced by the heavy-tailed waiting-time distribution. The ratio $\delta_x^2/(2\tau^\alpha)$ thus quantifies the balance between particle mobility and trapping effects. In the classical diffusive regime $\alpha=1$, this expression reduces to the standard diffusion coefficient $\mathcal{D}=\delta_x^2/(2\tau)$, which has physical dimension $[\mathrm{length}^2/\mathrm{time}]$. Hence, $\delta_x^2/(2\tau^\alpha)$ can be regarded as a natural generalization of the classical diffusion coefficient to anomalous diffusion, linking the microscopic stochastic parameters $(\delta_x,\tau)$ to the macroscopic transport behavior characterized by $\mathcal{D}$.
\end{remark}

\begin{remark}
The field $v(x,t)$ primarily characterizes the influence of the chemoattractant concentration on bacterial chemotaxis, governing the directional bias of cell movement. Varying $v(x,t)$ provides considerable modeling flexibility and naturally leads to different macroscopic formulations. For instance, setting $v(x,t) = c(x,t)$ recovers the macroscopic equation derived in this work. Alternatively, choosing the exponential form $v(x,t)=\exp(-\beta c(x,t))$ (see, e.g., Ref.~\refcite{Langlands10}) regularizes the chemotactic sensitivity and yields a non-singular drift term of the form $-\mathcal{D}\chi\nabla\cdot\big(n(x,t)\nabla c(x,t)\big)$.
\end{remark}

Similarly, to describe slime diffusion in a multi‑particle setting, we consider the subordinated Langevin equation
\begin{equation}\label{eq:Langevin-c}
\dot{x}_{i}(t)=\sqrt{2\mathcal{D}}\,\dot{\mathcal{W}}^{(i)}_{\mathcal{S}_{i}(t)},
\end{equation}
which corresponds to purely diffusive motion without directional bias. Here, ${\mathcal{W}^{(i)}}{i\ge1}$ are mutually independent Wiener processes, and ${\mathcal{S}_{i}(t)}_{i\ge1}$ are independent inverse $\alpha$-stable subordinators that account for trapping effects and anomalous waiting times. Applying the same subordination argument as above yields the time-fractional Fokker--Planck equation
\begin{equation}
{_{0}^{C}\mathfrak{D}^{\alpha}_t}c(x,t)
=\frac{\delta_x^2}{2\tau^{\alpha}}\frac{\partial^2c(x,t)}{\partial x ^2},
\end{equation}
which describes the anomalous diffusion of the slime concentration field $c(x,t)$.

The remaining reaction terms in \eqref{Pro:KS2} cannot be derived directly from the stochastic dynamics \eqref{eq:Langevin-c}, as they originate from biological processes rather than random motion. Specifically, following Refs.~\refcite{Ma26,Stevens00}, the term $-\gamma c(x,t)$ models the natural degradation of slime at rate $\gamma$, while $+n(x,t)$ represents slime production induced by myxobacterial aggregation and activity. Incorporating these biologically motivated mechanisms into the fractional diffusion equation yields the macroscopic model \eqref{Pro:KS2}.

\section{Solution Theory}
\label{sec:ST}
\subsection{Preliminaries}
Let $\Omega\subset\mathbb{R}^{d}$ (with $d\geq2$) denote an open bounded domain with smooth boundary $\partial\Omega$. We begin by briefly recalling several function spaces used throughout this work.

For any $\kappa\in\mathbb{N}_{+}$ and $1\leq p\leq\infty$, let $L^{p}(\Omega)$ and $W^{\kappa,p}(\Omega)$ denote the usual Lebesgue and Sobolev spaces, respectively, as outlined in Refs. \refcite{Adams03,Grisvard11,Jin21book}. When $0<\kappa<1$ and $1\leq p<\infty$, the Sobolev space $W^{\kappa,p}(\Omega)$ is referred to as fractional Sobolev space, which is defined as
\begin{displaymath}
W^{\kappa,p}(\Omega)
=\left\{\omega\in L^{p}(\Omega):\big|\omega\big|_{W^{\kappa,p}}
:=\left(\int_{\Omega}\int_{\Omega}
\frac{|\omega(\mathbf{x})-\omega(\mathbf{y})|^p}{|\mathbf{x}-\mathbf{y}|^{d+p\kappa}}
\mathrm{d}\mathbf{x}\,\mathrm{d}\mathbf{y}\right)^{1/p}<\infty\right\},
\end{displaymath}
equipped with the norm of $\|\omega(\mathbf{x},t)\|_{W^{\kappa,p}}:=(\|\omega(\mathbf{x},t)\|^{p}_{L^{p}}+|\omega(\mathbf{x},t)|^{p}_{W^{\kappa,p}})^{1/p}$. Under the smoothness assumption on $\partial\Omega$, fractional Sobolev spaces may also be characterized via the $K$-method of interpolation. In particular, for $0<\kappa<1$ and $p=2$, the Hilbertian Sobolev spaces are denoted by $H^{\kappa}(\Omega):=[L^{2}(\Omega),H^{1}(\Omega)]_{\kappa}$ (see, e.g., Ref. \refcite{Jin21book}), with equivalent norms.

Let $\mathcal{A}=-\Delta$ be the Neumann Laplacian on $L^p(\Omega)$ with domain (e.g., Ref.~\refcite{Taira16})
\begin{displaymath}
D(\mathcal{A}):=\left\{\omega(\mathbf{x})\in W^{2,p}(\Omega)\,:\,
\frac{\partial \omega(\mathbf{x})}{\partial\nu}=0\ \text{on }\partial\Omega\right\}.
\end{displaymath}
where $\nu$ denotes the outward unit normal vector on $\partial\Omega$. It is well-known that for a bounded domain $\Omega$ with smooth boundary, $\mathcal{A}$ is a closed sectorial operator on $L^p(\Omega)$ and generates a bounded analytic semigroup $\{e^{-t\mathcal{A}}\}_{t \ge 0}$ on $L^p(\Omega)$. The associated eigenvalue problem reads
\begin{equation}
\left\{\begin{aligned}
\mathcal{A} \varphi_j &= \lambda_j \varphi_j,             && \text{in}~~\Omega, \\
\displaystyle\frac{\partial \varphi_j}{\partial\nu} &= 0, && \text{on}~~\partial\Omega,
\end{aligned}\right.
\qquad j = 1,2,\dots
\end{equation}
The sequence of eigenvalues satisfies  $0=\lambda_1<\lambda_2\le\lambda_3\le\cdots$, $\lambda_j\to\infty$ ($j\to\infty$), and the corresponding eigenfunctions $\{\varphi_j\}_{j=1}^\infty$ form a complete orthonormal basis of $L^2(\Omega)$, i.e., $\|\varphi_j\|_{L^2}=1$, $(\varphi_i,\varphi_j)_{\Omega}=\delta_{ij}$ with $(\cdot,\cdot)_{\Omega}$ denoting the $L^{2}$ inner product defined on $\Omega$. Accordingly, for $s\geq0$, the spectral Sobolev space associated with the Neumann Laplacian  defined by
\begin{displaymath}
H_{\mathcal{A}}^{s}(\Omega)=\mathcal{D}\big((I+\mathcal{A})^{s/2}\big)
:=\Bigg\{\omega\in L^{2}(\Omega)\;:\;\sum_{j=1}^{\infty}(1+\lambda_{j})^{s}\Big(\int_{\Omega}\omega\varphi_j\mathrm{d}\mathbf{x}\Big)^2<\infty\Bigg\}.
\end{displaymath}
The norm is naturally given by
$\|\omega\|^2_{H_{\mathcal{A}}^{s}}=\sum_{j=1}^{\infty}(1+\lambda_{j})^{s}|(\omega,\varphi_j)_{\Omega}|^2$.

\begin{remark}\label{rmk:space}
Since $\Omega\subset\mathbb{R}^d$ ($d\geq2$) is a bounded domain with smooth boundary, the spectral spaces
$H_{\mathcal{A}}^{s}(\Omega)$ are equivalent to the classical Sobolev spaces $H^{s}(\Omega)$ (see, e.g., Refs.~\refcite{Amann19,Burenkov02,Fujiwara67}). Consequently, the standard Sobolev and Morrey embeddings remain valid for $H_{\mathcal{A}}^{s}(\Omega)$.
More precisely, the following continuous embeddings hold.
\begin{description}
\item[$(i)$] If $0\le s<d/2$ and $2\le p\le\frac{2d}{d-2s}$, then $H_{\mathcal{A}}^{s}(\Omega)\hookrightarrow L^p(\Omega)$. If $s=d/2$, then $H_{\mathcal{A}}^{s}(\Omega)\hookrightarrow L^p(\Omega)$, $2\le p<\infty$.
\item[$(ii)$]If $s>d/2$, then $H_{\mathcal{A}}^{s}(\Omega)\hookrightarrow C^{0,\mu}(\overline\Omega)$ with $\mu\in(0,s-d/2)$.
\item[$(iii)$] If $-d/2<s\le0$ and $p\ge\frac{2d}{d-2s}$, then $L^p(\Omega)\hookrightarrow H_{\mathcal{A}}^{s}(\Omega)$.
\item[$(iv)$] For all $p>d$, the embedding $W^{1,p}(\Omega)\hookrightarrow C^{0,1-d/p}(\overline\Omega)\hookrightarrow L^\infty(\Omega)$ is valid.
\item[$(v)$] For every $1<p<\infty$, the shifted square-root norm of the Neumann Laplacian is equivalent to the $W^{1,p}$-norm; more precisely, $D\big((I+\mathcal A)^{1/2}\big)=W^{1,p}(\Omega)$ with equivalent norms, and $\|u\|_{W^{1,p}}\simeq\|(I+\mathcal A)^{1/2}u\|_{L^p}$. In particular, $\|u\|_{W^{1,p}}\le C\left(\|u\|_{L^p}+\|\mathcal A^{1/2}u\|_{L^p}\right)$.
\end{description}
The embeddings in $(i)$--$(iv)$ follow from the classical Sobolev embedding theorem together with the equivalence between the spectral and classical Sobolev scales on smooth bounded domains; see, for instance, Refs.~\refcite{Adams03,Burenkov02,Grisvard11,Taira16}. Furthermore, the assertion in $(v)$ is based on the $L^p$ square-root estimate for the Neumann Laplacian (see, e.g., Ref.~\refcite{Auscher01}) and on the $L^p$-boundedness of the associated Neumann Riesz transform (see, e.g., Ref.~\refcite{Jiang24,Mendez01}). More precisely, for every $1<p<\infty$, $\|\mathcal A_p^{1/2}u\|_{L^p}\simeq\|\nabla u\|_{L^p}$. Here the estimate $\|\mathcal A_p^{1/2}u\|_{L^p}\le C\|\nabla u\|_{L^p}$ is the $L^p$ square-root estimate, while the converse estimate $\|\nabla u\|_{L^p}\le C\|\mathcal A_p^{1/2}u\|_{L^p}$ is the boundedness of the Neumann Riesz transform $\nabla\mathcal A_p^{-1/2}$, understood through the functional calculus of the Neumann Laplacian. Consequently, $\|u\|_{W^{1,p}}\simeq\|u\|_{L^p}+\|\mathcal A_p^{1/2}u\|_{L^p}\simeq\|(I+\mathcal A_p)^{1/2}u\|_{L^p}$.
\end{remark}

Additionally, we present the following lemmas to support subsequent analysis and validation efforts.
\begin{lemma}[Comparison principle\cite{Al-Refai22}]\label{Lem:CPp}
Let $\Omega\subset\mathbb{R}^d$ $(d\geq2)$ be a bounded domain with smooth boundary, and let
$\mathcal{A}=-\Delta$. Suppose that $\omega(\mathbf{x},t)\in C([0,T];H^2(\Omega))$ satisfies
\begin{equation}
\left\{\begin{aligned}
{_{0}^{C}\mathfrak D_t^\alpha}\omega+\mathcal{A}\omega+\gamma\omega&\ge0,
&& \text{in}~~\Omega\times(0,T],\\
\partial_\nu\omega&=0,
&& \text{on}~~\partial\Omega\times(0,T],\\
\omega(\mathbf{x},0)&\ge0,
&& \text{in}~~\Omega,
\end{aligned}\right.
\end{equation}
where $\gamma\ge0$ is a constant. Then $\omega\geq0$ holds in $\Omega\times[0,T]$.
\end{lemma}

\begin{lemma}[Generalized Gr\"{o}nwall inequality\cite{Alikhanov10}]\label{Lem:Gronwall}
Let $y(t)$ be a non-negative, absolutely continuous function satisfying the fractional differential inequality
\begin{equation*}
{_{0}^{C}\mathfrak D_t^\alpha}y(t)\leq C_1\,y(t)+C_2(t),\quad 0<\alpha\leq 1,
\end{equation*}
for almost all $t$ in $[0,T]$, where $C_1>0$ and $C_2(t)$ is an integrable nonnegative function on $[0,T]$. Then
\begin{equation}
y(t)\leq y(0)E_{\alpha}(C_1\,t^\alpha)+\Gamma(\alpha)\,E_{\alpha,\alpha}(C_1\,t^\alpha){_{0}I_t^\alpha}\big(C_2(t)\big),
\end{equation}
where $E_{\alpha}(z)$ and $E_{\alpha,\beta}(z)$ are the one- and two-parameter Mittag-Leffler functions defined by $E_{\alpha}(z):=\sum_{n=0}^{\infty}z^{n}/\Gamma(\alpha n+1)$ and $E_{\alpha,\beta}(z):=\sum_{n=0}^{\infty}z^{n}/\Gamma(\alpha n+\beta)$, respectively. ${_{0}I^{\alpha}_{t}}u(t):=\frac{1}{\Gamma(\alpha)}\int_{0}^{t}(t-s)^{1-\alpha}u(s)\mathrm{d}s$ denotes the Riemann-Liouville fractional integral.
\end{lemma}

\subsection{Main results}
Building upon the functional spaces and notations established earlier, we now present the primary well-posedness results for problem \eqref{Pro:FKS}. Detailed proofs of these theorems will be provided in the subsequent sections.
To this end, we assume that the initial data $n_0(\mathbf{x})$ and $c_0(\mathbf{x})$ satisfy the following assumptions
\begin{equation}\label{Eq:Assuption}
\left\{
\begin{aligned}
&n_0(\mathbf{x})\in C^{0}(\overline{\Omega})~\text{with}~n_0\geq0~\text{in}~\Omega~\text{and}~{n_0 \not\equiv0};\\
&c_0(\mathbf{x})\in W^{1,\infty}(\Omega)~\text{such that}~c_0>0~\text{in}~\overline{\Omega}.
\end{aligned}\right.
\end{equation}
These regularity and positivity assumptions are natural from a modeling perspective and essential for establishing the well-posedness of the system \eqref{Pro:FKS}. The continuity of $n_0(\mathbf{x})$ and $W^{1,\infty}(\Omega)$ regularity of $c_0(\mathbf{x})$ ensure that the initial configurations are physically meaningful, while the non-negativity and positivity conditions reflect the biological interpretations of myxobacteria density and slime concentration, respectively.

Let $C(\overline{\Omega}\times[0,\infty))$ be the space of continuous functions on $\overline{\Omega}\times[0,\infty)$, and let $C([0,\infty);W^{1,p}(\Omega))$ be the space of continuous functions from $[0,\infty)$ into $W^{1,p}(\Omega)$, with $p>d$. Under the assumptions given in \eqref{Eq:Assuption}, the primary results on the solution theory for problem \eqref{Pro:FKS} are stated in the following theorems.
\begin{theorem}\label{Thm:golablel}
Let $\Omega \subset \mathbb{R}^{d}$ ($d \geq 2$) be a bounded domain with smooth boundary. Assume that the initial data $n_0(\mathbf{x})$ and $c_0(\mathbf{x})$ satisfy \eqref{Eq:Assuption}. Then, for all $t > 0$, the solution component $n(\mathbf{x},t)$ of problem \eqref{Pro:FKS} satisfies the mass conservation law,
\begin{equation}
\int_{\Omega} n(\mathbf{x},t) \, \mathrm{d}\mathbf{x} = \int_{\Omega} n_0(\mathbf{x}) \, \mathrm{d}\mathbf{x},
\end{equation}
while the component $c(\mathbf{x},t)$ satisfies the following estimate,
\begin{equation}
\int_{\Omega} c(\mathbf{x},t)\,\mathrm{d}\mathbf{x}
\leq \max \left\{ \frac{t^{\alpha}}{\Gamma(1+\alpha)} \int_{\Omega} n_{0}(\mathbf{x}) \, \mathrm{d}\mathbf{x},\,\int_{\Omega} c_{0}(\mathbf{x})\, \mathrm{d}\mathbf{x} \right\}.
\end{equation}
\end{theorem}

To facilitate further analysis of system \eqref{Pro:FKS}, we introduce a novel fractional Lyapunov functional of the form (detailed in Sec.~\ref{Sec:sssCLS}) as follows
\begin{equation}
\mathcal{E}[(n,c)](t):
=\int_{0}^{t}(t-s)^{\alpha-1}\mathcal{F}(s)\mathrm{d}s,\quad t>0,
\end{equation}
where the integrand $\mathcal{F}$ is a jointly convex functional given by
\begin{equation}
\mathcal{F}(t):=\mathcal{F}[(n,c)](t)
=\int_{\Omega}n\log\left(\frac{n}{\bar{n}}\right)\mathrm{d}\mathbf{x}
+\theta\int_{\Omega}\frac{|\nabla c|^2}{c}\mathrm{d}\mathbf{x}.
\end{equation}
Here, $\bar{n}:=\frac{1}{|\Omega|}\int_{\Omega}n\mathrm{d}\mathbf{x}$ denotes the spatial average of $n$, and $\theta$ is specified positive constant. We show that $\mathcal{F}(t)\leq\mathcal{F}(0)$ and $\mathcal{E}[(n,c)](t)<\infty$ for all $t\in[0,\infty)$ (see Lemma \ref{Lem:FracLyapunov}). Crucially, this uniform upper bound provides a key priori estimate that prevents finite‑time blow‑up and underpins the global existence of solutions. The resulting regularity and stability properties are stated in Theorem \ref{Thm:well-posed}.
\begin{theorem}\label{Thm:well-posed}
Let $\Omega \subset \mathbb{R}^{d}$ ($d=2,3$) be a bounded domain with smooth boundary. Suppose $0<\chi<1/2$, and assume the initial data $n_0(\mathbf{x})$ and $c_0(\mathbf{x})$ satisfy \eqref{Eq:Assuption}. Then, for $p>d$, problem \eqref{Pro:FKS} admits a unique global mild solution $(n(\mathbf{x},t), c(\mathbf{x},t))$ possessing the following regularity properties:
\begin{equation}\label{Eq:regularity}
\left\{
\begin{aligned}
&n \in C\big([0,\infty); L^{\infty}(\Omega)\big) \cap C^{0,\alpha/2}\big((0,\infty); W^{1,p}(\Omega)\big) \cap C\big((0,\infty); D(\mathcal{A})\big),\\
&c \in C\big([0,\infty); W^{1,p}(\Omega)\big) \cap C^{0,\alpha/2}\big((0,\infty); W^{2,p}(\Omega)\big) \cap C\big((0,\infty); D(\mathcal{A})\big).
\end{aligned}\right.
\end{equation}
Furthermore, the solution satisfies $n(\mathbf{x},t)\geq0$ and $c(\mathbf{x},t)>0$ in $\Omega \times (0,\infty)$, along with the uniform bound
\begin{equation}\label{eq:bounddd}
\sup_{t>0}\left(\big\|n(\cdot,t)\big\|_{L^{\infty}}+\big\|c(\cdot,t)\big\|_{W^{1,p}}\right) \leq M.
\end{equation}
\end{theorem}

The proofs of Theorems \ref{Thm:golablel} and \ref{Thm:well-posed} are carried out in the subsequent sections via a series of lemmas. Our approach integrates novel energy estimates, refined regularity analysis, and several new PDE techniques to establish a robust framework capable of handling the intrinsic difficulties of System \eqref{Pro:FKS}.

Before turning to the proofs, we fix the convention that, throughout the estimates below, $C$ denotes a generic positive constant independent of $t$ and $T$, whose value may vary from line to line.

\section{Fundamental Lemmas for the Solution Theory}
\label{sec:fLST}
Before we get into the main proofs, let's lay down some basic tools. In this section, we introduce a few auxiliary functions and some handy lemmas will support our construction of the solution theory.

Let's begin with a useful special function. For $\kappa>-1$ and $\lambda\in\mathbb{C}$, the Wright function $W(\kappa,\lambda;z)$ is defined by the series (see Ref.~\refcite[Sec.~1.11]{Kilbas06})
\begin{equation}
 W(\kappa,\lambda;z):=\sum_{j=0}^{\infty}\frac{z^{j}}{j!\Gamma(\kappa j+\lambda)}, \quad \forall~  z\in\mathbb{C},
\end{equation}
where $\Gamma(\cdot)$ denotes the Euler--Gamma function. When $\kappa>-1$, this series converges for all $z$, so $W(\kappa,\lambda;z)$ is an entire function. A particularly important special case is when $\kappa=-\alpha$ and $\lambda=1-\alpha$. Then the Wright function becomes the Mainardi function $M_{\alpha}(z)$, which satisfies the neat integral identity\cite{Costa23,Jin21book,Ma25b,Ma25}
\begin{equation}\label{eq:identity}
\int_{0}^{\infty}t^{\gamma}M_{\alpha}(t)\mathrm{d}t=\frac{\Gamma(1+\gamma)}{\Gamma(1+\alpha\gamma)}, \quad \gamma>-1.
\end{equation}

Now we set up the operators needed for our analysis. Recall that $\mathcal{A}:=-\Delta$ be the Neumann Laplacian on $L^p(\Omega)$, and $\mathcal{A}_\gamma:=\mathcal{A}+\gamma I$ with $\gamma>0$. Then, $-\mathcal{A}$ generates the usual Neumann heat semigroup $e^{-t\mathcal{A}}=e^{t\Delta}$, while $-\mathcal{A}_{\gamma}$ generates its exponentially damped version $e^{-t\mathcal{A}_{\gamma}}=e^{t(\Delta-\gamma)}=e^{-\gamma t}e^{t\Delta}$ for $t>0$. Both semigroups are bounded on $L^p(\Omega)$ and enjoy the standard heat-kernel smoothing estimates from $L^p(\Omega)$ to $L^p(\Omega)$ for suitable exponents $1\leq q\leq p\leq\infty$ (see, e.g., Refs.~\refcite{Ma23,Winkler2010}).

To represent mild solutions of System~\eqref{Pro:FKS}, we require the Mittag-Leffler operator families associated with $\mathcal{A}$ and $\mathcal{A}_\gamma$. Using the scalar functions $E_{\alpha,\beta}$ (with $E_\alpha:=E_{\alpha,1}$) recalled from the previous section, we now define the families $\{E_{\alpha}(-t^{\alpha}\mathcal{A})\}_{t\geq0}$, $\{E_{\alpha,\alpha}(-t^{\alpha}\mathcal{A})\}_{t\geq0}$, along with their $\mathcal{A}_\gamma$-analogues $\mathcal{A}_\gamma$-analogues $\{E_{\alpha}(-t^{\alpha}\mathcal{A}_\gamma)\}_{t\geq0}$, $\{E_{\alpha,\alpha}(-t^{\alpha}\mathcal{A}_\gamma)\}_{t\geq0}$.

A useful bridge between these time-nonlocal operator families and the standard heat semigroup is provided by the Mainardi function $M_{\alpha}$, via the integral identities
\begin{equation}\label{eq:EDelta}
\left\{\begin{aligned}
&E_{\alpha}(-t^{\alpha}\mathcal{A})
   \coloneqq \int_{0}^{\infty} M_{\alpha}(s) e^{s t^{\alpha} \Delta} \mathrm{d}s,
   &&E_{\alpha}\bigl(-t^{\alpha}\mathcal{A}_{\gamma}\bigr)
   \coloneqq \int_{0}^{\infty} M_{\alpha}(s) e^{s t^{\alpha} (\Delta - \gamma)} \mathrm{d}s, \\[4pt]
&E_{\alpha,\alpha}(-t^{\alpha}\mathcal{A})
   \coloneqq \int_{0}^{\infty} \alpha s M_{\alpha}(s) e^{s t^{\alpha} \Delta} \mathrm{d}s,
   &&E_{\alpha,\alpha}\bigl(-t^{\alpha}\mathcal{A}_{\gamma}\bigr)
   \coloneqq \int_{0}^{\infty} \alpha s M_{\alpha}(s) e^{s t^{\alpha} (\Delta - \gamma)} \mathrm{d}s.
\end{aligned}\right.
\end{equation}
These representations will be instrumental in the subsequent  derivation of smoothing estimates and time-regularity properties. We collect their main properties in the following lemmas.

\begin{lemma}[Continuity\cite{Costa23,Ma25}]\label{lem:continuous}
For each $1\leq q<\infty$, the Mittag-Leffler operator families associated with $\mathcal{A}$ and $\mathcal{A}_\gamma$, namely $\{E_{\alpha}(-t^{\alpha}\mathcal{A})\}_{t\geq0}$, $\{E_{\alpha}(-t^{\alpha}\mathcal{A}_{\gamma})\}_{t\geq0}$, $\{E_{\alpha,\alpha}(-t^{\alpha}\mathcal{A})\}_{t\geq0}$, and $\{E_{\alpha,\alpha}(-t^{\alpha}\mathcal{A}_{\gamma})\}_{t\geq0}$ are strongly continuous in $t>0$ on $L^{q}(\Omega)$.
\end{lemma}

Utilizing the $L^p$-$L^q$ estimates for the Neumann heat semigroup from  Refs.~\refcite{Ma25,Winkler2010}, we derive the following estimates for the Mittag-Leffler operators.

\begin{lemma}[$L^{p}$-$L^{q}$ estimates]\label{lem:MLbounds}
Let $\mathcal{A}=-\Delta$ be the Neumann Laplacian on $L^p(\Omega)$ with homogeneous Neumann boundary conditions, and set $\mathcal{A}_\gamma:=\mathcal{A}+\gamma I$ with $\gamma>0$. The operators $E_{\alpha}(-t^{\alpha}\mathcal{A})$, $E_{\alpha,\alpha}(-t^{\alpha}\mathcal{A})$, and their ${A}_{\gamma}$ counterparts are defined as in \eqref{eq:EDelta}. For every $t>0$ and $1\leq q\leq p\leq\infty$, there exist positive constants $\{c_j\}_{j\in \mathbb{N}_{+}}$, dependent only on $\Omega$, $\alpha$, $d$, $q$ and $p$, such that the following estimates hold:
\begin{align}
&\big\|E_{\alpha}(-t^{\alpha}\mathcal{A})\big\|_{L^{q}\rightarrow L^{p}}
          \leq c_1\left(1+t^{-\frac{\alpha d}{2}(\frac{1}{q}-\frac{1}{p})}\right),
      &&\frac{d}{2}\Big(\frac{1}{q}-\frac{1}{p}\Big)<1; \label{eq:Ea}                  \\
&\big\|E_{\alpha,\alpha}(-t^{\alpha}\mathcal{A})\nabla\cdot\big\|_{L^{q}
          \rightarrow L^{p}}\leq c_2\left(1+t^{-\frac{\alpha}{2}-\frac{\alpha d}{2}(\frac{1}{q}-\frac{1}{p})}\right),
      &&\frac{d}{2}\Big(\frac{1}{q}-\frac{1}{p}\Big)<\frac{3}{2};\label{eq:Eb}  \\
&\big\|\mathcal{A}^{1/2} E_{\alpha,\alpha}(-t^{\alpha}\mathcal{A})\big\|_{L^{q}
          \rightarrow L^{p}}\leq c_3\left(1+t^{-\frac{\alpha}{2}-\frac{\alpha d}{2}(\frac{1}{q}-\frac{1}{p})}\right),
      &&\frac{d}{2}\Big(\frac{1}{q}-\frac{1}{p}\Big)<\frac{3}{2};\label{eq:Ec}  \\
&\big\|E_{\alpha,\alpha}(-t^{\alpha}\mathcal{A}_{\gamma})\big\|_{L^{q}
          \rightarrow L^{p}}\leq c_4 \left(1+t^{-\frac{\alpha d}{2}(\frac{1}{q}-\frac{1}{p})}\right),
      &&\frac{d}{2}\Big(\frac{1}{q}-\frac{1}{p}\Big)<1; \label{eq:Ed}                  \\
&\big\|E_{\alpha,\alpha}(-t^{\alpha}\mathcal{A}_{\gamma})\nabla\cdot\big\|_{L^{q}
          \rightarrow L^{p}}\leq c_5\left(1+t^{-\frac{\alpha}{2}-\frac{\alpha d}{2}(\frac{1}{q}-\frac{1}{p})}\right),
      &&\frac{d}{2}\Big(\frac{1}{q}-\frac{1}{p}\Big)<\frac{3}{2};\label{eq:Ee}  \\
&\big\|\mathcal{A}^{1/2} E_{\alpha,\alpha}(-t^{\alpha}\mathcal{A}_{\gamma})\big\|_{L^{q}
          \rightarrow L^{p}}\leq c_6\left(1+t^{-\frac{\alpha}{2}-\frac{\alpha d}{2}(\frac{1}{q}-\frac{1}{p})}\right),
      &&\frac{d}{2}\Big(\frac{1}{q}-\frac{1}{p}\Big)<\frac{3}{2}.\label{eq:Ef}
\end{align}
\end{lemma}
\begin{proof}
Fix $1\leq q\leq\infty$ and $\omega(\mathbf{x},t)\in L^{q}(\Omega)$. To prove \eqref{eq:Ea}, we use the identity \eqref{eq:identity} together with the heat semigroup estimate under homogenous Neumann boundary conditions (Lemma 3.1 (17) in Ref. \refcite{Ma25}). This gives
\begin{displaymath}
\begin{aligned}
\big\|E_{\alpha}(-t^{\alpha}\mathcal{A})\omega(\mathbf{x},t)\big\|_{L^{p}}
&\leq\int_{0}^{\infty}M_{\alpha}(s)\left\|
    e^{st^{\alpha}\Delta}\omega(\mathbf{x},t)\right\|_{L^{p}}\mathrm{d}s \\
&\leq C\int_{0}^{\infty}M_{\alpha}(s)\Big(1+(st^{\alpha})^{-\frac{d}{2}(\frac{1}{q}-\frac{1}{p})}\Big)
    \big\|\omega(\mathbf{x},t)\big\|_{L^{q}}\mathrm{d}s                 \\
&\leq C\left(1+t^{-\frac{\alpha d}{2}(\frac{1}{q}-\frac{1}{p})}
    \frac{\Gamma\big(1-\frac{d}{2}(\frac{1}{q}-\frac{1}{p})\big)}
    {\Gamma\big(1+\alpha\big(-\frac{d}{2}(\frac{1}{q}-\frac{1}{p})\big)\big)}\right)\big\|\omega(\mathbf{x},t)\big\|_{L^{q}}\\
&\leq C\Big(1+t^{-\frac{\alpha d}{2}(\frac{1}{q}-\frac{1}{p})}\Big)\big\|\omega(\mathbf{x},t)\big\|_{L^{q}},\quad t>0.
\end{aligned}
\end{displaymath}
The condition $\frac{d}{2}(\frac{1}{q}-\frac{1}{p})<1$ is required for the convergence of the singular integral and implies both $1-\frac{d}{2}(\frac{1}{q}-\frac{1}{p})>0$ and $1+\alpha(1-\frac{d}{2}(\frac{1}{q}-\frac{1}{p}))>0$ hold, thereby it ensures that the arguments of the Gamma functions are well defined. The estimate \eqref{eq:Ed} follows by the same argument.

Next, we establish \eqref{eq:Eb}. Using identity \eqref{eq:identity} and estimate (19) from Lemma 3.1 in Ref. \refcite{Ma25}, we derive
\begin{displaymath}
\begin{aligned}
&\big\|E_{\alpha,\alpha}(-t^{\alpha}\mathcal{A})\nabla\cdot\omega(\mathbf{x},t)\big\|_{L^{p}}
    \leq\int_{0}^{\infty}\alpha s M_{\alpha}(s)
         \left\|e^{st^{\alpha}\Delta}\nabla\cdot\omega(\mathbf{x},t)\right\|_{L^{p}}\mathrm{d}s \\
    &\leq C\int_{0}^{\infty}\alpha s M_{\alpha}(s)\left(1+(st^{\alpha})^{-\frac{1}{2}-\frac{d}{2}(\frac{1}{q}-\frac{1}{p})}\right)\big\|\omega(\mathbf{x},t)\big\|_{L^{q}}\mathrm{d}s           \\
    &\leq C\,\left(\frac{\Gamma(2)}{\Gamma(1+\alpha)}+t^{-\frac{\alpha}{2}-\frac{\alpha d}{2}(\frac{1}{q}-\frac{1}{p})}\frac{\Gamma\big(\frac{3}{2}-\frac{d}{2}(\frac{1}{q}-\frac{1}{p})\big)}
        {\Gamma\big(1+\alpha\big(\frac{1}{2}-\frac{d}{2}(\frac{1}{q}-\frac{1}{p})\big)\big)}\right)
        \big\|\omega(\mathbf{x},t)\big\|_{L^{q}}\\
    &\leq C_2\Big(1+t^{-\frac{\alpha}{2}-\frac{\alpha d}{2}(\frac{1}{q}-\frac{1}{p})}\Big)\big\|\omega(\mathbf{x},t)\big\|_{L^{q}}, \quad t>0.
\end{aligned}
\end{displaymath}
Here, the condition $\frac{d}{2}(\frac{1}{q}-\frac{1}{p})<\frac{3}{2}$ ensures the convergence of the singular integral and well-definedness of the Gamma functions.

Finally, the remaining estimates \eqref{eq:Ec}, \eqref{eq:Ee}, and \eqref{eq:Ef} follow by the same reasoning, using \eqref{eq:identity} together with the heat semigroup estimates (17)-(19) from Lemma 3.1 in Ref. \refcite{Ma25}. This completes the proof.
\end{proof}

To simplify notation without loss of generality, we set $\mathcal{D}=1$ in \eqref{Pro:FKS}. Applying the Laplace transform to \eqref{Pro:FKS} and employing the identity $\mathfrak{L}_{t\rightarrow z}\{t^{\beta-1}E_{\alpha,\beta}(-t^{\alpha}\mathcal{A})\}
=z^{\alpha-\beta}(z^{\alpha}I+\mathcal{A})^{-1}$, valid for $\alpha>0$, $0<t\leq T$, and $\beta\in\mathbb{R}$ (see, e.g., Refs.~\refcite{Costa23,Ma25}), we obtain, after inverse Laplace transform, the following Duhamel integral system,
\begin{equation}\label{sys:mild}
\left\{\begin{aligned}
n(\mathbf{x},t)&
=E_{\alpha}\big(-t^{\alpha}\mathcal{A}\big)n_0(\mathbf{x})
      -\chi\int_{0}^{t}(t-s)^{\alpha-1}
      E_{\alpha,\alpha}\big(-(t-s)^{\alpha}\mathcal{A}\big)\nabla\cdot\Big(\frac{n}{c}\nabla c\Big)(\mathbf{x},s)\mathrm{d}s;\\
c(\mathbf{x},t)&
=E_{\alpha}\big(-t^{\alpha}\mathcal{A}_{\gamma}\big)c_0(\mathbf{x})+\int_{0}^{t}(t-s)^{\alpha-1}
      E_{\alpha,\alpha}\big(-(t-s)^{\alpha}\mathcal{A}_{\gamma}\big)n(\mathbf{x},s)\mathrm{d}s.
\end{aligned}\right.
\end{equation}

Following Definition 4.1 in Ref. \refcite{Ma25}, we call a pair $(n,c)$ satisfying \eqref{sys:mild} a mild solution of system \eqref{Pro:FKS}. Additionally, the forms of $n(\mathbf{x},t)$ and $c(\mathbf{x},t)$ in \eqref{sys:mild} directly give rise to the properties stated in the lemmas below.

\begin{lemma}[Conservation of Mass]\label{lem:MassCon-n}
Let the initial data $n_0(\mathbf{x})$ and $c_0(\mathbf{x})$ satisfy \eqref{Eq:Assuption}. For every $t>0$ and $\alpha\in(0,1)$, the total masses of $n(\mathbf{x},t)$ and $c(\mathbf{x},t)$ evolve according to the following identities,
\begin{equation}\label{eq:Mass-n}
\int_{\Omega}n(\mathbf{x},t)\mathrm{d}\mathbf{x}=\int_{\Omega}n_0(\mathbf{x})\mathrm{d}\mathbf{x},
\end{equation}
and
\begin{equation}\label{Eq:MassCon-c}
\begin{aligned}
\int_{\Omega}c(\mathbf{x},t)\mathrm{d}\mathbf{x}+\frac{\gamma}{\Gamma(\alpha)}
\int_{0}^{t}(t-s)^{\alpha-1}\int_{\Omega}c(\mathbf{x},s)\mathrm{d}\mathbf{x}\,\mathrm{d}s
=\frac{t^{\alpha}}{\Gamma(1+\alpha)}\int_{\Omega}n_0(\mathbf{x})\mathrm{d}\mathbf{x}
 +\int_{\Omega}c_0(\mathbf{x})\mathrm{d}\mathbf{x}.
\end{aligned}
\end{equation}
\end{lemma}

\begin{proof}
Applying the Riemann-Liouvillee fractional integral operator ${_{0}I^{1-\alpha}_{t}}$ to both sides of \eqref{Pro:KS1} gives
\begin{displaymath}
n(\mathbf{x},t)-n_{0}(\mathbf{x})
=\frac{1}{\Gamma(\alpha)}\int_{0}^{t}(t-s)^{\alpha-1}
\left(\Delta n(\mathbf{x},s)-\chi\nabla\cdot\Big(\frac{n}{c}\nabla c\Big)(\mathbf{x},s)\right)\mathrm{d}s.
\end{displaymath}
Integrating over $\Omega$ and using the divergence theorem, we obtain, for all $t>0$
\begin{displaymath}
\begin{aligned}
\int_{\Omega}\big(n(\mathbf{x},t)-n_{0}(\mathbf{x})\big)\mathrm{d}\mathbf{x}
&=\frac{1}{\Gamma(\alpha)}\int_{0}^{t}(t-s)^{\alpha-1}\int_{\Omega}\Big(\Delta n-\chi\nabla\cdot\Big(\frac{n}{c}\nabla c\Big)\Big)(\mathbf{x},s)\mathrm{d}\mathbf{x}\,\mathrm{d}s\\
&=\frac{1}{\Gamma(\alpha)}\int_{0}^{t}(t-s)^{\alpha-1}\int_{\partial\Omega}
   \frac{\partial n(\mathbf{x},s)}{\partial\nu}\mathrm{d}S\,\mathrm{d}s         \\
&\quad -\frac{\chi}{\Gamma(\alpha)}\int_{0}^{t}(t-s)^{\alpha-1}\int_{\partial\Omega}
  \frac{n(\mathbf{x},s)}{c(\mathbf{x},s)}\frac{\partial c(\mathbf{x},s)}{\partial\nu}\mathrm{d}S\,\mathrm{d}s,
\end{aligned}
\end{displaymath}
where $\mathrm{d}S$ denotes the boundary area element. The homogeneous Neumann boundary conditions in \eqref{Pro:KS3} make both boundary integrals vanish, and \eqref{eq:Mass-n} follows.

The same argument applied to the equation of $c(x,t)$  gives \eqref{Eq:MassCon-c}, using the identity $\frac{1}{\Gamma(\alpha)}\int_{0}^{t}(t-s)^{\alpha-1}\mathrm{d}s=\frac{t^{\alpha}}{\Gamma(1+\alpha)}$. This completes the proof.
\end{proof}

\begin{lemma}\label{lem:bound-n}
Let $T>0$ be appropriately small, and the initial data $n_0(\mathbf{x})$ and $c_0(\mathbf{x})$ satisfy \eqref{Eq:Assuption}. Then, for all $(\mathbf{x},t)\in\Omega\times[0,T]$ with  $\Omega\subset \mathbb{R}^{d}$ ($d\geq2$), the solution component $n(\mathbf{x},t)$ of the mild solution to the system \eqref{Pro:FKS} remains nonnegative (i.e., $n\geq0$).
\end{lemma}

\begin{proof}
We prove the non-negativity of $n(\mathbf{x},t)$ by considering its negative part $u(\mathbf{x},t):=\max\{0,-n(\mathbf{x},t)\}$ (see Ref. \refcite{Huang25}). Clearly, $u(\mathbf{x},t)\ge0$, and by definition,
\begin{equation}\label{eq:reallllst}
\begin{cases}
n(\mathbf{x},t) = -u(\mathbf{x},t), \quad \nabla n(\mathbf{x},t)= -\nabla u(\mathbf{x},t), &~~ \text{for}~~ n(\mathbf{x},t) < 0, \\
u(\mathbf{x},t) = 0, \qquad\qquad \nabla u(\mathbf{x},t) = 0,  &~~ \text{for}~~ {n(\mathbf{x},t) \geq 0}.
\end{cases}
\end{equation}
Consequently, according to the definition of $u$, $n$ is non-negative if and only if $u(\mathbf{x},t)\equiv0$ throughout $\Omega\times[0, T]$.

The standard local existence theory for fractional parabolic equations in $L^q$ ensures that $n(\mathbf{x}, t)$ is absolutely continuous in $t$ for a.e. $\mathbf{x} \in \Omega$, justifing the use of the fractional convexity inequality. Multiply \eqref{Pro:KS1} by $\mathrm{sgn}(n)$ with $\mathrm{sgn}(\cdot)$ being defined as the standard sign function, and applying the fractional convexity inequality \eqref{eq:m1mL1} (Lemma \ref{lem:convexl1}), and integrating over $\Omega$, we obtain
\begin{displaymath}
{_{0}^{C}\mathfrak{D}^{\alpha}_t}\int_{\Omega}|n|\mathrm{d}\mathbf{x}
\leq \int_{\Omega}\text{sgn}(n)\Delta n\, \mathrm{d}\mathbf{x}-\chi\int_{\Omega} \text{sgn}(n)\nabla\cdot(n \mathbf{v})\, \mathrm{d}\mathbf{x},
\end{displaymath}
where $\mathbf{v}:=\frac{\nabla c}{c}$. Using $\nabla\cdot(n\mathbf{v})\text{sgn}(n)=\nabla\cdot(|n|\mathbf{v})$ and $\Delta n\,\text{sgn}(n)\leq\Delta|n|$ (Kato's inequality; see Refs.~\refcite{Huang25,Kato72}), we get
\begin{displaymath}
{_{0}^{C}\mathfrak{D}^{\alpha}_t}\int_{\Omega}|n|\, \mathrm{d}\mathbf{x}
\leq\int_{\Omega}\Delta|n|\,\mathrm{d}\mathbf{x}-\chi\int_{\Omega}\nabla\cdot(|n| \mathbf{v})\, \mathrm{d}\mathbf{x}.
\end{displaymath}
By the divergence theorem and the Neumann boundary conditions, we find $\int_{\Omega}\Delta|n|\mathrm{d}\mathbf{x}=\int_{\partial\Omega}\nabla|n|\cdot\vec{\nu}\mathrm{d}S=0$ and $\int_{\Omega}\nabla\cdot(|n| \mathbf{v})=\int_{\partial\Omega}\nabla|n|\mathbf{v}\cdot\vec{\nu}\mathrm{d}S=0$. Hence
\begin{displaymath}
{_{0}^{C}\mathfrak{D}^{\alpha}_t}\int_{\Omega}|n|\, \mathrm{d}\mathbf{x}\leq0.
\end{displaymath}
Meanwhile, Lemma \ref{lem:MassCon-n} gives ${_{0}^{C}\mathfrak{D}^{\alpha}_t}\int_{\Omega}n\mathrm{d}\mathbf{x}=0$. Subtracting yields
\begin{equation}\label{eq:maineq}
{_{0}^{C}\mathfrak{D}^{\alpha}_t}\int_{\Omega}(|n|-n)\, \mathrm{d}\mathbf{x} \leq 0.
\end{equation}
Since $|n|-n=2u$, by \eqref{eq:maineq}, we have
\begin{equation}\label{eq:maineq1}
{_{0}^{C}\mathfrak{D}^{\alpha}_t}\int_{\Omega}2u(\mathbf{x},t)\, \mathrm{d}\mathbf{x}\leq0.
\end{equation}
The initial non-negativity $n_0\geq0$ implies $u(\mathbf{x},0)=0$, hence $\int_{\Omega}u(\mathbf{x},0)\mathrm{d}\mathbf{x}=0$. Applying the generalized Gr\"{o}nwall lemma (Lemma \ref{Lem:Gronwall}) to \eqref{eq:maineq1} gives $\int_{\Omega}u(\mathbf{x},t)\mathrm{d}\mathbf{x}\le0\cdot E_{\alpha,1}(Ct^\alpha)=0$ for all $t\in[0,T]$, so
\begin{displaymath}
\int_{\Omega}u(\mathbf{x},t)\, \mathrm{d}\mathbf{x} \equiv 0,\quad \text{a.e.}~~ (\mathbf{x},t)\in\Omega\times[0,T].
\end{displaymath}
By the definition of $u$, implies $n\geq0$ almost everywhere on $(\mathbf{x},t)\in\Omega\times[0,T]$. This completes the proof.
\end{proof}

\begin{lemma}\label{lem:LowerBound-c}
Assume $c(\mathbf{x},0)$ satisfies the condition stated in \eqref{Eq:Assuption}. Then for $\alpha\in(0,1)$, the following hold:
\begin{description}
  \item[$(i)$] For every $t\in(0,T]$, the solution $c(\mathbf{x},t)$ of \eqref{Pro:FKS} satisfies
  \begin{equation}\label{eq:LowerBound-c}
      c(\mathbf{x},t)\geq \frac{t^{-\alpha}}{t^{-\alpha}+\gamma\,\Gamma(1-\alpha)}\inf_{\mathbf{x}\in\Omega}c(\mathbf{x},0),\quad  t>0,
  \end{equation}
      ensuring that $c(\mathbf{x},t)>0$ for all $(\mathbf{x},t)\in\Omega\times[0,T]$.
  \item[$(ii)$] There exists a constant $C_*>0$ such that $\inf_{\mathbf{x}\in\Omega, t\geq 0} c(\mathbf{x},t)\geq C_*$.
\end{description}
\end{lemma}

\begin{proof}
We first construct a spatially homogeneous comparison function $c_h(t)$ satisfying the time fractional ordinary differential equation ${_{0}^{C}\mathfrak{D}^{\alpha}_t}c_h(t)=-\gamma c_{h}(t)$ for all $t>0$, equipped with initial condition $c_h(0):=m_0=\inf_{\mathbf{x}\in\Omega}c(\mathbf{x},0)>0$. The explicit solution is
$c_h(t)=m_0E_{\alpha,1}(-\gamma t^\alpha)$. Define the auxiliary function $w(\mathbf{x},t):=c(\mathbf{x},t)-c_h(t)$. Applying the Caputo derivative operator to $w$ and using the second equation of \eqref{Pro:FKS}, we get, for $t>0$,
\begin{displaymath}
\begin{aligned}
{_{0}^{C}\mathfrak{D}^{\alpha}_t}w(\mathbf{x},t)
&={_{0}^{C}\mathfrak{D}^{\alpha}_t}c(\mathbf{x},t)-{_{0}^{C}\mathfrak{D}^{\alpha}_t}c_h(t)   \\
&=\Delta c(\mathbf{x},t)-\gamma c(\mathbf{x},t) + n(\mathbf{x},t)-\big(-\gamma c_h(t)\big)   \\
&=\Delta w(\mathbf{x},t)-\gamma w(\mathbf{x},t)+ n(\mathbf{x},t).
\end{aligned}\end{displaymath}
Treating the term $n(\mathbf{x},t)$ as a source term, and using its non-negativity guaranteed by Lemma~\ref{lem:bound-n}, the solution $w(\mathbf{x},t)$ satisfies
\begin{equation}
{_{0}^{C}\mathfrak{D}^{\alpha}_t}w-\Delta w+\gamma w\geq 0,\quad (\mathbf{x},t)\in\Omega\times(0,T].
\end{equation}
Since $w(\mathbf{x},0)=c(\mathbf{x},0)-c_h(0)\geq0$ and $\partial w/\partial\mathbf{\nu}=0$, Lemma \ref{Lem:CPp} implies $w(\mathbf{x},t)\geq0$, i.e., $c(\mathbf{x},t)\geq c_h(t)$ for all $\mathbf{x}\in\Omega$ and $t\geq0$. Using the standard lower bound for the Mittag–Leffler function (valid for $\gamma>0$, $\alpha\in(0,1)$),
\begin{equation}
E_{\alpha,1}\big(-\gamma t^\alpha\big)\geq\frac{1}{1+\gamma\,\Gamma(1-\alpha)\,t^\alpha},
\end{equation}
we obtain $c(\mathbf{x},t)\geq m_0 E_{\alpha,1}(-\gamma t^\alpha)\geq \frac{t^{-\alpha}}{t^{-\alpha}+\gamma\,\Gamma(1-\alpha)}m_0$. This yields the desired estimate \eqref{eq:LowerBound-c}. Since $E_{\alpha,1}(-z)>0$ for all $z\geq 0$, the strict positivity of $c$ follows immediately.

It remains to prove ($ii$), i.e., $c(\mathbf{x},t)$ admits a strictly positive lower bound $C_*>0$ for all $t\in[0,\infty)$. To establish this lower bound as $t\to\infty$, we analyze the continuous contribution of the source term $n(\mathbf{x},t)$, expressed as $\mathcal{K}(\mathbf{x},t):=\int_{0}^{t}(t-s)^{\alpha-1} E_{\alpha,\alpha}\big(-(t-s)^{\alpha}\mathcal{A}_{\gamma}\big)n(\mathbf{x},s)\mathrm{d}s$ in \eqref{sys:mild}, and split the time axis into two regimes.
\begin{itemize}
  \item \textbf{Short time interval $[0,1]$}: For $t\in[0,1]$, by Lemma \ref{lem:continuous} and from ($i$) with the lower bound \eqref{eq:LowerBound-c}, it holds that $c(\mathbf{x},t)\geq \min_{t\in[0,1]}m_0 E_{\alpha,1}(-\gamma t^\alpha)\geq m_0 E_{\alpha,1}(-\gamma):=C_1>0$.
  \item \textbf{Long time interval $(1,\infty)$}: For $t>1$, we consider the integral contribution in $\mathcal{K}(\mathbf{x},t)$ specifically over the recent history window $s\in[t-1,t-\frac{1}{2}]$. In this interval, $\tau=t-s\in[\frac{1}{2},1]$, and $\tau^{\alpha-1}\geq1$ since $\alpha\in(0,1)$, we have
  \begin{equation}\label{eq:longt}
   \begin{aligned}
  c(\mathbf{x},t)
  &\geq \int_{t-1}^{t-1/2} (t-s)^{\alpha-1} E_{\alpha,\alpha}\big(-(t-s)^{\alpha}\mathcal{A}_{\gamma}\big)n(\mathbf{x},s)\mathrm{d}s \\
  & \geq \int_{t-1}^{t-1/2}1\cdot(\eta_0 \tilde{M})\mathrm{d}s =\frac{1}{2}\eta_0 \tilde{M} := C_2 > 0.
   \end{aligned}
  \end{equation}
  The constant $C_2$ is strictly positive and independent of $t$.
\end{itemize}
In the last inequality in \eqref{eq:longt} is established by leveraging the subordination principle for the Mittag-Leffler operator. Specifically, we have
\begin{displaymath}
\big[E_{\alpha,\alpha}(-\tau^\alpha\mathcal{A}_\gamma)n\big](\mathbf{x},t)=\int_0^\infty\alpha z M_\alpha(z)e^{-\gamma z\tau^\alpha}\left(\int_\Omega p(\mathbf{x},\mathbf{y},z\tau^\alpha)n(\mathbf{y},t-\tau)\mathrm{d}\mathbf{y}\right)\mathrm{d}z,
\end{displaymath}
where $p(\mathbf{x},\mathbf{y},\theta)$ is the Neumann heat kernel on $\Omega$. For bounded connected domains $\Omega$ satisfying the uniform interior cone condition, it is well-established that the Neumann heat kernel admits the Gaussian lower bound (see, e.g., Thm.~3.10 in Ref.~\refcite{Gyrya11}). For any $\theta>0$, there exist constants $m_1>0$ and $m_2>0$ such that
\begin{displaymath}
p(\mathbf{x},\mathbf{y},\theta)
\geq \frac{m_1}{|\Omega|}\exp\left(-\frac{m_2\,\text{diam}(\Omega)^2}{\theta}\right)
:=p_*(\theta)>0,
\end{displaymath}
with $\text{diam}(\Omega)=\sup_{\mathbf{x},\mathbf{y}\in\Omega}|\mathbf{x}-\mathbf{y}|<\infty$ is the finite diameter of the domain. This estimate provides a uniform lower bound for the heat kernel that is independent of the spatial coordinates $\mathbf{x},\mathbf{y}$. Consequently, by utilizing the mass conservation law $\int_\Omega n(\mathbf{y},\cdot)\mathrm{d}\mathbf{y}=\tilde{M}$, we can bound the inner spatial integral as
\begin{equation}
\int_\Omega p\big(\mathbf{x},\mathbf{y},z\tau^\alpha)n(\mathbf{y},t-\tau\big)\mathrm{d}\mathbf{y}
\geq p_*\big(z \tau^\alpha\big)\tilde{M}.
\end{equation}
Define the auxiliary function $\mathcal{H}(\tau):=\alpha\tilde{M}\int_0^\infty z M_\alpha(z)e^{-\gamma z\tau^\alpha}p_*(z\tau^\alpha)\mathrm{d}z$. Note that $M_\alpha(z)>0$ for all $z>0$, and the exponential structure of the Gaussian bound $p_*$ ensures that the integrand is strictly positive on $(0,\infty)$. Furthermore, the Mainardi function $M_\alpha(z)$ exhibits super-exponential decay as $z\to 0^+$ (specifically, $M_\alpha(z)\sim\exp(-z^{-\frac{1}{1-\alpha}})$, see e.g., Thm.3.8~(iv) in Ref.~\refcite{Jin21book}), which effectively counteracts the singularity potentially arising from the term $\exp(-1/z)$ in the heat kernel estimate. By the Lebesgue dominated convergence theorem, $\mathcal{H}(\tau)$ is continuous on the compact interval $\tau\in[\frac{1}{2},1]$. Applying the Weierstrass extreme value theorem, we conclude that $\mathcal{H}(\tau)$ attains a strictly positive minimum
\begin{displaymath}
\min_{\tau\in[1/2,1]}\mathcal{H}(\tau):=\eta_0>0.
\end{displaymath}
Substitute this uniform bound into the temporal integral, \eqref{eq:longt} holds directly.

Finally, setting $C_*=\min\{C_1, C_2\}>0$ yields $c(\mathbf{x},t)\geq C_*$ for all $\mathbf{x}\in\Omega$ and $t\geq0$. This completes the proof.
\end{proof}

\begin{remark}
Combing Lemmas \ref{lem:bound-n} and \ref{lem:LowerBound-c} yields the non-negativity of $n(\mathbf{x},t)$ and the strict positivity of $c(\mathbf{x},t)>0$ almost everywhere on $(\mathbf{x},t)\in\Omega\times[0,T]$. In particular, the lower bound established in Lemma \ref{lem:LowerBound-c} guarantees the well-posedness of the logarithmic chemotaxis term and will be frequently invoked in the forthcoming analysis.
\end{remark}
\begin{lemma}\label{lem:regularity-c}
Let $T>0$, $d\geq2$ and $1\leq p,q\leq\infty$. Suppose the initial data $n_0(\mathbf{x})$ and $c_0(\mathbf{x})$ satisfy \eqref{Eq:Assuption}. Then the following estimates hold for $t\in(0,T]$:
\begin{description}
  \item[$(i)$] If $\frac{1}{q}-\frac{1}{p}<\frac{2}{d}$, there exists a positive constant $C$ such that
  \begin{equation}\label{eq:estimate-c-r1}
  \big\|c(\mathbf{x},t)\big\|_{L^{p}}
  \leq C\Big(1+\max\Big\{T^{\alpha},T^{\alpha-\frac{\alpha d}{2}\max\{0,\frac{1}{q}-\frac{1}{p}\}}\Big\}\sup_{t\in(0,T)}\big\|n(\mathbf{x},t)\big\|_{L^{q}}\Big).
  \end{equation}
  \item[$(ii$)] If $\frac{1}{q}-\frac{1}{p}<\frac{1}{d}$, there exists a positive constant $C$ such that
  \begin{equation}\label{eq:estimate-c-r2}
  \big\| c(\mathbf{x},t)\big\|_{W^{1,p}}
  \leq C\Big(1
  +\max\Big\{T^{\alpha},T^{\frac{\alpha}{2}-\frac{\alpha d}{2}\max\{0,\frac{1}{q}-\frac{1}{p}\}}\Big\}\sup_{t\in(0,T)}\big\|n(\mathbf{x},t)\big\|_{L^{q}}\Big).
  \end{equation}
\end{description}
\end{lemma}

\begin{proof}
We divide the proof of ($i$) into two distinct cases: $p\geq q$ and $p<q$.

\textbf{Case 1: $p\geq q$.} Apply estimates \eqref{eq:Ea} and \eqref{eq:Ed} from Lemma \ref{lem:MLbounds}, together with H\"{o}lder's inequality, yields, for $t>0$
\begin{equation}\label{eq:eeetedr}
\begin{aligned}
\big\|c(\mathbf{x},t)\big\|_{L^{p}}
&\leq\big\|E_{\alpha}(-t^{\alpha}\mathcal{A})c_0\big\|_{L^{p}}+\int_{0}^{t}(t-s)^{\alpha-1}
   \left\|E_{\alpha,\alpha}\big(-(t-s)^{\alpha}\mathcal{A}_{\gamma}\big)n(\mathbf{x},s)\right\|_{L^{p}}\mathrm{d}s\\
&\leq |\Omega|^{1/p}\big\|E_{\alpha}(-t^{\alpha}\mathcal{A})c_0\big\|_{L^{\infty}}
   +C\int_{0}^{t}(t-s)^{\alpha-1}\big\|n(\mathbf{x},s)\big\|_{L^{q}}\mathrm{d}s\\
&\quad+C\int_{0}^{t}(t-s)^{\alpha-1-\frac{\alpha d}{2}(\frac{1}{q}-\frac{1}{p})}\big\|n(\mathbf{x},s)\big\|_{L^{q}}\mathrm{d}s\\
&\leq C\big\|c_0\big\|_{L^{\infty}}+C\max\Big\{T^{\alpha},T^{\alpha-\frac{\alpha d}{2}(\frac{1}{q}-\frac{1}{p})}\Big\}\sup_{t\in(0,T)}\big\|n(\mathbf{x},t)\big\|_{L^{q}}.
\end{aligned}
\end{equation}
In the last inequality, we used the fact that if $0\leq\frac{1}{q}-\frac{1}{p}<\frac{2}{d}$, $0<1-\alpha+\frac{\alpha d}{2}(\frac{1}{q}-\frac{1}{p})<1$, then the second integral is convergent and bounded as above.

\textbf{Case 2: $p<q$.} The estimate follows directly from H\"{o}lder's inequality. Combining the two cases gives \eqref{eq:eeetedr} for all $1\leq p,q\leq\infty$. Under the assumptions in \eqref{Eq:Assuption}, the estimate \eqref{eq:estimate-c-r1} follows immediately.

For the gradient estimate in ($ii$), we use a similar argument together with the commutator identity $\mathcal{A}^{1/2}E_{\alpha}(-t^{\alpha}\mathcal{A}))c_0=E_{\alpha}(-t^{\alpha}\mathcal{A}))\mathcal{A}^{1/2}c_0$, which follows from the corresponding identity for the heat semigroup $\mathcal{A}^{1/2}e^{t\Delta}c_0=e^{t\Delta}\mathcal{A}^{1/2} c_0$. Applying the estimates \eqref{eq:Ea} and \eqref{eq:Ef} from Lemma \ref{lem:MLbounds}, and noting that the condition $0\leq\frac{1}{q}-\frac{1}{p}<\frac{1}{d}$ implies $1-\frac{\alpha}{2}+\frac{\alpha d}{2}(\frac{1}{q}-\frac{1}{p})<1$, we obtain \eqref{eq:estimate-c-r2}. This completes the proof of the lemma.
\end{proof}

\begin{remark}
Lemma \ref{lem:regularity-c} requires the exponents $p$, $q$ to satisfy $1\leq p,q\leq\infty$ and $\frac{1}{q}-\frac{1}{p}<\frac{1}{d}$ with $d\geq2$. Without loss of generality, we may make $q\in[\frac{d}{2},d)$, which implies $\frac{dq}{d-q}>d$. This allows us to choose $p>d$ with $p<\frac{dq}{d-q}$. Such a choice simultaneously satisfies the hypotheses of Lemma \ref{lem:regularity-c} and the assumptions needed in Theorem \ref{Thm:well-posed}.
\end{remark}

The proofs of these lemmas are based on standard arguments, yet their results reveal a clear departure from the classical logarithmic Keller–Segel system. This distinction provides new insight into the time‑nonlocal system \eqref{Pro:FKS}.

\section{Local Well-Posedness of Mild Solution in Arbitrary Dimensions}
\label{sec:LEMS}
Let $\Omega\subset\mathbb{R}^{d}$ ($d\geq2$) be a bounded domain with smooth boundary. Following the framework adopted in, e.g., Refs.~\refcite{Jin20,Winkler10,Winkler22}, we introduce the Banach space
\begin{equation}
\mathbb{X}:=C\big([0,T];L^{\infty}(\Omega)\big)\times C\big([0,T];W^{1,p}(\Omega)\big),\quad \text{\rm for}~ ~ p>d\geq2,
\end{equation}
endowed with the norm
\begin{equation}
\big\|(n,c)\big\|_{\mathbb{X}}:=\big\|n\big\|_{C([0,T];L^{\infty}(\Omega))}+\big\|c\big\|_{C([0,T];W^{1,p}(\Omega))}.
\end{equation}
Since $p>d\geq2$, the Sobolev embedding theorem ensures that $W^{1,p}(\Omega)\hookrightarrow C(\overline{\Omega})$. This continuity allows us to define pointwise bounds for $c(\mathbf{x},t)$, which is further justified by Lemma \ref{lem:LowerBound-c}. For fixed constants $\mathcal{R}>0$ and $C_{*}>0$, we define the closed subset $\mathfrak{B}$ as follows,
\begin{equation}
\mathfrak{B} := \left\{
(n,c) \in \mathbb{X} \;\middle|\;
\begin{aligned}
&\|(n,c)\|_{\mathbb{X}} \leq \mathcal{R},\; n \geq 0,\; c \geq C_* \,\text{in } \Omega \times [0,T], \\
&\partial_{\nu} n = \partial_{\nu} c = 0 \, \text{on } \partial\Omega
\end{aligned}
\right\}.
\end{equation}
We focus on the mapping $\mathcal{M}(n,c):=\big(\mathcal{M}_1(n,c),\mathcal{M}_2(n,c)\big)$, defined for $t\in(0,T]$ by
\begin{equation}\label{eq:mappings}
\left\{\begin{aligned}
\mathcal{M}_1(n,c) &= E_{\alpha}\bigl(-t^{\alpha}\mathcal{A}\bigr)n_0
      -\chi\int_{0}^{t} (t-s)^{\alpha-1} E_{\alpha,\alpha}\bigl(-(t-s)^{\alpha}\mathcal{A}\bigr)
         \nabla\cdot\Bigl(\frac{n}{c}\nabla c\Bigr)(s) \,\mathrm{d}s,\\[4pt]
\mathcal{M}_2(n,c) &= E_{\alpha}\bigl(-t^{\alpha}\mathcal{A}_{\gamma}\bigr)c_0
      +\int_{0}^{t} (t-s)^{\alpha-1} E_{\alpha,\alpha}\bigl(-(t-s)^{\alpha}\mathcal{A}_{\gamma}\bigr) n(s) \,\mathrm{d}s.
\end{aligned}\right.
\end{equation}
These components satisfy
\begin{equation}\label{Pro:mapping}
\left\{
\begin{aligned}
&{}_{0}^{C}\mathfrak{D}^{\alpha}_t \mathcal{M}_1(n,c)(\mathbf{x},t) = \mathcal{D}\Delta \mathcal{M}_1(n,c)(\mathbf{x},t)
      - \mathcal{D}\chi \nabla\cdot\Bigl(\frac{n}{c}\nabla c\Bigr)(\mathbf{x},t),\\[4pt]
&{}_{0}^{C}\mathfrak{D}^{\alpha}_t \mathcal{M}_2(n,c)(\mathbf{x},t) = \mathcal{D}\Delta \mathcal{M}_2(n,c)(\mathbf{x},t)
      - \gamma \mathcal{M}_2(n,c)(\mathbf{x},t) + n(\mathbf{x},t),\\[4pt]
&\mathcal{M}_1(n,c)(\mathbf{x},0) = n_0(\mathbf{x}), \quad \mathcal{M}_2(n,c)(\mathbf{x},0) = c_0(\mathbf{x}),\\[4pt]
&\partial_{\nu}\mathcal{M}_1(n,c)(\mathbf{x},t) = 0, \qquad \partial_{\nu}\mathcal{M}_2(n,c)(\mathbf{x},t) = 0.
\end{aligned}
\right.
\end{equation}

Our goal is to establish the following properties of $\mathcal{M}$:
\begin{itemize}
  \item $\mathcal{M}(n,c)\in\mathbb{X}$ for all $(n,c)\in\mathbb{X}$, $t\in(0,T]$;
  \item for sufficiently large $\mathcal{R}>0$ and appropriately small $T>0$, $\mathcal{M}$ maps $\mathfrak{B}$ into itself; and
  \item under the same conditions, $\mathcal{M}$ is a contraction on $\mathfrak{B}$.
\end{itemize}
We first show that $\mathcal{M}$ is well-defined on $\mathfrak{B}$ and, under suitable assumptions, is self-mapping.

\begin{lemma}\label{Lem:Bijection}
Let the initial data $n_0(\mathbf{x})$ and $c_0(\mathbf{x})$ satisfy \eqref{Eq:Assuption} and let $\chi > 0$. Define the mapping $\mathcal{M}(n,c):=\big(\mathcal{M}_1(n,c),\mathcal{M}_2(n,c)\big)$ by \eqref{eq:mappings}. Then $\mathcal{M}(n,c)\in\mathbb{X}$ for all $(n,c)\in\mathbb{X}$. Moreover, for any $t\in(0,T]$, if $\mathcal{R}>0$ is suitably large and $T>0$ is appropriately small, the mapping $\mathcal{M}$ is well-defined on the closed subset $\mathfrak{B}$ and maps $\mathfrak{B}$ into itself, i.e., $\mathcal{M}:\mathfrak{B}\to\mathfrak{B}$.
\end{lemma}

\begin{proof}
We assume without loss of generality that $t\in(0,T]$ with $T<1$ and $p>d$. For any $(n,c)\in\mathfrak{B}$, we have $n(t,x)\geq0$ and $c(x,t)\ge C_*$ on $\Omega\times(0,T)$. Applying estimates \eqref{eq:Ea} and \eqref{eq:Eb} from Lemma \ref{lem:MLbounds},  together with H\"{o}lder's inequality, yields
\begin{displaymath}
\begin{aligned}
\big\|\mathcal{M}_1(n,c)\big\|_{L^{\infty}}
&\leq C\big\|n_0\big\|_{L^{\infty}}
     +\chi\int_{0}^{t}s^{\alpha-1}\left\|E_{\alpha,\alpha}\big(-s^{\alpha}\mathcal{A}\big)
     \nabla\cdot\Big(\frac{n}{c}\nabla c\Big)(t-s)\right\|_{L^{\infty}}\mathrm{d}s\\
&\leq C\big\|n_0\big\|_{L^{\infty}}+C\int_{0}^{t}(t-s)^{\frac{\alpha}{2}-\frac{\alpha d}{2p}-1}
      \left\|c(s)^{-1}\right\|_{L^{\infty}}\big\|n(s)\nabla c(s)\big\|_{L^{p}}\mathrm{d}s\\
&\leq C\big\|n_0\big\|_{L^{\infty}}+\frac{C}{C_{*}}\int_{0}^{t}
     (t-s)^{\frac{\alpha}{2}-\frac{\alpha d}{2p}-1}
      \big\|n(s)\big\|_{L^{\infty}}
      \big\|\nabla c(s)\big\|_{L^{p}}\mathrm{d}s\\
&\leq C\big\|n_0\big\|_{L^{\infty}}+C\mathcal{R}^2
     \int_{0}^{t}(t-s)^{\frac{\alpha}{2}-\frac{\alpha d}{2p}-1}\mathrm{d}s \\
&\leq C\big\|n_0\big\|_{L^{\infty}}+C\mathcal{R}^2 T^{\frac{\alpha}{2}(1-\frac{d}{p})}.
\end{aligned}
\end{displaymath}
Choosing $\mathcal{R}>0$ suitably large and $T>0$ appropriately small (depending on $n_0$ and $c_0$ only through their norms in $L^{\infty}(\Omega)$ and $W^{1,p}(\Omega)$), we obtain
\begin{equation}\label{eq:EM1}
\sup_{t\in[0,T]}\big\|\mathcal{M}_1(n,c)\big\|_{L^{\infty}}\leq \mathcal{R}/2.
\end{equation}

Similarly, using eatimates \eqref{eq:Ea} and \eqref{eq:Ef} from  Lemma \ref{lem:MLbounds}, for $p>d$, we have
\begin{equation}\label{eq:EM2}
\begin{aligned}
\Big\|\mathcal{A}^{1/2}\mathcal{M}_2(n,c)\Big\|_{L^{p}}
&\leq \left\|E_{\alpha}(-t^\alpha\mathcal{A}_{\gamma})\mathcal{A}^{1/2}c_0\right\|_{L^{p}}
      +\int_{0}^{t}s^{\alpha-1}\Big\|\mathcal{A}^{1/2}E_{\alpha,\alpha}\big(-s^{\alpha}
      \mathcal{A}_{\gamma}\big)n(t-s)\Big\|_{L^{p}}\mathrm{d}s\\
&\leq C\big\|\nabla c_0\big\|_{L^{\infty}}
      +C\int_{0}^{t}(t-s)^{\frac{\alpha}{2}-1}\big\|n(s)\big\|_{L^{p}}\mathrm{d}s\\
&\leq C\big\|\nabla c_0\big\|_{L^{\infty}}+CT^{\frac{\alpha}{2}}\,\mathcal{R},
\end{aligned}
\end{equation}
and
\begin{equation}\label{eq:EM3}
\big\|\mathcal{M}_2(n,c)\big\|_{L^{p}}\leq C\|c_0\|_{L^{\infty}}+C T^{\alpha}\mathcal{R}.
\end{equation}
Thus, for $\mathcal{R}$ and $T$ chosen as above, we get
\begin{equation}\label{eq:EM4}
\sup_{t\in[0,T]}\big\|\mathcal{M}_2(n,c)\big\|_{W^{1,p}}\leq \mathcal{R}/2.
\end{equation}
Combining \eqref{eq:EM1} and \eqref{eq:EM4}, we conclude that for all $t\geq0$,
\begin{displaymath}
\big\|\mathcal{M}(n,c)\big\|_{\mathbb{X}}
\leq\mathcal{R},\qquad \forall~~(n,c)\in\mathfrak{B}.
\end{displaymath}

To complete the proof that $\mathcal{M}(\mathfrak{B})\subset\mathfrak{B}$, it remains to verify that $\mathcal{M}_1$ preserves the non-negativity of elements in $\mathfrak{B}$ (i.e., $\mathcal{M}_1\geq0$), and that $\mathcal{M}_2$ preserves the lower bound constraint $c\geq C_*>0$ that defines $\mathfrak{B}$. This is done by adapting the technique from the proofs of Lemmas \ref{lem:bound-n} and \ref{lem:LowerBound-c}.

\textbf{Non-negativity of $\mathcal{M}_1$.} Following the argument in Lemma~\ref{lem:bound-n}, Multiply the first equation in \eqref{Pro:mapping} by $\operatorname{sgn}(\mathcal{M}_1)$, integrate over $\Omega$, and apply the Neumann boundary conditions. This yields
\begin{displaymath}
{}_{0}^{C}\mathfrak{D}^{\alpha}_t\int_{\Omega}|\mathcal{M}_1|\,\mathrm{d}x\le 0.
\end{displaymath}
On the other hand, mass conservation (derived from the same equation) gives ${}_{0}^{C}\mathfrak{D}^{\alpha}_t\int_{\Omega}\mathcal{M}_1\mathrm{d}x=0$. Subtracting yields
\begin{displaymath}
{}_{0}^{C}\mathfrak{D}^{\alpha}_t\int_{\Omega}\bigl(|\mathcal{M}_1|-\mathcal{M}_1\bigr)\,\mathrm{d}x \le 0.
\end{displaymath}
Since $|\mathcal{M}_1|-\mathcal{M}_1\ge 0$ and vanishes at $t=0$, it follows that $|\mathcal{M}_1|-\mathcal{M}_1\equiv 0$, i.e., $\mathcal{M}_1(n,c)\ge 0$ for all $(n,c)\in\mathfrak{B}$.

\textbf{Lower bound for $\mathcal{M}_2$}. Recall from Lemma~\ref{lem:LowerBound-c} that
\begin{equation}
    C_* = \min\{C_1,\, C_2\} > 0, \quad
    C_1 = m_0 E_{\alpha,1}(-\gamma), \quad
    m_0 = \inf_{\overline{\Omega}}c_0(\mathbf{x})> 0,
\end{equation}
satisfies $c(\mathbf{x},t)\ge C_*$ for the original solution. We show the same for $\mathcal{M}_2(n,c)$. Using the mild representation of $\mathcal{M}_2$, we have
\begin{displaymath}
\mathcal{M}_2(n,c)(\mathbf{x},t)
= E_{\alpha}\bigl(-t^{\alpha}\mathcal{A}_{\gamma}\bigr)c_0(\mathbf{x})
  + \int_{0}^{t} (t-s)^{\alpha-1}
  E_{\alpha,\alpha}\bigl(-(t-s)^{\alpha}\mathcal{A}_{\gamma}\bigr)n(\mathbf{x},s)\,\mathrm{d}s.
\end{displaymath}
Since $(n,c)\in\mathfrak{B}$ implies $n\ge 0$ and  the Mittag–Leffler operators $E_{\alpha}(-t^{\alpha}\mathcal{A}_{\gamma})$, $E_{\alpha,\alpha}(-t^{\alpha}\mathcal{A}_{\gamma})$ are positivity-preserving (see, e.g., Ref.~\refcite{Wang12}), the integral term is non-negative. Hence,
\begin{displaymath}
\mathcal{M}_2(n,c)(\mathbf{x},t) \ge E_{\alpha}\bigl(-t^{\alpha}\mathcal{A}_{\gamma}\bigr)c_0(\mathbf{x}).
\end{displaymath}
Combining the lower bound $c_0\ge m_0$ with the positivity-preserving property, we obtain
\begin{displaymath}
E_{\alpha}\bigl(-t^{\alpha}\mathcal{A}_{\gamma}\bigr)c_0
\ge m_0\,E_{\alpha}\bigl(-t^{\alpha}\mathcal{A}_{\gamma}\bigr)1.
\end{displaymath}
Using the eigenfunction expansion and the identity $\mathcal{A}_{\gamma}\,1=\gamma\,1$,
it follows that $E_{\alpha}\bigl(-t^{\alpha}\mathcal{A}_{\gamma}\bigr)\mathbf{1}=E_{\alpha,1}(-\gamma t^{\alpha})\,1$. Thus,
\begin{equation}\label{eq:M2lower}
\mathcal{M}_2(n,c)(\mathbf{x},t) \ge m_0\,E_{\alpha,1}(-\gamma t^{\alpha}),
\qquad (\mathbf{x},t)\in\Omega\times[0,T].
\end{equation}
In the present local existence argument we may assume without loss of generality that $T\le 1$ (indeed the statement already restricts to $T<1$). Since the Mittag-Leffler function $E_{\alpha,1}(-x)$ is strictly decreasing on $[0,\infty)$ and $t^{\alpha}\le 1$, we have
\begin{displaymath}
m_0\,E_{\alpha,1}(-\gamma t^{\alpha})
\ge m_0\,E_{\alpha,1}(-\gamma T^{\alpha})
\ge m_0\,E_{\alpha,1}(-\gamma)
= C_1 \ge C_* .
\end{displaymath}
Combining this with \eqref{eq:M2lower} yields
\begin{equation}\label{eq:EM5}
\mathcal{M}_2(n,c)(\mathbf{x},t) \ge C_*, \qquad
(\mathbf{x},t)\in\Omega\times[0,T].
\end{equation}
Thus $\mathcal{M}_2\geq C_*$, and the constant $C_*$ here coincides with that obtained in Lemma~\ref{lem:LowerBound-c}.

Combining the non-negativity of $\mathcal{M}_1$, the uniform positive lower bound of $\mathcal{M}_2$ established in \eqref{eq:EM5}, and the norm bounds \eqref{eq:EM1}--\eqref{eq:EM4}, we conclude that $\mathcal{M}(\mathfrak{B})\subset\mathfrak{B}$. Moreover, the estimates \eqref{eq:EM1}--\eqref{eq:EM3} already guarantee that $\mathcal{M}(n,c)(t)\in\mathbb{X}$ for all $(n,c)\in\mathbb{X}$.

Finally, the continuity of $\mathcal{M}(n,c)$ in $t\in[0,T]$  follows from strong continuity of the Mittag-Leffler operator (see Lemma \ref{lem:continuous}) and the elementary properties of the Neumann heat semigroup. With this, it immediately follows that for any $t_1$, $t_2\in[0, T]$ with $t_1<t_2$, $\lim_{t_{1}\to t_{2}^{-}}\|\mathcal{M}(n,c)(t_1)-\mathcal{M}(n,c)(t_2)\|_{\mathbb{X}}=0$, which implies $\mathcal{M}(n,c)\in C\big([0,T]; L^{\infty}(\Omega)\big)$. In establishing this continuity result, we also naturally prove that $n(\mathbf{x},t)\in C\big([0,T];L^{\infty}(\Omega)\big)$, $c(\mathbf{x},t)\in C\big([0,T];W^{1,p}(\Omega)\big)$, $p>d$. Since this argument is standard (see, e.g., Refs.~\refcite{Costa23,Ma25}), the detailed verification is omitted. The proof is now complete.
\end{proof}

\begin{lemma}[Contraction mapping]\label{Lem:Contraction}
Under the assumptions of Lemma \ref{Lem:Bijection}, the mapping $\mathcal{M}(n,c):\mathfrak{B}\rightarrow\mathfrak{B}$ is a contraction for all $t\in(0,T]$, provided $\mathcal{R}>0$ is suitably large and $T>0$ is appropriately small.
\end{lemma}

\begin{proof}
Let $(n,c)$ and $(\widetilde{n},\widetilde{c})$ belong to $\mathfrak{B}$. Applying estimates \eqref{eq:Eb} from Lemma \ref{lem:MLbounds}, together with Lemma \ref{lem:LowerBound-c} and H\"{o}lder's inequality, we obtain the following bound
\begin{displaymath}
\begin{aligned}
&\big\|\mathcal{M}_1(n,c)-\mathcal{M}_1(\widetilde{n},\widetilde{c})\big\|_{L^{\infty}}\\
&\leq \chi\int_{0}^{t}(t-s)^{\alpha-1}\left\|E_{\alpha,\alpha}\big(-(t-s)^{\alpha}\mathcal{A}\big)\nabla\cdot
     \Big(\frac{n}{c}\nabla c-\frac{\widetilde{n}}{\widetilde{c}}\nabla \widetilde{c}\Big)(s)\right\|_{L^{\infty}}\mathrm{d}s                   \\
&\leq C\int_{0}^{t}(t-s)^{\frac{\alpha}{2}-\frac{\alpha d}{2p}-1}
\left\|\left(\frac{n}{c}\nabla c-\frac{\widetilde{n}}{\widetilde{c}}\nabla
       \widetilde{c}\right)(s)\right\|_{L^{p}}\mathrm{d}s  \\
&\leq C\int_{0}^{t}(t-s)^{\frac{\alpha}{2}-\frac{\alpha d}{2p}-1}
     \Big\{\big\|\big(n-\widetilde{n}\big)(s)\big\|_{L^{\infty}}
     \big\|\,c(s)\,\big\|^{-1}_{L^{\infty}}
     \big\|\nabla c(s)\big\|_{L^{p}}                      \\
&\qquad
     \qquad~+\big\|\,\widetilde{n}(s)\,\big\|_{L^{\infty}}\big\|\,c(s)\,\big\|^{-1}_{L^{\infty}}
     \big\|\big(\nabla c-\nabla \widetilde{c}\big)(s)\big\|_{L^{p}}  \\
&\qquad
     \qquad~+\big\|\,\widetilde{n}(s)\,\big\|_{L^{\infty}}
     \big\|\big(1/c-1/\widetilde{c}\big)(s)\,\big\|_{L^{p}}
     \big\|\nabla c(s)\big\|_{L^{\vartheta}}\Big\}\mathrm{d}s        \\
&\leq \big(C\mathcal{R}/C_*\big)\big(1+(\mathcal{R}/C_*)\big)
      T^{\frac{\alpha}{2}\left(1-\frac{d}{p}\right)}\big\|(n,c)-(\widetilde{n},\widetilde{c})\big\|_{\mathbb{X}}.
\end{aligned}
\end{displaymath}
Hence,
\begin{equation}\label{eq:EcM1}
\sup_{t\in(0,T]}
\big\|\mathcal{M}_1(n,c)-\mathcal{M}_1(\tilde{n},\tilde{c})\big\|_{L^\infty(\Omega)}
\le C\mathcal{R}^{2}T^{\frac{\alpha}{2}\left(1-\frac{d}{p}\right)}
\big\|(n,c)-(\tilde{n},\tilde{c})\big\|_{\mathbb{X}}.
\end{equation}

Similarly, using estimate \eqref{eq:Ef} from Lemma \ref{lem:MLbounds}, we obtain, for $t\in(0,T]$,
\begin{displaymath}
\begin{aligned}
\left\|\mathcal{A}^{1/2}\big(\mathcal{M}_{2}(n,c)-\mathcal{M}_{2}(\widetilde{n},\widetilde{c})\big)\right\|_{L^{p}}
&\leq \int_{0}^{t}(t-s)^{\alpha-1}\left\|\mathcal{A}^{1/2}E_{\alpha,\alpha}\big(-(t-s)^{\alpha}\mathcal{A}_{\gamma}\big)
     \big(n(s)-\widetilde{n}(s)\big)\right\|_{L^{p}}\mathrm{d}s  \\
&\leq C\int_{0}^{t}(t-s)^{\frac{\alpha}{2}-1}
     \big\|n(s)-\widetilde{n}(s)\big\|_{L^{p}}\mathrm{d}s
\leq CT^{\frac{\alpha}{2}}\big\|(n,c)-(\widetilde{n},\widetilde{c})\big\|_{\mathbb{X}},
\end{aligned}
\end{displaymath}
and
\begin{displaymath}
\big\|\mathcal{M}_{2}(n,c)-\mathcal{M}_{2}(\widetilde{n},\widetilde{c})\big\|_{L^{p}}
\leq CT^{\alpha}\big\|(n,c)-(\widetilde{n},\widetilde{c})\big\|_{\mathbb{X}}.
\end{displaymath}
Combining these two bounds yields
\begin{equation}\label{eq:EcM2}
\sup_{t\in(0,T]}
\big\|\mathcal M_2(n,c)-\mathcal M_2(\tilde n,\tilde c)\big\|_{W^{1,p}}
\le C \max\big\{T^{\alpha/2},T^{\alpha}\big\}
\big\|(n,c)-(\tilde n,\tilde c)\big\|_{\mathbb X}.
\end{equation}

By integrating the estimates \eqref{eq:EcM1} and \eqref{eq:EcM2}, we can establish that choosing a suitably small $T>0$ and an appropriately large radius $\mathcal{R}>0$ ensures the following result holds, for all $t\in(0,T]$,
\begin{displaymath}
\big\|\mathcal{M}(n,c)-\mathcal{M}(\widetilde{n},\widetilde{c})\big\|_{\mathbb{X}}
\leq C_\mathfrak{B}\big\|(n,c)-(\widetilde{n},\widetilde{c})\big\|_{\mathbb{X}},
\end{displaymath}
where the contractive constant $C_\mathfrak{B}$ is defined as
\begin{displaymath}
C_\mathfrak{B}:=C\max\Big\{\mathcal{R}^{2}
T^{\frac{\alpha}{2}(1-\frac{d}{p})}, T^{\frac{\alpha}{2}}, T^{\alpha}\Big\}<1.
\end{displaymath}
As a result, the mapping $\mathcal{M}$ is a contraction on the closed subset $\mathfrak{B}$.
\end{proof}

\begin{theorem}
\label{Lem:WelPod1}
Let the initial data $n_0(\mathbf{x})$, $c_0(\mathbf{x})$ satisfy \eqref{Eq:Assuption}. Then there exists an appropriately small $T>0$ such that system \eqref{Pro:FKS} admits a unique local mild solution $(n,c)\in\mathfrak{B}$ for $t\in[0,T]$, satisfying
\begin{equation}
n(\mathbf{x},t)\in C\big([0,T];L^{\infty}(\Omega)\big),\quad
c(\mathbf{x},t)\in C\big([0,T];W^{1,p}(\Omega)\big),\quad p>d\geq2.
\end{equation}
Furthermore, for $p>d$, as $t\rightarrow0^{+}$, it holds that
\begin{equation}\label{eq:asymtic}
\lim_{t\rightarrow0^{+}}\big\|n(\mathbf{x},t)-n_0(\mathbf{x})\big\|_{L^{p}}\rightarrow0,\quad
\lim_{t\rightarrow0^{+}}\big\|c(\mathbf{x},t)-c_0(\mathbf{x})\big\|_{W^{1,p}}\rightarrow0.
\end{equation}
\end{theorem}

\begin{proof}
By Lemmas \ref{Lem:Bijection} and \ref{Lem:Contraction}, the mapping $\mathcal{M}(n,c)(t): \mathfrak{B} \rightarrow \mathfrak{B}$ defined in \eqref{eq:mappings} is well-defined and strictly contractive on the complete metric space $\mathfrak{B}$ for all $t\in(0, T]$, with $T>0$ appropriately small. The Banach fixed-point theorem (see Theorem. 5.7 in Ref. \refcite{Brezis11}) then guarantees a unique fixed point $(n,c)\in \mathfrak{B}$, which yields the unique local mild solution to \eqref{Pro:FKS} on $[0,T]$. This establishes the local well-posedness.

We now verify the initial data attainment. By applying the estimates from Lemma \ref{lem:MLbounds}, we directly derive
\begin{align*}
\big\|n(\cdot,t)-n_0\big\|_{L^{p}}
&\leq\big\|E_{\alpha}(-t^\alpha\mathcal{A})n_0-n_0\big\|_{L^{p}}
     +\chi\int_{0}^{t}s^{\alpha-1}\left\|E_{\alpha,\alpha}\big(-s^{\alpha}\mathcal{A}\big)
     \nabla\cdot\Big(\frac{n}{c}\nabla c\Big)(t-s)\right\|_{L^{p}}\mathrm{d}s\\
&\leq |\Omega|^{1/p}\big\|E_{\alpha}(-t^\alpha\mathcal{A})n_0-n_0\big\|_{L^{\infty}}
      +\frac{C}{C_{*}}\int_{0}^{t}(t-s)^{\frac{\alpha}{2}-\frac{\alpha d}{2p}-1}\mathrm{d}s\,\big\|(n,c)\big\|^2_{\mathbb{X}}  \\
&\leq |\Omega|^{1/p}\big\|E_{\alpha}(-t^\alpha\mathcal{A})n_0-n_0\big\|_{L^{\infty}}
      +C t^{\frac{\alpha}{2}\left(1-\frac{d}{p}\right)}\,\big\|(n,c)\big\|^2_{\mathbb{X}}.
\end{align*}
By the strong continuity of the Mittag–Leffler operator (Lemma \ref{lem:continuous}), $\|n(\cdot,t)-n(\cdot,0)\|_{L^{p}}\rightarrow0$ as $t\rightarrow0^{+}$. Hence $\big\|n(\cdot,t)-n_0\big\|_{L^{p}}\to 0$.

Similarly, for the gradient part of $c$
\begin{align*}
\left\|\mathcal{A}^{1/2}c(\cdot,t)-\mathcal{A}^{1/2}c_0\right\|_{L^{p}}
&\leq \left\|\mathcal{A}^{1/2}E_{\alpha}(-t^\alpha\mathcal{A}_{\gamma})c_0-\mathcal{A}^{1/2}c_0\right\|_{L^{p}} \\
&\quad+\int_{0}^{t}s^{\alpha-1}\big\|\mathcal{A}^{1/2}E_{\alpha,\alpha}\big(-s^{\alpha}\mathcal{A}_{\gamma}\big)n(t-s)\big\|_{L^{p}}\mathrm{d}s\\
&\leq \left\|E_{\alpha}(-t^\alpha\mathcal{A}_{\gamma})\mathcal{A}^{1/2}c_0-\mathcal{A}^{1/2}c_0\right\|_{L^{p}}
      +C\int_{0}^{t}(t-s)^{\frac{\alpha}{2}-\frac{\alpha d}{2p}-1}
      \big\|n(s)\big\|_{L^{\infty}}\mathrm{d}s \\
&\leq \left\|E_{\alpha}(-t^\alpha\mathcal{A}_{\gamma})\mathcal{A}^{1/2}c_0-\mathcal{A}^{1/2}c_0\right\|_{L^{p}}+Ct^{\frac{\alpha}{2}\left(1-\frac{d}{p}\right)}
\big\|(n,c)\big\|_{\mathbb{X}} \rightarrow0,
\end{align*}
as $t\rightarrow0^{+}$, where we used the commutativity $\mathcal{A}^{1/2}E_{\alpha}(-t^\alpha\mathcal{A}_{\gamma})=E_{\alpha}(-t^\alpha\mathcal{A}_{\gamma})\mathcal{A}^{1/2}$ and the strong continuity of the semigroup.

For the $L^p$-norm of $c$, we similarly obtain
\begin{align*}
\big\|c(\cdot,t)-c_0\big\|_{L^{p}}
\leq \big\|E_{\alpha}(-t^\alpha\mathcal{A}_{\gamma})c_0-c_0\big\|_{L^{p}}
+Ct^{\alpha\left(1-\frac{d}{2p}\right)}
\big\|(n,c)\big\|_{\mathbb{X}} \rightarrow 0,\quad t\rightarrow0^{+}.
\end{align*}

Combining the above convergence results yields \eqref{eq:asymtic}. This completes the proof.
\end{proof}

Theorem \ref{Lem:WelPod1} establishes local well-posedness of system \eqref{sys:mild} in $\mathbb{X}$, and provides the key estimate \eqref{eq:asymtic} for solutions $(n,c)(\cdot,t)$, which guarantees continuous dependence on initial data.

We now aim to improve the temporal and spatial regularity of this mild solution. To facilitate the analysis, we introduce the shorthand $\mathcal{G}(s):=\nabla\cdot(\frac{n(\mathbf{x},s)}{c(\mathbf{x},s)}\nabla c(\mathbf{x},s))$, together with the Mittag-Leffler families
\begin{equation}\label{Not:notiess}
\left\{\begin{aligned}
\mathcal{S}_\alpha(t):=&E_\alpha(-t^\alpha\mathcal{A}),\quad && \mathcal{P}_\alpha(t):=t^{\alpha-1}E_{\alpha,\alpha}(-t^\alpha\mathcal{A}),\\
\mathcal{S}^{\gamma}_\alpha(t):=&E_\alpha(-t^\alpha\mathcal{A}_\gamma),\quad
&&\mathcal{P}^{\gamma}_\alpha(t):=t^{\alpha-1}E_{\alpha,\alpha}(-t^\alpha\mathcal{A}_\gamma).
\end{aligned}\right.
\end{equation}
Estimates for these operators are collected in Lemma~\ref{Der:Lemma}. With these tools in hand, we present the following lemmas on enhanced regularity.

\begin{lemma}[Improved temporal regularity]\label{Lem:timR}
Suppose $p>d\geq2$ and $T>\tau>0$, where $T$ is appropriately small. Let $(n,c)$ be the local mild solution obtained in Theorem~\ref{Lem:WelPod1} for Problem~\eqref{Pro:FKS}. Then, for $\alpha\in(0,1)$,  the following regularity properties hold on $[\tau,T]$,
\begin{align*}
&c(\mathbf{x},t)\in L^{\infty}([\tau,T];W^{2,p}(\Omega))
   \cap C^{0,\alpha/2}([\tau,T];W^{2,p}(\Omega)),\\
&n(\mathbf{x},t)\in L^{\infty}([\tau,T];W^{1,p}(\Omega))
   \cap C^{0,\alpha}([\tau,T];L^p(\Omega)).
\end{align*}
Moreover, for any $\tau>0$, there also hold
\begin{displaymath}
\sup_{t\in[\tau,T]}\Big(\big\|n(\mathbf{x},t)\big\|_{L^\infty}
+\big\|\nabla n(\mathbf{x},t)\big\|_{L^p}
+\big\|c(\mathbf{x},t)\big\|_{L^\infty}
+\big\|\nabla c(\mathbf{x},t)\big\|_{L^\infty}\Big)<\infty.
\end{displaymath}
\end{lemma}

\begin{proof}
To help navigate the proof, we provide a schematic diagram in Fig.\ref{Fig:diagram} that outlines the logical flow of the regularity improvements.

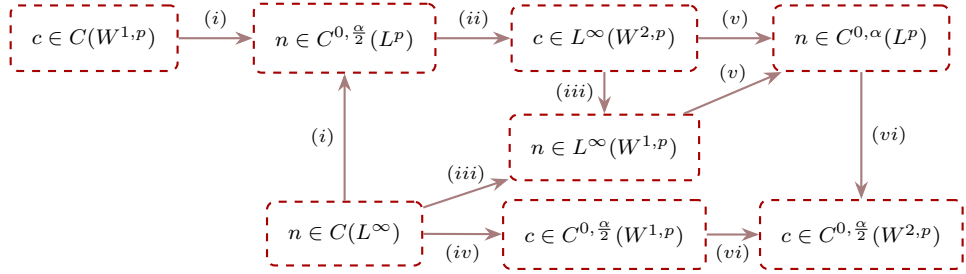
\begin{figure}[htbp]
\centering
\begin{tikzpicture}[
    node distance=1.68cm and 0.98cm,
    mainnode/.style={align=center,font=\footnotesize,inner sep=8pt,rounded corners=3pt,
    draw=red!65!black,thick,dashed,fill=none},>=Stealth,thick]
\node[mainnode] (A) {$c\in C(W^{1,p})$};
\node[mainnode] (C) [right=of A] {$n\in C^{0,\frac{\alpha}{2}}(L^p)$};
\node[mainnode] (D) [right=of C] {$c\in L^\infty(W^{2,p})$};
\node[mainnode] (F) [right=of D] {$n\in C^{0,\alpha}(L^p)$};
\node[mainnode] (B) [below=of C] {$n\in C(L^\infty)$};
\node[mainnode] (G) [below=of D] {$c\in C^{0,\frac{\alpha}{2}}(W^{1,p})$};
\node[mainnode] (E) [below=0.55cm of D] {$n\in L^\infty(W^{1,p})$};
\node[mainnode] (H) [below=of F] {$c\in C^{0,\frac{\alpha}{2}}(W^{2,p})$};
\draw[->, draw=pink!65!black] (A) -- (C) node[midway,above] {\scriptsize($i$)};
\draw[->, draw=pink!65!black] (B) -- (C) node[midway,left] {\scriptsize($i$)};
\draw[->, draw=pink!65!black] (B) -- (G) node[midway,below] {\scriptsize($iv$)};
\draw[->, draw=pink!65!black] (B) -- (E) node[midway,above] {\scriptsize($iii$)};
\draw[->, draw=pink!65!black] (C) -- (D) node[midway,above] {\scriptsize($ii$)};
\draw[->, draw=pink!65!black] (D) -- (F) node[midway,above] {\scriptsize($v$)};
\draw[->, draw=pink!65!black] (D) -- (E) node[midway,left] {\scriptsize($iii$)};
\draw[->, draw=pink!65!black] (E) -- (F) node[midway,above] {\scriptsize($v$)};
\draw[->, draw=pink!65!black] (F) -- (H) node[midway,right] {\scriptsize($vi$)};
\draw[->, draw=pink!65!black] (G) -- (H) node[midway,below] {\scriptsize($vi$)};
\end{tikzpicture}
\caption{The diagram for illustrating the proof of improved regularity of $(n,c)$.}
\label{Fig:diagram}
\end{figure}

Recall that $\mathcal{A}=-\Delta$ is the Neumann Laplacian on $L^p(\Omega)$, $1<p<\infty$, and $\mathcal{A}_\gamma=\mathcal{A}+\gamma I$ is the damped Neumann Laplacian with $\gamma>0$. Since $\gamma I$ is bounded on $L^p(\Omega)$, we have $D(\mathcal{A}_\gamma)=D(\mathcal{A})$ and the norm equivalence $\|u\|_{W^{2,p}}\simeq\|\mathcal{A}u\|_{L^p}+\|u\|_{L^p}\simeq\|\mathcal{A}_\gamma u\|_{L^p}+\|u\|_{L^p}$ (see, e.g., Ref.~\refcite{Lindemulder26}). We also recall the notation $\mathcal{G}(s):=\nabla\cdot(\frac{n(\mathbf{x},s)}{c(\mathbf{x},s)}\nabla c(\mathbf{x},s))$ and the Mittag-Leffler families defined in  \eqref{Not:notiess}. With these preparations, we prove the lemma in the following steps.

\begin{itemize}
    \item \emph{\textbf{Step $(i)$}}: $C^{0,\alpha/2}([\tau,T];L^{p}(\Omega))$-regularity of $n$.
\end{itemize}

We first establish a temporal regularity estimate for $n$,   using the already available bounds $n\in L^\infty((0,T)\times\Omega)$ and
$c\in L^\infty([\tau,T]; W^{1,p}(\Omega))$ from Theorem~\ref{Lem:WelPod1}. Since $p > d\geq2$, Sobolev embedding gives $c\in L^\infty([\tau,T]; L^\infty(\Omega))$. Together with the lower bound $c\ge C_* > 0$ (Lemma~\ref{lem:LowerBound-c}), this implies $F:= -n\frac{\nabla c}{c}$ belongs to $L^\infty([\tau,T]; L^p(\Omega))$. Recalling $\mathcal{G}=\nabla\cdot F$, the mild solution for $n$ reads
\begin{displaymath}
n(t) = \mathcal{S}_\alpha(t)n_0 + \int_0^t \mathcal{P}_\alpha(t-s)\mathcal{G}(s)\,\mathrm{d}s.
\end{displaymath}
Let $\tau \le t_1 < t_2 \le T$. We decompose the difference $n(t_2) - n(t_1) = \mathcal{J}_{1} + \mathcal{J}_{2} + \mathcal{J}_{3}$, where
\begin{equation}\label{eq:dedee}
\left\{
\begin{aligned}
\mathcal{J}_1 &:= \big(\mathcal{S}_\alpha(t_2) - \mathcal{S}_\alpha(t_1)\big)n_0,\\
\mathcal{J}_2 &:= \int_{0}^{t_1} \big(\mathcal{P}_\alpha(t_2-s) - \mathcal{P}_\alpha(t_1-s)\big)\mathcal{G}(s)\,\mathrm{d}s, \\
\mathcal{J}_3 &:= \int_{t_1}^{t_2} \mathcal{P}_\alpha(t_2-s)\mathcal{G}(s)\,\mathrm{d}s.
\end{aligned}\right.
\end{equation}
For $\mathcal{J}_3$, using the kernel bound $\|\mathcal{P}_\alpha(t)\nabla \cdot f\|_{L^p}\le C t^{\frac{\alpha}{2}-1}\|f\|_{L^p}$, we eatimate
\begin{displaymath}
\big\|J_3\big\|_{L^p} \le C \int_{t_1}^{t_2} (t_2-s)^{\alpha/2-1}\big\|F(s)\big\|_{L^p}\,\mathrm{d}s \le C|t_2-t_1|^{\alpha/2}.
\end{displaymath}
For $\mathcal{J}_2$, applying the fundamental theorem of calculus yields
\begin{align*}
\big\|J_2\big\|_{L^p}
\le C \int_{t_1}^{t_2} \int_0^{t_1} (r-s)^{\frac{\alpha}{2}-2}\big\|F(s)\big\|_{L^p}\,\mathrm{d}s\,\mathrm{d}r
\le C \int_{t_1}^{t_2} (r-t_1)^{\frac{\alpha}{2}-1}\,\mathrm{d}r = C|t_2-t_1|^{\frac{\alpha}{2}}.
\end{align*}
For $J_1$, the continuous differentiability of the operator $\mathcal{S}_\alpha(t)$ for $t > 0$ gives $\|J_1\|_{L^p} \le C_\tau|t_2-t_1|$. Combining the estimates for $J_1$, $J_2$, and $J_3$, we conclude that
\begin{displaymath}
\big\|n(t_2)-n(t_1)\big\|_{L^p} \le C\big|t_2-t_1\big|^{\frac{\alpha}{2}}, \quad \forall~ \tau \le t_1 < t_2 \le T.
\end{displaymath}
Therefore, the desired H\"{o}lder regularity has been established, that is $n\in C^{0,\frac{\alpha}{2}}\big([\tau,T]; L^p(\Omega)\big)$.

\begin{itemize}
    \item \emph{\textbf{Step $(ii)$}}: $L^\infty([\tau,T];W^{2,p}(\Omega))$-regularity of $c$.
\end{itemize}

For $t\in[\tau,T]$, applying the shifted operator $\mathcal{A}_\gamma$ to the mild formulation of $c$ yields
\begin{align*}
\big\|\mathcal{A}_\gamma c(\mathbf{x},t)\big\|_{L^p}
\le \big\|\mathcal{A}_\gamma \mathcal{S}_\alpha^\gamma(t)c_0\big\|_{L^p}
+ \int_0^t \big\|\mathcal{A}_\gamma \mathcal{P}_\alpha^\gamma(t-s)n(\mathbf{x},s)\big\|_{L^p}\,\mathrm{d}s.
\end{align*}
From the resolvent estimate for the Mittag-Leffler family (see Lemma~\ref{lem:MLbounds}), for $t>0$, we have $\|\mathcal{A}_\gamma\mathcal{S}_\alpha^\gamma(t)\|_{L^p\to L^p}\le C t^{-\alpha}$. Hence, for $t\geq\tau>0$,
\begin{displaymath}
\big\|\mathcal{A}_\gamma\mathcal{S}^{\gamma}_\alpha(t)c_0\big\|_{L^{p}}
\leq C t^{-\alpha}\big\|c_0\big\|_{L^{p}}
\leq C(\tau)\big\|c_0\big\|_{L^{p}}.
\end{displaymath}
For the convolution term, write $n(\mathbf{x},s)=n(\mathbf{x},t)+(n(\mathbf{x},s)-n(\mathbf{x},t))$.
Then
\begin{align*}
\int_0^t \mathcal{A}_\gamma\mathcal{P}^{\gamma}_\alpha(t-s)n(\mathbf{x},s)\,\mathrm{d}s
=&\int_0^t \mathcal{A}_\gamma\mathcal{P}^{\gamma}_\alpha(t-s)\big(n(\mathbf{x},s)-n(\mathbf{x},t)\big)\,\mathrm{d}s \\
&+\Big(\int_0^t\mathcal{A}_\gamma\mathcal{P}^{\gamma}_\alpha(t-s)\,\mathrm{d}s\Big)\,n(\mathbf{x},t).
\end{align*}
Using the resolvent identity $\int_0^t\mathcal{A}_\gamma\mathcal{P}_\alpha^\gamma(r)\mathrm{d}r=I-\mathcal{S}_\alpha^\gamma(t)$, we obtain
\begin{displaymath}
\mathcal{A}_\gamma c(\mathbf{x},t)=\mathcal{A}_\gamma\mathcal{S}_\alpha^\gamma(t)c_0
+\big(I-\mathcal{S}_\alpha^\gamma(t)\big)n(\mathbf{x},t)
+\int_0^t\mathcal{A}_\gamma\mathcal{P}_\alpha^\gamma(t-s)\big(n(\mathbf{x},s)-n(\mathbf{x},t)\big)\,\mathrm{d}s.
\end{displaymath}
By Lemma \ref{lem:MLbounds}, we obtain $\|\mathcal{S}_\alpha^\gamma(t)\|_{L^p\to L^p}\le C$. Since $n(\mathbf{x},t)\in C([0,T];L^p(\Omega))$ from Theorem~\ref{Lem:WelPod1}, the second term remains uniformly bounded on $[\tau,T]$. Moreover, Lemma~\ref{Der:Lemma} gives $\|\mathcal{A}_\gamma\mathcal{P}_\alpha^\gamma(t)\|_{L^p\to L^p}\le C t^{-1}(1+t^{\alpha})^{-1}$. Thus, for the third term, by the temporal H\"{o}lder continuity $n\in C^{0,\alpha/2}([\tau,T];L^p)$, we have, for some
\(\theta\in(0,\alpha/2]\),
\begin{displaymath}
\begin{aligned}
&\int_0^t\left\|\mathcal{A}_\gamma\mathcal{P}_\alpha^\gamma(t-s)\bigl(n(\mathbf{x},s)-n(\mathbf{x},t)\bigr)\right\|_{L^p}\,\mathrm{d}s \\
&\le C\int_0^t\frac{\|n(\mathbf{x},s)-n(\mathbf{x},t)\|_{L^p}}{(t-s)(1+(t-s)^\alpha)}\mathrm{d}s
\le C\int_0^t \frac{|t-s|^\theta}{(t-s)(1+(t-s)^\alpha)}\,\mathrm{d}s.
\end{aligned}
\end{displaymath}
Letting $r=t-s$, we obtain $\int_0^t \frac{r^{\theta-1}}{1+r^\alpha}\,\mathrm{d}r\le\int_0^T r^{\theta-1}\,\mathrm{d}r
= \frac{T^\theta}{\theta}<\infty$, since $t\le T$ and $\theta>0$.
Therefore the convolution term is uniformly bounded in $t\in[\tau,T]$, and we conclude
\begin{displaymath}
\sup_{t\in[\tau,T]}\|\mathcal{A}_\gamma c(\mathbf{x},t)\|_{L^p}<\infty.
\end{displaymath}
Elliptic regularity for the Neumann Laplacian implies $\|c(\mathbf{x},t)\|_{W^{2,p}}\simeq\|\mathcal{A}_\gamma c(\mathbf{x},t)\|_{L^p}+\|c(\mathbf{x},t)\|_{L^p}$ (see, e.g., Ref.~\refcite{Lindemulder26}). Since $c(\mathbf{x},t)\in C([0,T];L^p(\Omega))$ from  Theorem~\ref{Lem:WelPod1}, we conclude $c(\mathbf{x},t)\in L^\infty([\tau,T];W^{2,p}(\Omega))$. Moreover, since $p>d$, the embedding $W^{2,p}(\Omega)\hookrightarrow W^{1,\infty}(\Omega)$ yields
\begin{equation}
\sup_{t\in[\tau,T]}\Big(\big\|c(\mathbf{x},t)\big\|_{L^\infty}+\big\|\nabla c(\mathbf{x},t)\big\|_{L^\infty}\Big)<\infty.
\end{equation}

\begin{itemize}
    \item \emph{\textbf{Step $(iii)$}}: $L^{\infty}([\tau,T];W^{1,p})$-regularity of $n$.
\end{itemize}

Expanding the chemotaxis term $\mathcal{G}(t):=\nabla\cdot\left(\frac{n}{c}\nabla c\right)$, a direct expansion gives
\begin{displaymath}
\mathcal{G}=\frac{1}{c}\nabla n\cdot\nabla c-\frac{n}{c^2}|\nabla c|^2+\frac{n}{c}\Delta c.
\end{displaymath}
By Lemma \ref{lem:LowerBound-c}, $c$ is uniformly positive on $[\tau,T]$, i.e., $c\ge C_*>0$. Combined with $n\in C([\tau,T];L^\infty(\Omega))$ and $c\in L^\infty([\tau,T];W^{2,p}(\Omega))$ from Theorem \ref{Lem:WelPod1}, along with the embedding $W^{2,p}(\Omega)\hookrightarrow W^{1,\infty}(\Omega)$ for $p>d\geq 2$ and the uniform bound $\|n(t)\|_{L^\infty}\le C_\tau$, we establish the estimate, for all $t\in[\tau,T]$,
\begin{align}
\|\mathcal{G}(t)\|_{L^p}
&\le C\left(\|\nabla n(t)\|_{L^p}\|\nabla c(t)\|_{L^\infty}+\|n(t)\|_{L^\infty}\|\nabla c(t)\|_{L^\infty}^2
+\|n(t)\|_{L^\infty}\|\Delta c(t)\|_{L^p}\right)  \notag\\
&\le C_\tau\left(1+\|\nabla n(t)\|_{L^p}\right).  \label{est:GGG}
\end{align}

Returning to the integral representation of $n$ prescribed in \eqref{sys:mild}, applying the gradient operator and using the smoothing property of the resolvent family given by Lemma~\ref{Der:Lemma}, namely $\|\mathcal{A}^{1/2}\mathcal{P}_\alpha(t)\|_{L^p\to L^p}\le Ct^{\frac{\alpha}{2}-1}$, we obtain
\begin{align*}
\big\|\mathcal{A}^{1/2}n(t)\big\|_{L^p}
&\le\big\|\mathcal{A}^{1/2}\mathcal{S}_\alpha(t)n_0\big\|_{L^p}+C\int_0^t(t-s)^{\frac{\alpha}{2}-1}\big\|\mathcal{G}(s)\big\|_{L^p}\,\mathrm{d}s \\
&\le C_\tau+C_\tau\int_0^t(t-s)^{\frac{\alpha}{2}-1}\big\|\mathcal{A}^{1/2} n(s)\big\|_{L^p}\,\mathrm{d}s,
\end{align*}
where we used estimate \eqref{est:GGG}. Setting $u(t):=\|\mathcal{A}^{1/2} n(\mathbf{x},t)\|_{L^{p}(\Omega)}$, we form the Volterra-type inequality for $t \in [\tau,T]$,
\begin{equation}\label{ineq:volterra-clean}
u(t)\leq C_\tau+C_\tau\int_{0}^{t}(t-s)^{\frac{\alpha}{2}-1}u(s)\mathrm{d}s,
\end{equation}
Since the kernel $k(t)=C_2t^{\frac{\alpha}{2}-1}\in L^1(0,T)$ for any $\alpha>0$. Applying the Volterra-type Gr\"{o}nwall inequality (see Lemma 8.2 in Chapter 9 of Ref. \refcite{Gripenberg90}) to \eqref{ineq:volterra-clean} yields
\begin{equation}\label{es:regim}
u(t)=\big\|\mathcal{A}^{1/2} n(\mathbf{x},t)\big\|_{L^{p}}\leq C,\quad \forall~ t\in[\tau, T].
\end{equation}
Combined with the already established bound $\|n(\mathbf{x},t)\|_{L^{p}}\leq C$, it follows that $n\in L^\infty([\tau,T];W^{1,p}(\Omega))$. Since $p>d$, the Sobolev embedding $W^{1,p}(\Omega) \hookrightarrow L^\infty(\Omega)$ holds. Together with the uniform estimate of $\|n(t)\|_{W^{1,p}}$ obtained above, we conclude that $n$ is pointwise bounded on $[\tau,T]$, i.e.,
\begin{displaymath}
\sup_{t\in[\tau,T]}\|n(t)\|_{L^\infty}<\infty.
\end{displaymath}

\begin{itemize}
    \item \emph{\textbf{Step $(iv)$}}: $C^{0,\alpha/2}([\tau,T];W^{1,p}(\Omega))$-regularity of $c$.
\end{itemize}

Let $\tau\leq t_1<t_2\leq T$.  Leveraging the mild formulation of $c(\mathbf{x},t)$, we decompose
\begin{align*}
c(\mathbf{x},t_2)-c(\mathbf{x},t_1)
&=\big(\mathcal{S}^{\gamma}_{\alpha}(t_2)-\mathcal{S}^{\gamma}_{\alpha}(t_1)\big)c_0
  +\int_{0}^{t_1}\big(\mathcal{P}^{\gamma}_\alpha(t_2-s)-\mathcal{P}^{\gamma}_\alpha(t_1-s)\big)n(\mathbf{x},s)\mathrm{d}s\\
&\quad + \int_{t_1}^{t_2}\mathcal{P}^{\gamma}_\alpha(t_2-s)n(\mathbf{x},s)\mathrm{d}s
=:\mathcal{K}_1+\mathcal{K}_2+\mathcal{K}_3.
\end{align*}
We estimate each term in $W^{1,p}(\Omega)$.

For term $\mathcal{K}_1$, since $\mathcal{S}^{\gamma}_\alpha(t)$ is differentiable on $(0,\infty)$ as a bounded operator on $L^p(\Omega)$ with $1<p<\infty$, we obtain
\begin{equation}
\mathcal{K}_1=\big(\mathcal{S}^{\gamma}_{\alpha}(t_2)-\mathcal{S}^{\gamma}_{\alpha}(t_1)\big)c_0
=\int_{t_1}^{t_2}\partial_{s}\,\mathcal{S}^{\gamma}_{\alpha}(s)\,c_0\,\mathrm{d}s.
\end{equation}
By taking the $W^{1,p}$-norm and applying Minkowski's inequality,
\begin{equation}\label{eq:dsrt}
\big\|\mathcal{K}_1\big\|_{W^{1,p}}
\le\int_{t_1}^{t_2}\big\|\partial_s\,\mathcal{S}_\alpha^\gamma(s)\,c_0\big\|_{W^{1,p}}\,\mathrm{d}s.
\end{equation}
Recall the derivative identity (see, e.g., Ref.~\refcite{Podlubny99})
\begin{equation}\label{eq:dsrt1}
\partial_s\,\mathcal{S}^{\gamma}_{\alpha}(s)
=\partial_s\,E_\alpha(-s^\alpha\mathcal{A}_\gamma)
=-s^{\alpha-1}\mathcal{A}_{\gamma}\,E_{\alpha,\alpha}(-s^\alpha\mathcal{A}_\gamma)
=-\mathcal{A}_\gamma\,\mathcal{P}_\alpha^\gamma(s).
\end{equation}
Using the norm equivalence $\|u\|_{W^{1,p}}\leq C(\|\mathcal{A}_\gamma^{1/2}u\|_{L^p}+\|u\|_{L^p})$, and and the commutativity of $\mathcal{A}_\gamma^{1/2}$ with $\mathcal{A}_\gamma \mathcal{P}_\alpha^\gamma(s)$, we get
\begin{align*}
\big\|\partial_s\,\mathcal{S}_\alpha^\gamma(s)c_0\big\|_{W^{1,p}}
=\big\|\mathcal{A}_\gamma \mathcal{P}_\alpha^\gamma(s)c_0\big\|_{W^{1,p}}
\leq C\left(\big\|\mathcal{A}_\gamma \mathcal{P}_\alpha^\gamma(s)\mathcal{A}_\gamma^{1/2}c_0\big\|_{L^p}+\big\|\mathcal{A}_\gamma\mathcal{P}_\alpha^\gamma(s)c_0\big\|_{L^p}\right).
\end{align*}
By employing the explicit resolvent-type estimate from Lemma~\ref{Der:Lemma}, which provides $\|\mathcal{A}_\gamma \mathcal P_\alpha^\gamma(s)\|_{L^p\to L^p}\le C s^{-1}(1+s^\alpha)^{-1}$, we obtain
\begin{align*}
\big\|\partial_s\,\mathcal{S}_\alpha^\gamma(s)c_0\big\|_{W^{1,p}}
&\leq C s^{-1}(1+s^\alpha)^{-1}\left(\big\|\mathcal{A}_\gamma^{1/2}c_0\big\|_{L^p}+\big\|c_0\big\|_{L^p} \right)\\
&\leq C s^{-1}(1+s^\alpha)^{-1} \big\|c_0\big\|_{W^{1,p}}
\leq C s^{-1} \big\|c_0\big\|_{W^{1,p}},
\end{align*}
where we used the property $(1+s^\alpha)^{-1}<1$ for all $s>0$. Substituting this bound into \eqref{eq:dsrt}, we infer that
\begin{equation}
\big\|\mathcal{K}_1\big\|_{W^{1,p}}
\le C\big\|c_0\big\|_{W^{1,p}}\int_{t_1}^{t_2}s^{-1}\,\mathrm{d}s.
\end{equation}
Since $t_1, t_2\in[\tau,T]$ with $\tau>0$, the integrand is uniformly bounded by $\tau^{-1}$. Thus, the integral evaluates to $\int_{t_1}^{t_2}s^{-1}\mathrm{d}s \le \tau^{-1}|t_2-t_1|$. By the embedding $C^{0,1}\hookrightarrow C^{0,\alpha/2}$ for any $0<\alpha<1$, we obtain
\begin{displaymath}
\|\mathcal{K}_1\|_{W^{1,p}}\le C|t_2-t_1|^{\alpha/2}.
\end{displaymath}
This establishes the desired H\"{o}lder continuity on $[\tau,T]$.

For term $\mathcal{K}_3$, by applying Minkowski's integral inequality, we can decompose the estimate as follows:
\begin{align*}
\big\|\mathcal{K}_3\big\|_{W^{1,p}}
&\le C \left(\big\|\mathcal{A}_\gamma^{1/2}\mathcal{K}_3\big\|_{L^p}
   +\big\|\mathcal{K}_3\big\|_{L^p}\right)  \\
&\le C \int_{t_1}^{t_2}\left(\big\|\mathcal{A}_\gamma^{1/2}\mathcal{P}_\alpha^\gamma(t_2-s)n(\cdot,s)\big\|_{L^p}+\big\|\mathcal{P}_\alpha^\gamma(t_2-s)n(\cdot,s)\big\|_{L^p}\right) \mathrm{d}s.
\end{align*}
Recall from Lemma~\ref{Der:Lemma} that $\|\mathcal{A}_\gamma^{1/2}\mathcal{P}_\alpha^\gamma(t)\|_{L^p\to L^{p}}\le Ct^{\frac{\alpha}{2}-1}(1+t^{\alpha})^{-1}$, and the standard resolvent estimate gives $\|\mathcal{P}_\alpha^\gamma(t)\|_{L^p\to L^{p}}\le Ct^{\alpha-1}(1+t^{\alpha})^{-1}$. Since $n\in C([\tau,T];L^\infty(\Omega))$ according to Theorem~\ref{Lem:WelPod1}, the source term is uniformly bounded, i.e., $\max_{s\in[\tau,T]}\|n(\cdot,s)\|_{L^p}\le C_\tau$. Substituting these into the inequality yields
\begin{align*}
\big\|\mathcal{K}_3\big\|_{W^{1,p}}
\le C_\tau \int_{t_1}^{t_2}
 &\left[(t_2-s)^{\frac{\alpha}{2}-1}\big(1+(t_2-s)^\alpha\big)^{-1} \right.\\
 &\quad\left.+ (t_2-s)^{\alpha-1}\big(1+(t_2-s)^\alpha\big)^{-1} \right] \mathrm{d}s.
\end{align*}
Note that for all $s \in [t_1, t_2]$, the bounded algebraic factors satisfy $(1+(t_2-s)^\alpha)^{-1/2} \le 1$ and $(1+(t_2-s)^\alpha)^{-1} \le 1$. Furthermore, since $\alpha > 0$, the temporal exponent satisfies $\frac{\alpha}{2}-1 < \alpha-1$, meaning that the singularity $(t_2-s)^{\frac{\alpha}{2}-1}$ dominates as $s \to t_2$. Consequently, on the bounded time interval, the sum inside the bracket can be sharply capped by the more singular term, leading to
\begin{equation}
\big\|\mathcal{K}_3\big\|_{W^{1,p}}
\le C_\tau \int_{t_1}^{t_2} (t_2-s)^{\frac{\alpha}{2}-1} \,\mathrm{d}s
  = \frac{2C_\tau}{\alpha}|t_2-t_1|^{\alpha/2}.
\end{equation}

For term $\mathcal{K}_2$, we first apply the fundamental theorem of calculus to represent the kernel difference as $\mathcal{P}_\alpha^\gamma(t_2-s)-\mathcal{P}_\alpha^\gamma(t_1-s)=\int_{t_1}^{t_2} \partial_r\mathcal{P}_\alpha^\gamma(r-s)\,\mathrm{d}r$. Hence
\begin{displaymath}
\mathcal{K}_2=\int_{0}^{t_1}\int_{t_1}^{t_2}  \partial_r\mathcal{P}_\alpha^\gamma(r-s)n(s)\,\mathrm{d}r\,\mathrm{d}s.
\end{displaymath}
Invoking the norm equivalence $\|u\|_{W^{1,p}}\leq C(\|\mathcal{A}_\gamma^{1/2}u\|_{L^p}+\|u\|_{L^p})$ once again, together with Minkowski's inequality, we obtain
\begin{displaymath}
\big\|\mathcal{K}_2\big\|_{W^{1,p}}
\leq C \int_{0}^{t_1}\int_{t_1}^{t_2} \left(\big\|\mathcal{A}_\gamma^{1/2}\partial_r\mathcal{P}_\alpha^\gamma(r-s)n(s)\big\|_{L^p}+\big\|\partial_r\mathcal{P}_\alpha^\gamma(r-s)n(s)\big\|_{L^p}\right)\mathrm{d}r\,\mathrm{d}s.
\end{displaymath}
Since $n\in C([0,T];L^{p}(\Omega))$, there exists a constant $\bar{M}>0$ such that $\|n\|_{L^{p}}\leq \bar{M}$ for $0\leq s\leq T$.
Moreover, by the standard Mittag–Leffler operator estimates, $\|\mathcal{A}_\gamma^{1/2}\partial_r\mathcal{P}_\alpha^\gamma(t)\|_{L^{p}\to L^{p}}\leq Ct^{\alpha/2-2}$, and $\|\partial_t\mathcal{P}_\alpha^\gamma(t)\|_{L^{p}\to L^{p}}\leq C t^{\alpha-2}$ (see Lemma~\ref{Der:Lemma}). Therefore,
\begin{displaymath}
\big\|\mathcal{K}_2\big\|_{W^{1,p}}
\leq C\bar{M}\int_{0}^{t_1}\int_{t_1}^{t_2} \left(
(r-s)^{\frac{\alpha}{2}-2}+(r-s)^{\alpha-2}
\right)\mathrm{d}r\,\mathrm{d}s.
\end{displaymath}
We first estimate the more singular contribution. By Fubini's theorem,
\begin{displaymath}
\int_{0}^{t_1}\int_{t_1}^{t_2}(r-s)^{\frac{\alpha}{2}-2}\mathrm{d}r\,\mathrm{d}s=\int_{t_1}^{t_2}\int_{0}^{t_1}(r-s)^{\frac{\alpha}{2}-2}\mathrm{d}s\,\mathrm{d}r.
\end{displaymath}
Introducing the change of variables $u=r-s$, we obtain $\int_{0}^{t_1}(r-s)^{\frac{\alpha}{2}-2}\mathrm{d}s=\int_{r-t_1}^{r}u^{\frac{\alpha}{2}-2}\mathrm{d}u$. Since $\alpha/2-2<-1$, direct integration yields $\int_{r-t_1}^{r}u^{\frac{\alpha}{2}-2}\mathrm{d}u=\frac{1}{1-\alpha/2}[(r-t_1)^{\alpha/2-1}-r^{\alpha/2-1}]$. Because $\alpha/2-1<0$, it follows that $\int_{0}^{t_1}(r-s)^{\frac{\alpha}{2}-2}\mathrm{d}s\leq C(r-t_1)^{\alpha/2-1}$. Integrating once more with respect to $r$, we arrive at $\int_{t_1}^{t_2}(r-t_1)^{\alpha/2-1}\mathrm{d}r=\frac{2}{\alpha}|t_2-t_1|^{\alpha/2}$. Hence,
\begin{displaymath}
\int_{0}^{t_1}\int_{t_1}^{t_2}(r-s)^{\frac{\alpha}{2}-2}\mathrm{d}r\,\mathrm{d}s\leq C\big|t_2-t_1\big|^{\frac{\alpha}{2}}.
\end{displaymath}
Similarly,
\begin{displaymath}
\int_{0}^{t_1}\int_{t_1}^{t_2}(r-s)^{\alpha-2}\mathrm{d}r\,\mathrm{d}s\leq C\big|t_2-t_1\big|^{\alpha}.
\end{displaymath}
Combining the above estimates and observing that $|t_2-t_1|^{\alpha}\leq T^{\alpha/2}|t_2-t_1|^{\alpha/2}$, we conclude that
\begin{equation}
\big\|\mathcal{K}_2\big\|_{W^{1,p}}\le C_{\tau,T}\big|t_2-t_1\big|^{\alpha/2}.
\end{equation}
Therefore, $\mathcal{K}_2$ satisfies the desired H\"{o}lder's regularity in $W^{1,p}(\Omega)$.

Collecting all of the aforementioned estimates,
we may conclude that $c(\mathbf{x},t)\in C^{0,\alpha/2}([\tau,T];W^{1,p}(\Omega))$
for $\tau>0$ and $p>d\geq2$.

\begin{itemize}
    \item \emph{\textbf{Step $(v)$}}: $C^{0,\alpha}([\tau,T];L^{p}(\Omega))$-regularity of $n$.
\end{itemize}

In view of the estimates \eqref{est:GGG} and \eqref{es:regim}, with $c\in L^\infty([\tau,T];W^{2,p}(\Omega))$ and $n\in L^\infty([\tau,T]; W^{1,p})$, we now have $\mathcal{G}\in L^\infty([\tau,T];L^p(\Omega))$. Based on this, to rigorously establish that $n\in C^{0,\alpha}([\tau, T]; L^p(\Omega))$, we choose $\tau \le t_1 < t_2 \le T$ and recall the decomposition of $n(t_2) - n(t_1)$ in \eqref{eq:dedee}.

For $\mathcal{J}_1$, since $t_1 \ge \tau > 0$, the differentiability of $\mathcal{S}_\alpha(t)$ yields a localized Lipschitz bound $\|\mathcal{J}_1\|_{L^p} \le C_\tau |t_2-t_1|$, which algebraically embeds into $C^{0,\alpha}$. For $\mathcal{J}_3$, the standard resolvent estimate $\|\mathcal{P}_\alpha(t)\|_{L^p \to L^p} \le C t^{\alpha-1}$ (see Lemma \ref{Der:Lemma}) coupled with the result $\mathcal{G} \in L^\infty([\tau,T]; L^p(\Omega))$ directly implies $\|\mathcal{J}_3\|_{L^p} \le C \|\mathcal{G}\|_{L^\infty} \int_{t_1}^{t_2} (t_2-s)^{\alpha-1}\,\mathrm{d}s \le C |t_2-t_1|^\alpha$. Finally, for $\mathcal{J}_2$, we express the kernel difference as $\int_{t_1}^{t_2} \partial_\theta \mathcal{P}_\alpha(\theta-s)\,\mathrm{d}\theta$ and utilize the temporal derivative decay estimate $\|\partial_\theta \mathcal{P}_\alpha(\theta-s)\|_{L^p \to L^p} \le C(\theta-s)^{\alpha-2}$. Applying Fubini's theorem to exchange the integration order over the domain yields the sharp bound $\int_{t_1}^{t_2} \int_{0}^{t_1} (\theta-s)^{\alpha-2}\,\mathrm{d}s\,\mathrm{d}\theta = \frac{1}{1-\alpha}\int_{t_1}^{t_2}[(\theta-t_1)^{\alpha-1}-\theta^{\alpha-1} ]\le C|t_2-t_1|^\alpha$, thereby confirming that $\|\mathcal{J}_2\|_{L^p} \le C |t_2-t_1|^\alpha$. Combining these three bounds yields
\begin{equation}\label{eq:nHolder}
\big\|n(t_2)-n(t_1)\big\|_{L^p}\le C_{\tau,T}\big|t_2-t_1\big|^\alpha,
\end{equation}
which implies $n\in C^{0,\alpha}([\tau,T];L^p(\Omega))$.

\begin{itemize}
    \item \emph{\textbf{Step $(vi)$}}: $C^{0,\alpha/2}([\tau,T];W^{2,p}(\Omega))$-regularity of $c$.
\end{itemize}

We now proceed to establish the temporal H\"{o}lder continuity of $c$ in $W^{2,p}(\Omega)$. Recall the mild solution for $c$ in \eqref{sys:mild} that
\begin{displaymath}
\mathcal{A}_\gamma c(t)
= \mathcal{A}_\gamma\mathcal S_\alpha^\gamma(t)c_0 + \bigl(I-\mathcal{S}_\alpha^\gamma(t)\bigr)n(t)
 +\int_0^t \mathcal{A}_\gamma\mathcal{P}_\alpha^\gamma(t-s)\bigl(n(s)-n(t)\bigr)\,\mathrm{d}s.
\end{displaymath}
For convenience, we define the three terms on the right-hand side as $H_1(t)$, $H_2(t)$, and $H_3(t)$, respectively, $H_1(t):=\mathcal{A}_\gamma\mathcal{S}_\alpha^\gamma(t)c_0$, $H_2(t):=(I-\mathcal S_\alpha^\gamma(t))n(t)$, $H_3(t):=\int_0^t\mathcal{A}_\gamma\mathcal{P}_\alpha^\gamma(t-s)(n(s)-n(t))\mathrm{d}s$. Let $\tau\le t_1<t_2\le T$. We estimate the increments of $H_1$, $H_2$, and $H_3$ separately.

For $H_1$, if $t_1\ge\tau>0$, the operator $\mathcal{A}_\gamma\mathcal{S}_\alpha^\gamma(t)$ is smooth in $t$, so $H_1$ is Lipschitz on $[\tau,T]$, which implies $\|H_1(t_2)-H_1(t_1)\|_{L^p}\le C_\tau|t_2-t_1|\,\|c_0\|_{L^p}$.

For $H_2$, we write
\begin{displaymath}
H_2(t_2)-H_2(t_1)
= \bigl(I-\mathcal{S}_\alpha^\gamma(t_2)\bigr)\bigl(n(t_2)-n(t_1)\bigr)
 + \bigl(\mathcal{S}_\alpha^\gamma(t_1)-\mathcal{S}_\alpha^\gamma(t_2)\bigr)n(t_1).
\end{displaymath}
Since $n\in C^{0,\alpha}([\tau,T];L^p(\Omega))$ and $\mathcal{S}_\alpha^\gamma(t)$ is uniformly bounded on $L^p(\Omega)$, it follows that $\|(I-\mathcal{S}_\alpha^\gamma(t_2))(n(t_2)-n(t_1))\|_{L^p}\le C|t_2-t_1|^\alpha$. Moreover, arguing as above,
\begin{align*}
\big\|\bigl(\mathcal{S}_\alpha^\gamma(t_1)-\mathcal{S}_\alpha^\gamma(t_2)\bigr)n(t_1)\big\|_{L^p}
\le\int_{t_1}^{t_2}\big\|\mathcal{A}_\gamma\mathcal{P}_\alpha^\gamma(r)\big\|_{L^p\to L^p}\,\mathrm{d}r\,\big\|n(t_1)\big\|_{L^p}\le C_\tau\big|t_2-t_1\big|.
\end{align*}
Therefore,
\begin{displaymath}
\big\|H_2(t_2)-H_2(t_1)\big\|_{L^p}\le C\big|t_2-t_1\big|^\alpha.
\end{displaymath}

We now proceed to estimate $H_3$. To this end, we decompose the increment as $H_3(t_2)-H_3(t_1)=I_1+I_2$, where
\begin{displaymath}
I_2=\int_{t_1}^{t_2}\mathcal{A}_\gamma\mathcal{P}_\alpha^\gamma(t_2-s)\bigl(n(s)-n(t_2)\bigr)\,\mathrm{d}s,
\end{displaymath}
and
\begin{displaymath}
I_1=\int_0^{t_1}\left[\mathcal{A}_\gamma\mathcal{P}_\alpha^\gamma(t_2-s)\bigl(n(s)-n(t_2)\bigr)
   -\mathcal{A}_\gamma\mathcal{P}_\alpha^\gamma(t_1-s)\bigl(n(s)-n(t_1)\bigr)\right]\,\mathrm{d}s.
\end{displaymath}
For the term $I_2$, we employ the H\"{o}lder continuity of $n$ to obtain
\begin{align*}
\big\|I_2\big\|_{L^p}
&\le C \int_{t_1}^{t_2} (t_2-s)^{-1}\big\|n(s)-n(t_2)\big\|_{L^p}\,\mathrm{d}s \\
&\le C \int_{t_1}^{t_2} (t_2-s)^{\alpha-1}\,\mathrm{d}s \le C\big|t_2-t_1\big|^\alpha.
\end{align*}
To estimate $I_1$,  we further split it as $I_1=I_{11}+I_{12}$, where
\begin{displaymath}
\left\{
\begin{aligned}
I_{11}&=\int_0^{t_1}\bigl(\mathcal{A}_\gamma\mathcal{P}_\alpha^\gamma(t_2-s)
       -\mathcal{A}_\gamma\mathcal{P}_\alpha^\gamma(t_1-s)\bigr)\bigl(n(s)-n(t_1)\bigr)\,\mathrm{d}s, \\
I_{12}&=\int_0^{t_1}\mathcal{A}_\gamma\mathcal{P}_\alpha^\gamma(t_2-s)\bigl(n(t_1)-n(t_2)\bigr)\,\mathrm{d}s.
\end{aligned}\right.
\end{displaymath}
For $I_{11}$, applying the fundamental theorem of calculus yields
\begin{displaymath}
\mathcal{A}_\gamma\mathcal{P}_\alpha^\gamma(t_2-s)-\mathcal{A}_\gamma\mathcal{P}_\alpha^\gamma(t_1-s)
=\int_{t_1}^{t_2}\partial_r\bigl(\mathcal{A}_\gamma\mathcal{P}_\alpha^\gamma(r-s)\bigr)\,\mathrm{d}r.
\end{displaymath}
By Lemma \ref{Der:Lemma}, we have $\|\partial_r(\mathcal{A}_\gamma\mathcal{P}_\alpha^\gamma(r-s))\|_{L^p\to L^p}\le C(r-s)^{-2}$. Therefore,
\begin{align*}
\|I_{11}\|_{L^p}
&\le C \int_{t_1}^{t_2}\int_0^{t_1} (r-s)^{-2}\big\|n(s)-n(t_1)\big\|_{L^p}\,\mathrm{d}s\,\mathrm{d}r \\
&\le C \int_{t_1}^{t_2}\int_0^{t_1} (r-s)^{-2}(t_1-s)^\alpha\,\mathrm{d}s\,\mathrm{d}r.
\end{align*}
By introducing the change of variables $u=t_1-s$, we have $r-s=(r-t_1)+u$. This allows us to estimate the inner integral as follows
\begin{align*}
\int_0^{t_1} (r-s)^{-2}(t_1-s)^\alpha\,\mathrm{d}s
&= \int_0^{t_1} \bigl((r-t_1)+u\bigr)^{-2} u^\alpha\,\mathrm{d}u \\
&\le \int_0^\infty \bigl((r-t_1)+u\bigr)^{-2} u^\alpha\,\mathrm{d}u
= C(r-t_1)^{\alpha-1}.
\end{align*}
Consequently,
\begin{displaymath}
\big\|I_{11}\big\|_{L^p} \le C \int_{t_1}^{t_2} (r-t_1)^{\alpha-1}\,\mathrm{d}r
\le C|t_2-t_1|^\alpha.
\end{displaymath}
For $I_{12}$, we use the identity $\int_0^{t_1}\mathcal{A}_\gamma\mathcal{P}_\alpha^\gamma(t_2-s)\mathrm{d}s
=\mathcal{S}_\alpha^\gamma(t_2-t_1)-\mathcal{S}_\alpha^\gamma(t_2)$ to explicitly evaluate it as $I_{12}=\left(\mathcal{S}_\alpha^\gamma(t_2-t_1)-\mathcal{S}_\alpha^\gamma(t_2)\right)\left(n(t_1)-n(t_2)\right)$. Since the operator $\mathcal{S}_\alpha^\gamma(t)$ is uniformly bounded on $L^p(\Omega)$, we obtain the estimate
\begin{displaymath}
\big\|I_{12}\big\|_{L^p}
\le C\big\|n(t_1)-n(t_2)\big\|_{L^p}
\le C\big|t_2-t_1\big|^\alpha.
\end{displaymath}
Combining the bounds for $I_{11}$, $I_{12}$, and $I_2$, we conclude that
\begin{displaymath}
\big\|H_3(t_2)-H_3(t_1)\big\|_{L^p} \le C\big|t_2-t_1\big|^\alpha.
\end{displaymath}

Furthermore, aggregating the estimates for $H_1$, $H_2$, and $H_3$ yields
\begin{displaymath}
\big\|\mathcal{A}_\gamma c(t_2)-\mathcal{A}_\gamma c(t_1)\big\|_{L^p} \le C\big|t_2-t_1\big|^\alpha,
\end{displaymath}
which implies that $\mathcal A_\gamma c \in C^{0,\alpha/2}\bigl([\tau,T];L^p(\Omega)\bigr)$. On the other hand,  it has been established in Step $(iv)$ that $c\in C^{0,\alpha/2}\bigl([\tau,T];W^{1,p}(\Omega)\bigr)$. Leveraging the continuous embedding $W^{1,p}(\Omega) \hookrightarrow L^p(\Omega)$, we deduce
\begin{displaymath}
\big\|c(t_2)-c(t_1)\big\|_{L^p} \le C\big|t_2-t_1\big|^{\alpha/2}.
\end{displaymath}

Finally, utilizing the norm equivalence $\|u\|_{W^{2,p}(\Omega)} \le C\bigl(\|\mathcal{A}_\gamma u\|_{L^p(\Omega)}+\|u\|_{L^p(\Omega)}\bigr)$ for $u\in D(\mathcal{A}_\gamma)$, we obtain
\begin{align*}
\big\|c(t_2)-c(t_1)\big\|_{W^{2,p}}
&\le C\big\|\mathcal A_\gamma(c(t_2)-c(t_1))\big\|_{L^p} + C\big\|c(t_2)-c(t_1)\big\|_{L^p}
\le C\big|t_2-t_1\big|^{\alpha/2}.
\end{align*}
Hence, $c\in C^{0,\alpha}\bigl([\tau,T];W^{2,p}(\Omega)\bigr)$. In particular, $D^2c \in C^{0,\alpha/2}\bigl([\tau,T];L^p(\Omega)\bigr)$. This completes the proof.

\end{proof}

\begin{lemma}[Improved spatial regularity]\label{Lem:Rnnn}
Let $T>0$ be appropriately small and let $(n,c)$ be the local mild solution of  System \eqref{Pro:FKS}. Then $n(\mathbf{x},t)\in C^{0,\alpha/2}((0,T];W^{1,p}(\Omega))\cap C((0,T];D(\mathcal{A}))$ and $c(\mathbf{x},t)\in C((0,T];D(\mathcal{A}))$ for $p>d\geq2$.
\end{lemma}

\begin{proof}
Let $\mathcal{A}:=-\Delta$ be the Neumann Laplacian on $L^p(\Omega)$. To establish the desired temporal continuity in $D(\mathcal{A})$, it suffices to show that $n$, $c\in C([\tau,T]; D(\mathcal{A}))$ for every $\tau>0$.

The spatial regularity of $c$ follows directly from Lemma~\ref{Lem:timR}, which guarantees that $c\in C^{0,\alpha}([\tau,T];W^{2,p}(\Omega))\subset C([\tau,T];W^{2,p}(\Omega))$. Since $c$ satisfies the Neumann boundary condition, $c(t)\in D(\mathcal{A}_\gamma)=D(\mathcal{A})$. The equivalence between the graph norm of $D(\mathcal{A})$ and the $W^{2,p}(\Omega)$ norm, then yields $c\in C([\tau,T]; D(\mathcal{A}))$.

For the regularity of $n$, we use the mild formulation
\begin{displaymath}
n(t)=\mathcal{S}_\alpha(t)n_0+\int_0^t\mathcal{P}_\alpha(t-s)\mathcal{G}(s)\mathrm{d}s,
\end{displaymath}
where $\mathcal{G}:=\chi\big(\frac{1}{c}\nabla n\cdot\nabla c-\frac{n}{c^2}|\nabla c|^2+\frac{n}{c}\Delta c\big)$. From Lemma~\ref{Lem:timR}, we know $\mathcal{G}\in L^\infty((0,T];L^p(\Omega))$. Fix $\tau>0$ and introduce a strict time-truncation parameter $\delta:=\tau/2> 0$. The proof proceeds in the following sequential stages.

\begin{itemize}
    \item \emph{\textbf{Step $(i)$}}: The H\"{o}lder continuity of $n$ bounded away from $t=0$.
\end{itemize}

To this end, we estimate the regularity of $n(t)$ for $t\in[\delta,T]$. Since $t\ge\delta>0$, the initial value term satisfies $\mathcal{S}_\alpha(t)n_0\in C^\infty([\delta,T];D(\mathcal{A}))$. For the convolution integral, utilizing the estimate $\|\mathcal{A}^{1/2}\mathcal{P}_\alpha(t)\|_{L^p\to L^p}\le C t^{\alpha/2-1}$ from Lemma \ref{Der:Lemma} and arguing exactly as in Step $(iii)$ of Lemma~\ref{Lem:timR}, we rigorously deduce that $\mathcal{A}^{1/2}n\in C^{0,\alpha/2}([\delta,T];L^p(\Omega))$. Therefore, via domain characterization, we obtain $n\in C^{0,\alpha/2}([\delta,T];W^{1,p}(\Omega))$.

\begin{itemize}
    \item \emph{\textbf{Step $(ii)$}}: The H\"{o}lder continuity of the forcing term $\mathcal{G}$.
\end{itemize}

Since $p>d\ge 2$, the embeddings $W^{1,p}(\Omega) \hookrightarrow L^\infty(\Omega)$ and $W^{2,p}(\Omega)\hookrightarrow W^{1,\infty}(\Omega)$ hold. We collect the regularities on the restricted interval $[\delta,T]$: $c\in C^{0,\alpha}([\delta,T];W^{2,p}(\Omega))$, which gives $\Delta c\in C^{0,\alpha}([\delta,T];L^p(\Omega))$ and $\nabla c$, $c^{-1}\in C^{0,\alpha}([\delta,T];L^\infty(\Omega))$. $n\in C^{0,\alpha/2}([\delta,T]; W^{1,p}(\Omega))$, which gives $\nabla n\in C^{0,\alpha/2}([\delta,T];L^p(\Omega))$ and $n\in C^{0,\alpha/2}([\delta,T];L^\infty(\Omega))$. Applying standard product rules for H\"{o}lder spaces, we deduce that each nonlinear component of $\mathcal{G}$ is H\"{o}lder continuous. Consequently $\mathcal{G}\in C^{0,\alpha/2}([\delta,T];L^p(\Omega))$.

\begin{itemize}
    \item \emph{\textbf{Step $(iii)$}}: The continuity of $\mathcal{A} n$.
\end{itemize}

Define $\mathcal{J}(t)=\int_0^t \mathcal{P}_\alpha(t-s)\mathcal{G}(s)\mathrm{d}s$. To rigorously justify that $\mathcal{J}(t)\in D(\mathcal{A})$ and $\mathcal{A}\mathcal{J}(t)$ is continuous, we split $\mathcal{J}(t)$ into a regular history part and a singular local part, namely
\begin{equation}
\mathcal{J}(t)
=\int_0^{t-\delta} \mathcal{P}_\alpha(t-s)\mathcal{G}(s)\mathrm{d}s
+\int_{t-\delta}^t\mathcal{P}_\alpha(t-s)\mathcal{G}(s)\mathrm{d}s:=\mathcal{J}_1(t)+\mathcal{J}_2(t).
\end{equation}
For $\mathcal{J}_1(t)$, the temporal variable $s$ is strictly separated from $t$ such that $t-s\ge\delta>0$. In this regime, $\mathcal{A}\mathcal{P}_\alpha(t-s)$ is uniformly bounded and differentiable. Thus, $\mathcal{A}$ is allowed to pass through the integral, yielding
\begin{displaymath}
\mathcal{I}_1(t):=\mathcal{A}\mathcal{J}_1(t)\in C([\tau,T];L^p(\Omega)).
\end{displaymath}
For $\mathcal{J}_2(t)$, the operator $\mathcal{A}$ cannot be directly pulled inside the singular integral. Instead, we rewrite $\mathcal{J}_2(t)$ by adding and subtracting $\mathcal{G}$, i.e.,
\begin{displaymath}
\mathcal{J}_2(t)
=\int_{t-\delta}^t \mathcal{P}_\alpha(t-s)\big(\mathcal{G}(s)-\mathcal{G}(t)\big)\mathrm{d}s
 +\Bigg(\int_0^\delta \mathcal{P}_\alpha(r)\mathrm{d}r\Bigg)\mathcal{G}(t).
\end{displaymath}
Applying $\mathcal{A}$ to this expression is now mathematically well-defined. Using the fundamental relation $\frac{d}{dt}\mathcal{S}_\alpha(t)=-\mathcal{A}\mathcal{P}_\alpha(t)$, we have $\mathcal{A}\int_0^\delta \mathcal{P}_\alpha(r)\mathrm{d}r=I-\mathcal{S}_\alpha(\delta)$. Therefore, we obtain
\begin{displaymath}
\mathcal{I}_2(t)
:= \mathcal{A}\mathcal{J}_2(t)
= \int_{t-\delta}^t \mathcal{A}\mathcal{P}_\alpha(t-s)\big(\mathcal{G}(s)-\mathcal{G}(t)\big)\mathrm{d}s
 + \big(I-\mathcal{S}_\alpha(\delta)\big)\mathcal{G}(t).
\end{displaymath}
The boundary term $(I-\mathcal{S}_\alpha(\delta))\mathcal{G}(t)$ is continuous since $\mathcal{G}\in C([\delta,T];L^p)$ and $I-\mathcal{S}_\alpha(\delta)$ is a bounded operator. For the convolution integral, since both $s$ and $t$ lie within $[\delta,T]$, we apply the H\"{o}lder continuity of $\mathcal{G}$ established as before, and thus there is
\begin{displaymath}
\left\|\mathcal{A}\mathcal{P}_\alpha(t-s)\big(\mathcal{G}(s)-\mathcal{G}(t)\big)\right\|_{L^p}
\le C(t-s)^{-1}|t-s|^\gamma=C(t-s)^{\gamma-1}.
\end{displaymath}
Since $(t-s)^{\gamma-1}\in L^1(0,\delta)$ for $\gamma>0$, the integral converges absolutely. Moreover, the integrand depends continuously on $t$ and is dominated by the integrable function $C(t-s)^{\gamma-1}$. Hence the continuity of the integral follows from the Lebesgue dominated convergence theorem. Thus, $\mathcal{I}_2\in C([\tau, T]; L^p(\Omega))$.

Combining the above estimates yields $\mathcal{A}n(t)\in C([\tau,T];L^p(\Omega))$. Since the mild formulation already ensures $n\in C([\tau,T];L^p(\Omega))$, it follows from the graph norm characterization of $D(\mathcal{A})$ that $n\in C([\tau,T];D(\mathcal{A}))$ for every $\tau>0$. Therefore, we conclude $n\in C((0,T];D(\mathcal{A}))$. The proof is complete.
\end{proof}
\begin{theorem}\label{Thm:Lwell-posed}
Let $\Omega\subset\mathbb{R}^{d}$ ($d\geq2$) be a bounded domain with smooth boundary $\partial\Omega$. If the initial data $n_0(\mathbf{x})$ and $c_0(\mathbf{x})$ satisfy the conditions specified in \eqref{Eq:Assuption}, then for $p>d$, there exists $T>0$ properly small such that Problem \eqref{Pro:FKS} admits a unique local mild solution $(n,c)(\mathbf{x},t)$ satisfying
\begin{displaymath}
\left\{
\begin{aligned}
&n(\mathbf{x},t)\in C\big([0,T];L^{\infty}(\Omega)\big)\cap C^{0,\alpha/2}\big((0,T];W^{1,p}(\Omega)\big)\cap C\big((0,T];D(\mathcal{A})\big), \\
&c(\mathbf{x},t)\in C\big([0,T];W^{1,p}(\Omega)\big)\cap C^{0,\alpha/2}\big((0,T];W^{2,p}(\Omega)\big)\cap C\big((0,T];D(\mathcal{A})\big).
\end{aligned}\right.
\end{displaymath}
Furthermore, the solution components $n$ and $c$ satisfy $n(\cdot,t)\geq0$ and $c(\cdot,t)>0$ for all $t>0$, respectively.
\end{theorem}

\begin{proof}
By Theorem \ref{Lem:WelPod1}, problem \eqref{Pro:FKS} admits a unique local mild solution $(n,c)$ on the time interval $[0,T]$ for some $T>0$.
Higher-order temporal and spatial regularity is subsequently established via a bootstrap argument in Lemmas \ref{Lem:timR} and \ref{Lem:Rnnn}. What's more, the non-negativity of $n(\mathbf{x},t)$ and the strict positivity of $c(\mathbf{x},t)$ are rigorously guaranteed by Lemmas \ref{lem:bound-n} and \ref{lem:regularity-c}, respectively.
\end{proof}

\section{Global Well-Posedness of Mild Solutions for Two- and Three-Dimensional Systems}
\label{Sec:CLS}

\subsection{A novel Lyapunov functional}\label{Sec:sssCLS}

We establish a conditional energy functional method to prove the global boundedness of solutions to system \eqref{Pro:FKS}. Motivated by previous works on chemotaxis systems (see, e.g., \refcite{Chen21,Dai22,Jin20,Pang21}), the core idea of this approach lies in constructing and analyzing a carefully designed time-dependent energy functional whose dissipation controls the growth of the solution, thereby preventing finite-time blow-up. However, the classical logarithmic sensitivity model is not directly applicable to Problem \eqref{Pro:FKS}, and the analysis is further complicated by the lack of available theoretical results for time-nonlocal KS models with logarithmic sensitivity.

To address the aforementioned difficulties, we adopt the strategy of Ref.~\refcite{Quan20} for time-fractional phase-field equations, which resolves the issue of constructing a dissipation-preserving energy functional for Caputo fractional derivatives. Inspired by their use of a time-weighted average of the classical energy, we introduce the following novel Lyapunov functional for our problem:
\begin{equation}\label{F:Lyapunov}
\mathcal{E}[(n,c)](t):
=\int_{0}^{t}(t-s)^{\alpha-1}\mathcal{F}(s)\mathrm{d}s,\quad t>0,
\end{equation}
where $\mathcal{F}(\cdot)$ is a jointly convex function defined by
\begin{equation}\label{Fun:JCF}
\mathcal{F}(t):=\mathcal{F}[(n,c)](t)
=\underbrace{\int_{\Omega}n\log\left(\frac{n}{\bar{n}}\right)\mathrm{d}\mathbf{x}}_{:=\mathcal{F}_1(n)}
+\underbrace{\theta\int_{\Omega}\frac{|\nabla c|^2}{c}\mathrm{d}\mathbf{x}}_{:=\mathcal{F}_2(c,\nabla c)},
\end{equation}
with $\bar{n}:=\frac{1}{|\Omega|}\int_{\Omega}n\mathrm{d}\mathbf{x}$ and $\theta>0$ a constant. The non-negativity of $\mathcal{F}(t)$ is fundamental to our analysis. Indeed, since $\varphi(n)=n\log n$ is strictly convex, Jensen's inequality gives $\int_{\Omega}n\log n\mathrm{d}\mathbf{x}\geq|\Omega|\bar{n}\log\bar{n}$, which ensures $\mathcal{F}_1(n)\geq 0$. Combined with the obvious non-negativity of the term $\mathcal{F}_2(c,\nabla c)$, we conclude $\mathcal{F}(t)\geq0$ for all $t>0$. The non-negativity of $\mathcal{F}(t)$ on $[0, T]$ ensures that $\mathcal{E}(t)\ge 0$ as a direct consequence of the positivity of the integral kernel $(t-s)^{\alpha-1}$.

\begin{remark}
The Lyapunov functional \eqref{F:Lyapunov} is defined as a time-convolution of the free energy \eqref{Fun:JCF} with the positive kernel $(t-s)^{\alpha-1}$. This structure captures history-dependent energy accumulation associated with the evolution of system \eqref{Pro:FKS}.  It tracks the cumulative effect of the instantaneous free energy over time, rather than relying solely on pointwise-in-time estimates, and provides a convenient framework for establishing boundedness of mild solutions.
\end{remark}

Let $\phi(n):=n\log\left(\frac{n}{\bar{n}}\right)$ for $n\in[0,\infty)$. The functional $\mathcal{F}_1(n)$ is strictly convex on the set $\{n\in L^1(\Omega):n\ge0; \int_{\Omega}n\mathrm{d}\mathbf{x}=\tilde{M}\}$. Defining $\phi(0):=0$ and using $\lim_{n\to0^+}\phi(n)=0$, we obtain the continuity of $\phi$ on $[0,\infty)$. Moreover, $\phi$ is strictly convex on $(0,\infty)$ since $\phi''(n)=1/n>0$. By the fundamental theorem of convex analysis, for any $n_1,n_2\in[0,\infty)$ with $n_1\neq n_2$ and $\lambda\in(0,1)$, the following inequality
\begin{equation}\label{ieq:f11}
\phi\big(\lambda n_1+(1-\lambda)n_2\big)<\lambda\phi(n_1)+(1-\lambda)\phi(n_2)
\end{equation}
holds. Integrating \eqref{ieq:f11} over $\Omega$ and employing the linearity of integrals yields
\begin{displaymath}
\mathcal{F}_1\big(\lambda n_1+(1-\lambda)n_2\big)<\lambda\mathcal{F}_1(n_1)+(1-\lambda)\mathcal{F}_1(n_2),
\end{displaymath}
establishing that $\mathcal{F}_1(n)$ is strictly convex for $n\ge0$.

For $\mathcal{F}_2(c,\nabla c)$, which depends on both $c$ and $\nabla c$, to confirm its strict convexity, we need to verify the convexity of the bivariate function $f(u,\mathbf{v})=\frac{|\mathbf{v}|^2}{u}$, where $u=c>0$ and $\mathbf{v}=\nabla c$. Specifically, define $f:\mathbb{R}^{+}\times\mathbb{R}^d\to\mathbb{R}$ as $f(u,\mathbf{v})=\frac{\sum_{i=1}^d v_i^2}{u}$. Its Hessian matrix $\mathbf{H}_f$ is given by
\begin{displaymath}
\mathbf{H}_f
= \frac{2}{u^3}
\begin{pmatrix}
|\mathbf{v}|^2 & -u v_1 & \dots    & -u v_d \\
-u v_1         & u^2    & \dots    & 0      \\
\vdots         & \vdots & \ddots   & \vdots \\
-u v_d         & 0      & \dots    & u^2
\end{pmatrix}.
\end{displaymath}
For any vector $\mathbf{z}=(a,b_1,\dots,b_d)^T\in\mathbb{R}^{d+1}$, we have
\begin{equation}\label{eq:nonegative}
\mathbf{z}^T\mathbf{H}_f\mathbf{z}
=\frac{2}{u^3}\big(a^2|\mathbf{v}|^2-2au\mathbf{v}\cdot\mathbf{b}+u^2|\mathbf{b}|^2\big)
=\frac{2}{u^3}\big|a\mathbf{v}-u\mathbf{b}\big|^2
\geq0,
\end{equation}
indicating that the Hessian matrix $\mathbf{H}_f$ is positive semi-definite, which implies that $\mathcal{F}_2(c,\nabla c)$ is jointly convex.

Combining the above analyses, we conclude that the functional $\mathcal{F}(t)$ is jointly convex on the admissible set. Moreover, the function $\mathcal{F}(t)$ enjoys the following significant property.

\begin{lemma}[Variational inequality]\label{Lem:varieq}
Let $\mathcal{F}(t)$ be the convex functional defined in \eqref{Fun:JCF}. Then for all $t>0$,
\begin{equation}\label{ieq:varieq}
{_{0}^{C}\mathfrak{D}^{\alpha}_t} \mathcal{F}(t)
\le \int_{\Omega}\frac{\delta\mathcal{F}}{\delta n}{_{0}^{C}\mathfrak{D}^{\alpha}_t}n\,\mathrm{d}\mathbf{x}
   + \int_{\Omega}\frac{\delta\mathcal{F}}{\delta c}{_{0}^{C}\mathfrak{D}^{\alpha}_t}c\,\mathrm{d}\mathbf{x},
\end{equation}
where $\frac{\delta\mathcal{F}}{\delta n}$ and $\frac{\delta\mathcal{F}}{\delta c}$ denote the variational derivatives of $\mathcal{F}$ with respect to $n$ and $c$, respectively.
\end{lemma}

\begin{proof}
By the linearity of the Caputo fractional derivative ${_{0}^{C}\mathfrak{D}^{\alpha}_t}$, it follows that
\begin{displaymath}
{_{0}^{C}\mathfrak{D}^{\alpha}_t}\mathcal{F}(t)
= {_{0}^{C}\mathfrak{D}^{\alpha}_t}\mathcal{F}_1(t)
  + {_{0}^{C}\mathfrak{D}^{\alpha}_t}\mathcal{F}_2(t).
\end{displaymath}
To obtain the result of \eqref{ieq:varieq}, we estimate each term separately.

For $\mathcal{F}_1$, by employing the extended fractional convex inequality (Lemma 2.12 in Ref.~\refcite{Jin21book}) and invoking the strict convexity of $\mathcal{F}_1(t)$, we obtain, for all $t>0$,
\begin{equation}\label{ieq:AF1}
\begin{aligned}
{_{0}^{C}\mathfrak{D}^{\alpha}_t}\int_{\Omega}n\log\left(\frac{n}{\bar{n}}\right)\mathrm{d}\mathbf{x}
\le\int_{\Omega}\left(\log\left(\frac{n}{\bar{n}}\right)+1\right){_{0}^{C}\mathfrak{D}^{\alpha}_t}n\,\mathrm{d}\mathbf{x}
=\int_{\Omega}\frac{\delta \mathcal{F}_1}{\delta n} {_{0}^{C}\mathfrak{D}^{\alpha}_t}n\,\mathrm{d}\mathbf{x},
\end{aligned}
\end{equation}
with $\frac{\delta \mathcal{F}_1}{\delta n}=\log\left(\frac{n}{\bar{n}}\right)+1=\frac{\delta \mathcal{F}}{\delta n}$. The final equality follows from the homogeneous Neumann boundary condition ($\mathbf{J}\cdot \mathbf{\nu}=0$, which indicates that there is no matter entering or leaving the boundary). By Gauss's divergence theorem, we have  $\int_{\Omega}{_{0}^{C}\mathfrak{D}^{\alpha}_t}n\mathrm{d}\mathbf{x}= -\int_{\Omega}\nabla\cdot\mathbf{J}\mathrm{d}\mathbf{x}
=-\int_{\partial\Omega}\mathbf{J}\cdot\mathbf{\nu}\mathrm{d}\sigma=0$, with the flux $\mathbf{J}$ defined as $\mathbf{J}:=-\mathcal{D}\nabla n+\mathcal{D}\chi\frac{n}{c}\nabla c$. This is also another characterization of the law of conservation of mass.

For $\mathcal{F}_2$, the presence of spatial gradients prevents a direct application of convex inequality. To preserve the proof structure and address this issue, we first introduce the Bregman’s distance (see, e.g., Ref.~\refcite{Bregman67}) to manage the jointly convex function $\mathcal{F}_2(c,\nabla c)$. Define $f(t)=f(c(t),\nabla c(t)):=|\nabla c|^2/c$. At any time $t$, for past moments $s$ and a fixed point $\mathbf{x}\in\Omega\subset\mathbb{R}^{d}$, by virtue of the joint convexity of $f(c,\nabla c)$ (see \eqref{eq:nonegative}), the following exact identity holds
\begin{equation}\label{eq:Bregman}
\begin{aligned}
f(c(s),\nabla c(s))
&=f(c(t),\nabla c(t))+\frac{\partial f}{\partial c}(t)\big[c(s)-c(t)\big]   \\
&\quad+\frac{\partial f}{\partial(\nabla c)}(t)\cdot\big[\nabla c(s)-\nabla c(t)\big]+\mathfrak{B}(s,t),
\end{aligned}
\end{equation}
where $\mathfrak{B}(s,t)\ge0$ (for $s\in[0,t]$) represents the Bregman distance associated with $f$ when $f$ is jointly convex. Note that $\mathfrak{B}(s,t)=0$ if and only if $(c(s),\nabla c(s))=(c(t),\nabla c(t))$.

Using the equivalent representation of the Caputo derivative (see, e.g., Ref. \refcite{Jin21book})
\begin{displaymath}
{_{0}^{C}\mathfrak{D}^{\alpha}_t}f(t)
= \frac{f(t)-f(0)}{\Gamma(1-\alpha)t^\alpha}
  +\frac{\alpha}{\Gamma(1-\alpha)}\int_0^t\frac{f(t)-f(s)}{(t-s)^{\alpha+1}}\mathrm{d}s,
\end{displaymath}
and substituting the expression for $f(t)-f(s)$ derived from \eqref{eq:Bregman} into the pointwise derivative, we obtain
\begin{equation}\label{eq:Bregman2}
\begin{aligned}
{_{0}^{C}\mathfrak{D}^{\alpha}_t}f(t)
&= \underbrace{\frac{\partial f(t)}{\partial c}\left[\frac{c(t)-c(0)}{\Gamma(1-\alpha)t^\alpha}
+\frac{\alpha}{\Gamma(1-\alpha)}\int_0^t\frac{c(t)-c(s)}{(t-s)^{\alpha+1}}\mathrm{d}s\right]}_{\frac{\partial f}{\partial c}\cdot{_{0}^{C}\mathfrak{D}^{\alpha}_t} c} \\
&\quad +\underbrace{\sum_{i=1}^d\frac{\partial f(t)}{\partial(\partial_{x_i}c)}
\left[\frac{\partial_{x_i}c(t)-\partial_{x_i}c(0)}{\Gamma(1-\alpha)t^\alpha}
      +\frac{\alpha}{\Gamma(1-\alpha)}\int_0^t\frac{\partial_{x_i}c(t)-\partial_{x_i}c(s)}{(t-s)^{\alpha+1}}\mathrm{d}s\right] }_{\frac{\partial f}{\partial\nabla c}\cdot{_{0}^{C}\mathfrak{D}^{\alpha}_t}\nabla c}\\
&\quad-\underbrace{\left[ \frac{\mathfrak{B}(0,t)}{\Gamma(1-\alpha)t^\alpha}
+\frac{\alpha}{\Gamma(1-\alpha)}\int_0^t \frac{\mathfrak{B}(s,t)}{(t-s)^{\alpha+1}}\mathrm{d}s\right]}_{\mathcal{R}(t)\ge0}  \\
&=\frac{\partial f}{\partial c}{_{0}^{C}\mathfrak{D}^{\alpha}_t}c +\frac{\partial f}{\partial\nabla c}\cdot{_{0}^{C}\mathfrak{D}^{\alpha}_t}\nabla c-\mathcal{R}(\mathbf{x},t).
\end{aligned}
\end{equation}
Since the kernel $(t-s)^{-\alpha}$ is independent of the spatial variable $\mathbf{x}$, by Leibniz's integral rule and $c$ is $C^1$ in $\mathbf{x}$, the gradient operator commutes with the time fractional derivative operator, i.e., ${_0^C\mathfrak{D}^\alpha_t}\nabla c=\nabla{_0^C\mathfrak{D}^\alpha_t}c$. Integrating both side of \eqref{eq:Bregman2} over $\Omega$ yields
\begin{displaymath}
{_{0}^{C}\mathfrak{D}^{\alpha}_t}\mathcal{F}_2
=\theta\int_{\Omega}\frac{\partial f}{\partial c}{_{0}^{C}\mathfrak{D}^{\alpha}_t}c\,\mathrm{d}\mathbf{x}
 +\theta\int_{\Omega}\frac{\partial f}{\partial \nabla c}\cdot \nabla ({_{0}^{C}\mathfrak{D}^{\alpha}_t}c)\,\mathrm{d}\mathbf{x}
 -\theta\int_{\Omega}\mathcal{R}(x, t)\mathrm{d}\mathbf{x}.
\end{displaymath}
Applying Green's formula (integration by parts) with the homogeneous Neumann condition, we get
\begin{displaymath}
{_{0}^{C}\mathfrak{D}^{\alpha}_t}\mathcal{F}_2
= \int_{\Omega}\left[\theta\left(-\frac{|\nabla c|^2}{c^2}-2\nabla\cdot\frac{\nabla c}{c}\right)\right] {_{0}^{C}\mathfrak{D}^{\alpha}_t}c\,\mathrm{d}\mathbf{x}-\theta\int_{\Omega}\mathcal{R}(\mathbf{x},t)\mathrm{d}\mathbf{x}.
\end{displaymath}
Recalling that the variational derivative is given by $\frac{\delta\mathcal{F}_2}{\delta c}=\theta\big(\frac{\partial f}{\partial c}-\nabla\cdot\frac{\partial f}{\partial\nabla c}\big)=\theta\big(-\frac{|\nabla c|^2}{c^2}-2\nabla\cdot\frac{\nabla c}{c}\big)=\frac{\delta\mathcal{F}}{\delta c}$, and noting the non-negativity of the remainder term $\int_{\Omega}\mathcal{R}(\mathbf{x},t)\mathrm{d}\mathbf{x}\ge0$, it follows that
\begin{equation}\label{ieq:AF3}
{_{0}^{C}\mathfrak{D}^{\alpha}_t}\mathcal{F}_2
\leq \int_{\Omega}\frac{\delta \mathcal{F}_2}{\delta c}\,{_{0}^{C}\mathfrak{D}^{\alpha}_t}c\,\mathrm{d}\mathbf{x},\quad t>0.
\end{equation}
Combining \eqref{ieq:AF1} and \eqref{ieq:AF3} yields the desired inequality \eqref{ieq:varieq}. This completes the proof of the lemma.
\end{proof}

Crucially, Lemma \ref{Lem:varieq} establishes a variational inequality \eqref{ieq:varieq} that bridges the fractional energy law and the system's long-term decay. This result is indispensable for proving the convergence of the Lyapunov functional and represents one of the central technical innovations of this work.

\begin{lemma}\label{lem:dissapation}
Let $\Omega\subset \mathbb{R}^d$ ($d\geq2$) be a bounded domain with smooth boundary, and $(n,c)$ be a solution of System \eqref{Pro:FKS}. For the convex functional $\mathcal{F}(t)$  defined in \eqref{Fun:JCF}, if $0<\chi<1/2$, then it holds that ${_{0}^{C}\mathfrak{D}^{\alpha}_t}\mathcal{F}(t)<0$ for all $t>0$. Furthermore, there is a positive constant $\lambda>0$ such that the dissipation estimate
\begin{equation}\label{est:dissap}
{_{0}^{C}\mathfrak{D}^{\alpha}_t} \mathcal{F}(t)
\leq-\lambda\mathcal{F}(t),\quad ~ t>0,
\end{equation}
holds, which implies
\begin{equation}
\mathcal{F}(t)\leq \mathcal{F}(0) E_\alpha\big(-\lambda t^\alpha\big)\leq \mathcal{F}(0)<\infty,\quad t>0.
\end{equation}
\end{lemma}

\begin{proof}
For brevity, we denote the variational derivatives by $\mu_n:=\frac{\delta\mathcal{F}}{\delta n}$ and $\mu_c:=\frac{\delta\mathcal{F}}{\delta c}$. By invoking Lemma \ref{Lem:varieq} and substituting the dynamics from System \eqref{Pro:FKS} into the variational inequality \eqref{ieq:varieq}, followed by integration by parts, results in
\begin{equation}\label{eq:KKKt}
{_{0}^{C}\mathfrak{D}^{\alpha}_t}\mathcal{F}
\le \big\langle\mu_n,{_{0}^{C}\mathfrak{D}^{\alpha}_t}n\big\rangle
    +\big\langle\mu_c,{_{0}^{C}\mathfrak{D}^{\alpha}_t}c\big\rangle
:=\mathcal{I}_n(t)+\mathcal{I}_c(t).
\end{equation}
A direct computation yields
\begin{align*}
\mathcal{I}_n(t)
&:= \int_{\Omega}
\left(\log\left(\frac{n}{\bar n}\right)+1\right){_{0}^{C}\mathfrak{D}^{\alpha}_t}n\,\mathrm{d}\mathbf{x} \\
&= \int_{\Omega}\log\left(\frac{n}{\bar n}\right)\,\Big[\mathcal{D}\Delta n
   -\mathcal{D}\chi\nabla\cdot\left(\frac{n}{c}\nabla c\right)\Big]\,\mathrm{d}\mathbf{x}
\qquad\quad\Big(\because\int_\Omega{_{0}^{C}\mathfrak{D}^{\alpha}_t}n\,\mathrm{d}\mathbf{x}=0\Big)\\
&= -\mathcal{D}\int_{\Omega}\nabla\log n\cdot\nabla n\,\mathrm{d}\mathbf{x}
   +\mathcal{D}\chi\int_{\Omega}\nabla\log n\cdot\frac{n}{c}\nabla c\,\mathrm{d}\mathbf{x}\\
&= -4\mathcal{D}\int_{\Omega}\big|\nabla\sqrt n\big|^2\,\mathrm d\mathbf{x}
   +2\mathcal{D}\chi\int_{\Omega}\frac{\sqrt n}{c}\nabla\sqrt n\cdot\nabla c\,\mathrm{d}\mathbf{x}.
\end{align*}
Similarly,
\begin{displaymath}
\begin{aligned}
\mathcal{I}_c(t)
&:=\int_{\Omega}\theta\left(-2\frac{\Delta c}{c}+\frac{|\nabla c|^2}{c^2}\right)
     \big[\mathcal{D}\Delta c-\gamma c+n\big]\mathrm{d}\mathbf{x}\\
&=-2\theta \mathcal{D}\int_{\Omega}\frac{(\Delta c)^2}{c}\mathrm{d}\mathbf{x}
  +\theta\mathcal{D}\int_{\Omega}\frac{|\nabla c|^2\Delta c}{c^2}\mathrm{d}\mathbf{x}
  -\gamma\theta\int_{\Omega}\frac{|\nabla c|^2}{c}\mathrm{d}\mathbf{x} \\
&\quad +4\theta \int_{\Omega}\frac{\sqrt{n}\nabla \sqrt{n}\cdot\nabla c}{c}\mathrm{d}\mathbf{x}
  -\theta\int_{\Omega}\frac{n |\nabla c|^2}{c^2}\mathrm{d}\mathbf{x}.
\end{aligned}
\end{displaymath}
Combining $\mathcal{I}_n$ and $\mathcal{I}_c$, we get
\begin{equation}\label{eq:KKIt}
\begin{aligned}
\mathcal{I}(t)
=&\underbrace{-4\mathcal{D}\int_{\Omega}\big|\nabla \sqrt{n}\big|^2\mathrm{d}\mathbf{x}
+2(\mathcal{D}\chi+2\theta)\int_{\Omega}\frac{\sqrt{n}\nabla \sqrt{n}\cdot\nabla c}{c}\mathrm{d}\mathbf{x}
-\theta\int_{\Omega}\frac{n|\nabla c|^2}{c^2}\mathrm{d}\mathbf{x}}_{:=\mathcal{J}_1} \\
& \underbrace{-2\theta\mathcal{D}\int_{\Omega}\frac{(\Delta c)^2}{c}\mathrm{d}\mathbf{x}
+\theta\mathcal{D}\int_{\Omega}\frac{|\nabla c|^2\Delta c}{c^2}\mathrm{d}\mathbf{x}}_{:=\mathcal{J}_2}
-\gamma\theta\int_{\Omega}\frac{|\nabla c|^2}{c}\mathrm{d}\mathbf{x}\\
=:&\mathcal{J}_1+\mathcal{J}_2-\gamma\theta\int_{\Omega}\frac{|\nabla c|^2}{c}\mathrm{d}\mathbf{x}.
\end{aligned}
\end{equation}

Defining $\mathbf{u}(\mathbf{x},t) := \nabla\sqrt{n}$ and $\mathbf{v}(\mathbf{x},t):=\frac{\sqrt{n}\nabla c}{c}$, then the term $\mathcal{J}_1$ can be rewritten as
\begin{displaymath}
\mathcal{J}_1
=\int_{\Omega}\Big[-4\mathcal{D}|\mathbf{u}|^2
+2(\mathcal{D}\chi+2\theta)\mathbf{u}\cdot\mathbf{v}-\theta|\mathbf{v}|^2\Big]\mathrm{d}\mathbf{x}.
\end{displaymath}
The integrand is the quadratic form
\begin{displaymath}
\mathcal{Q}(\mathbf{u},\mathbf{v}):=
\begin{pmatrix}
\mathbf{u}\\
\mathbf{v}\end{pmatrix}^T
\mathbf{A}
\begin{pmatrix}\mathbf{u}\\
\mathbf{v}
\end{pmatrix},
\quad\textrm{with}\quad
\mathbf{A}:=
\begin{pmatrix}-4\mathcal{D} & \mathcal{D}\chi+2\theta \\
\mathcal{D}\chi+2\theta & -\theta
\end{pmatrix}.
\end{displaymath}
By Sylvester's criterion (for $-\mathbf{A}$ positive definite), its negative definite necessary and sufficient condition is that the determinant $\det(\mathbf{A})>0$ (and the first-order principal form $-4\mathcal{D}<0$ obviously true), i.e.,
\begin{equation}\label{ieq:Fesabe}
4\mathcal{D}\theta-(\mathcal{D}\chi+2\theta)^2>0
\iff 4\theta^2-4\mathcal{D}(1-\chi)\theta+\mathcal{D}^2\chi^2<0.
\end{equation}
When $0<\chi<1/2$, the discriminant of the above quadratic inequality is $\Delta=16\mathcal{D}^2(1-2\chi)> 0$, and the solution set of this inequality is exactly $\theta\in(\theta_-,\theta_+)$ with $\theta_{\pm}=\frac{\mathcal{D}}{2}(1-\chi\pm\sqrt{1-2\chi})$. The reason for setting $0<\chi<1/2$ is clearly observed from Fig.~\ref{Fig:chi}. Choosing $\theta$ in this interval, the quadratic form is strongly negative definite, thus, there exists a constant $\mu_0>0$, such that

\begin{equation}\label{ieq:J11}
\mathcal{J}_1
\leq-\mu_0\int_{\Omega}\Bigg(\big|\nabla\sqrt{n}\big|^2+\frac{n\big|\nabla c\big|^2}{c^2}\Bigg)\mathrm{d}\mathbf{x}.
\end{equation}
\begin{figure}[t!]
 \centering
 \centerline{\includegraphics[width=0.86\textwidth]{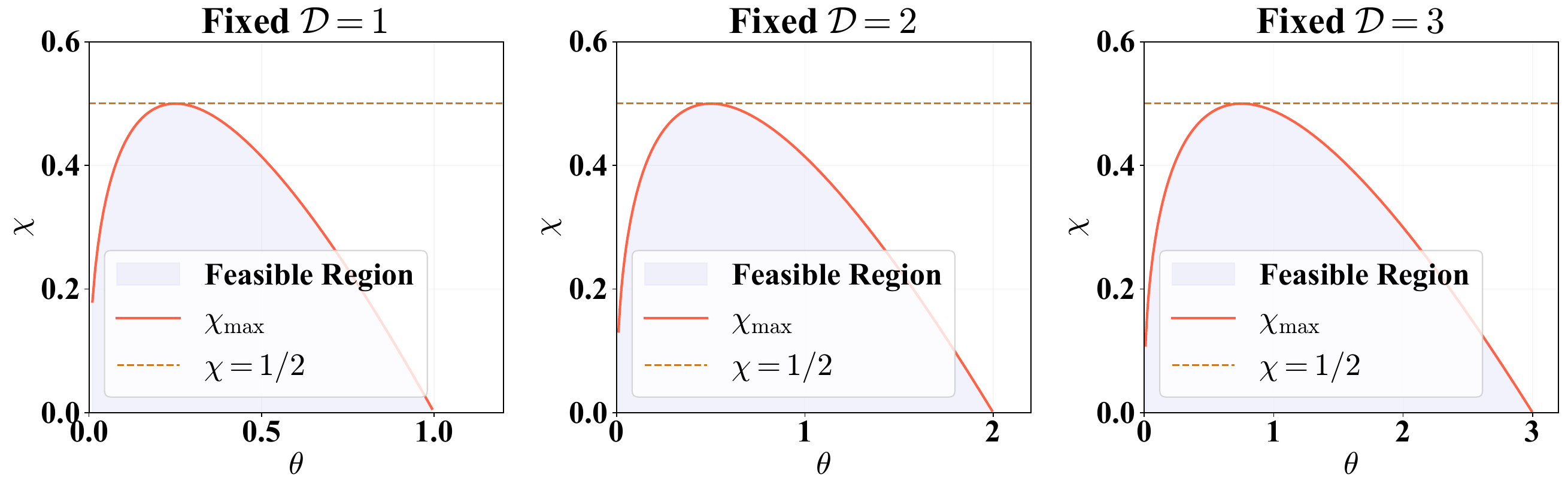}}
 \caption{{\rm(}Color online{\rm)}  Numerical illustration  of the feasible region determined by Problem \eqref{ieq:Fesabe} under various fixed $\mathcal{D}$.}
\label{Fig:chi}
\end{figure}

For $\mathcal{J}_2$, since $\Omega$ is convex and leveraging the lower bound $c\geq C_{\aleph}>0$ from Lemma \ref{lem:LowerBound-c}, we apply the pointwise identity $\frac{1}{2}\Delta|\nabla c|^2=|\nabla^2 c|^2+\nabla c\cdot\nabla(\Delta c)$ (cf. the proof of Lemma 3.1 in Ref. \refcite{Winkler14}) and the fact that $\partial\nu|\nabla c|^2\leq0$ (see Lemma 3.1 in Ref.~\refcite{Lankeit16}) to obtain
\begin{equation}\label{ieq:J21}
\begin{aligned}
\mathcal{J}_2
&= -2\theta\mathcal{D}\int_{\Omega}\frac{(\Delta c)^2}{c}\mathrm{d}\mathbf{x} + \theta\mathcal{D}\int_{\Omega}\frac{|\nabla c|^2\Delta c}{c^2}\mathrm{d}\mathbf{x} \\
&= -2\theta\mathcal{D}\int_{\Omega}c\left|\nabla^2\log c\right|^2\mathrm{d}\mathbf{x}+\theta\mathcal{D}\int_{\partial\Omega}\frac{1}{c}\frac{\partial|\nabla c|^2}{\partial\nu}\mathrm{d}S\leq0.
\end{aligned}
\end{equation}
This finding is further supported by Ref. \refcite{Chen21}, specifically Lemma 3.1. Substituting the estimates \eqref{ieq:J11} and \eqref{ieq:J21} into \eqref{eq:KKKt} or \eqref{eq:KKIt}, and invoking the lower bound of $c$, we arrive at the total energy dissipation:
\begin{equation}\label{eq:total-dissipation}
{_{0}^{C}\mathfrak{D}^{\alpha}_t}\mathcal{F}(t)
\leq -\mu_0\int_{\Omega}\big|\nabla\sqrt{n}\big|^2 \mathrm{d}\mathbf{x}
     -\gamma\theta\int_{\Omega}\frac{|\nabla c|^2}{c} \mathrm{d}\mathbf{x}
     < 0, \quad t>0.
\end{equation}
This confirms that the functional $\mathcal{F}(t)$ is non-increasing along the trajectories of the system.

To rigorously establish the dissipation estimate \eqref{est:dissap}, we utilize the Logarithmic Sobolev Inequality (LSI)\cite{Gross1975}. Since $\Omega\subset \mathbb{R}^d$ is a bounded domain with smooth boundary, $\bar{\Omega}$ can be regard as a compact smooth Riemannian manifold with boundary, endowed with the Euclidean metric. In this geometric setting, the Ricci curvature is zero, and the convexity of $\partial\Omega$ ensures that the second fundamental form of the boundary is non-negative. Following the framework of Deuschel and Stroock Ref.~\refcite{Deuschel90}, Chung and Yau Ref.~\refcite{chung1996}, the Neumann Laplace operator on such a manifold satisfies the logarithmic Sobolev inequality (LSI). Specifically, we define the normalized probability measure $\mathrm{d}\mu=\mathrm{d}\mathbf{x}/|\Omega|$ and the corresponding density $f=n/\bar{n}$, where $\bar{n}=\tilde{M}/|\Omega|$ is the spatial average under mass conservation. Employing the identity $|\nabla\sqrt{f}|^2=\frac{1}{4}\frac{|\nabla f|^2}{f}$, on the set $\{f>0\}$, the LSI is given by
\begin{displaymath}
\int_{\Omega} f\log f \mathrm{d}\mu
\leq \frac{1}{2\rho}\int_{\Omega}\frac{|\nabla f|^2}{f}\mathrm{d}\mu,
\end{displaymath}
with $\rho\geq\min\left\{\frac{\lambda_1}{8e},\frac{1}{d \cdot D(\Omega)^2}\right\}>0$, and where $D(\Omega)$ denotes the diameter of $\Omega$, $\lambda_1$ is the first eigenvalue of the Laplacian, and $e$ denots the Euler constant. To map this to our physical variables, we observe that
\begin{displaymath}
\int_{\Omega}\frac{n}{\bar{n}}\log\left(\frac{n}{\bar{n}}\right)\frac{\mathrm{d}\mathbf{x}}{|\Omega|}
=\frac{1}{\tilde{M}}\int_{\Omega}n\ln\left(\frac{n}{\bar{n}}\right)\mathrm{d}\mathbf{x}
=\frac{\mathcal{F}_1(t)}{\tilde{M}},
\end{displaymath}
and
\begin{displaymath}
\int_{\Omega}\frac{|\nabla(n/\bar{n})|^2}{n/\bar{n}}\frac{\mathrm{d}\mathbf{x}}{|\Omega|}
=\frac{1}{\bar{n}|\Omega|}\int_{\Omega}\frac{|\nabla n|^2}{n}\mathrm{d}\mathbf{x}
=\frac{4}{\tilde{M}}\int_{\Omega}\big|\nabla\sqrt{n}\big|^2\mathrm{d}\mathbf{x}.
\end{displaymath}
Substituting these back into LSI and eliminating the mass $\tilde{M}$, we obtain the functional bound
\begin{displaymath}
\mathcal{F}_1(t)\leq\frac{2}{\rho}\int_{\Omega}\big|\nabla \sqrt{n}\big|^2 \mathrm{d}\mathbf{x},\quad t>0.
\end{displaymath}
Hence, the first dissipation term in the energy law satisfies
\begin{equation}\label{eq:est-F1}
\mu_0 \int_{\Omega}\big|\nabla \sqrt{n}\big|^2 \mathrm{d}\mathbf{x}\geq\tilde{\lambda}\mathcal{F}_1(t),\quad t>0
\end{equation}
where the decay rate $\tilde{\lambda}:=\mu_0\rho/2>0$ is determined by the geometric spectral gap of $\Omega$.

For the second component $\mathcal{F}_2(t)=\theta\int_{\Omega} \frac{|\nabla c|^2}{c}\mathrm{d}\mathbf{x}$,  the linear degradation term directly yields
\begin{equation}\label{eq:est-F2}
\gamma\theta \int_{\Omega}\frac{|\nabla c|^2}{c}\mathrm{d}\mathbf{x}= \gamma\mathcal{F}_2(t).
\end{equation}
Inserting \eqref{eq:est-F1} and \eqref{eq:est-F2} into the energy evolution equation \eqref{eq:total-dissipation}, we arrive at
\begin{equation}
{_{0}^{C}\mathfrak{D}^{\alpha}_t}\mathcal{F}(t)
\leq -\tilde{\lambda}\mathcal{F}_1(t)-\gamma\big(\theta\mathcal{F}_2(t)\big),\quad t>0.
\end{equation}
By setting $\lambda:=\min\{\tilde{\lambda},\gamma\}>0$, the inequality simplifies to
\begin{equation}
{_{0}^{C}\mathfrak{D}^{\alpha}_t}\mathcal{F}(t)
\leq -\lambda\big(\mathcal{F}_1(t)+\theta\mathcal{F}_2(t)\big)
=-\lambda \mathcal{F}(t),\quad t>0.
\end{equation}
By Lemma \ref{Lem:Gronwall}, we conclude that
\begin{equation}
\mathcal{F}(t)\leq \mathcal{F}(0) E_\alpha(-\lambda t^\alpha),\quad t>0,
\end{equation}
where $E_\alpha(\cdot)$ is the Mittag-Leffler function. This result characterizes the algebraic decay of the system toward its equilibrium state, completing the proof.
\end{proof}

\begin{lemma}[Nonlocal Lyapunov dissipation]\label{Lem:FracLyapunov}
Let $\Omega\subset\mathbb{R}^d$ $(d\ge2)$ be a bounded domain with smooth boundary. Suppose $(n, c)$ is the mild global solution of Problem \eqref{Pro:FKS} with $\gamma>0$, $0<\chi<1/2$, and $\alpha\in(0,1)$. Let $\mathcal{E}[(n,c)](t)$ be the Lyapunov functional defined in \eqref{F:Lyapunov}, with associated energy $\mathcal{F}(t)$ given by \eqref{Fun:JCF}. Then for some $\lambda>0$,
\begin{equation}\label{ineq:FracEnergy}
\mathcal{F}(t)+\frac{\lambda}{\Gamma(\alpha)}\mathcal{E}[(n,c)](t)\le \mathcal{F}(0),\quad~ t>0.
\end{equation}
As a consequence, the solution satisfies the uniform bounds
\begin{displaymath}
\mathcal{F}(t)\le \mathcal{F}(0),\quad\textrm{and}\quad
\mathcal{E}[(n,c)](t)\le\frac{\Gamma(\alpha)}{\lambda}\mathcal{F}(0), \quad t>0.
\end{displaymath}
\end{lemma}

\begin{proof}
Applying the Riemann-Liouvillee fractional integral ${{}_0I_t^\alpha}$ to the dissipation inequality \eqref{est:dissap} established in Lemma \ref{lem:dissapation}, we obtain
\begin{displaymath}
{}_0I_t^\alpha\bigl({}_0^C\mathfrak D_t^\alpha \mathcal F(t)\bigr)\le-\lambda\,{}_0I_t^\alpha\big(\mathcal{F}(t)\big).
\end{displaymath}
By Invoking the fundamental theorem of fractional calculus for Caputo derivatives, which states that for $0<\alpha<1$, ${{}_0I_t^\alpha}\bigl({}_0^C\mathfrak{D}_t^\alpha f(t)\bigr)=f(t)-f(0)$ (cf. Ref.~\refcite{Podlubny99}), we deduce that
\begin{displaymath}
\mathcal{F}(t)-\mathcal{F}(0)\le-\lambda\,{}_0I_t^\alpha\mathcal{F}(t),\quad t>0.
\end{displaymath}
According to \eqref{F:Lyapunov}, since ${}_0I_t^\alpha\mathcal{F}(t)=\frac1{\Gamma(\alpha)}\mathcal{E}[(n,c)](t)$, we obtain
\begin{displaymath}
\mathcal{F}(t)+\frac{\lambda}{\Gamma(\alpha)}\mathcal{E}(t)\le\mathcal{F}(0),\quad t>0,
\end{displaymath}
which is exactly \eqref{ineq:FracEnergy}. By virtue of the non-negativity of $\mathcal{E}[(n, c)](t)$ immediately gives $\mathcal{F}(t)\le\mathcal{F}(0)$. Moreover, since $\mathcal{F}(t)$ is also non-negative, rearranging \eqref{ineq:FracEnergy} directly implies
\begin{displaymath}
\mathcal{E}[(n,c)](t)\le\frac{\Gamma(\alpha)}{\lambda}\mathcal{F}(0),\quad \forall~ t>0.
\end{displaymath}
This completes the proof.
\end{proof}

\subsection{Global well-posedness}
According to Theorem \ref{Thm:Lwell-posed}, the well-posedness of local mild solutions is known. To extend globally, it suffices to show that the working norm remains uniformly bounded on the time interval $[0,\infty)$.

\begin{lemma}\label{Lem:bounddde}
Let $d=2$, $\gamma>0$, and let $\Omega\subset\mathbb R^2$ be a bounded domain with smooth boundary. Suppose $0<\chi<\frac{1}{2}$, and let $(n,c)$ be a global mild solution
of System~\eqref{Pro:FKS} on $[0,\infty)$. Then, for every finite $q>2$, there exists a constant $C=C(\|n_0\|_{L^\infty},\|c_0\|_{W^{1,\infty}},q,\Omega)>0$ such that
\begin{displaymath}
\sup_{t\ge0}\big(\|n(\cdot,t)\|_{L^\infty}+\|c(\cdot,t)\|_{W^{1,q}}\big)\le C.
\end{displaymath}
Here the dependence of $C$ on the fixed parameters $\alpha,\mathcal{D},\gamma,\chi$ and $C_*$ (lower bound of $c$) is suppressed.
\end{lemma}

\begin{proof}
We prove the estimate on an arbitrary finite interval $[0,T]$, with constants independent of $T$. Since $T>0$ is arbitrary, the desired uniform-in-time estimate follows by letting $T\to\infty$. To this end, we proceed by breaking the proof into the following sequence of implications, with the outline presented in Fig. \ref{Fig:bootstrap-diagram}.

\begin{figure}[htbp]
\centering
\begin{tikzpicture}[
    node distance=1.0cm and 1.2cm,
    mainnode/.style={align=center, font=\footnotesize, inner sep=1pt, rounded corners=3pt,
                     draw=red!65!black, thick, dashed, fill=none,
                     text width=3.45cm, minimum height=1.36cm,
                     text centered},
    doublearrow/.style={->,double,double distance=1.2pt,draw=pink!65!black,line width=0.55pt,>=Stealth},
    arrowlabel/.style={font=\scriptsize,inner sep=1.5pt},
    >=Stealth,
    thick
]
\node[mainnode] (A)
{$\begin{aligned}
&\text{Mass conservation}\\
&\text{Lyapunov estimate}
\end{aligned}$};
\node[mainnode] (B) [right=of A]
{$\begin{aligned}
&c\in L^{\infty}(L^m),\,1\le m<\infty\\[1mm]
&\displaystyle \mathcal Y(t):=\int_\Omega n^2c^{-\frac{1}{2}}\,\mathrm d\mathbf x\le C
\end{aligned}$};
\node[mainnode] (C) [right=of B]
{$\begin{aligned}
&n\in L^\infty(L^s),\,1<s<2
\end{aligned}$};
\node[mainnode] (D) [below=0.95cm of C]
{$\begin{aligned}
&\chi\frac{\nabla c}{c}
\in L^{\infty}(L^{q_0}),\,q_0>2
\end{aligned}$};
\node[mainnode] (E) [left=of D]
{$n\in L^\infty(0,T;L^\infty(\Omega))$};
\node[mainnode] (F) [left=of E]
{$\begin{aligned}
&c\in L^\infty(W^{1,q}),\,2<q<\infty
\end{aligned}$};
\draw[->, draw=pink!65!black] (A) -- (B)
node[midway,above,arrowlabel] {$(i)$};
\draw[->, draw=pink!65!black] (B) -- (C)
node[midway,above,arrowlabel] {$(ii)$};
\draw[->, draw=pink!65!black] (C) -- (D)
node[midway,right,arrowlabel] {$(iii)$};
\draw[->, draw=pink!65!black] (D) -- (E)
node[midway,above,arrowlabel] {$(iv)$} node[midway, below] {\scriptsize Moser};
\draw[->, draw=pink!65!black] (E) -- (F)
node[midway,above,arrowlabel] {$(v)$};
\end{tikzpicture}
\caption{Schematic diagram of the bootstrap argument  used to prove uniform boundedness of $(n,c)$.}
\label{Fig:bootstrap-diagram}
\end{figure}
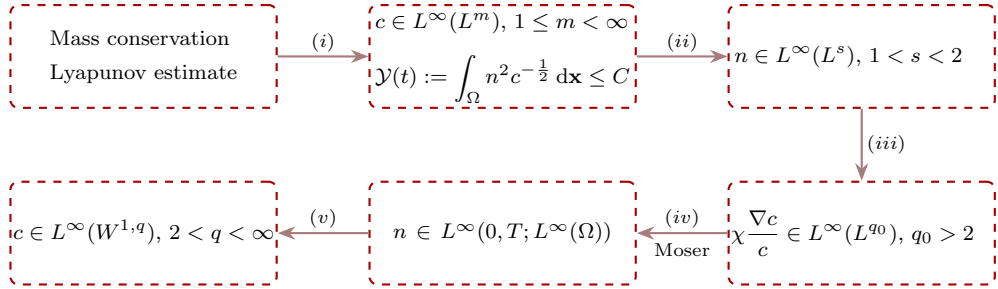

\begin{itemize}
    \item \emph{\textbf{Step $(i)$}}: $L^m$-bound for $c$, $1\leq m<\infty$.
\end{itemize}

We start from the a priori estimates already obtained above. By virtue of the dissipation property established in Lemma \ref{lem:dissapation}, for some $\lambda>0$, we have
\begin{equation}\label{eqt:hkhk}
\mathcal{F}(t)
\leq\mathcal{F}(0)E_\alpha(-\lambda t^\alpha)\leq\mathcal{F}(0)<\infty, \quad~ t>0.
\end{equation}
Recalling the definition of $\mathcal{F}(t)$ in \eqref{Fun:JCF}, the uniform bound on $\mathcal{F}(t)$ directly implies that the relative entropy term $\int_{\Omega}n\log(n/\bar{n})\mathrm{d}\mathbf{x}$ is bounded from \eqref{eqt:hkhk}, where $\bar{n}=\frac{1}{|\Omega|}\int_{\Omega}n\mathrm{d}\mathbf{x}$. By utilizing the properties of the logarithm, and combining it with the mass conservation $\|n(\mathbf{x},t)\|_{L^1}=\|n_0\|_{L^1}\equiv\tilde{M}$ from Lemma~\ref{lem:MassCon-n}, we can decompose the integral as
\begin{equation}\label{eq:NFKK}
\int_{\Omega}n\log\left(\frac{n}{\bar{n}}\right)\mathrm{d}\mathbf{x}
=\int_{\Omega}n\log n\,\mathrm{d}\mathbf{x}-\tilde{M}\log\bar{n}\leq C, \quad 0<t<T.
\end{equation}
Similarly, as both terms in $\mathcal F$ are nonnegative, this gives in particular
\begin{equation}\label{eq:entropy-fisher-bound}
\int_\Omega\frac{|\nabla c|^2}{c}\,d\mathbf{x}\le C,\qquad 0<t<T.
\end{equation}

We next extract from \eqref{Eq:MassCon-c} a uniform $L^1$-bound for $c$. To this end, set $C_c(t):=\int_\Omega c(\mathbf{x},t)\mathrm{d}\mathbf{x}$. Then \eqref{Eq:MassCon-c} can be written in the compact form
\begin{equation}\label{eq:Cc-Volterra}
C_c(t)+\gamma {_{0}I_{t}^{\alpha}} C_c(t)=C_c(0)+\tilde{M}\,{_{0}I_{t}^{\alpha}} 1(t),
\end{equation}
where ${_{0}I_{t}^{\alpha}}$ denotes the Riemann-Liouvillee fractional integral. Indeed, ${_{0}I_{t}^{\alpha}} 1(t)=\frac{t^\alpha}{\Gamma(1+\alpha)}$. Equivalently, \eqref{eq:Cc-Volterra} can be rearranged as
\begin{displaymath}
C_c(t)-C_c(0)
={_{0}I_{t}^{\alpha}}\big(\tilde{M}-\gamma C_c\big)(t).
\end{displaymath}
This is precisely the Volterra integral formulation of the scalar Caputo equation
\begin{displaymath}
{_0^C\mathfrak{D}_t^\alpha}C_c(t)+\gamma C_c(t)=\tilde{M},\qquad
C_c(0)=\int_\Omega c_0(\mathbf x)\,\mathrm{d}\mathbf{x}.
\end{displaymath}
Thus the spatial total mass of $c$ is governed by a one-dimensional fractional relaxation equation with a constant source term $\tilde{M}$.

For completeness, we derive the explicit representation. Taking the Laplace
transform of \eqref{eq:Cc-Volterra}, and using $\mathcal L\{{_{0}I_{t}^{\alpha}} f\}(z)=z^{-\alpha}\mathcal L\{f\}(z)$ (cf. Ref.~\refcite{Podlubny99}), we obtain
\begin{displaymath}
\widehat{C_c}(z)+\gamma z^{-\alpha}\widehat{C_c}(z)
=C_c(0)z^{-1}+\tilde{M} z^{-\alpha-1}.
\end{displaymath}
Equivalently,
\begin{equation}\label{eq:Cc-Laplace-2}
\widehat{C_c}(z)
=\frac{C_c(0)z^{\alpha-1}}{z^\alpha+\gamma}+\frac{\tilde{M} z^{-1}}{z^\alpha+\gamma}.
\end{equation}
By employing the standard Laplace transform identity for the Mittag-Leffler function, namely $\mathcal{L}\left\{t^{\beta-1}E_{\alpha,\beta}(-\gamma t^\alpha)\right\}(z)
=\frac{z^{\alpha-\beta}}{z^\alpha+\gamma}$ for any $\beta>0$, we infer from \eqref{eq:Cc-Laplace-2} that
\begin{equation}\label{eq:Cc-explicit}
C_c(t)=C_c(0)E_\alpha(-\gamma t^\alpha)+\tilde{M}t^\alpha E_{\alpha,\alpha+1}(-\gamma t^\alpha).
\end{equation}
The second term in \eqref{eq:Cc-explicit} is uniformly bounded. Indeed, by the
series definition of the Mittag-Leffler functions,
\begin{displaymath}
\begin{aligned}
\gamma t^\alpha E_{\alpha,\alpha+1}(-\gamma t^\alpha)
&=\gamma t^\alpha\sum_{j=0}^{\infty}\frac{(-\gamma t^\alpha)^j}{\Gamma(\alpha j+\alpha+1)}
=\sum_{k=1}^{\infty}\frac{(-1)^{k-1}(\gamma t^\alpha)^k}{\Gamma(\alpha k+1)}\\
&=1-\sum_{k=0}^{\infty}\frac{(-\gamma t^\alpha)^k}{\Gamma(\alpha k+1)}
=1-E_\alpha(-\gamma t^\alpha).
\end{aligned}
\end{displaymath}
Therefore,
\begin{equation}\label{eq:Cc-explicit-bounded-form}
C_c(t) = C_c(0)E_\alpha(-\gamma t^\alpha)+
    \frac{\tilde M}{\gamma}
    \big(1-E_\alpha(-\gamma t^\alpha)\big).
\end{equation}
Since $0<\alpha<1$, the function $E_\alpha(-\gamma t^\alpha)$ is completely
monotone on $t>0$, and in particular $0\le E_\alpha(-\gamma t^\alpha)\le1$ for $t\ge0$ (cf. Refs.~\refcite{Jin21book,Rudolf20}). Hence \eqref{eq:Cc-explicit-bounded-form} gives
\begin{displaymath}\label{eq:Cc-uniform-bound}
0\le C_c(t)\le C_c(0)+\frac{\tilde M}{\gamma},\qquad t>0.
\end{displaymath}
Since $c>0$, we have $\|c(\cdot,t)\|_{L^1(\Omega)}=C_c(t)$. Consequently,
\begin{equation}\label{eq:c-L1-uniform-used}
\sup_{0<t<T}\big\|c(\cdot,t)\big\|_{L^1(\Omega)}
\le C_c(0)+\frac{\tilde M}{\gamma}\le C.
\end{equation}
The estimate in \eqref{eq:entropy-fisher-bound} gives
\begin{displaymath}
\big\|\nabla\sqrt c\big\|_{L^2}^2
=\frac{1}{4}\int_\Omega\frac{|\nabla c|^2}{c}\,d\mathbf{x}\le C.
\end{displaymath}
Together with \eqref{eq:c-L1-uniform-used}, this yields
\begin{displaymath}
\sup_{0<t<T}\big\|\sqrt c(\cdot,t)\big\|_{H^1(\Omega)}\le C .
\end{displaymath}
Since $d=2$, the Sobolev embedding $H^1(\Omega)\hookrightarrow L^r(\Omega)$ holds for every finite $r\ge1$. Hence, for every finite $m\ge1$,
\begin{equation}\label{eq:c-Lm-uniform}
\sup_{0<t<T}\big\|c(\cdot,t)\big\|_{L^m(\Omega)}\le C_m .
\end{equation}

\begin{itemize}
    \item \emph{\textbf{Step $(ii)$}}: $L^s$-bound for $n$ with $s\in(1,2)$.
\end{itemize}

We now introduce the auxiliary weighted functional $\mathcal{Y}(t):=\int_\Omega n^2c^{-\frac{1}{2}}\,\mathrm{d}\mathbf{x}$. This quantity serves as an intermediate bootstrap functional, upgrading the entropy-level information to an $L^s$-bound for $n$ with some exponent $s>1$.

Consider the density $\Phi(n,c):=n^2c^{-\frac12}$. A direct computation gives $\Phi_{nn}=2c^{-\frac{1}{2}}$, $\Phi_{cc}=\frac{3}{4}n^2c^{-\frac{5}{2}}$, and $\Phi_{nc}=-nc^{-\frac{3}{2}}$. Moreover,
\begin{displaymath}
\Phi_{nn}\Phi_{cc}-\Phi_{nc}^2=\frac{1}{2}n^2c^{-3}\ge0.
\end{displaymath}
Thus $\Phi$ is convex on $\mathbb{R}\times(0,\infty)$. Applying the Caputo variational inequality in Lemma~\ref{Lem:varieq}, or equivalently the same convexity argument used in its proof, gives
\begin{equation}\label{eq:Y-Caputo-variational}
{_0^C\mathfrak{D}_t^\alpha}\mathcal{Y}(t)
\le\int_\Omega\left(2nc^{-\frac{1}{2}}{_0^C\mathfrak{D}_t^\alpha}n
-\frac{1}{2} n^2c^{-\frac{3}{2}}{_0^C\mathfrak{D}_t^\alpha}c\right)\mathrm{d}\mathbf{x}.
\end{equation}
Substituting System~\eqref{Pro:FKS} into \eqref{eq:Y-Caputo-variational} and
integrating by parts, we get
\begin{equation}\label{eq:Y-energy-before-Young}
\begin{aligned}
{_0^C\mathfrak{D}_t^\alpha}\mathcal{Y}(t)
&\le\mathcal{D}\Big[-2\int_\Omega c^{-\frac{1}{2}}\big|\nabla n\big|^2\,\mathrm{d}\mathbf{x}
+(2+2\chi)\int_\Omega nc^{-\frac32}\nabla n\cdot\nabla c\,\mathrm{d}\mathbf{x}\\
&\qquad
-\left(\frac{3}{4}+\chi\right)\int_\Omega n^2c^{-\frac{5}{2}}\big|\nabla c\big|^2\,\mathrm{d}\mathbf{x}
\Big]
+\frac{\gamma}{2}\int_\Omega n^2c^{-\frac{1}{2}}\,\mathrm{d}\mathbf{x}
-\frac{1}{2}\int_\Omega n^3c^{-\frac{3}{2}}\,\mathrm{d}\mathbf{x}.
\end{aligned}
\end{equation}
Let $A:=c^{-\frac{1}{4}}|\nabla n|$ and $B:=nc^{-\frac{5}{4}}|\nabla c|$. Then $nc^{-\frac32}|\nabla n||\nabla c|=AB$. By Young's inequality,
\begin{displaymath}
(2+2\chi)AB \le 2A^2+\frac{(2+2\chi)^2}{8}B^2.
\end{displaymath}
Therefore the gradient contribution in \eqref{eq:Y-energy-before-Young} is
bounded from above by
\begin{displaymath}
-\mathcal D
\left(\left(\frac{3}{4}+\chi\right)-\frac{(2+2\chi)^2}{8}\right)
\int_\Omega n^2c^{-\frac{5}{2}}|\nabla c|^2\,\mathrm{d}\mathbf{x} .
\end{displaymath}
Since $\left(\frac{3}{4}+\chi\right)-\frac{(2+2\chi)^2}{8}
=\frac{1}{4}-\frac{\chi^2}{2}>0$ for $0<\chi<\frac{1}{2}$, this contribution is non-positive. Hence
\begin{equation}\label{eq:Y-differential-raw}
{_0^C\mathfrak{D}_t^\alpha}\mathcal{Y}(t)
\le\frac{\gamma}{2}\mathcal{Y}(t)
-\frac{1}{2}\int_\Omega n^3c^{-\frac{3}{2}}\,\mathrm{d}\mathbf{x}.
\end{equation}

It remains to relate the last integral to $\mathcal{Y}(t)$. Since $\mathcal{Y}(t)=\int_\Omega\left(n^3c^{-3/2}\right)^{2/3}c^{1/2}\mathrm{d}\mathbf{x}$, H\"{o}lder's inequality and \eqref{eq:c-Lm-uniform} with $m=3/2$ imply
\begin{displaymath}
\mathcal{Y}(t)\le
\left(\int_\Omega n^3c^{-\frac{3}{2}}\,\mathrm{d}\mathbf{x}\right)^{2/3}
\left(\int_\Omega c^{\frac{3}{2}}\,\mathrm{d}\mathbf{x}\right)^{1/3}
\le
C\left(\int_\Omega n^3c^{-\frac{3}{2}}\,\mathrm{d}\mathbf{x}\right)^{2/3}.
\end{displaymath}
Consequently,
\begin{equation}\label{eq:cubic-lower}
\int_\Omega n^3c^{-\frac{3}{2}}\,\mathrm{d}\mathbf{x}
\ge C\mathcal{Y}(t)^{3/2}.
\end{equation}
Combining \eqref{eq:Y-differential-raw} and \eqref{eq:cubic-lower}, we obtain
\begin{displaymath}
{_0^C\mathfrak{D}_t^\alpha}\mathcal{Y}(t)
\le a\mathcal{Y}(t)-b\mathcal{Y}(t)^{3/2}
\end{displaymath}
for some $a,b>0$. The scalar comparison principle
Lemma~\ref{lem:Caputo-logistic} yields
\begin{equation}\label{eq:Y-weighted-uniform}
\sup_{0<t<T}\mathcal{Y}(t)\le C.
\end{equation}

Let $s\in(1,2)$. Since $n^s=\left(n^2c^{-1/2}\right)^{s/2}c^{s/4}$, H\"{o}lder's inequality gives
\begin{displaymath}
\int_\Omega n^s\,\mathrm{d}\mathbf{x}
\le \mathcal{Y}(t)^{\frac{s}{2}}
\left(\int_\Omega c^{\frac{s}{2(2-s)}}\,\mathrm{d}\mathbf{x}\right)^{(2-s)/2}.
\end{displaymath}
Using \eqref{eq:c-Lm-uniform} and \eqref{eq:Y-weighted-uniform}, we infer
\begin{equation}\label{eq:n-Ls-uniform}
\sup_{0<t<T}\big\|n(\cdot,t)\big\|_{L^s(\Omega)}\le C_s,\quad 1<s<2 .
\end{equation}

\begin{itemize}
    \item \emph{\textbf{Step $(iii)$}}: The estimate for $\mathbf{V}:=\chi\frac{\nabla c}{c}$.
\end{itemize}

We now improve the spatial regularity of $c$. By employing the mild formulation, as written with the operators defined in \eqref{eq:EDelta}, and applying $\mathcal{A}^{1/2}$ to it, we obtain
\begin{equation}\label{eq:Ahalf-c-mild}
\begin{aligned}
\mathcal{A}^{1/2}c(t)
=\mathcal{A}^{1/2}E_\alpha(-t^\alpha\mathcal{A}_\gamma)c_0+\int_0^t(t-\tau)^{\alpha-1}\mathcal{A}^{1/2}
E_{\alpha,\alpha}\bigl(-(t-\tau)^\alpha\mathcal{A}_\gamma\bigr)n(\tau)\,\mathrm{d}\tau.
\end{aligned}
\end{equation}
The initial term is bounded in $L^{q_0}(\Omega)$ for every finite $q_0>1$, because $c_0\in W^{1,\infty}(\Omega)$ and the Mittag-Leffler family is bounded on the corresponding fractional domain.

Fix $s\in(1,2)$ as in \eqref{eq:n-Ls-uniform}. Choose $q_0>2$, sufficiently
close to $2$, such that $\mu:=\frac{1}{s}-\frac{1}{{q_0}}<\frac{1}{2}$. This is possible because $s>1$. By the refined estimate from Ref.~\refcite{Bezerra24}, Lemma~2.7 and Remark~2.8 (ii), we note that for $1\leq p \leq q \leq\infty$ satisfying $\frac{1}{p}-\frac{1}{q}<\frac{1}{d}$, the following inequality holds
\begin{equation}\label{eqt:refEAA}
\big\|\mathcal{A}^{1/2} E_{\alpha,\alpha}\big(-t^{\alpha}\mathcal{A}_{\gamma}\big)n\big\|_{L^{q}}
\leq C\times\left\{
\begin{aligned}
& t^{-\frac{\alpha}{2}-\frac{\alpha d}{2}\left(\frac{1}{p}-\frac{1}{q}\right)}
    \big\|n(\mathbf{x},t)\big\|_{L^{p}}, &&\text{for}~ 0<t<1,\\
& t^{-2\alpha+\frac{\alpha}{2}+\frac{\alpha d}{2}\left(\frac{1}{p}-\frac{1}{q}\right)}
    \big\|n(\mathbf{x},t)\big\|_{L^{p}}, &&\text{for}~ t>1,
\end{aligned}\right.
\end{equation}
Consequently, we obtain
\begin{displaymath}
\left\|
\mathcal{A}^{1/2}E_{\alpha,\alpha}
\bigl(-\tau^\alpha\mathcal{A}_\gamma\bigr)f\right\|_{L^{q_0}}
\le C
\begin{cases}
\tau^{-\frac{\alpha}{2}-\alpha\mu}\big\|f\big\|_{L^s},
&0<\tau<1,\\[1mm]
\tau^{-2\alpha+\frac{\alpha}{2}+\alpha\mu}\big\|f\big\|_{L^s},
&\tau>1 .
\end{cases}
\end{displaymath}
After multiplication by the convolution factor $\tau^{\alpha-1}$, the kernels
become $\tau^{-1+\alpha(\frac12-\mu)}$ for $0<\tau<1$, and $\tau^{-1-\alpha(\frac12-\mu)}$ for $\tau>1$. Both are integrable because $\mu<\frac12$. Hence, using
\eqref{eq:n-Ls-uniform} in \eqref{eq:Ahalf-c-mild}, we obtain
\begin{equation}\label{eq:Ahalf-c-q0}
\sup_{0<t<T}\big\|\mathcal{A}^{1/2}c(\cdot,t)\big\|_{L^{q_0}}\le C .
\end{equation}
By the graph norm equivalence for the Neumann Laplacian, $\|\nabla c\|_{L^{q_0}}\le C\|\mathcal{A}^{1/2}c\|_{L^{q_0}}$, up to the harmless constant mode. Therefore $\sup_{0<t<T}\|\nabla c(\cdot,t)\|_{L^{q_0}}\le C$. Using the previously established lower bound $c\ge C_*>0$, we obtain
\begin{equation}\label{eq:V-Lq0-uniform}
\sup_{0<t<T}\big\|\mathbf{V}(\cdot,t)\big\|_{L^{q_0}}\le C,\qquad
\mathbf V:=\chi\frac{\nabla c}{c}.
\end{equation}

\begin{itemize}
    \item \emph{\textbf{Step $(iv)$}}: $L^\infty$-bootstrap for $n$ via Moser iteration.
\end{itemize}

The first equation in System~\eqref{Pro:FKS} can be rewritten as ${_0^C\mathfrak{D}_t^\alpha}n=\mathcal{D}\Delta n-\mathcal{D}\nabla\cdot(n\mathbf{V})$. For $r\ge1$, set $Y_r(t):=\|n(\cdot,t)\|_{L^r}^r=\int_\Omega n^r\,d\mathbf {x}$. Let $p\ge2$ and put $w=n^{p/2}$. Testing the equation for $n$ by $pn^{p-1}$, using the fractional convex inequality for $z\mapsto z^p$, and integrating by parts, we obtain
\begin{displaymath}
{_0^C\mathfrak{D}_t^\alpha}Y_p
+\frac{4\mathcal{D}(p-1)}{p}\big\|\nabla w\big\|_{L^2}^2
\le2\mathcal{D}(p-1)\int_\Omega w\big|\mathbf V\big|\big|\nabla w\big|\,\mathrm{d}\mathbf{x}.
\end{displaymath}
Young's inequality yields
\begin{equation}\label{eq:Moser-test-2}
{_0^C\mathfrak{D}_t^\alpha}Y_p
+C_1\big\|\nabla w\big\|_{L^2}^2
\le C_2p^2\int_\Omega w^2\big|\mathbf{V}\big|^2\,\mathrm{d}\mathbf{x},
\end{equation}
where $C_1,C_2>0$ are independent of $p\ge2$. By H\"{o}lder's inequality and \eqref{eq:V-Lq0-uniform},
\begin{equation}\label{eq:V-holder}
\int_\Omega w^2|\mathbf{V}|^2\,\mathrm{d}\mathbf{x}
\le C\big\|w\big\|_{L^{\frac{2q_0}{q_0-2}}}^2.
\end{equation}
Since $d=2$, the Gagliardo--Nirenberg inequality gives
\begin{equation}\label{eq:GN-w-r}
\big\|w\big\|_{L^{\frac{2q_0}{q_0-2}}}
\le C\big\|\nabla w\big\|_{L^2}^{\vartheta}\big\|w\big\|_{L^1}^{1-\vartheta}+C\big\|w\big\|_{L^1},
\end{equation}
where $\vartheta=1-\frac{q_0-2}{2q_0}=\frac{q_0+2}{2q_0}\in(0,1)$. Since $\|w\|_{L^1}=\int_\Omega n^{p/2}\mathrm{d}\mathbf{x}=Y_{p/2}(t)$, we infer from \eqref{eq:GN-w-r} that
\begin{equation}\label{eq:w-r-bound}
\big\|w\big\|_{L^{\frac{2q_0}{q_0-2}}}^2
\le C\big\|\nabla w\big\|_{L^2}^{2\vartheta}Y_{p/2}(t)^{2(1-\vartheta)}+CY_{p/2}(t)^2.
\end{equation}
Combining \eqref{eq:Moser-test-2}, \eqref{eq:V-holder}, and
\eqref{eq:w-r-bound}, we get
\begin{displaymath}
C_2p^2\int_\Omega w^2\big|\mathbf{V}\big|^2\,d\mathbf{x}
\le Cp^2\big\|\nabla w\big\|_{L^2}^{2\vartheta} Y_{p/2}(t)^{2(1-\vartheta)}+Cp^2Y_{p/2}(t)^2.
\end{displaymath}
By Young's inequality, for every $\varepsilon>0$,
\begin{displaymath}
Cp^2 \big\|\nabla w\big\|_{L^2}^{2\vartheta}Y_{p/2}(t)^{2(1-\vartheta)}
\le \varepsilon\big\|\nabla w\big\|_{L^2}^2
+C_\varepsilon p^{\frac{2}{1-\vartheta}}Y_{p/2}(t)^2.
\end{displaymath}
Thus, after enlarging the exponent if necessary, there exists $\beta:=\max\left\{2,\frac{2}{1-\vartheta}\right\}>0$
such that
\begin{displaymath}
C_2p^2\int_\Omega w^2\big|\mathbf{V}\big|^2\,d\mathbf x
\le\varepsilon\big\|\nabla w\big\|_{L^2}^2+Cp^\beta Y_{p/2}(t)^2.
\end{displaymath}
Choosing $\varepsilon>0$ sufficiently small, we obtain
\begin{equation}\label{eq:Moser-gradient-inequality}
{_0^C\mathfrak{D}_t^\alpha}Y_p
+\kappa\big\|\nabla w\big\|_{L^2}^2
\le Cp^\beta Y_{p/2}(t)^2 .
\end{equation}

We also use the two-dimensional Gagliardo--Nirenberg inequality
\begin{displaymath}
\big\|w\big\|_{L^2}\le C\big\|\nabla w\big\|_{L^2}^{1/2}\big\|w\big\|_{L^1}^{1/2}+C\big\|w\big\|_{L^1}.
\end{displaymath}
After squaring and applying Young's inequality, for any fixed $\delta>0$,
\begin{displaymath}
Y_p(t)=\big\|w\big\|_{L^2}^2\le\delta\big\|\nabla w\big\|_{L^2}^2+C_\delta Y_{p/2}(t)^2 .
\end{displaymath}
Equivalently,
\begin{equation}\label{eq:gradient-lower-Yp}
\big\|\nabla w\big\|_{L^2}^2\ge\delta^{-1}Y_p(t)-C_\delta\delta^{-1}Y_{p/2}(t)^2 .
\end{equation}
Combining \eqref{eq:Moser-gradient-inequality} and
\eqref{eq:gradient-lower-Yp}, and then fixing $\delta>0$, we arrive at
\begin{equation}\label{eq:Moser-recursive}
{_0^C\mathfrak{D}_t^\alpha}Y_p+\eta Y_p\le C p^\beta Y_{p/2}(t)^2, \qquad p\ge2 ,
\end{equation}
where $\eta>0$ and $C>0$ are independent of $p$.

We now perform the Moser iteration. Let $p_k:=s2^k$ and $A_k:=\sup_{0<t<T}\|n(\cdot,t)\|_{L^{p_k}(\Omega)}$. By \eqref{eq:n-Ls-uniform}, $A_0<\infty$. Since $p_k=s2^k\ge2s>2$ for $k\ge1$, we apply \eqref{eq:Moser-recursive} with $p=p_k$. The scalar fractional variation-of-constants formula gives
\begin{displaymath}
Y_{p_k}(t)\le Y_{p_k}(0)E_\alpha(-\eta t^\alpha)
+C p_k^\beta\int_0^t(t-\tau)^{\alpha-1}E_{\alpha,\alpha}(-\eta(t-\tau)^\alpha)Y_{p_{k-1}}(\tau)^2\,d\tau .
\end{displaymath}
Using $0\le E_\alpha(-\eta t^\alpha)\le1$ and $\int_0^\infty\tau^{\alpha-1}E_{\alpha,\alpha}(-\eta\tau^\alpha)\,d\tau=\frac1\eta$, we infer
\begin{displaymath}
\sup_{0<t<T}Y_{p_k}(t)\le Y_{p_k}(0)+C p_k^\beta\sup_{0<t<T}Y_{p_{k-1}}(t)^2.
\end{displaymath}
Since $p_k=2p_{k-1}$, $A_k^{p_k}\le Y_{p_k}(0)+Cp_k^\beta A_{k-1}^{p_k}$. Because $n_0\in L^\infty(\Omega)$, there exists $C_0\ge1$ such that $Y_{p_k}(0)=\|n_0\|_{L^{p_k}}^{p_k}\le C_0^{p_k}$. Therefore
\begin{displaymath}
A_k^{p_k}\le C_0^{p_k}+C(s2^k)^\beta A_{k-1}^{p_k}.
\end{displaymath}
Set $\widetilde{A}_k:=\max\{A_k,A_0,C_0,1\}$. Then
\begin{displaymath}
\widetilde{A}_k^{p_k}\le C\big(1+2^{k\beta}\big)\widetilde{A}_{k-1}^{p_k}.
\end{displaymath}
Taking the $p_k=s2^k$-th root gives
\begin{displaymath}
\widetilde{A}_k\le\left(C(1+2^{k\beta})\right)^{1/(s2^k)}\widetilde{A}_{k-1}.
\end{displaymath}
Thus
\begin{displaymath}
\widetilde{A}_k\le\widetilde{A}_0\prod_{j=1}^k\left(C(1+2^{j\beta})\right)^{1/(s2^j)} .
\end{displaymath}
Since $\sum_{j=1}^{\infty}\frac{\log(C(1+2^{j\beta}))}{s2^j}<\infty$, the product is finite. Hence $\sup_{k\ge0}A_k<\infty $. Let $M_*:=\sup_{k\ge0}A_k$. For fixed $t\in(0,T)$, one has $\|n(\cdot,t)\|_{L^{p_k}}\le M_*$, $k=0,1,2,\dots$. If there existed $\lambda>M_*$ such that $|\{x\in\Omega:n(x,t)>\lambda\}|>0$, then
\begin{displaymath}
\lambda\big|\{n(\cdot,t)>\lambda\}\big|^{1/p_k}
\le\big\|n(\cdot,t)\big\|_{L^{p_k}}\le M_*
\end{displaymath}
for all $k$. Letting $k\to\infty$ gives $\lambda\le M_*$, a contradiction.
Therefore
\begin{equation}\label{eq:n-Linfty-final}
    \sup_{0<t<T}\big\|n(\cdot,t)\big\|_{L^\infty}\le C .
\end{equation}

\begin{itemize}
    \item \emph{\textbf{Step $(v)$}}: $W^{1,q}$-regularity estimate for $c$.
\end{itemize}

It remains to estimate $c$ in $W^{1,q}(\Omega)$ for arbitrary finite $q>2$. By \eqref{eq:n-Linfty-final}, $\sup_{0<t<T}\|n(\cdot,t)\|_{L^q}\le C$. Applying again the mild representation of $c$ and the refined estimate
\eqref{eqt:refEAA}, now with $p=q$, we obtain
\begin{displaymath}
\left\|\mathcal{A}^{1/2}E_{\alpha,\alpha}\bigl(-\tau^\alpha\mathcal{A}_\gamma\bigr)f\right\|_{L^q}
\le C
\begin{cases}
\tau^{-\frac{\alpha}{2}}\big\|f\big\|_{L^q},
&0<\tau<1,\\[1mm]
\tau^{-\frac{3\alpha}{2}}\big\|f\big\|_{L^q},
&\tau>1.
\end{cases}
\end{displaymath}
After multiplication by $\tau^{\alpha-1}$, the kernels become $\tau^{-1+\frac{\alpha}{2}}$ for $0<\tau<1$, and $\tau^{-1-\frac{\alpha}{2}}$, $\tau>1$. Both are integrable. Hence $\sup_{0<t<T}\|\mathcal A^{1/2}c(\cdot,t)\|_{L^q}\le C$. Together with \eqref{eq:c-Lm-uniform} and the graph norm equivalence for the
Neumann Laplacian, this gives
\begin{displaymath}
\sup_{0<t<T}\big\|c(\cdot,t)\big\|_{W^{1,q}}\le C .
\end{displaymath}
Combining this estimate with \eqref{eq:n-Linfty-final}, we conclude that
\begin{displaymath}
\sup_{0<t<T}\left(\big\|n(\cdot,t)\big\|_{L^\infty}+\big\|c(\cdot,t)\big\|_{W^{1,q}}\right)\le C.
\end{displaymath}
The constant is independent of $T$. Since $T>0$ was arbitrary, letting
$T\to\infty$ proves
\begin{displaymath}
\sup_{t\ge0}\left(\big\|n(\cdot,t)\big\|_{L^\infty}+\big\|c(\cdot,t)\big\|_{W^{1,q}}\right)\le C.
\end{displaymath}
The proof is complete.
\end{proof}

\begin{remark}
The proof above is essentially two-dimensional. The main reason is that the
estimate $\mathcal{Y}(t):=\int_\Omega n^2c^{-\frac{1}{2}}\,\mathrm{d}\mathbf x\le C$ only provides an initial improvement of the integrability of $n$. In dimension three, the same strategy may still be adapted: indeed, one can choose $s\in\left(\frac{3}{2},\frac{12}{7}\right)$ and then choose $q_0>3$ such that $\frac1s-\frac1{q_0}<\frac13$. This would allow the refined smoothing estimate for the $c$-equation to yield
the subcritical drift estimate $\mathbf V:=\chi\frac{\nabla c}{c}\in L^\infty(0,T;L^{q_0}(\Omega))$, $q_0>3$, which is the condition needed to close the Moser iteration in three dimensions. For dimensions $d\ge4$, however, the present bootstrap mechanism no longer closes. More precisely, the weighted estimate for $\mathcal{Y}(t)$ does not provide an exponent $s>d/2$. Consequently, one cannot obtain, by this argument,
a subcritical drift bound of the form $\mathbf{V}\in L^\infty(0,T;L^{q_0}(\Omega))$, $q_0>d$. This subcritical drift estimate is precisely the ingredient required in the Moser iteration for the drift-diffusion equation satisfied by $n$. Thus the present proof should be regarded as a two-dimensional argument; a possible three-dimensional extension would require a separate treatment of the exponents, whereas dimensions $d\ge4$ cannot be covered by this bootstrap scheme.  Whether global bounded solutions continue to exist for arbitrary dimensions $d\ge4$ remains an interesting open problem.
\end{remark}

\begin{remark}
By virtue of the uniform bounds for $n$ and $c$ established in Lemma \ref{Lem:bounddde} for all $t>0$, and following the arguments in Lemmas \ref{Lem:timR} and \ref{Lem:Rnnn}, the regularity results derived in those lemmas remain valid for $t\in[\tau,\infty)$ with an arbitrarily small $\tau>0$.
\end{remark}

\vspace{0.5cm}
\begin{MyProof}[Proof of Theorem \ref{Thm:well-posed}]
The proof of this theorem is established by combining the theoretical results derived in the preceding sections. Specifically, the uniform upper bound obtained in Lemma \ref{Lem:bounddde} ensures the existence of a positive constant $M$, independent of $t$, such that
\begin{equation}\label{eq:proof-blowup}
\sup_{t>0}\Big(\big\|n(\mathbf{x},t)\big\|_{L^{\infty}}
+\big\|c(\mathbf{x},t)\big\|_{W^{1,p}}\Big)
\leq M, \quad~p>d.
\end{equation}
This establishes the bound mentioned in \eqref{eq:bounddd}. Furthermore, by coupling the mass conservation from Lemma \ref{lem:MassCon-n} and the local well-posedness from Theorem \ref{Thm:Lwell-posed} with the non-blowup criterion in \eqref{eq:proof-blowup}, a standard continuation argument guarantees that the system admits a unique global mild solution. Concurrently, leveraging the uniform boundedness from Lemma \ref{Lem:bounddde}, the mass conservation property established in Lemma \ref{lem:MassCon-n} and applying regularity proof techniques analogous to those in Lemmas \ref{Lem:timR} and \ref{Lem:Rnnn}, the regularity properties in \eqref{Eq:regularity} are readily verified. We omit the detailed steps for brevity, thereby completing the proof.
\end{MyProof}

\vspace{0.5cm}
\begin{MyProof}[Proof of Theorem \ref{Thm:golablel}]
The proof follows by combining the mass conservation property established in Lemma \ref{lem:MassCon-n} with the global well-posedness  provided in Theorem \ref{Thm:well-posed}.
\end{MyProof}

\section{Simulations with Non-Negativity/Positivity-Preserving PINNs}
\label{sec:SPINN}
In this section, we perform numerical simulations using the popular Physics-Informed Neural Networks (PINNs) to visualize the solution behavior of System \eqref{Pro:FKS} and to validate the theoretical results established above.

\subsection{Methodology: DNN architecture and implementations}
This subsection describes the numerical method used to solve the system \eqref{Pro:FKS}. The problem is highly nonlinear and strongly coupled, with a time-nonlocal structure. To handle these difficulties, a time-marching PINNs algorithm is employed. A distinguishing feature of the algorithm is that it preserves the non-negativity of $n(x,y,t)$ and the strict positivity of $c(x,y,t)$, which are properties established earlier in Lemmas \ref{lem:bound-n} and \ref{lem:regularity-c}.

Let  $\Omega\subset\mathbb{R}^{d}$ ($d=2$) be a bounded domain with smooth boundary. For $(x,y)\in\Omega$, the following variable transformations are introduced. These are motivated by Refs.~\refcite{Huang23,Liu18,Wang25}.
\begin{equation}\label{eq:transform}
\left\{\begin{aligned}
&n(\cdot,t)=\rho(\cdot,t)^2\geq0, \quad &\text{for}~~\rho(\cdot,t): \Omega\times\mathbb{R}_{+}\rightarrow \mathbb{R};\\
&c(\cdot,t)=\exp\big(v(\cdot,t)\big)>0, \quad &\text{for}~~v(\cdot,t): \Omega\times\mathbb{R}_{+}\rightarrow \mathbb{R}.
\end{aligned}\right.
\end{equation}
Substituting the transformation \eqref{eq:transform} into the original problem \eqref{Pro:FKS} leads to the following reformulated system in terms of $(\rho(\cdot,t),v(\cdot,t))$:
\begin{equation}\label{Pro:tranFKS}
\left\{\begin{aligned}
&{_{0}^{C}\mathfrak{D}^{\alpha}_t}
\big(\rho^2(\cdot,t)\big)=\mathcal{D}\Delta\big(\rho^2(\cdot,t)\big)
               -\mathcal{D}\chi\nabla\cdot\big(\rho^2(\cdot,t)\nabla v(\cdot,t)\big),
&& \Omega\times(0,T],\\
&{_{0}^{C}\mathfrak{D}^{\alpha}_t}
\big(e^{v(\cdot,t)}\big)=\mathcal{D}\Delta\big(e^{v(\cdot,t)}\big)-\gamma e^{v(\cdot,t)}+\rho^2(\cdot,t),
&& \Omega\times(0,T],\\
&\rho(\cdot,0)=\sqrt{n_0(x)}, \quad v(\cdot,0)=\ln\big(c_0(\cdot)\big),
&& \Omega\times\{0\},\\
&\nabla \rho(\cdot,t) \cdot \mathbf{\nu} = 0, \quad \nabla v(\cdot,t) \cdot \mathbf{\nu} = 0,
&& \partial\Omega \times (0,T].
\end{aligned}\right.
\end{equation}
Here, $\mathbf{\nu}$ denotes the outward normal vector. This formulation has two useful features.  It guarantees the non-negativity of $n(\cdot,t)$ and the strict positivity of $c(\cdot,t)$ by construction. It also replaces the logarithmic gradient term $\nabla c(\cdot,t)/c(\cdot,t)$ with the simpler linear gradient $\nabla v(\cdot,t)$, which is numerically more favorable.

To solve System \eqref{Pro:tranFKS} numerically, a Deep Neural Network (DNN) framework is employed. The algorithm preserves the non-negativity and positivity of the solution variables.

The numerical procedure begins with the defining of the solution domain $\Omega$. Collocation points are sampled from the interior and the boundary, denoted respectively by $\xi^{\imath n}_\ell$ ($\ell=1,2,\ldots,M_{\imath n}$) and $\xi^{bc}_\jmath$ ($\jmath=1,2,\ldots,M_{bc}$), respectively. These spatial coordinates are subsequently fed into the network's input layer. Accordingly, the detailed implementation steps are  described as follows.

\begin{itemize}
    \item \textit{\textbf{Step $(i)$}}: Time semi-discretization. The $L_1$ scheme is employed for the temporal semi-discretization of the coupled system \eqref{Pro:tranFKS}.
\end{itemize}

The $L_1$ scheme is a widely approach for discretizing the Caputo fractional derivative in time\cite{Angstmann21,Jin16}. To implement the scheme, we first partition the time interval $[0,T]$ uniformly into $N$ subintervals of length $\tau=T/N$, giving the grid points $t_j=j\tau$ for $j=0,1,\ldots,N$. The core principle of the $L_1$ scheme lies in approximating the integrand of the Caputo derivative via piecewise linear interpolation over the grid pints $t_j$. The discrete approximation of the Caputo derivative at $t=t_j$ is given by\cite{Gao14,Sun06}
\begin{equation}\label{eq:L1}
_{0}^{C}\mathfrak{D}_{t}^{\alpha}f(t_j)\approx\frac{\tau^{-\alpha}}{\Gamma(2-\alpha)}
\left(a_0f^j-\sum_{k=1}^{j-1}(a_{j-k-1}-a_{j-k})f^k-a_{j-1}f^0\right),
\end{equation}
where $f^j:=f(t_j)$, $\alpha\in(0,1)$, and the coefficients are defied as $a_k:=(k+1)^{1-\alpha}-k^{1-\alpha}$ for $k=0,1,\ldots,j-1$. In particular, $a_0=1$, and the sequence ${a_k}$ is positive and monotonically decreasing.

The $L_1$ scheme \eqref{eq:L1} is unconditionally stable for discretizing the Caputo derivative and attains a global convergence rate of $O(\tau^{2-\alpha})$ on uniform grids\cite{Gao14,Jin16}. As $\alpha\rightarrow1$, it reduces to the classical first-order backward Euler method. The non-local property of the time non-local derivatives is explicitly reflected in the summation term, which involves all previous time levels $f^k$. This differs from integer-order derivatives, which depend only on local information.

Let $\rho^j=\rho(\cdot,t_j)$ and $v^j=v(\cdot,t_j)$ denote the semi-discrete solutions at time $t_j$. Accordingly, $n^j=(\rho^j)^2$ and $c^j = \exp(v^j)$. Applying the $L_1$ scheme \eqref{eq:L1} to System \eqref{Pro:tranFKS} at $t=t_j$ for $j\in\{0,1,2,\ldots,N\}$ yields the following semi-discrete system
\begin{equation}\label{Pro:tranFKSdist}
\left\{\begin{aligned}
&\frac{\tau^{-\alpha}}{\Gamma(2-\alpha)}\left[\big(\rho^j\big)^2
  -\sum_{k=1}^{j-1}\big(a_{j-k-1}-a_{j-k}\big)\big(\rho^k\big)^2-a_{j-1}\big(\rho^0\big)^2\right]\\
&\qquad\qquad = \mathcal{D}\Delta(\rho^j)^2-\mathcal{D}\chi\nabla\cdot\big((\rho^j)^2\nabla v^j\big),              \\
&\frac{\tau^{-\alpha}}{\Gamma(2-\alpha)}\left[\exp\big(v^j\big)
  -\sum_{k=1}^{j-1}\big(a_{j-k-1}-a_{j-k}\big)\exp\big(v^k\big)-a_{j-1}\exp\big(v^0\big)\right] \\
&\qquad\qquad=
\mathcal{D}\Delta(\exp(v^j))-\gamma \exp(v^j)+(\rho^j)^2,                                           \\
&\rho^0=\sqrt{n_0}, \quad v^0=\ln(c_0), \quad \nabla \rho^j\cdot\mathbf{\nu}=0, \quad \nabla v^j\cdot\mathbf{\nu}=0.
\end{aligned}\right.
\end{equation}

To streamline the notation, we rewrite the semi-discrete system \eqref{Pro:tranFKSdist}  in the following operator form
\begin{equation}\label{Pro:tranFKSdistS}
\left\{
\begin{aligned}
&\mathcal{P}_{\rho}[\rho^{j},v^{j}]=0,\quad
&&\mathcal{P}_{v}[\rho^{j},v^{j}]=0,      \\
&\Xi[\rho^{0},v^{0}]=0,\quad
&&\widetilde{\mathcal{B}}[\rho^{j},v^{j}]=0.
\end{aligned}\right.
\end{equation}
Here, $\mathcal{P}_{\rho}$ and $\mathcal{P}_{v}$ represent the discrete differential operators corresponding to the PDEs for $\rho^j$ and $v^j$; $\widetilde{\mathcal{B}}$ encapsulates the boundary conditions, and $\Xi$ collects the initial conditions. All are formulated abstractly based on the semi-discrete system \eqref{Pro:tranFKSdist}.

With the semi-discrete scheme \eqref{Pro:tranFKSdist}, we employ the PINNs to solve the resulting system \eqref{Pro:tranFKSdist}, leading to a  a mesh-free, positivity-preserving,
multi-objective optimization-based time-marching PINNs algorithm. For an overall picture of the framework, we refer the reader to Fig. \ref{Fig:exampNN}. The detailed algorithmic procedure is presented as follows.

\begin{figure}[!ht]
 \centering
 \centerline{\includegraphics[width=0.89\textwidth]{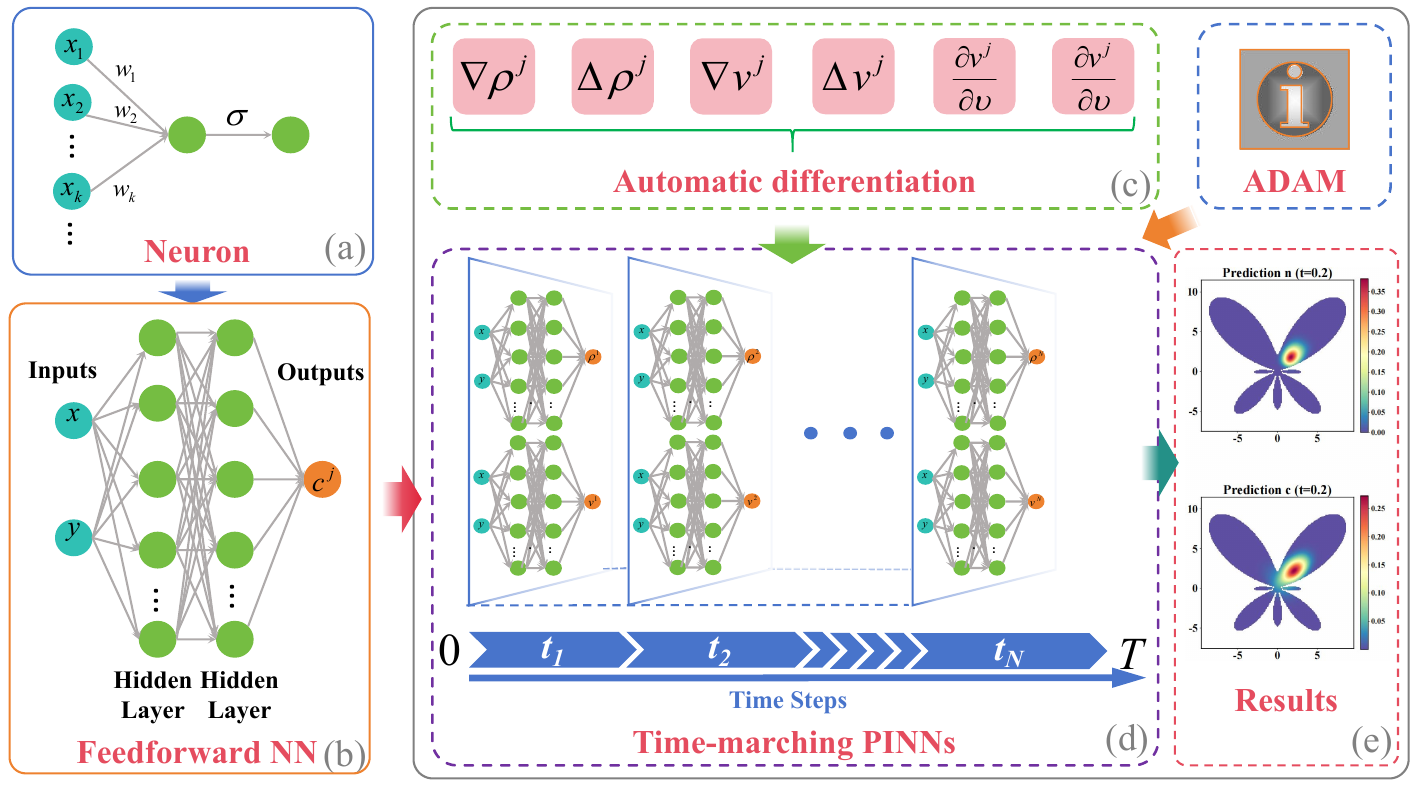}}
\caption{Schematic workflow of the multi-objective positivity/non-negativity-preserving time-marching PINNs. The panels illustrate: (a) a single neuron model; (b) the multi-layer perceptron (MLP) architecture; (c)  the automatic differentiation mechanism; (d)  the overall time‑marching PINN framework; and (e) snapshots of typical numerical results.}
\label{Fig:exampNN}
\end{figure}

\begin{itemize}
    \item \textit{\textbf{Step $(ii)$}}: Solution representation via Deep Neural Networks (DNNs).
\end{itemize}

At each time step $t_j$, the unknown functions $\rho(\cdot,t_j)$ and $v(\cdot,t_j)$ are approximated by two independent feedforward neural networks, denoted by $\mathcal{N}_{\theta_{j}}$ and $\mathcal{N}_{\phi_{j}}$, respectively. Unlike recurrent neural networks, feedforward neural networks are characterized by an acyclic connectivity pattern, where information flows strictly in one direction from the input to the output.

As illustrated in Fig. \ref{Fig:exampNN}, the network consists of $L$ sequential layers. Formally, let $\mathbf{x}^{(i)}$ denote the output vector of the $i$-th layer. The transition from one layer to the next is given by Ref.~\refcite{Bishop24}:
\begin{equation}
\mathbf{x}^{(i+1)}
=\sigma\left(\mathbf{W}^{(i)}\mathbf{x}^{(i)}+\mathbf{b}^{(i)}\right), \quad i = 0, 1, \ldots, L-1,
\end{equation}
where $\sigma$ is a non-linear activation function, $\mathbf{W}^{(i)}$ is the learnable weight matrix, and $\mathbf{b}^{(i)}$ is the bias vector. Specifically, the dimensions of these components are as follows:
\begin{itemize}
  \item For the input layer ($i=0$): $\mathbf{x}^{(0)}=\mathbf{x}\in\Omega\subset\mathbb{R}^d$, where $d$ being the spatial dimension. The weight matrix $\mathbf{W}^{(0)}\in\mathbb{R}^{m\times d}$ and $\mathbf{b}^{(0)}\in\mathbb{R}^m$, with $m$ being the fixed number of neurons in each hidden layer.
  \item For hidden layers ($1\leq i\leq L-2$): $\mathbf{x}^{(i)}\in\mathbb{R}^m$. with $\mathbf{W}^{(i)}\in\mathbb{R}^{m\times m}$ and $\mathbf{b}^{(i)}\in\mathbb{R}^m$.
  \item For the output layer ($i=L-1$): the final output $\mathbf{x}^{(L)}\in\mathbb{R}^1$ (representing $\rho^j$ or $v^j$) is produced  with $\mathbf{W}^{(L-1)}\in\mathbb{R}^{1\times m}$ and $\mathbf{b}^{(L-1)}\in\mathbb{R}$.
\end{itemize}
This architecture is commonly known as a multi-layer perceptron (MLP). The hidden layers are typically fully connected, meaning each neuron in one layer connects to every neuron in the next\cite{Bishop24}. This dense connectivity enables the network to approximate the complex nonlinear mappings required by the time-nonlocal system.

In this framework, the neural networks take as input the collocation points $\xi^{\imath n}_\ell\in\Omega$ ($\ell=1,2,\ldots,M_{\imath n}$) and $\xi^{bc}_\jmath\in\partial\Omega$ ($\jmath=1,2,\ldots,M_{bc}$), and output the corresponding field values $\rho^j$ and $v^j$ at each time step $t_j$, as illustrated in Fig. \ref{Fig:exampNN} (b). Mathematically, the neural network approximations are expressed as\cite{Raissi19}
\begin{align}
\rho(\mathbf{x},t_j)\approx\rho^j_{\theta_{j}}(\mathbf{x}):=\mathcal{N}_{\theta_{j}}
&=\zeta^{(L-1)}_{\theta_{j}}\circ\sigma\circ\zeta^{(L-2)}_{\theta_{j}}
\circ\sigma\circ\cdots\circ\zeta^{(0)}_{\theta_{j}}(\mathbf{x}),             \\
v(\mathbf{x},t_j)\approx v^j_{\phi_{j}}(\mathbf{x}):=\mathcal{N}_{\phi_{j}}
&=\zeta^{(L-1)}_{\phi_{j}}\circ\sigma\circ\zeta^{(L-2)}_{\phi_{j}}
\circ\sigma\circ\cdots\circ\zeta^{(0)}_{\phi_{j}}(\mathbf{x}),
\end{align}
where $\sigma$ is the activation function, applied element-wise to the output of each hidden layer to introduce the necessary non-linearity to the network. For PINNs, $\sigma$  is required to be at least $C^2$-continuous, so that second-order spatial derivatives such as the Laplacian $\Delta$ are well-defined. Here we use the Sigmoid Linear Unit function, defined as $\sigma(\cdot)=\verb"SiLU"(\cdot)$, given by $\text{SiLU}(x)=x\cdot\frac{1}{1+ e^{-x}}$. This function is infinitely differentiable and helps mitigate the vanishing gradient problem during backpropagation, making it a common choice for solving PDEs. The affine transformation at each layer $i$ is given as
\begin{align}
\zeta^{(i)}_{\theta_{j}}(\mathbf{x})=\mathbf{W}^{(i)}_{\theta_{j}}\mathbf{x}
+\mathbf{b}^{(i)}_{\theta_{j}},\quad
\zeta^{(i)}_{\phi_{j}}(\mathbf{x})=\mathbf{W}^{(i)}_{\phi_{j}}\mathbf{x}
+\mathbf{b}^{(i)}_{\phi_{j}},\quad i=0,1,\ldots,L-1.
\end{align}
For a hidden layer width $m$ and spatial dimension $d$, the weight matrices satisfy $\mathbf{W}^{(0)}\in\mathbb{R}^{m\times d}$, $\mathbf{W}^{(i)}\in\mathbb{R}^{m\times m}$ for $1<i< L-2$, and $\mathbf{W}^{(L-1)}\in\mathbb{R}^{1\times m}$. To maintain accuracy, we use the same network architecture (depth $L$ and width $m$) across all time steps.

Since the semi-discrete system \eqref{Pro:tranFKSdist} has a time-stepping structure, the solution is computed sequentially at each time step $t_j$ for $j=1, 2,\ldots, N$, starting from the initial state at $j=0$. We adopt a time-marching strategy based on PINNs\cite{Raissi19}, in which a separate neural  network is constructed at each time level $t_j$ to approximate $\rho^j(\mathbf{x})$ and $v^j(\mathbf{x})$. As depicted in Fig. \ref{Fig:exampNN} (d), this step-by-step procedure ensures that the historical values required by the $L_1$ scheme are already fixed when solving for the current time level.

\begin{itemize}
    \item \textit{\textbf{Step $(iii)$}}: Spatial derivative computation.
\end{itemize}

In the time-marching PINNs implementation, all required spatial derivatives, including $\nabla \rho^{j}$, $\Delta \rho^{j}$, $\nabla v^{j}$, $\Delta v^{j}$, and the normal derivatives $\frac{\partial \rho^{j}}{\partial \mathbf{\nu}}$ and $\frac{\partial v^{j}}{\partial \mathbf{\nu}}$ for $j\in\{0,1,2,\ldots,N\}$, are computed via the automatic differentiation (AD)\cite{Atilim18}, which is natively supported in modern deep learning frameworks. This approach avoids manual derivation or numerical discretization of spatial operators, and helps maintain high accuracy in enforcing the physical constraints in \eqref{Pro:tranFKSdist}.

\begin{itemize}
    \item \textit{\textbf{Step $(iv)$}}: Construction of the loss functions.
\end{itemize}

To train the neural networks at each time step $t_j$, a composite loss function is constructed. It consists of the residuals from the semi-discrete PDE system, boundary conditions, and initial conditions.

Specifically, at each $t_j$, we define the loss function as the weighted sum
\begin{equation}
\begin{aligned}
\mathcal{L}oss(\theta_{j},\phi_j;t_j)
=&\lambda_{PDE}\,\mathcal{L}oss_{RS}\big(\theta_{j},\phi_j;t_j\big)
  +\lambda_{BC}\,\mathcal{L}oss_{BC}\big(\theta_{j},\phi_j;t_j\big)\\
 &+\lambda_{IC}\,\mathcal{L}oss_{IC}\big(\theta_{0},\phi_0;t_0\big),
\end{aligned}
\end{equation}
where $\phi_{j}=\{\mathbf{W}_{\phi_j}^{(i)}, \mathbf{b}_{\phi_j}^{(i)}\}$ and $\theta_{j}=\{\mathbf{W}_{\theta_j}^{(i)}, \mathbf{b}_{\theta_j}^{(i)}\}$ ($j=1,2,\ldots,N$ and $i=0,1,\ldots,L-1$) are the trainable parameters of the networks at time level $t_j$. The weights $\lambda_{PDE}$, $\lambda_{BC}$ and $\lambda_{IC}$ are penalty parameters that balance the contributions of the PDE residuals, boundary conditions, and initial conditions, respectively.

The total loss $\mathcal{L}oss(\theta_{j},\phi_{j};t_j)$ is evaluated using collocation points sampled inside the domain $\Omega$ and on the boundary $\partial\Omega$, with $M_{\imath n}$ interior points and $M_{bc}$ boundary points. The individual loss components are defined below.

\begin{description}
  \item[$\bullet$] Residuals of the semi-discrete PDE system within the domain.
\end{description}

The interior loss at $t_j$ is defined as the sum of the residuals of the PDEs for the Myxobacteria density and the chemoattractant concentration
  \begin{equation}
     \mathcal{L}oss_{RS}\big(\theta_{j},\phi_{j};t_j\big)
     =\mathcal{L}oss^{\rho}_{RS}\big(\theta_{j},\phi_{j};t_j\big)
     +\mathcal{L}oss^{v}_{RS}\big(\theta_{j},\phi_{j};t_j\big),\quad j=1,2,\ldots,N,
  \end{equation}
where the individual residuals are computed as the Mean Squared Error (MSE) over the interior collocation points $\xi_{\ell}^{in}\in\Omega$:
\begin{equation*}
\left\{\begin{aligned}
     &\mathcal{L}oss^{\rho}_{RS}(\theta_{j},\phi_{j};t_j)
     =\frac{1}{M_{\imath n}}\sum_{\ell=1}^{M_{\imath n}}\big(\mathcal{P}_{\rho}[\mathcal{N}_{\theta_j}\odot\mathcal{N}_{\phi_j}](\xi^{\imath n}_{\ell})\big)^2,~ \xi^{\imath n}_{\ell}\in\Omega,~ \ell=1,\ldots,M_{\imath n},\\
     &\mathcal{L}oss^{v}_{RS}(\theta_{j},\phi_{j};t_j)
     =\frac{1}{M_{\imath n}}\sum_{\ell=1}^{M_{\imath n}}\big(\mathcal{P}_{v}[\mathcal{N}_{\theta_j}\odot\mathcal{N}_{\phi_j}](\xi^{\imath n}_{\ell})\big)^2,~ \xi^{\imath n}_{\ell}\in\Omega,~ \ell=1,\ldots,M_{\imath n}.
  \end{aligned}\right.
\end{equation*}
Here the symbol ``$\,\odot\,$''  indicates the mutual coupling between the two neural networks, which arises from the coupling between the PDEs in the system.

\begin{description}
  \item[$\bullet$] Boundary condition loss.
\end{description}

The Neumann boundary conditions on $\partial\Omega$ are enforced through the boundary loss
  \begin{equation}
     \mathcal{L}oss_{BC}\big(\theta_{j},\phi_j;t_j\big)
     =\mathcal{L}oss^{\rho}_{BC}\big(\theta_{j};t_j\big)
     +\mathcal{L}oss^{v}_{BC}\big(\phi_{j};t_j\big),\quad j=1,2,\ldots,N,
  \end{equation}
where the residuals are evaluated at boundary collocation points $\xi_{k}^{bc}\in\partial\Omega$ as
\begin{equation*}
\left\{\begin{aligned}
     &\mathcal{L}oss^{\rho}_{BC}(\theta_{j};t_j)
     =\frac{1}{M_{bc}}\sum_{\jmath=1}^{M_{bc}}
     \big(\widetilde{\mathcal{B}}[\mathcal{N}_{\theta_j}](\xi^{bc}_{\jmath})\big)^2,~ \xi^{bc}_{\jmath}\in\partial\Omega,~ \jmath=1,\ldots,M_{bc},\\
     &\mathcal{L}oss^{v}_{BC}(\phi_j;t_j)
     =\frac{1}{M_{bc}}\sum_{\jmath=1}^{M_{bc}}
     \big(\widetilde{\mathcal{B}}[\mathcal{N}_{\phi_j}](\xi^{bc}_{\jmath})\big)^2,~ \xi^{bc}_{\jmath}\in\partial\Omega,~ \jmath=1,\ldots,M_{bc}.
  \end{aligned}\right.
\end{equation*}

\begin{description}
  \item[$\bullet$] Initial condition loss.
\end{description}

To ensure consistency with the initial data $\rho^0$ and $v^0$, the initial loss is defined as
  \begin{equation}
     \mathcal{L}oss_{IC}(\theta_{j},\phi_j;t_0)
     =\mathcal{L}oss^{\rho}_{IC}(\theta_{j};t_0)+\mathcal{L}oss^{v}_{IC}(\phi_j;t_0),\quad j=1,2,\ldots,N,
  \end{equation}
with residuals
\begin{equation*}
\left\{\begin{aligned}
     &\mathcal{L}oss^{\rho}_{IC}(\theta_{j};t_0)
     =\frac{1}{M_{\imath n}}\sum_{\ell=1}^{M_{\imath n}}
     \big(\Xi[\mathcal{N}_{\theta_j}](\xi^{\imath n}_{\ell})-\rho^0(\xi^{\imath n}_{\ell})\big)^2,~ \xi^{\imath n}_{\ell}\in\Omega,~ \ell=1,\ldots,M_{\imath n},\\
     &\mathcal{L}oss^{v}_{IC}(\phi_{j};t_0)
     =\frac{1}{M_{\imath n}}\sum_{\ell=1}^{M_{bc}}
     \big(\Xi[\mathcal{N}_{\phi_j}](\xi^{\imath n}_{\ell})-v^0(\xi^{\imath n}_{\ell})\big)^2,~ \xi^{\imath n}_{\ell}\in\Omega,~ \ell=1,\ldots,M_{\imath n}.
  \end{aligned}\right.
\end{equation*}

\begin{description}
  \item[Step V:] Sequential time-stepping optimization with parameter transfer.
\end{description}

At each time level $t_j$, the optimal network parameters $\theta^*_j$ and $\phi^*_j$ are obtained by solving the minimization problem
\begin{equation}
\big(\theta^*_j,\phi^*_j\big)
=\operatorname{arg\,min}\limits_{(\theta_j,\phi_j)} \mathcal{L}oss\big(\theta_{j},\phi_j;t_j\big).
\end{equation}
This is carried out using the ADAM optimizer\cite{Adam-2014}, a stochastic gradient-based method widely used in deep learning.

The semi-discrete systems are solved sequentially in a time-marching manner. For each step $j\in\{1,2,\ldots,N\}$, the network parameters are initialized with the optimized values from the previous time level, i.e., $\theta^*_{j-1}$ and $\phi^*_{j-1}$ (with the initial condition used for $j=1$). This parameter transfer strategy exploits the temporal continuity of the solution, providing a good initial guess that facilitates efficient and stable convergence. After the loss is minimized iteratively via ADAM, the resulting approximations  $\rho^j$ and $v^j$ are stored as historical data for the subsequent steps.

In this time-marching PINNs framework, error accumulation is mitigated by training independent networks at each time step. The parameters from previous steps serve only as a warm start; they are fully re-optimized through the physics-constrained loss. This strategy confines approximation errors to individual time levels, while still benefiting from the stability of the $L_1$ discretization and the expressive power of deep neural networks.

\subsection{Simulations: 2D case with known source terms}
\label{subsec:example1}
To validate the numerical performance of the proposed algorithm, we consider a two-dimensional circular domain $\Omega=\{(x,y):x^{2}+y^2\leq R^2\}$ with radius $R>0$. Defining the auxiliary spatial function $\phi=(r^{2}-R^{2})^2$ where $r^2=x^2+y^2$. We employ the method of manufactured solutions by adding source terms $f(x,y,t)$ and $g(x,y,t)$ to the governing equations for $\rho$ and $v$ in \eqref{Pro:tranFKS}, respectively:
\begin{align*}
f(x,y,t)
&=e^{-\phi}\Big\{\Gamma(\alpha+1)-(1+t^\alpha)
\Big(\mathcal{D}-\frac{\mathcal{D}\chi}{3}\Big)\Big(16r^2(r^2-R^2)^2-16r^2+8R^2\Big)\Big\},\\
g(x,y,t)
&= e^{-\phi/3}\Gamma(\alpha+1)-\mathcal{D}\big(1+t^\alpha\big)e^{-\phi/3}
\Big(\frac{16}{9}r^2\big(r^2-R^2\big)^2-\frac{16}{3}r^2+\frac{8}{3}R^2\Big)\\
&\quad+\gamma e^{-\phi/3}(1+t^\alpha)-e^{-\phi}\big(1+t^\alpha\big).
\end{align*}
This construction yields exact solutions with weak temporal regularity ($c^{\alpha}(0,T)$, $0<\alpha<1$), given by
\begin{equation}
\left\{
\begin{aligned}
n(x,y,t)&=\rho(x,y,t)^2= e^{-\phi}\,\big(1+t^\alpha\big),\\
c(x,y,t)&=\exp(v(x,y,t))= e^{-\phi/3}\,\big(1+t^\alpha\big).
\end{aligned}
\right.
\end{equation}
These exact solutions are consistent with the prescribed initial and boundary conditions.
\begin{figure}[!ht]
 \centering
\begin{minipage}[c]{0.8\textwidth}
 \centerline{\includegraphics[width=1\textwidth]{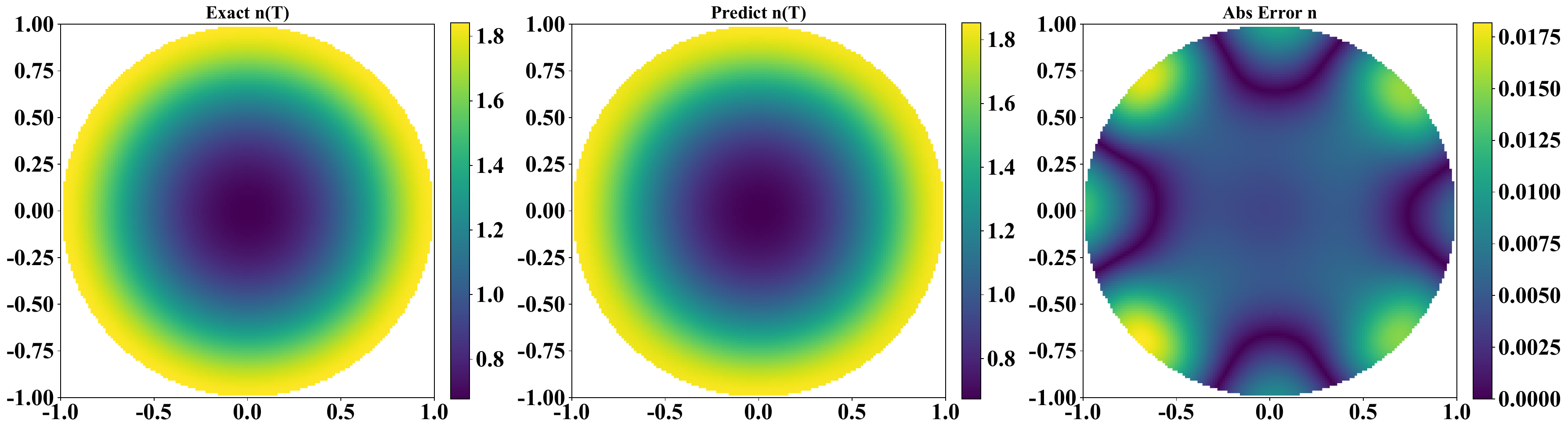}}
\end{minipage}
\begin{minipage}[c]{0.8\textwidth}
 \centerline{\includegraphics[width=1\textwidth]{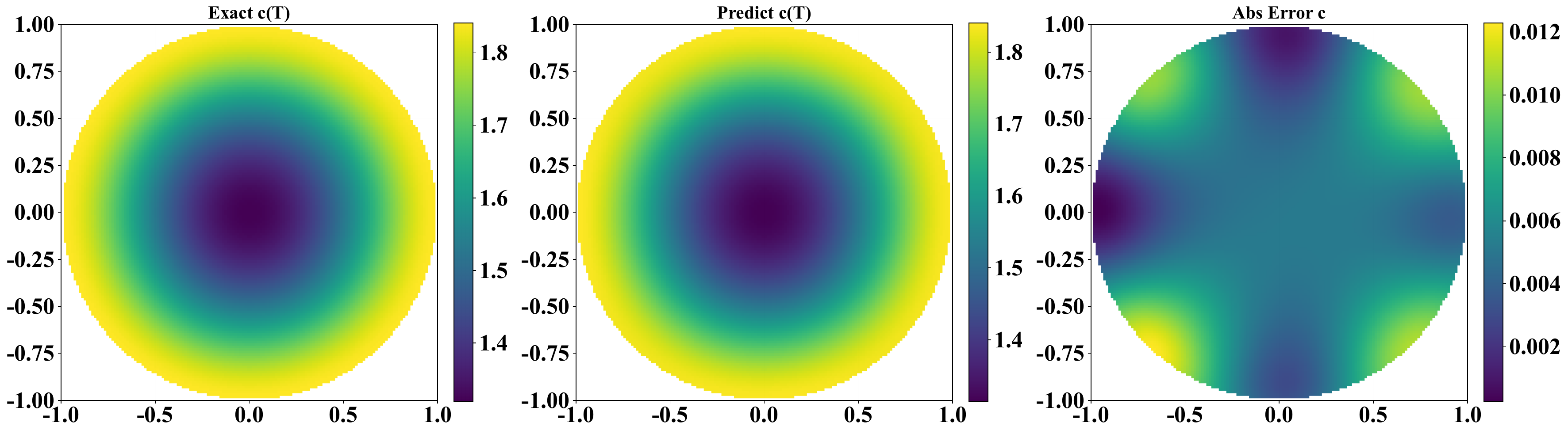}}
\end{minipage}
\caption{(Color online) Visual comparison between the reference and approximate solutions for $\alpha=0.25$ at the final time $T=0.5$.}
\label{Fig:exp1-refapprx25}
\end{figure}

\begin{figure}[!ht]
 \centering
 \centerline{\includegraphics[width=1\textwidth]{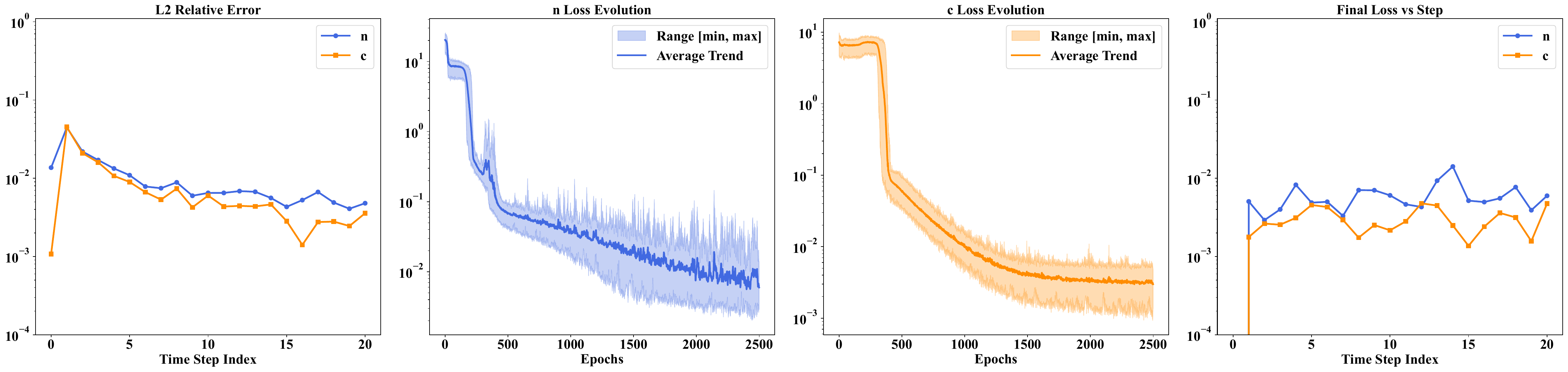}}
\caption{(Color online) Numerical performance of the positivity-preserving time-marching PINN for $\alpha=0.25$ at final time $T=0.5$. The panels show, from left to right: the relative $L^{2}$ error, the the loss evolution for $n(x,y,t)$ and $c(x,y,t)$ at each epoch, and the evolution of the total loss against the number of time steps.}
\label{Fig:exp1-numerp25}
\end{figure}

For the numerical experiments, the parameters are set as $R=1$, $\mathcal{D}=1.0$, $\chi=0.25$ (with $\chi\in(0,1/2)$), and $\gamma=1.0$. The time interval $[0, T]$ with $T=0.5$ is uniformly  divided into $N=20.0$ steps. We use $M_{\imath n}=1000.0$ interior collocation points and $M_{bc}=1000.0$ boundary points. The neural network has  $3$ hidden layers, each with $64$ neurons. At each time step, the model is trained for $2500$ iterations with a decreasing learning rate $lr=1e-3$. The loss weights are set to $\lambda_{PDE}=1.0$, $\lambda_{BC}=100.0$, and $\lambda_{IC}=100.0$. To evaluate the accuracy of the algorithm, we define the relative $L^{2}$ error
\begin{equation}
\text{Relative } L^2 \text{ Error}
=\frac{\|u_{predic}-u_{exact}\|_{L^{2}(\Omega)}}{\|u_{exact}\|_{L^{2}(\Omega)}},\quad u\in\{n,c\},
\end{equation}
for both solution components.

\begin{figure}[!ht]
 \centering
\begin{minipage}[c]{1\textwidth}
 \centerline{\includegraphics[width=0.8\textwidth]{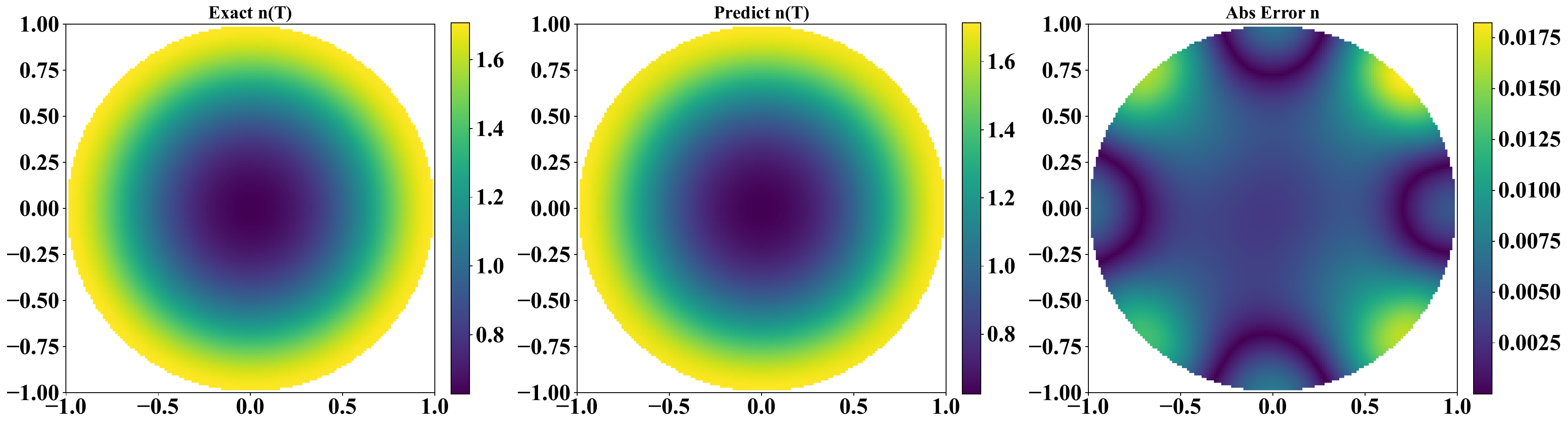}}
\end{minipage}
\begin{minipage}[c]{1\textwidth}
 \centerline{\includegraphics[width=0.8\textwidth]{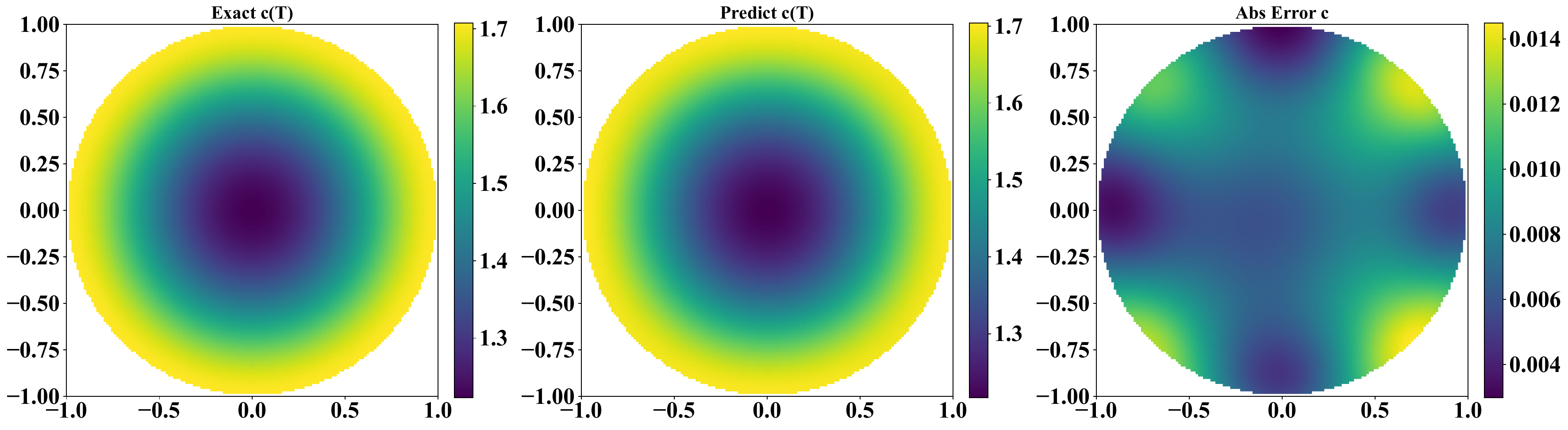}}
\end{minipage}
\caption{{\rm(}Color online{\rm)} Visual comparison of the reference solution and the approximate solution for $\alpha=0.50$ at the final time $T=0.5$.}
\label{Fig:exp1-refapprx50}
\end{figure}

\begin{figure}[!ht]
 \centering
 \centerline{\includegraphics[width=1\textwidth]{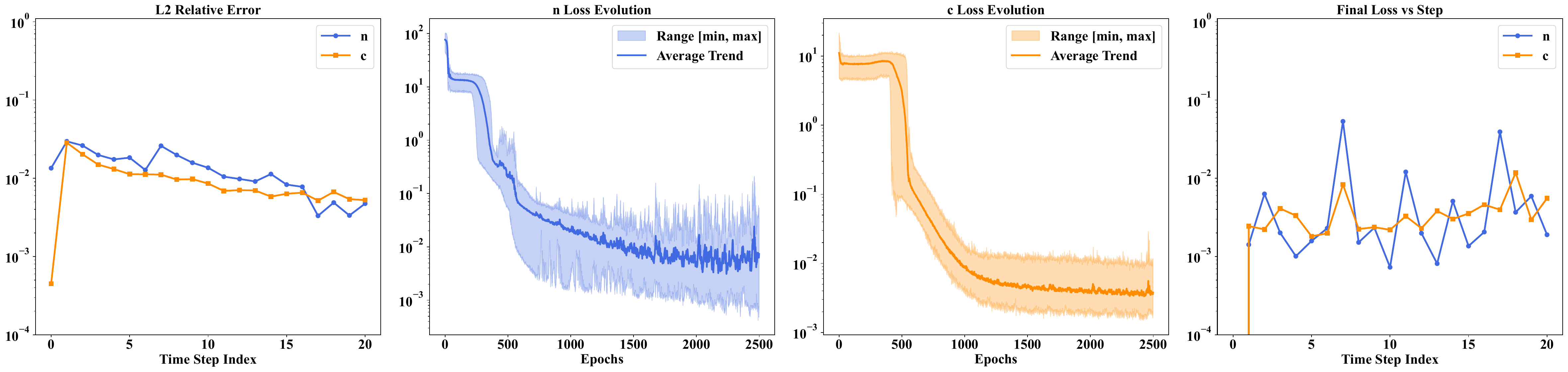}}
\caption{(Color online) The numerical performance of the non-negativity/positivity-preserving time-marching PINNs algorithm when $\alpha=0.50$ and $T=0.5$. From left to right, the relative $L^{2}$  error, the training loss for $n$ and $c$ at each time step, and the final loss versus time steps.}
\label{Fig:exp1-numerp50}
\end{figure}

The numerical results on the unit disk are presented in Figures \ref{Fig:exp1-refapprx25}, \ref{Fig:exp1-numerp25}, \ref{Fig:exp1-refapprx50}, \ref{Fig:exp1-numerp50}, \ref{Fig:exp1-refapprx75}, and \ref{Fig:exp1-numerp75}. These experiments demonstrate that a compact feedforward neural network is capable of approximating the solutions with satisfactory accuracy. Three main observations can be highlighted:
\begin{itemize}
  \item \textbf{Good accuracy}: The predicted solutions agree well with the exact ones, with absolute errors generally within 2\% and the relative $L^2$ error remaining stable across all time steps. This confirms that the proposed PINNs framework can effectively capture the solution behavior even with relatively limited network capacity.
  \item \textbf{Stable Training}: The training processes for both $n(x,y,t)$ and $c(x,y,t)$ are smooth, with the loss functions decreasing steadily toward convergence. The independent network architecture and parameter transfer strategy appear to contribute positively to training stability, without introducing noticeable oscillations or overshooting.
  \item \textbf{Robustness}: The algorithm performs consistently well across different values of the fractional order $\alpha$, indicating that the time-marching scheme, combined with the $L_1$ discretization, is adaptable to various degrees of nonlocal memory effects.
\end{itemize}

Overall, these results indicate that the proposed method is reliable and computationally efficient, and offers sufficient accuracy for the present test cases.

\begin{figure}[!ht]
 \centering
\begin{minipage}[c]{0.8\textwidth}
 \centerline{\includegraphics[width=1\textwidth]{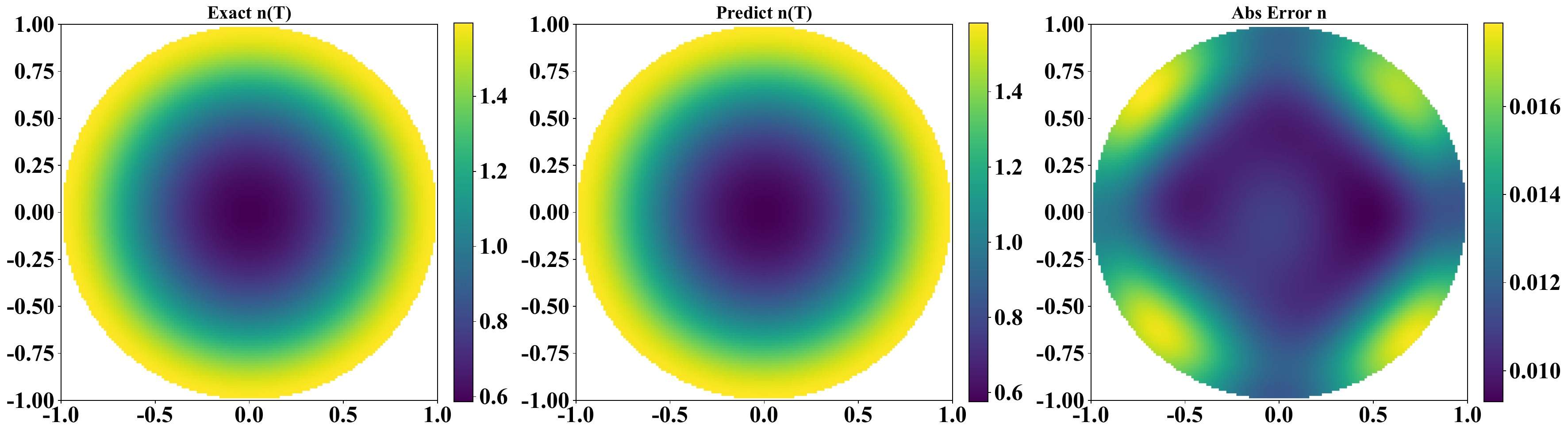}}
\end{minipage}
\begin{minipage}[c]{0.8\textwidth}
 \centerline{\includegraphics[width=1\textwidth]{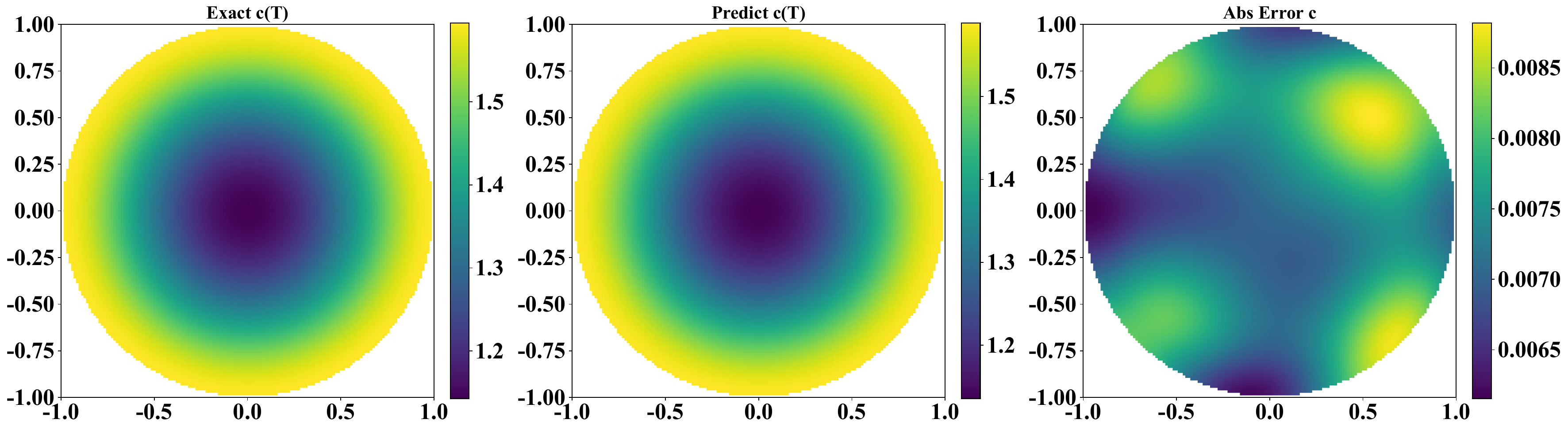}}
\end{minipage}
\caption{{\rm(}Color online{\rm)} Visual comparison of the reference solution and the approximate solution for $\alpha=0.75$ at the final time $T=0.5$.}
\label{Fig:exp1-refapprx75}
\end{figure}

\begin{figure}[!ht]
 \centering
 \centerline{\includegraphics[width=1\textwidth]{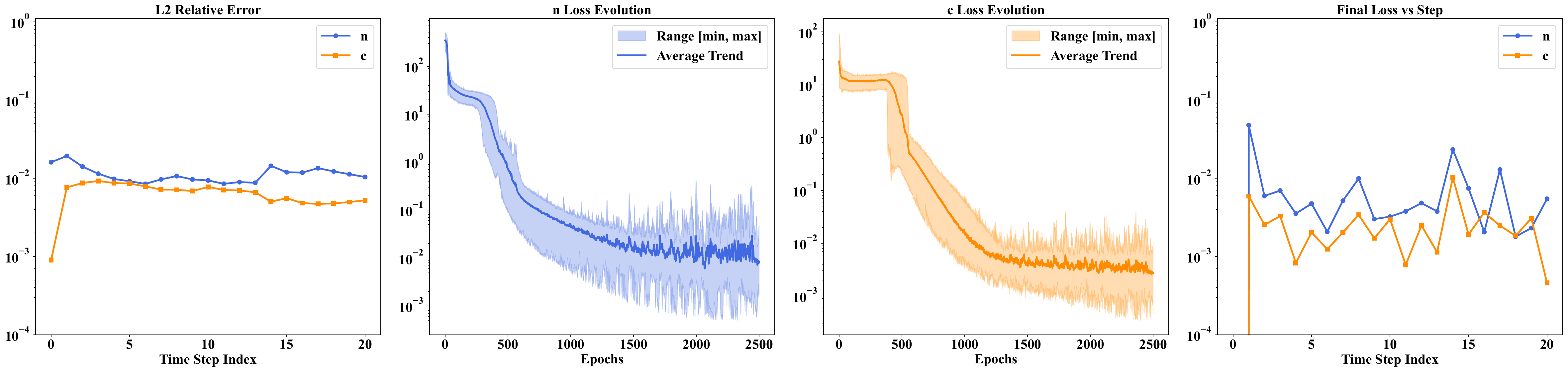}}
\caption{{\rm(}Color online{\rm)} The numerical performance of the non-negativity/positivity-preserving time-marching PINNs algorithm when $\alpha=0.75$ and $T=0.5$. From left to right, the figures illustrate the $L^{2}$ relative error, the total loss trends of $n$ and $c$ at each time step, and the final loss versus time steps.}
\label{Fig:exp1-numerp75}
\end{figure}

\subsection{Simulations: 2D case with Gaussian initial data}
To evaluate the numerical performance of the proposed time-marching PINNs on complex geometries, we consider a two-dimensional bounded domain $\Omega\subset\mathbb{R}^2$, often referred to as the ``butterfly" domain. Its boundary is parameterized in polar coordinates $(r,\phi)$ by
\begin{equation}\label{domain:butter}
\partial\Omega:=\Bigg\{(r,\phi)\big|r(\phi) = 3\times\Big|\;e^{\sin(\phi)}-2\cos(4\phi)+\sin^5\Big(\frac{2\phi-\pi}{24}\Big)\Big|, \phi \in [0, 2\pi]\Bigg\}.
\end{equation}
This highly irregular and non-convex region is chosen intentionally to illustrate the flexibility of deep neural networks in handling complex spatial domains. Unlike traditional grid-based methods, the present approach does not require mesh generation, which becomes particularly advantageous for such intricate geometries.

\begin{figure}[htbp]
 \centering
 \centerline{\includegraphics[width=0.36\textwidth]{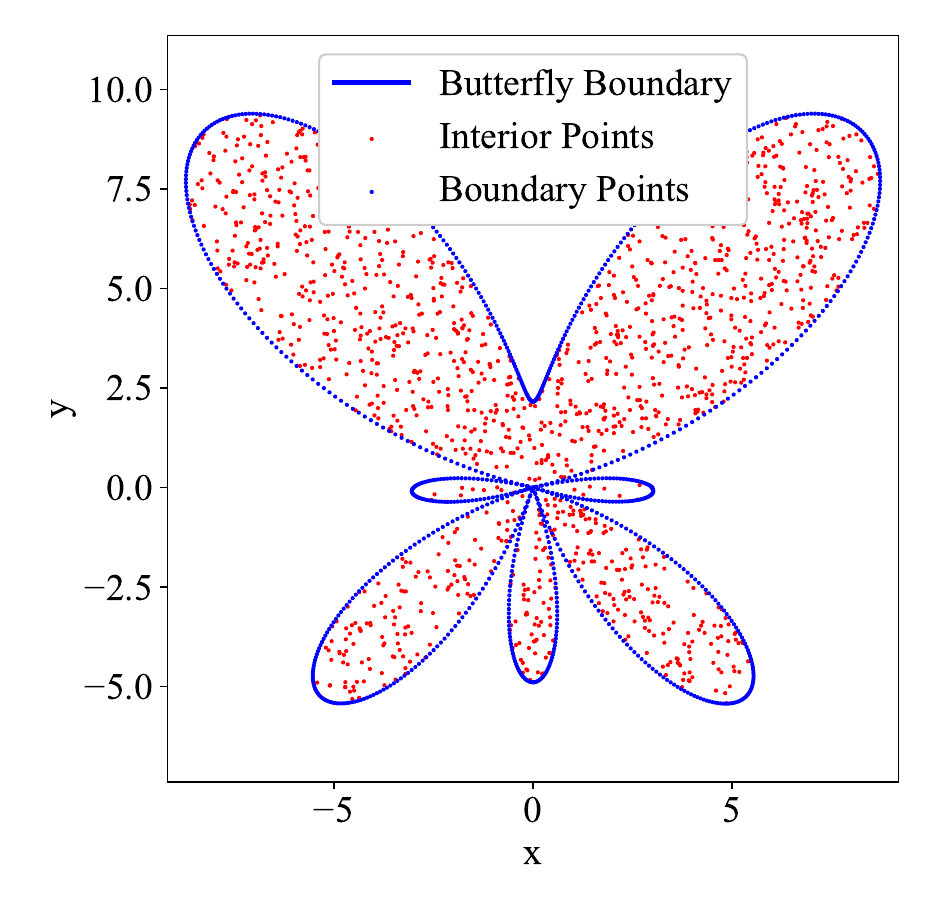}}
\caption{{\rm(}Color online{\rm)} Numerical visualization of the 2D butterfly-shaped computational domain, showing the distribution of interior collocation points and boundary sampling points.}
\label{Fig:Domain}
\end{figure}

\begin{figure}[!ht]
 \centering
 \centerline{\includegraphics[width=0.88\textwidth]{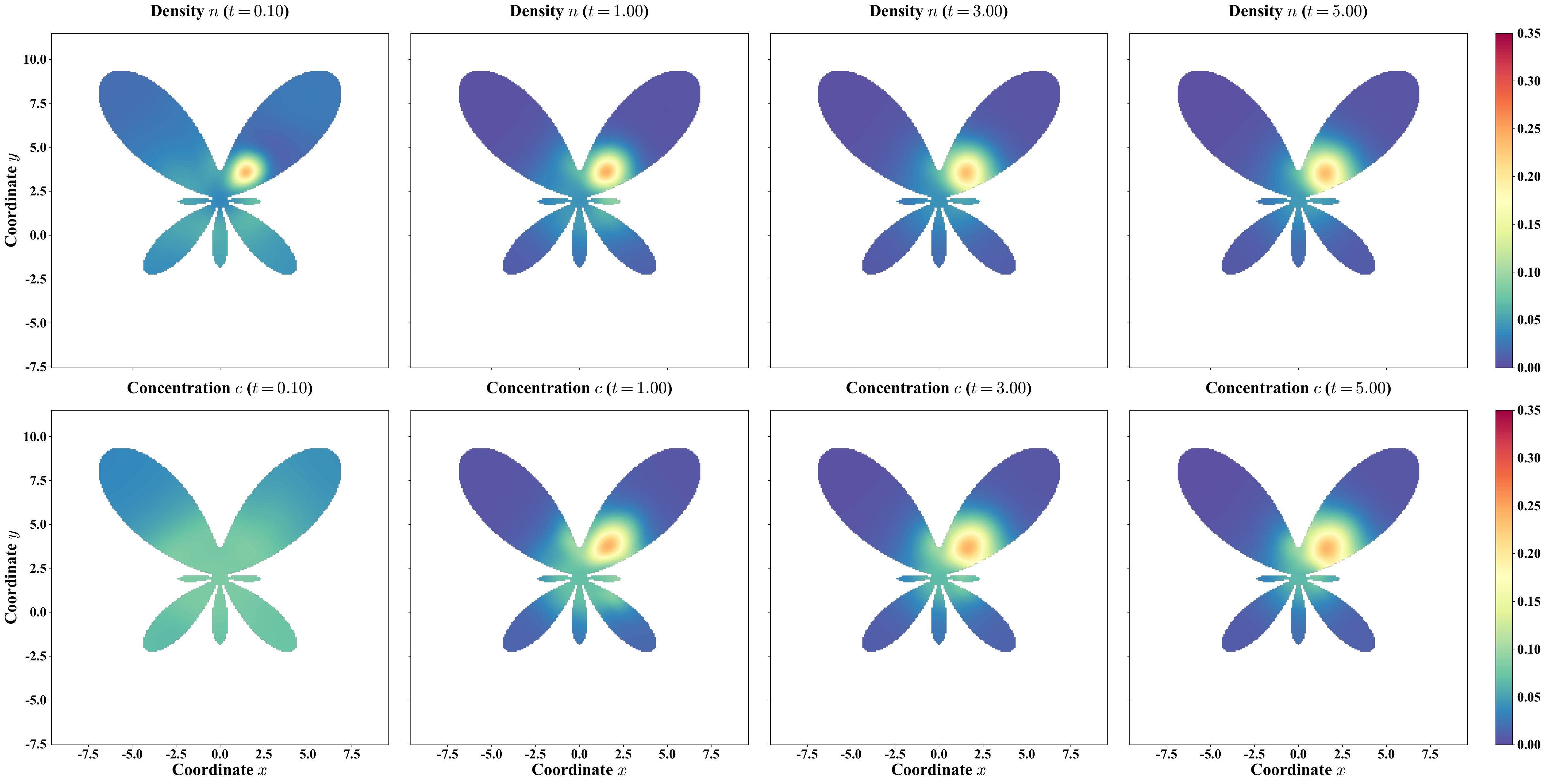}}
\caption{{\rm(}Color online{\rm)}  Evolution of solutions on the butterfly-shaped domain computed by the time-marching PINNs for $\alpha=0.25$. The simulation parameters are $N=50$, $T=5$, $\mathcal{D}=1$, $\chi=0.25$, $\gamma=1$. The neural network consists of five hidden layers with $100$ neurons each, employs $\text{SiLU}(\cdot)$ activation functions, and is trained for $1500$ iterations per time step with a decreasing learning rate of $1e-3$.}
\label{Fig:exampFKS25}
\end{figure}

\begin{figure}[!ht]
 \centering
 \centerline{\includegraphics[width=1\textwidth]{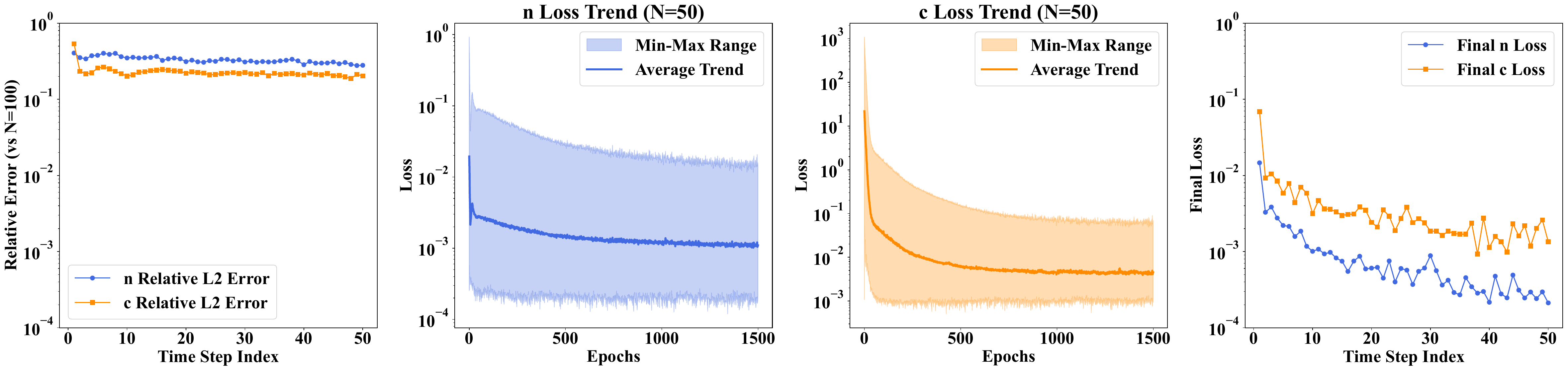}}
\caption{{\rm(}Color online{\rm)} The numerical performance of the non-negativity/positivity-preserving time-marching PINNs algorithm with $\alpha=0.25$. From left to right, the figures illustrate the $L^{2}$ relative error, the total Loss trends of $n$ and $c$ at each time step, and the final loss versus time steps.}
\label{Fig:exampLoss25}
\end{figure}

\begin{figure}[!ht]
 \centering
 \centerline{\includegraphics[width=0.69\textwidth]{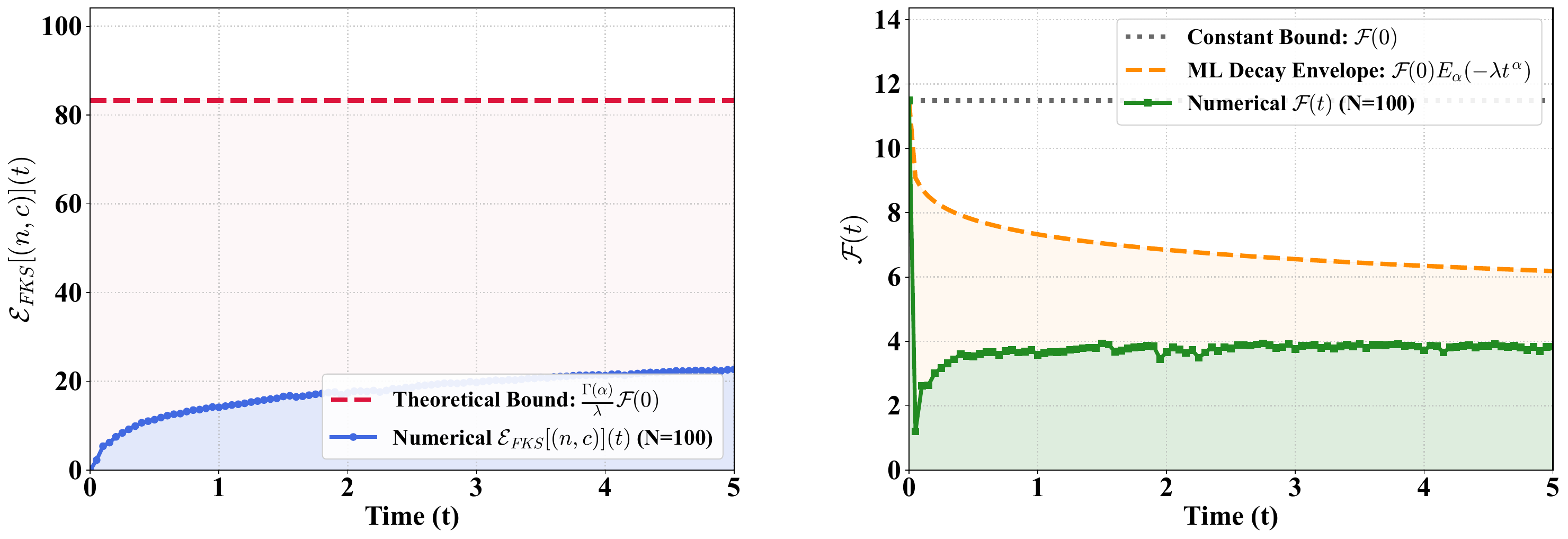}}
\caption{{\rm(}Color online{\rm)} The theoretical and numerical results of the energy functions with $\alpha=0.25$. From left to right, the figures illustrate the evolution of the novel Lyapunov function, defined by \eqref{F:Lyapunov}, and the jointly convex function $\mathcal{F}(t)$ (defined in \eqref{Fun:JCF}) over time $t$.}
\label{Fig:LF25}
\end{figure}

To ensure mathematical consistency between the initial data and the homogeneous Neumann boundary conditions, a quadratic masking factor $\Psi(x,y)$ is introduced. The initial distributions $n(x,y,t)$ and $c(x,y,t)$ are defined as
\begin{equation}
\left\{\begin{aligned}
n(x,y,0)
&=\pi\left|\cos(\pi x)\cos(\pi y)\right|\exp\left(-0.75\left[(x-2)^2+(y-2)^2\right]\right)\cdot\Psi(x,y),\\
c(x,y,0)
&=\pi\left|\cos(\pi x)\cos(\pi y)\right|\exp\left(-0.25\left[(x-2)^2+(y-2)^2\right]\right)\cdot\Psi(x,y).
\end{aligned}\right.
\end{equation}
The compatibility factor $\Psi(x,y)$ is given by
\begin{equation}
\Psi(x,y):=\left(1-\sigma^2(x,y)\right)^2,\quad
\sigma(x,y):=\frac{\sqrt{x^2+y^2}}{r(\operatorname{atan2}(y,x))},
\end{equation}
where $r(\phi)$ is the radial distance at angle $\phi$, as defined in \eqref{domain:butter}, and $\phi=\operatorname{atan2}(y,x)$ is the four-quadrant inverse tangent (available in  PyTorch code \verb"torch.atan2(y,x)").

By construction, the factor $\Psi(x,y)$ vanishes quadratically at the boundary, where $\sigma=1$.  This ensures two key properties:
\begin{itemize}
  \item Boundary value vanishing: $\Psi(1) = 0$, so both $n(x,y,t)$ and $c(x,y,t)$ vanish on $\partial\Omega$.
  \item Gradient vanishing: $\frac{\partial \Psi}{\partial\sigma}\big|_{\sigma=1}=0$, which guarantees that the normal derivatives $\frac{\partial n(x,y,t)}{\partial \mathbf{\nu}}$ and $\frac{\partial c(x,y,t)}{\partial\mathbf{\nu}}$ are zero regardless of the boundary curvature.
\end{itemize}
This construction provides a robust foundation for the convergence of the time-marching PINNs framework on irregular domains.

\begin{figure}[!ht]
 \centering
 \centerline{\includegraphics[width=0.88\textwidth]{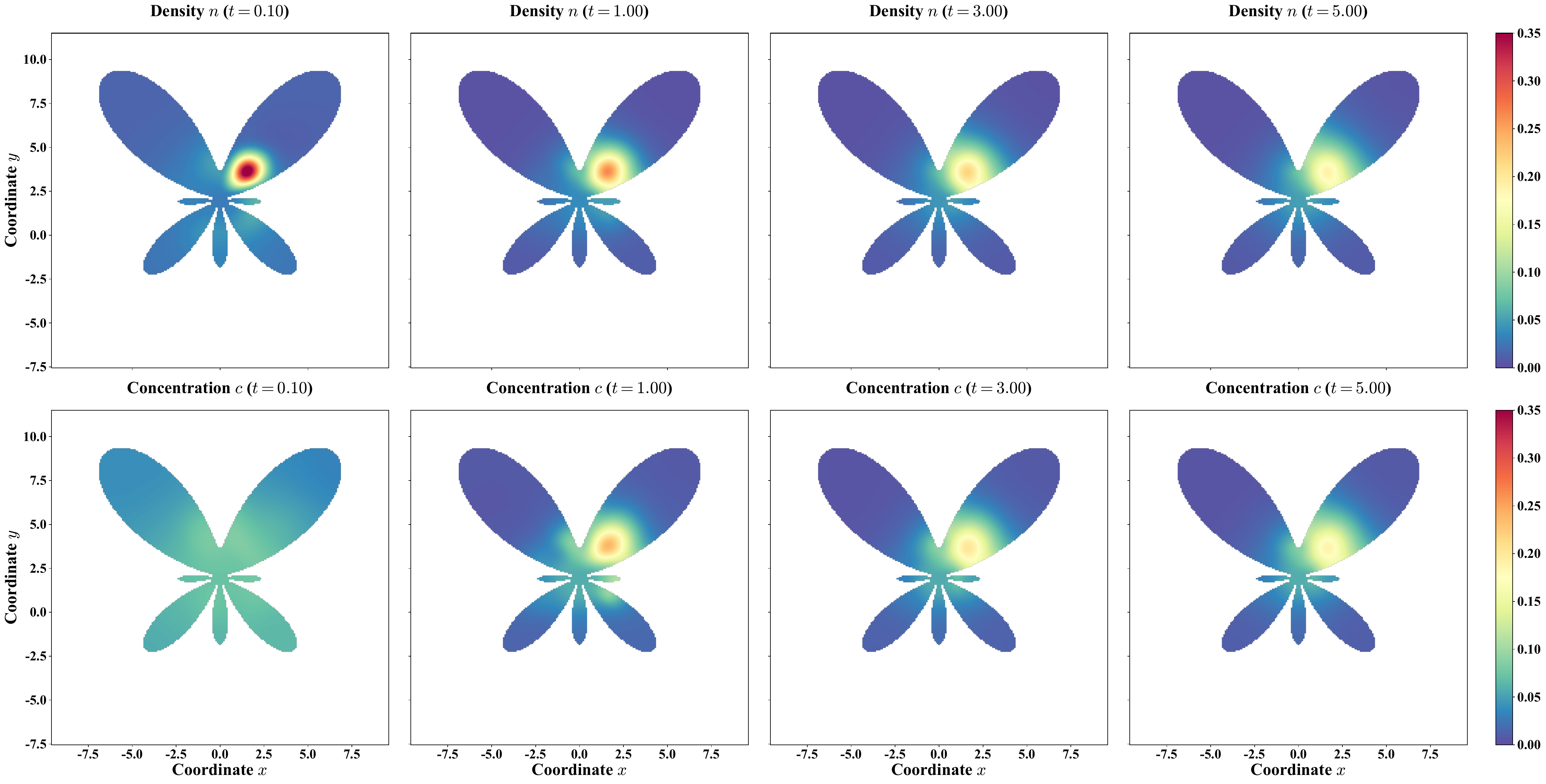}}
\caption{(Color online) Spatiotemporal evolution of the solutions on a bounded butterfly domain with $\alpha = 0.50$ and $N = 50$. All other parameters remain consistent with those in Fig.~\ref{Fig:exampFKS25}.}
\label{Fig:exampFKS50}
\end{figure}

\begin{figure}[!ht]
 \centering
 \centerline{\includegraphics[width=1\textwidth]{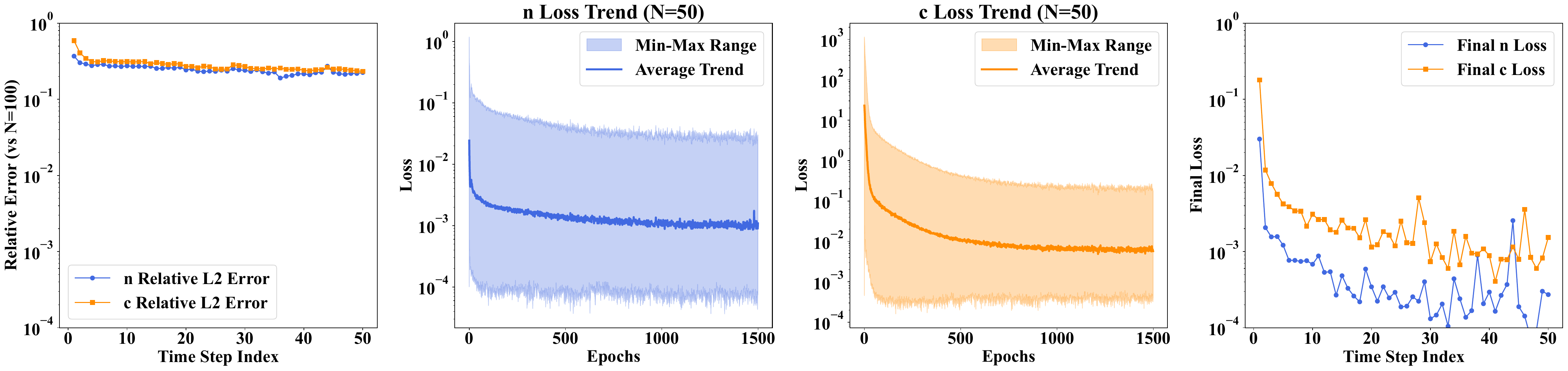}}
\caption{{\rm(}Color online{\rm)} Numerical performance of the non-negativity/positivity-preserving time-marching PINNs algorithm for $\alpha=0.50$. All other parameters are consistent with those in Fig.~\ref{Fig:exampLoss25}.}
\label{Fig:exampLoss50}
\end{figure}

\begin{figure}[!ht]
 \centering
 \centerline{\includegraphics[width=0.69\textwidth]{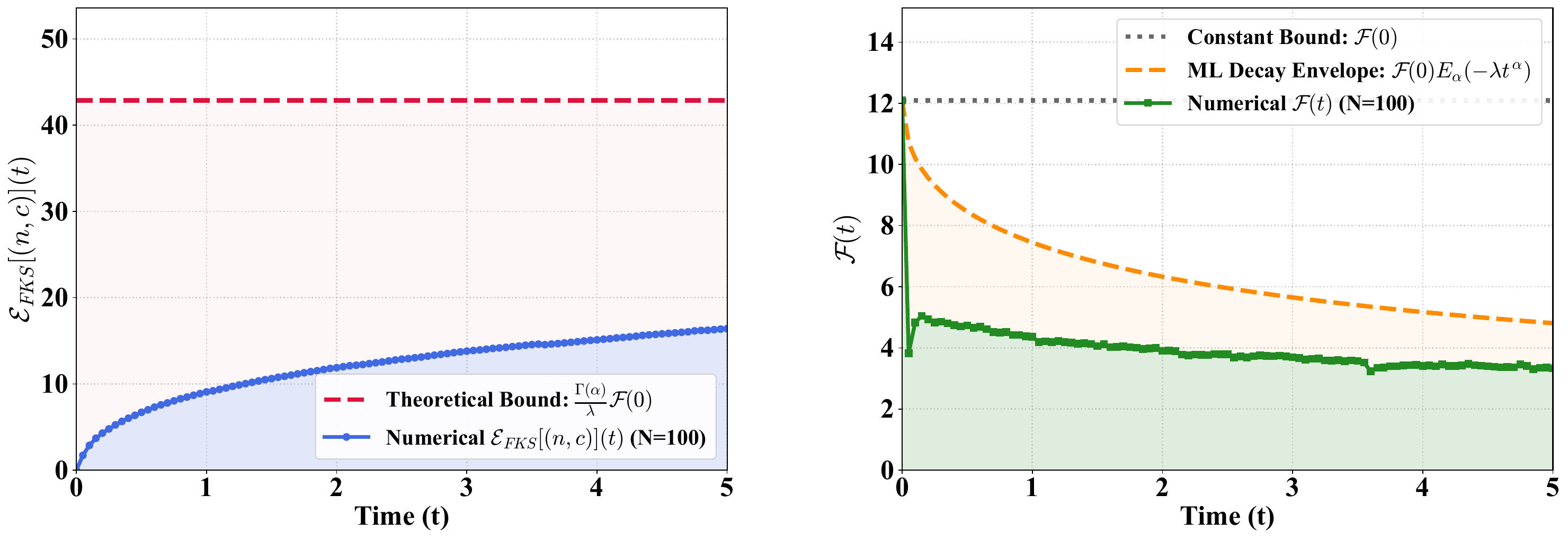}}
\caption{{\rm(}Color online{\rm)} The theoretical and numerical results of the energy functions with $\alpha=0.5$. From left to right, the figures illustrate the evolution of the novel Lyapunov function, defined by \eqref{F:Lyapunov}, and the jointly convex function $\mathcal{F}(t)$ (defined in \eqref{Fun:JCF}) over time $t$.}
\label{Fig:LF50}
\end{figure}

This numerical experiment aims to visualize the dynamical evolution of the system on a complex  geometry, specifically the butterfly-shaped domain \eqref{domain:butter}), and to provide numerical evidence supporting global existence. The network architecture and parameter settings are kept consistent with those in Subsection \ref{subsec:example1} to ensure compatibility and robustness.

Since no analytical solution is available, a high-fidelity reference solution ($n_{\text{Ref}, N=100}$, $c_{\text{Ref}, N=100}$) is generated using the proposed time-marching PINNs algorithm with $N=100$ time steps, $1500$ iterations per step, and $3$ hidden layers each with $64$ neurons. To evaluate the accuracy of the algorithm, the relative $L^2$ error is defined as
\begin{equation}
\text{Relative } L^2 \text{ Error} = \frac{\|u_{predic}-u_{Ref.,N=100}\|_{L^{2}(\Omega)}}{\|u_{Ref.,N=100}\|_{L^{2}(\Omega)}},\quad u\in\{n,c\}.
\end{equation}

The numerical results are presented in Figures \ref{Fig:exampFKS25} -- \ref{Fig:LF75}. Several observations can be drawn:
\begin{itemize}
    \item The evolution of $(n, c)$ shown in Figures  \ref{Fig:exampFKS25}, \ref{Fig:exampFKS50} and \ref{Fig:exampFKS75})  reveals that different values of $\alpha$ lead to different diffusion rates. In particular, a larger $\alpha$ corresponds to faster spreading. This is consistent with the modeling assumptions: since a larger $\alpha$ implies shorter waiting times, it results in higher mobility, which accelerates the dispersion of both myxobacteria and slime.

    \item The numerical performance illustrated in Figures \ref{Fig:exampLoss25}, \ref{Fig:exampLoss50}, and \ref{Fig:exampLoss75} demonstrates the stable training dynamics of our proposed deep learning algorithm. These results highlight the robustness of the DNN framework, especially for problems defined on non-trivial domains. While there is still room for quantitative improvement in the relative $L^2$ error, the current accuracy is sufficient for capturing the essential dynamics and meets the practical needs of the simulation.

    \item Figures \ref{Fig:LF25}, \ref{Fig:LF50}, and \ref{Fig:LF75} provide numerical verification of the new energy functionals defined in (see \eqref{F:Lyapunov} and \eqref{Fun:JCF}). The observed behavior is in good agreement with the theoretical results in Lemmas \ref{lem:dissapation} and \ref{Lem:FracLyapunov}. Specifically, the Lyapunov functional $\mathcal{E}$ remains uniformly bounded,  confirming the estimate $\mathcal{E}[(n,c)](t)\le\frac{\Gamma(\alpha)}{\lambda}\mathcal{F}(0)$ for $t>0$. Meanwhile, the convex functional $\mathcal{F}(t)$ exhibits the expected dissipation property, satisfying $\mathcal{F}(t)\leq\mathcal{F}(0) E_\alpha(-\lambda t^\alpha)$ for $t>0$.
\end{itemize}

In summary, the proposed method reproduces the expected diffusion behavior and respects the theoretical energy constraints on irregular geometries. The flexibility of the architecture also suggests that it can be extended to more complex coupled systems, offering a promising numerical tool for such problems.

\begin{figure}[!ht]
 \centering
 \centerline{\includegraphics[width=0.88\textwidth]{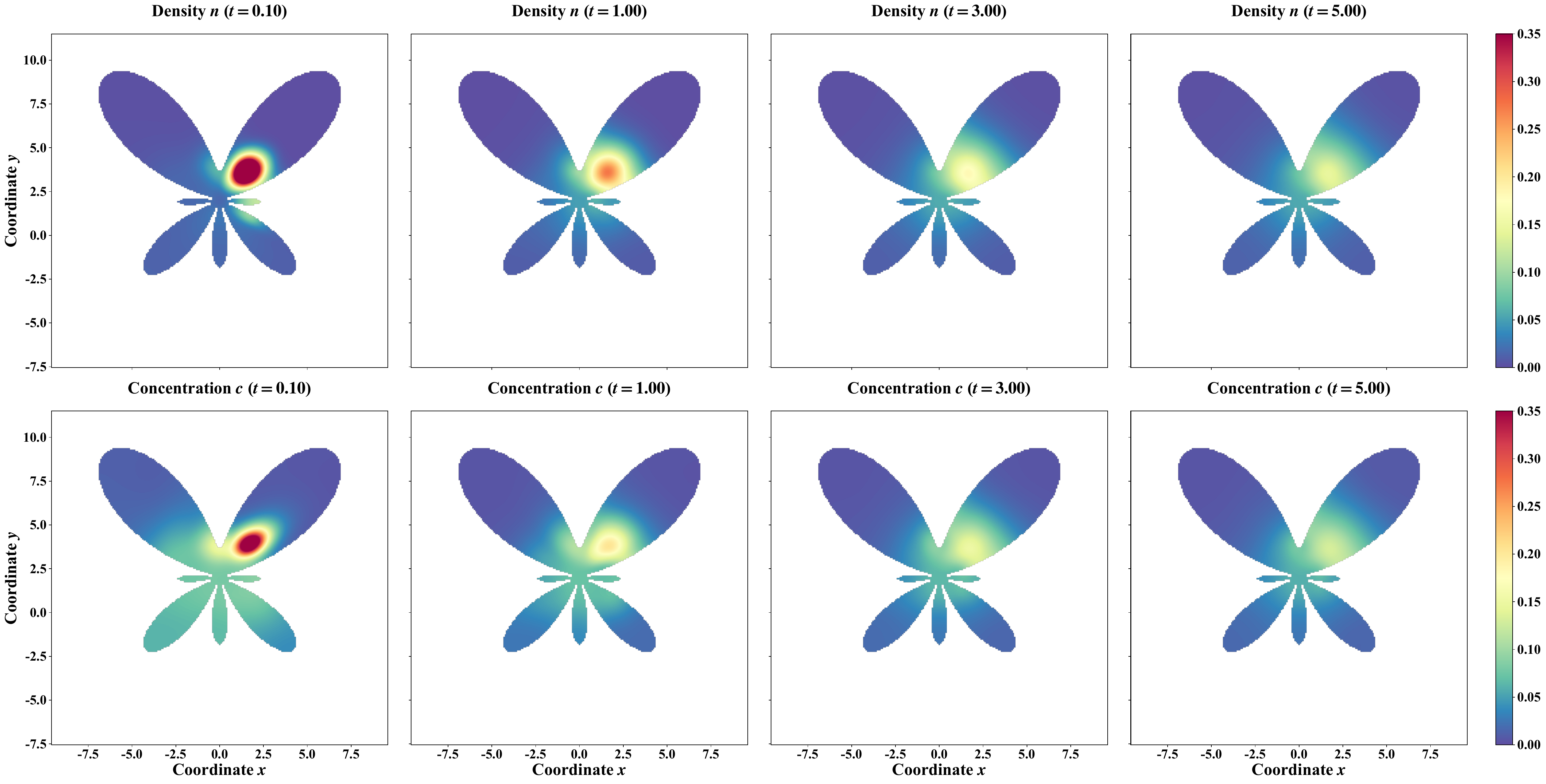}}
\caption{{\rm(}Color online{\rm)} The dynamical evolutions of solutions on the bounded butterfly domain with $\alpha=0.75$. All other parameters are set the same as Fig.~\ref{Fig:exampFKS25}.}
\label{Fig:exampFKS75}
\end{figure}

\begin{figure}[!ht]
 \centering
 \centerline{\includegraphics[width=1\textwidth]{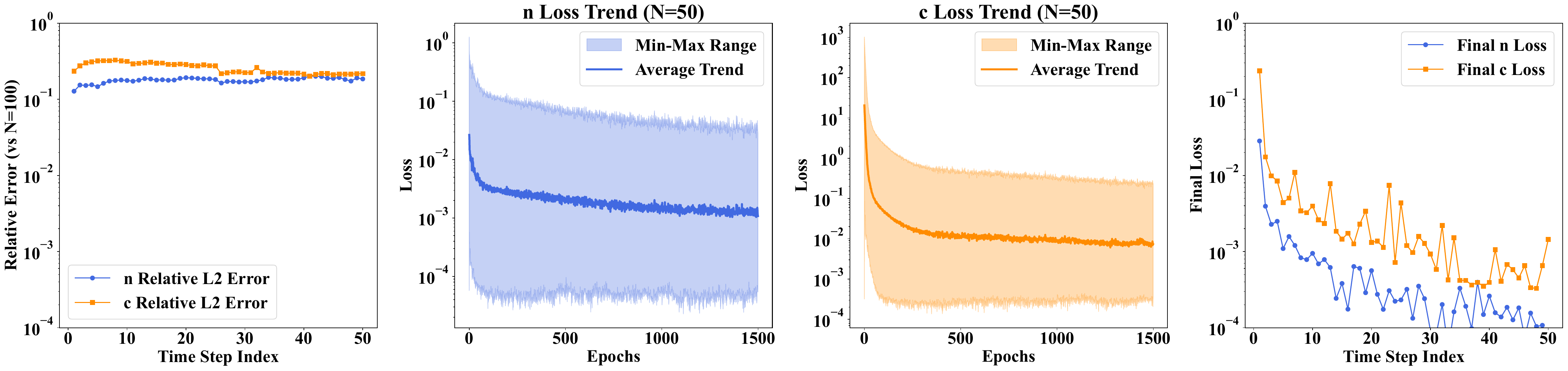}}
\caption{{\rm(}Color online{\rm)} The numerical performance of the non-negativity/positivity-preserving time-marching PINNs algorithm with $\alpha=0.75$. All other parameters are set the same as Fig.~\ref{Fig:exampLoss25}.}
\label{Fig:exampLoss75}
\end{figure}

\begin{figure}[!ht]
 \centering
 \centerline{\includegraphics[width=0.69\textwidth]{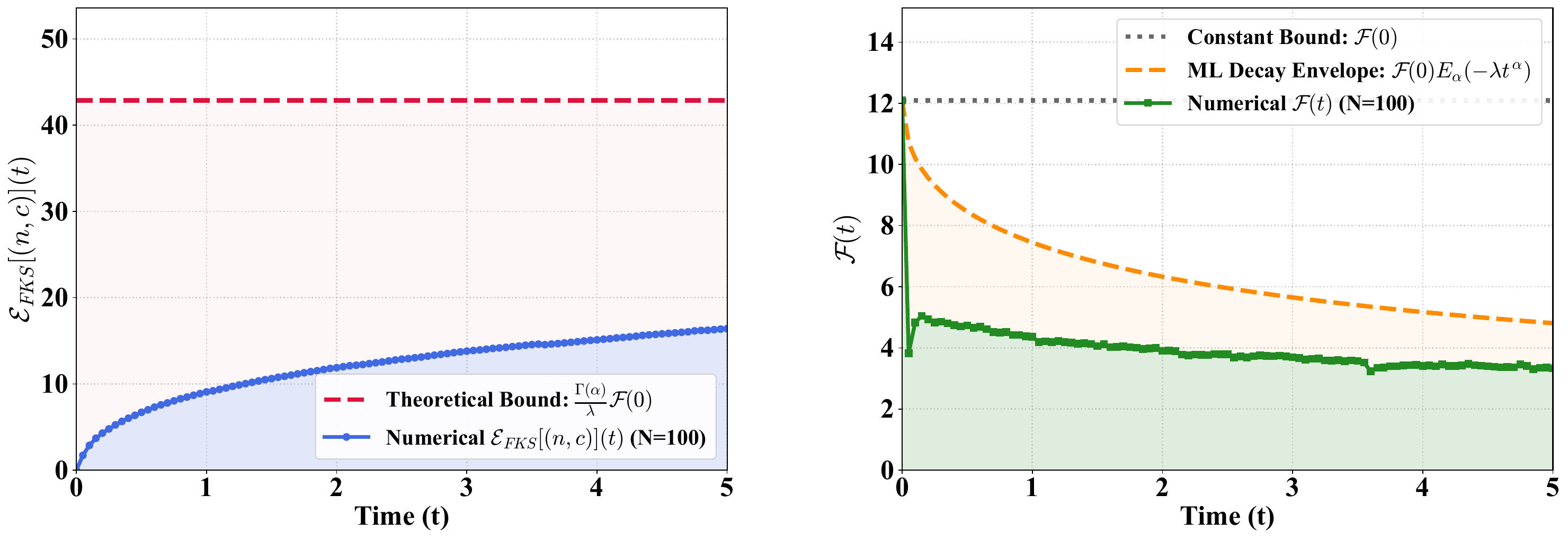}}
\caption{{\rm(}Color online{\rm)} The theoretical and numerical results of the energy functions with $\alpha=0.75$. From left to right, the figures illustrate the evolution of the novel Lyapunov function, defined by \eqref{F:Lyapunov}, and the jointly convex function $\mathcal{F}(t)$ (defined in \eqref{Fun:JCF}) over time $t$.}
\label{Fig:LF75}
\end{figure}

\section{Concluding Remarks}
\label{sec:conclusions}
In this paper, we establish a comprehensive framework to accurately characterize the global dynamic behavior of chemotactic gliding-diffusion and aggregation in myxobacteria. To this end, we synergistically integrate stochastic modeling, rigorous mathematical analysis, and deep learning-based simulations. We first established a robust physical foundation by constructing a lattice-based discrete agent model at the microscopic scale to capture actual kinetic behavior on rough soil surfaces, with our assumptions validated through comparative experimental data. Bridging biological insights with stochastic modeling, we rigorously derived the macroscopic model \eqref{Pro:FKS} from a subordinated Langevin equation, ensuring both physical and mathematical consistency. On the theoretical side, we establish a definitive solution theory by incorporating an analytical framework based on a novel Lyapunov function, innovative fractional convexity inequality, and generalized Sobolev spaces. This analytical framework successfully overcomes the inherent difficulties of time-nonlocal operators and establishes global well-posedness, mass conservation, and novel regularity results. To address numerical challenges, we design a multi-objective, positivity/non-negativity-preserving time-marching PINNs algorithm with independent network architectures and temporal semi-discretization. The method demonstrates exceptional robustness and high generalizability across a broad class of KS-type systems. Numerical benchmarks on complex geometries, most notably the `butterfly-shaped' domain, not only verify the accuracy of our scheme but also provide empirical validation of the global existence and long-term stability of the solutions.

In conclusion, this research offers insights into these complex processes by synergistically combining multiscale modeling, rigorous mathematical analysis, and computational methods, and thus provides a promising framework for future investigations of biological dynamics on rough surfaces and irregular geometries.

As an initial exploratory effort, this study opens several promising avenues for future research:
\begin{itemize}
    \item \textbf{Experimentally}, engaging in collaborating with microbiologists to integrate empirical data will further strengthen the model's reliability and predictive capacity.
    \item \textbf{Mathematically}, several directions are particularly worth pursuing. A natural next step is to investigate global existence and possible blow-up in higher-dimensional domains ($d\ge4$), and to establish sharper conditions on the chemotactic sensitivity coefficient $\chi$ that guarantee global boundedness. Other promising extensions include coupling the system with fluid dynamics to model more realistic environments (e.g., Refs.~\refcite{deAnna2020,Ma25,Tao13}), incorporating source terms to describe cell proliferation and death (e.g., Ref.~\refcite{Hillen13}), or considering the combined influence of both (e.g., Refs.~\refcite{Dai22,Dai23}). It should be noted that when source terms are present, mass conservation is lost, which introduces substantial additional difficulties in the analysis.
    \item \textbf{Numerically}, while the proposed deep neural network approach is effective, the non-local operators currently incur high computational costs. Future efforts will focus on optimizing network architecture to reduce  overhead without sacrificing physical fidelity.
    \item \textbf{Broad applications}, extending the current two-dimensional framework to three-dimensional complex geometries or multi-species interacting networks stands as a natural and impactful next step.
\end{itemize}

\appendix
\section{Auxiliary Results and Lemmas}
In this appendix, we collect several auxiliary lemmas and supporting results, together with their proofs, that are used repeatedly in the main text.

\subsection{Continuum Limit of the Transition Probabilities}\label{Der:drift_coefficient}
In the absence of environmental heterogeneity, the particle performs an unbiased nearest-neighbor random walk with transition probabilities $p_{r}=p_{l}=1/2$. This corresponds to isotropic local exploration without any directional preference. Biologically, when the surrounding slime field or chemoattractant distribution is spatially homogeneous, the cells do not possess sufficient directional information and therefore move randomly.

Chemotactic sensing is modeled as a weak perturbation of aforementioned symmetric motion rather than a deterministic steering mechanism. In other words, the environmental signal does not generate additional transition probability, but instead redistributes the directional preference between rightward and leftward motion. Following the standard framework of weakly biased random walks (see, e.g.,~\refcite{Hillen09,Stevens00}), we write the transition probabilities as symmetric perturbations around the unbiased state
\begin{equation}\label{eq:transP}
p_{r}=\cfrac{1}{2}+\varepsilon B(x,t),\quad
p_{l}=\cfrac{1}{2}-\varepsilon B(x,t),\quad \varepsilon>0,
\end{equation}
where $B(x,t)$ describes the local directional bias induced by the surrounding signal field, while $\varepsilon$ measures the strength of this bias and thus serves as the chemotactic sensitivity coefficient. This structure automatically preserves probability normalization, $p_{r}+p_{l}=1$, guarantees that the unbiased state is recovered whenever the local environment is spatially symmetric, namely when $B(x,t)=0$.

In the context of Myxobacterial aggregation, cells compare the signal intensities sensed within their local neighborhood before selecting a preferred moving direction (see \eqref{eq:Probbl}). Motivated by this mechanism, we define the directional bias through the normalized local contrast
\begin{equation}\label{eq:bias}
    B(x,t)=\frac{v(x+\delta_x,t)-v(x-\delta_x,t)}{v(x-\delta_x,t)+v(x+\delta_x,t)},
\end{equation}
where $v(x,t)>0$ denotes the local slime concentration or chemoattractant field. Here, the numerator measures the directional asymmetry of the perceived signal, while the denominator represents the overall background intensity within the sensing range. Consequently, the cellular response depends on the relative environmental contrast rather than the absolute concentration level, which is consistent with the experimentally observed adaptive sensing behavior in many biological aggregation processes.

According to \eqref{eq:bias}, the model possesses the following natural biological interpretations:
\begin{itemize}
    \item If $v(x+\delta_x,t)>v(x-\delta_x,t)$, then $B(x,t)>0$, meaning that the right-hand side contains a stronger slime signal, and the Myxobacteria are therefore more likely to glide toward the right;
    \item If $v(x+\delta_x,t)<v(x-\delta_x,t)$, then $B(x,t)<0$, indicating a stronger signal on the left-hand side, and the cells preferentially move leftward;
    \item If $v(x+\delta_x,t)=v(x-\delta_x,t)$, then $B(x,t)=0$, and the system naturally returns to the unbiased random walk $p_{r}=p_{l}=1/2$.
\end{itemize}
By gathering \eqref{eq:transP} and \eqref{eq:bias} and setting $\varepsilon=\frac{\chi}{4}$, the transition probabilities can be rewritten as
\begin{equation}\label{eq:prpl-def}
\left\{
\begin{aligned}
 p_r(x,t)=\frac{1}{2}+\frac{\chi}{4}\frac{v(x+\delta_x,t)-v(x-\delta_x,t)}{v(x-\delta_x,t)+v(x+\delta_x,t)},\\
 p_l(x,t)=\frac{1}{2}-\frac{\chi}{4}\frac{v(x+\delta_x,t)-v(x-\delta_x,t)}{v(x-\delta_x,t)+v(x+\delta_x,t)}.
\end{aligned}\right.
\end{equation}
where $\chi>0$ is the chemotactic sensitivity coefficient characterizing the strength of directional response to environmental heterogeneity. By construction, $p_r(x,t)+p_l(x,t)=1$. Let the random displacement during a single transition be denoted by  $\xi\in\{+\delta_x, -\delta_x\}$. Then the first two Kramers--Moyal coefficients (see, e.g., Ref.~\refcite{Risken89}) are given by
\begin{displaymath}
\left\{
\begin{aligned}
 m_1(x,t) &:= \mathbb E[\xi\mid x,t] = \delta_x\,\big( p_r(x,t) - p_l(x,t)\big), \\
 m_2(x,t) &:= \mathbb E[\xi^2\mid x,t] = (\delta_x)^2\,\big( p_r(x,t) + p_l(x,t)\big)=(\delta_x)^2.
\end{aligned}\right.
\end{displaymath}
Hence, the asymmetric component $p_r-p_l$ generates the effective directional drift, while the symmetric component $p_r+p_l$ determines the diffusive spreading.

Assuming that $v$ is sufficiently smooth, Taylor expansion around $x$ gives
\begin{displaymath}
v(x\pm\delta_x,t)=v(x,t)\pm\delta_x\partial_x v(x,t)+\mathcal{O}(\delta_x^2).
\end{displaymath}
Substituting these expansions into \eqref{eq:prpl-def}, we obtian
\begin{displaymath}
\begin{aligned}
 p_r-p_l
 &= \frac{\chi}{2}\frac{v(x+\delta_x,t)-v(x-\delta_x,t)}{v(x+\delta_x,t)+v(x-\delta_x,t)}
 = \frac{\chi}{2}\frac{2\delta_x\,\partial_xv(x)+\mathcal{O}(\delta_x^3)}{2v(x)+\mathcal{O}(\delta_x^2)} \\
 &= \frac{\chi}{2}\delta_x\,\partial_x\ln v(x)+\mathcal{O}(\delta_x^3).
\end{aligned}
\end{displaymath}
Therefore,
\begin{displaymath}
 m_1(x,t)=\frac{\chi}{2}\delta_x^2\,\partial_x\ln v(x,t)+\mathcal{O}(\delta_x^4),\qquad
 m_2(x,t)=(\delta_x)^2.
\end{displaymath}
Under the anomalous diffusive scaling $\delta_{x}^2\sim \tau^{\alpha}$ with $\alpha\in(0,1)$, the effective drift and diffusion coefficients are formally identified as\cite{Risken89}
\begin{displaymath}
D^{(1)}(x,t)=\lim_{\tau\to0}\frac{m_1(x,t)}{\tau^{\alpha}},\qquad
D^{(2)}(x,t)=\lim_{\tau\to0}\frac{m_2(x,t)}{2\tau^{\alpha}}.
\end{displaymath}
Introducing the generalized diffusion coefficient $\mathcal{D}=\frac{\delta_x^2}{2\tau^\alpha}$, we obtain
\begin{displaymath}
D^{(1)}(x,t)=\chi\mathcal{D}\,\partial_x\ln v(x,t),
\qquad
D^{(2)}(x,t)=\mathcal{D}.
\end{displaymath}
According to the standard Kramers-Moyal formalism for continuous diffusion limits (see, e.g., Ch.~4 in Ref.~\refcite{Risken89}), truncation at second order yields the Fokker-Planck equation
\begin{equation}\label{eq:FPE-KM}
\frac{\partial P(x,t)}{\partial t}
=-\frac{\partial}{\partial x}\big(D^{(1)}(x,t)P(x,t)\big)+\frac{\partial^2}{\partial x^2}\big(D^{(2)}(x,t)P(x,t)\big),
\end{equation}
which is statistically equivalent to the It\^o Langevin equation
\begin{equation}\label{eq:Langevin-KM}
\dot{x}(t)=D^{(1)}(x,t)+\sqrt{2D^{(2)}(x,t)}\,\dot{\mathcal{W}}(t).
\end{equation}
Substituting the expressions for $D^{(1)}$ and $D^{(2)}$ into \eqref{eq:Langevin-KM}, we arrive at
\begin{displaymath}
\dot{x}(t)=\chi\mathcal{D}\,\partial_x\ln v(x,t)+\sqrt{2\mathcal{D}}\,\dot{\mathcal{W}}(t).
\end{displaymath}
Therefore, the continuum limit of the weakly biased random walk naturally generates a macroscopic drift directed along the logarithmic gradient of the slime field, which quantitatively describes the tendency of Myxobacteria to aggregate toward regions with stronger environmental signals.

\subsection{Key Supporting Lemmas}
\begin{lemma}\label{lem:convexl1}
For $\alpha\in(0,1)$, and $f(t)\in AC(\bar{D})$, there holds
\begin{equation}\label{eq:m1mL1}
{_{0}^{C}\mathfrak{D}^{\alpha}_t}|f(t)|
\leq \mathrm{sgn}\big(f(t)\big)\cdot{_{0}^{C}\mathfrak{D}^{\alpha}_t}f(t),
\end{equation}
where $\mathrm{sgn}(\cdot)$ denotes the standard sign function, and $AC(\cdot)$ the class of absolutely continuous functions.
\end{lemma}

\begin{proof}
For the sign function $\mathrm{sgn}(\cdot)$ the sub-gradient property states that $|b|-|a|\geq\text{sgn}(a)(b-a)$ for all $a,b\in\mathbb{R}$. Taking $a=f(t)$ and $b=f(0)$ yields
\begin{equation}\label{sub-gradient1}
\frac{|f(t)|-|f(0)|}{\Gamma(1-\alpha)t^{\alpha}}\leq \frac{1}{\Gamma(1-\alpha)}\,\mathrm{sgn}\big(f(t)\big)\,\frac{f(t)-f(0)}{t^{\alpha}},\quad \forall~t>0.
\end{equation}
and similarly,
\begin{equation}\label{sub-gradient2}
\frac{|f(t)|-|f(s)|}{\Gamma(1-\alpha)(t-s)^{\alpha+1}}\leq \frac{1}{\Gamma(1-\alpha)}\,\mathrm{sgn}\big(f(t)\big)\,\frac{f(t)-f(s)}{(t-s)^{\alpha+1}},\quad \forall~s\in[0,t]\subset D.
\end{equation}

For $\alpha\in(0,1)$, Lemma 2.10 in Ref. \refcite{Jin21book} gives the following representation of the Caputo fractional derivative for $|f(t)|$,
\begin{equation}\label{Caputo}
{_{0}^{C}\mathfrak{D}^{\alpha}_t}\big|f(t)\big|
= \frac{1}{\Gamma(1-\alpha)}\left(\frac{|f(t)|-|f(0)|}{t^\alpha}+\alpha\int_0^t \frac{|f(t)|-|f(s)|}{(t-s)^{\alpha+1}}\,\mathrm{d}s \right).
\end{equation}
Substituting inequalities \eqref{sub-gradient1} and \eqref{sub-gradient2} into \eqref{Caputo} and carrying out straightforward sign manipulation leads to
\begin{displaymath}
{_{0}^{C}\mathfrak{D}^{\alpha}_t}\big|f(t)\big|
\leq \text{sgn}\big(f(t)\big)\cdot\frac{1}{\Gamma(1-\alpha)}\left(\frac{f(t)-f(0)}{t^\alpha}
+\alpha\int_0^t\frac{f(t)-f(s)}{(t-s)^{\alpha+1}}\,\mathrm{d}s\right).
\end{displaymath}
The expression inside the brackets coincides with the definition of ${_{0}^{C}\mathfrak{D}^{\alpha}_t}f(t)$. Therefore, the pointwise inequality \eqref{eq:m1mL1} is established.
\end{proof}

\begin{lemma}\label{Der:Lemma}
Let $\Omega\subset\mathbb{R}^d$ ($d\geq2$) be a bounded domain with smooth boundary, and let $\mathcal{A}=-\Delta$ denote the Neumann Laplacian on $L^p(\Omega)$ with $1<p<\infty$. Define $\mathcal{P}_{\alpha}(t)=t^{\alpha-1}E_{\alpha,\alpha}(-t^{\alpha}\mathcal{A})$. Then, for all $t>0$, there exist a constant $C>0$ such that
\begin{equation}\label{AE:estemat0}
\Bigl\|\mathcal{A}_\gamma^{\sigma}\mathcal{P}_\alpha^\gamma(t)\Bigr\|_{L^p\to L^p}
\le C\,t^{\alpha(1-\sigma)-m-1}, \qquad \sigma\in\{0,1/2,1\}.
\end{equation}
Furthermore, for $\gamma>0$ define $\mathcal{A}_\gamma:=\mathcal{A}+\gamma I$ and $\mathcal{P}^{\gamma}_{\alpha}(t)=t^{\alpha-1}E_{\alpha,\alpha}(-t^{\alpha}\mathcal{A}_\gamma)$. Then, for $t>0$, there exist additional constant $C>0$ such that
\begin{equation}\label{eq:LLv}
\Bigl\|\partial_t^{m}\mathcal{A}_\gamma^{\sigma}\mathcal{P}_\alpha^\gamma(t)\Bigr\|_{L^p\to L^p}
\le C\,t^{\alpha(1-\sigma)-m-1}(1+t^\alpha)^{-1},
\end{equation}
where $m=0,1$ and $\sigma\in\{0,1/2,1\}$.
\end{lemma}

\begin{proof}
We first prove the estimates for $\mathcal{P}_\alpha(t)$. By the Hankel contour representation of the Mittag-Leffler function (see, e.g., Ref.~\refcite{Podlubny99}), $E_{\alpha,\alpha}(\tau)
=\frac{1}{2\pi i}\int_{\Gamma}\frac{e^\zeta}{\zeta^\alpha-\tau}\,\mathrm{d}\zeta$, where $\Gamma$ is a sectorial contour contained in $\Sigma_{\theta}:=\{z\in\mathbb{C}:|\arg(z)|<\theta,z\neq0\}$ with $\theta\in(\pi/2,\pi)$. Since the Neumann Laplacian $\mathcal A=-\Delta$ on $L^p(\Omega)$ is sectorial, this representation extends to $\mathcal{A}$. By substituting $\tau=-t^\alpha\mathcal{A}$, performing the variable transformation $\zeta=z t$, and combining these with the result $\mathcal{L}\left\{t^{\alpha-1} E_{\alpha,\alpha}(-\lambda t^\alpha)\right\}=\frac{1}{z^{\alpha}+\lambda}$, we obtain
\begin{displaymath}
\mathcal P_\alpha(t)
=\frac{1}{2\pi i}\int_\Gamma
e^{z t}\big(z^\alpha+\mathcal{A}\big)^{-1}\,\mathrm{d}z.
\end{displaymath}

Since $\mathcal{A}$ is sectorial, there exists $\theta\in(\pi/2,\pi)$ such that $\|(z+\mathcal{A})^{-1}\|_{L^p\to L^p}\le \frac{C}{|z|}$ for all $z\notin\Sigma_{\theta}$. We choose $\Gamma$ to be the sectorial contour $\Gamma=\{z\in \mathbb{C}:|\arg(z)|=\theta, r=|z|\geq 1/t\}\cup\{z\in\mathbb{C}:r=|z|=1/t,|\arg (z)|\leq\theta\}$ oriented counter-clockwise. For $z\in\Gamma\subset\Sigma_\theta$, the term $z^\alpha$ also lies in a sector avoiding the negative real axis, and thus the resolvent bound yields $\|(z^{\alpha}+\mathcal{A})^{-1}\|_{L^p\to L^p}\le C|z|^{-\alpha}$. Moreover, $\Re(z)\le-c|z|$ on the branches of $\Gamma$, leading to  $|e^{z t}|\leq e^{-ct|z|}$. Following standard estimates by the Laplace transform method (see, e.g., Refs.~\refcite{Jin21book,Ma26,Ma23,Ma23b}), for $\sigma=0$ and $\sigma=1$ the estimate in \eqref{AE:estemat0} can be directly proved analogously to the proof of Theorem 6.4 in Ref.~\refcite{Jin21book}. Furthermore, for $1<p<\infty$, by the interpolation inequality (moment inequality) of sectorial operators (see e.g., Proposition~6.6.4 in Ref.~\refcite{Haase06}), we have $\|\mathcal{A}^{1/2}u\|_{L^p}\le C\|u\|_{L^p}^{1/2}\|\mathcal{A}u\|_{L^p}^{1/2}$. Applying this, we derive
\begin{align*}
\big\|\mathcal P_\alpha(t)\big\|_{L^p\to W^{1,p}}
\le C \left(t^{\alpha-1}\right)^{1/2}\left(t^{-1}\right)^{1/2}
\le C t^{\alpha/2-1}.
\end{align*}
This completes the estimate of \eqref{AE:estemat0}.

We now proceed to prove \eqref{eq:LLv}. For $m\in\{0,1\}$ and $\sigma\in\{0,\frac12,1\}$, the Hankel contour representation yields
\begin{displaymath}
\partial_t^{m}\mathcal{A}_\gamma^{\sigma}\mathcal{P}_\alpha^\gamma(t)
=\frac{1}{2\pi i}\int_\Gamma e^{zt}z^m\mathcal{A}_\gamma^\sigma \bigl(z^\alpha+\mathcal{A}_\gamma\bigr)^{-1} \,\mathrm{d}z,
\end{displaymath}
where $\Gamma$ is a standard sectorial contour, as denoted previously. We first consider the cases $\sigma=0$ and $\sigma=1$. Since $\mathcal{A}_\gamma$ is sectorial and $\sigma(\mathcal{A}_\gamma)\subset[\gamma,\infty)$, the resolvent estimate $\|(z^\alpha+\mathcal{A}_\gamma)^{-1}\|_{L^p\to L^p}\le C(|z|^\alpha+\gamma)^{-1}$ holds uniformly for all $z\in\Gamma$. Moreover, the contour $\Gamma$ can be chosen such that $\Re(z)\le -c|z|$ for $z\in\Gamma$, and hence $|e^{zt}|\le e^{-ct|z|}$. For $\sigma=0$, the resolvent estimate directly gives $\left\|\mathcal{A}_\gamma^0(z^\alpha+\mathcal{A}_\gamma)^{-1}\right\|_{L^p\to L^p}\le C(|z|^\alpha+\gamma)^{-1}$. For $\sigma=1$, using the identity $\mathcal{A}_\gamma(z^\alpha+\mathcal{A}_\gamma)^{-1}=I-z^\alpha(z^\alpha+\mathcal{A}_\gamma)^{-1}$, and observing that $\int_\Gamma e^{zt}z^m\mathrm{d}z=0$ by Cauchy's theorem, we obtain
\begin{displaymath}
\partial_t^m\mathcal{A}_\gamma\mathcal{P}_\alpha^\gamma(t)
= -\frac1{2\pi i}\int_\Gamma e^{zt}z^{m+\alpha}(z^\alpha+\mathcal{A}_\gamma)^{-1}\,\mathrm{d}z.
\end{displaymath}
Therefore, both cases $\sigma=0$ and $\sigma=1$, the integrand admits the unified bound $|e^{zt}||z|^{m+\alpha\sigma}\bigl(|z|^\alpha+\gamma\bigr)^{-1}$. Consequently,
\begin{displaymath}
\bigl\|\partial_t^{m}\mathcal{A}_\gamma^\sigma \mathcal{P}_\alpha^\gamma(t)\bigr\|_{L^p\to L^p}
\le C \int_{1/t}^\infty e^{tr\cos(\theta)}\frac{r^{m+\alpha\sigma}}{r^\alpha+\gamma}\,\mathrm{d}r+\int_{-\theta}^{\theta}e^{t\Re(z)}\frac{|z|^{m+\alpha\sigma}}{|z|^{\alpha}+\gamma}|\mathrm{d}z|:=I_1+I_2,
\end{displaymath}
where $I_1$ corresponds to the integration over the branches and $I_2$ corresponds to the circular arc.

For the circular arc, we have $|z|=1/t$ and $|\mathrm{d}z|=\frac{1}{t}\mathrm{d}\varphi$. Furthermore, $\Re(z)\le|z|\cos\theta$, so $|e^{zt}|\le e^{\cos\theta}\le C$. Thus, we have
\begin{displaymath}
I_2 = \int_{-\theta}^{\theta} \big|e^{zt}\big| \frac{|z|^{m+\alpha\sigma}}{|z|^\alpha+\gamma} |\mathrm{d}z|
\le C \int_{-\theta}^{\theta} \frac{t^{-(m+\alpha\sigma)}}{t^{-\alpha}+\gamma} \frac{1}{t} \,\mathrm{d}\varphi
\le C t^{-m-\alpha\sigma-1} \frac{1}{t^{-\alpha}+\gamma}.
\end{displaymath}
Since $\frac{1}{t^{-\alpha}+\gamma}=\frac{t^\alpha}{1+\gamma t^\alpha}\le C t^\alpha(1+t^\alpha)^{-1}$, we obtain
\begin{displaymath}
I_2 \le C t^{\alpha(1-\sigma)-m-1} (1+t^\alpha)^{-1}.
\end{displaymath}

For the branches, we have $\Re(z)\le -c|z|$ which implies $|e^{zt}|\le e^{-ctr}$. Thus, making the substitution $s=tr$, we obtain
\begin{displaymath}
I_1 \le C t^{\alpha(1-\sigma)-m-1} \int_1^\infty e^{-cs} \frac{s^{m+\alpha\sigma}} {s^\alpha+\gamma t^\alpha}\,\mathrm{d}s.
\end{displaymath}
To estimate the integral, we distinguish two cases.
\begin{itemize}
  \item For $0<t\le1$: Since $s^\alpha+\gamma t^\alpha\ge s^\alpha$, it follows that
  \begin{displaymath}
   \int_1^\infty e^{-cs}\frac{s^{m+\alpha\sigma}}{s^\alpha+\gamma t^\alpha}\,\mathrm{d}s
   \le \int_1^\infty e^{-cs}s^{m+\alpha(\sigma-1)}\,\mathrm{d}s\leq C.
  \end{displaymath}
  The last inequality holds for $m+\alpha(\sigma-1)\ge-\alpha>-1$, ensuring the above Gamma-type integral converges and yielding$\int_0^\infty e^{-cs}\frac{s^{m+\alpha\sigma}}{s^\alpha+\gamma t^\alpha}\,\mathrm{d}s\le C$.
  \item For $t>1$: We have $s^\alpha+\gamma t^\alpha\ge\gamma t^\alpha$, which gives
  \begin{displaymath}
   \int_1^\infty e^{-cs}\frac{s^{m+\alpha\sigma}}{s^\alpha+\gamma t^\alpha}\,\mathrm{d}s
   \le \frac1{\gamma t^\alpha}\int_1^\infty e^{-cs} s^{m+\alpha\sigma}\,\mathrm{d}s\le Ct^{-\alpha}.
  \end{displaymath}
  \end{itemize}
Combining the above estimates, we conclude that $\int_1^\infty e^{-cs}\frac{s^{m+\alpha\sigma}}{s^\alpha+\gamma t^\alpha}\mathrm{d}s\le C(1+t^\alpha)^{-1}$. Hence, for $\sigma\in\{0,1\}$, the following holds
\begin{displaymath}
\bigl\|\partial_t^{m}\mathcal{A}_\gamma^\sigma\mathcal{P}_\alpha^\gamma(t)\bigr\|_{L^p\to L^p}
\le C t^{\alpha(1-\sigma)-m-1}(1+t^\alpha)^{-1}.
\end{displaymath}

It remains to consider the case $\sigma=\frac{1}{2}$. By the moment inequality for sectorial operators, $\|\mathcal{A}_\gamma^{1/2}u\|_{L^p}\le C\|u\|_{L^p}^{1/2}\|\mathcal{A}_\gamma u\|_{L^p}^{1/2}$, we obtain
\begin{align*}
\bigl\|\partial_t^{m}\mathcal{A}_\gamma^{1/2}\mathcal{P}_\alpha^\gamma(t)\bigr\|_{L^p\to L^p}
&= \bigl\|\mathcal{A}_\gamma^{1/2}\big(\partial_t^m\mathcal{P}_\alpha^\gamma(t)\big)\bigr\|_{L^p\to L^p} \\
&\le C \bigl\|\partial_t^m\mathcal{P}_\alpha^\gamma(t)\bigr\|_{L^p\to L^p}^{1/2}\bigl\|\mathcal{A}_\gamma \partial_t^m \mathcal{P}_\alpha^\gamma(t)\bigr\|_{L^p\to L^p}^{1/2}.
\end{align*}
Substituting the estimates already established for $\sigma=0$ and $\sigma=1$, we arrive at
\begin{align*}
\bigl\|\partial_t^{m}\mathcal{A}_\gamma^{1/2}\mathcal{P}_\alpha^\gamma(t)\bigr\|_{L^p\to L^p}
&\le C \Bigl(t^{\alpha-m-1}(1+t^\alpha)^{-1}\Bigr)^{1/2}\Bigl(t^{-m-1}(1+t^\alpha)^{-1}\Bigr)^{1/2} \\
&= C t^{\alpha/2-m-1}(1+t^\alpha)^{-1}.
\end{align*}
Consequently, for all $m\in\{0,1\}$ and $\sigma\in\{0,\frac12,1\}$,
\begin{displaymath}
\bigl\| \partial_t^{m}\mathcal{A}_\gamma^\sigma\mathcal{P}_\alpha^\gamma(t)\bigr\|_{L^p\to L^p}
\le C t^{\alpha(1-\sigma)-m-1}(1+t^\alpha)^{-1}.
\end{displaymath}
This completes the proof.
\end{proof}

\begin{lemma}[A scalar comparison principle]\label{lem:Caputo-logistic}
Let $\alpha\in(0,1)$, $a,b>0$, and $\theta>0$. Suppose that $Y\ge0$ is
sufficiently regular so that
\begin{displaymath}
{_0^C D_t^\alpha}Y(t)
=\frac{Y(t)-Y(0)}{\Gamma(1-\alpha)t^\alpha}+
\frac{\alpha}{\Gamma(1-\alpha)}\int_0^t\frac{Y(t)-Y(s)}{(t-s)^{\alpha+1}}\,\mathrm{d}s,\qquad t>0.
\end{displaymath}
Assume moreover that
${_0^C\mathcal{D}_t^\alpha}Y(t)\le aY(t)-bY(t)^{1+\theta}$ for $t>0$.
Then
\begin{displaymath}
\sup_{t\ge0}Y(t)
\le\max\left\{Y(0),\left(\frac ab\right)^{1/\theta}\right\}.
\end{displaymath}
\end{lemma}

\begin{proof}
Set $K:=\max\left\{Y(0),\left(\frac ab\right)^{1/\theta}\right\}$. We prove that $Y(t)\le K$ for all $t\ge0$. Suppose otherwise. Then there exist
$T>0$ and $\tilde{t}\in[0,T]$ such that
\begin{displaymath}
    Y(\,\tilde{t}\,)=\max_{0\le t\le T}Y(t)>K.
\end{displaymath}
Since $Y(0)\le K$, we have $\tilde{t}>0$. Moreover, $Y(\,\tilde{t}\,)\ge Y(s)$, $0\le s\le \tilde{t}$. Therefore, by the above representation formula,
\begin{displaymath}
{_0^C D_t^\alpha}Y(\,\tilde{t}\,)
=\frac{Y(\,\tilde{t}\,)-Y(0)}{\Gamma(1-\alpha)\tilde{t}^\alpha}+
\frac{\alpha}{\Gamma(1-\alpha)}
\int_0^{\tilde{t}}\frac{Y(\,\tilde{t}\,)-Y(s)}{(\tilde{t}-s)^{\alpha+1}}\,\mathrm{d}s\ge0.
\end{displaymath}
On the other hand,
\begin{displaymath}
Y(\,\tilde{t}\,)>\left(\frac{a}{b}\right)^{1/\theta}
\end{displaymath}
implies
\begin{displaymath}
aY(\,\tilde{t}\,)-bY(\,\tilde{t}\,)^{1+\theta}
=Y(\,\tilde{t}\,)\left(a-bY(\,\tilde{t}\,)^\theta\right)<0,
\end{displaymath}
which contradicts the assumed differential inequality. Hence $Y(t)\le K$ for all $t\ge0$.
\end{proof}

\section*{Declarations}
\subsection*{Conflict of interest}
The authors declared that they have no conflict of interest.

\section*{Acknowledgments}
The first author would like to express sincere gratitude to Prof. Tiejun Li for his insightful guidance on AI4SC. Fugui Ma is supported by the Peking University Boya Postdoctoral Fellowship. Lei Wu is supported by NSF under grant DMS-2405161.

\bibliography{References}

@article{Adler1966,
  title = {Chemotaxis in bacteria: {M}otile {E}scherichia coli
            migrate in bands that are influenced by oxygen and organic nutrients.},
  volume = {153},
  ISSN = {1095-9203},
  url = {http://dx.doi.org/10.1126/science.153.3737.708},
  DOI = {10.1126/science.153.3737.708},
  number = {3737},
  journal = {Science},
  publisher = {American Association for the Advancement of Science (AAAS)},
  author = {Adler, Julius},
  year = {1966},
  month = aug,
  pages = {708--716}
}

@book {Adams03,
    AUTHOR = {Adams, Robert A. and Fournier, John J. F.},
     TITLE = {Sobolev Spaces},
    SERIES = {},
    VOLUME = {140},
   EDITION = {Second},
 PUBLISHER = {Elsevier/Academic Press, Amsterdam},
      YEAR = {2003},
     PAGES = {},
      ISBN = {},
   MRCLASS = {},
  MRNUMBER = {},
}

@article {Adam-2014,
    AUTHOR = {Adam, Kingma DP Ba J},
     TITLE = {A method for stochastic optimization},
   JOURNAL = {arXiv:1412.6980},
  FJOURNAL = {arXiv:1412.6980},
    VOLUME = {},
      YEAR = {2014},
     PAGES = {},
      ISSN = {},
   MRCLASS = {},
  MRNUMBER = {},
MRREVIEWER = {},
       DOI = {},
       URL = {},
}

@article {Aida05,
    AUTHOR = {Aida, Masashi and Osaki, Koichi and Tsujikawa, Tohru and Yagi,
              Atsushi and Mimura, Masayasu},
     TITLE = {Chemotaxis and growth system with singular sensitivity function},
   JOURNAL = {Nonlinear Anal.-Real World Appl.},
  FJOURNAL = {Nonlinear Analysis. Real World Applications. An International
              Multidisciplinary Journal},
    VOLUME = {6},
      YEAR = {2005},
    NUMBER = {2},
     PAGES = {323--336},
      ISSN = {1468-1218,1878-5719},
   MRCLASS = {92C17 (35K50 35K55)},
  MRNUMBER = {2111657},
MRREVIEWER = {Evelyn\ Sander},
       DOI = {10.1016/j.nonrwa.2004.08.011},
       URL = {https://doi.org/10.1016/j.nonrwa.2004.08.011},
}

@article{Atilim18,
  author  = {Atilim Gunes Baydin and Barak A. Pearlmutter and Alexey Andreyevich Radul and Jeffrey Mark Siskind},
  title   = {Automatic differentiation in machine learning: A Survey},
  journal = {J. Mach. Learn. Res.},
  year    = {2018},
  volume  = {18},
  number  = {153},
  pages   = {1--43},
}

@article {Angstmann21,
    AUTHOR = {Angstmann, Christopher N. and Erickson, Austen M. and Henry,
              Bruce I. and McGann, Anna V. and Murray, John M. and Nichols,
              James A.},
     TITLE = {A general framework for fractional order compartment models},
   JOURNAL = {SIAM Rev.},
  FJOURNAL = {SIAM Review},
    VOLUME = {63},
      YEAR = {2021},
    NUMBER = {2},
     PAGES = {375--392},
      ISSN = {0036-1445},
   MRCLASS = {60K40 (34A08 60G22 92C45 92D30)},
  MRNUMBER = {4253793},
       DOI = {10.1137/21M1398549},
}

@book {Applebaum09,
    AUTHOR = {Applebaum, David},
     TITLE = {L\'{e}vy Processes And Stochastic Calculus},
    SERIES = {},
    VOLUME = {116},
   EDITION = {Second},
 PUBLISHER = {Cambridge University Press, Cambridge},
      YEAR = {2009},
     PAGES = {xxx+460},
      ISBN = {978-0-521-73865-1},
   MRCLASS = {60-02 (60G44 60G51 60H05 60H30 60J35 60J60 60J65)},
  MRNUMBER = {2512800},
MRREVIEWER = {Dora Sele\v{s}i},
       DOI = {10.1017/CBO9780511809781},
}

@article {Alikhanov10,
    AUTHOR = {Alikhanov, A. A.},
     TITLE = {A priori estimates for solutions of boundary value problems
              for equations of fractional order},
   JOURNAL = {Differ. Uravn.},
  FJOURNAL = {Differentsial\cprime nye Uravneniya},
    VOLUME = {46},
      YEAR = {2010},
    NUMBER = {5},
     PAGES = {658--664},
      ISSN = {0374-0641},
   MRCLASS = {35R11 (26A33 35B45 35K20 35K57)},
  MRNUMBER = {2797545},
MRREVIEWER = {V.\ S.\ Rabinovich},
       DOI = {10.1134/S0012266110050058},
}

@article {Al-Refai22,
    AUTHOR = {Al-Refai, Mohammed and Luchko, Yuri},
     TITLE = {Comparison principles for solutions to the fractional
              differential inequalities with the general fractional
              derivatives and their applications},
   JOURNAL = {J. Differential Equations},
  FJOURNAL = {Journal of Differential Equations},
    VOLUME = {319},
      YEAR = {2022},
     PAGES = {312--324},
      ISSN = {0022-0396},
   MRCLASS = {34A08 (26A33 33E12 34A40 35S10 45K05)},
  MRNUMBER = {4389754},
MRREVIEWER = {Chunyan Luo},
       DOI = {10.1016/j.jde.2022.02.054},
}

@article {Arumugam21,
    AUTHOR = {Arumugam, Gurusamy and Tyagi, Jagmohan},
     TITLE = {Keller-{S}egel chemotaxis models: A review},
   JOURNAL = {Acta Appl. Math.},
  FJOURNAL = {Acta Applicandae Mathematicae},
    VOLUME = {171},
      YEAR = {2021},
     PAGES = {6, 82},
      ISSN = {0167-8019,1572-9036},
   MRCLASS = {92C17 (35A01 35B44 35D30 35K40 65M06 65M08)},
  MRNUMBER = {4188348},
       DOI = {10.1007/s10440-020-00374-2},
       URL = {https://doi.org/10.1007/s10440-020-00374-2},
}

@book {Amann19,
    AUTHOR = {Amann, Herbert},
     TITLE = {Linear and quasilinear parabolic problems. {V}ol. {II}},
    SERIES = {Monographs in Mathematics},
    VOLUME = {106},
      NOTE = {Function spaces},
 PUBLISHER = {Birkh\"{a}user/Springer, Cham},
      YEAR = {2019},
     PAGES = {xiv+464},
      ISBN = {978-3-030-11762-7; 978-3-030-11763-4},
   MRCLASS = {46-02 (35Kxx 42B35 46E35 46E40 46F05 46M35 46N20)},
  MRNUMBER = {3930629},
MRREVIEWER = {Jos\'{e}\ Luis\ Ansorena},
       DOI = {10.1007/978-3-030-11763-4},
       URL = {https://doi.org/10.1007/978-3-030-11763-4},
}

@article {Auscher01,
    AUTHOR = {Auscher, P. and Tchamitchian, Ph.},
     TITLE = {Square roots of elliptic second order divergence operators on
              strongly {L}ipschitz domains: {$L^p$} theory},
   JOURNAL = {Math. Ann.},
  FJOURNAL = {Mathematische Annalen},
    VOLUME = {320},
      YEAR = {2001},
    NUMBER = {3},
     PAGES = {577--623},
      ISSN = {0025-5831,1432-1807},
   MRCLASS = {47F05 (35B45 42B25)},
  MRNUMBER = {1846778},
MRREVIEWER = {Xuan\ Thinh\ Duong},
       DOI = {10.1007/PL00004487},
       URL = {https://doi.org/10.1007/PL00004487},
}

@article {Bellomo22,
    AUTHOR = {Bellomo, N. and Outada, N. and Soler, J. and Tao, Y. and
              Winkler, M.},
     TITLE = {Chemotaxis and cross-diffusion models in complex environments:
              Models and analytic problems toward a multiscale vision},
   JOURNAL = {Math. Models Methods Appl. Sci.},
  FJOURNAL = {Mathematical Models and Methods in Applied Sciences},
    VOLUME = {32},
      YEAR = {2022},
    NUMBER = {4},
     PAGES = {713--792},
      ISSN = {0218-2025,1793-6314},
   MRCLASS = {35A01 (35B40 35B44 35K55 35K57 35Q35 35Q92 92C17)},
  MRNUMBER = {4421216},
       DOI = {10.1142/S0218202522500166},
       URL = {https://doi.org/10.1142/S0218202522500166},
}

@article {Bezerra24,
    AUTHOR = {Bezerra, Mario and Cuevas, Claudio and Viana, Arl\'{u}cio},
     TITLE = {Local and global solutions for a subdiffusive
              parabolic-parabolic {K}eller-{S}egel system},
   JOURNAL = {Z. Angew. Math. Phys.},
  FJOURNAL = {Zeitschrift f\"{u}r Angewandte Mathematik und Physik. ZAMP.
              Journal of Applied Mathematics and Physics. Journal de
              Math\'{e}matiques et de Physique Appliqu\'{e}es},
    VOLUME = {75},
      YEAR = {2024},
    NUMBER = {5},
     PAGES = {Paper No. 172, 30},
      ISSN = {0044-2275,1420-9039},
   MRCLASS = {35A01 (35B44 35K55 35Q92 35R11 92C17)},
  MRNUMBER = {4790970},
       DOI = {10.1007/s00033-024-02316-6},
       URL = {https://doi.org/10.1007/s00033-024-02316-6},
}

@book {Bishop24,
    AUTHOR = {Bishop, Christopher M. and Bishop, Hugh},
     TITLE = {Deep Learning: Foundations and Concepts},
 PUBLISHER = {Springer, Cham},
      YEAR = {2024},
     PAGES = {xx+649},
      ISBN = {978-3-031-45467-7; 978-3-031-45468-4},
   MRCLASS = {68-01 (68T07)},
  MRNUMBER = {4719738},
       DOI = {10.1007/978-3-031-45468-4},
}

@article{Bouvard2022,
  title = {Direct measurement of the aerotactic response in a bacterial suspension},
  volume = {106},
  ISSN = {2470-0053},
  url = {http://dx.doi.org/10.1103/PhysRevE.106.034404},
  DOI = {10.1103/physreve.106.034404},
  number = {3},
  journal = {Phys. Rev. E},
  publisher = {American Physical Society (APS)},
  author = {Bouvard,  J. and Douarche,  C. and Mergaert,  P. and Auradou,  H. and Moisy,  F.},
  year = {2022},
  PAGES = {034404},
}

@book {Brezis11,
    AUTHOR = {Brezis, Haim},
     TITLE = {Functional Analysis, {S}obolev Spaces And Partial Differential Equations},
    SERIES = {Universitext},
 PUBLISHER = {Springer, New York},
      YEAR = {2011},
     PAGES = {xiv+599},
      ISBN = {978-0-387-70913-0},
   MRCLASS = {35-01 (46-01 46E35 46N20 47F05)},
  MRNUMBER = {2759829},
}

@article {Bregman67,
    AUTHOR = {Br\`egman, L. M.},
     TITLE = {A relaxation method of finding a common point of convex sets
              and its application to the solution of problems in convex
              programming},
   JOURNAL = {U.S.S.R. Comput. Math. Math. Phys.},
  FJOURNAL = {\v{Z}urnal Vy\v{c}islitel\cprime no\u{\i} Matematiki i
              Matemati\v{c}esko\u{\i} Fiziki},
    VOLUME = {7},
      YEAR = {1967},
     PAGES = {200--217},
      ISSN = {0044-4669},
   MRCLASS = {90.60 (65.00)},
  MRNUMBER = {215617},
       DOI = {https://doi.org/10.1016/0041-5553(67)90040-7},
MRREVIEWER = {P.\ Johansen},
}

@article{BERG1972,
  title = {Chemotaxis in {E}scherichia coli analysed by three-dimensional tracking},
  volume = {239},
  ISSN = {1476-4687},
  url = {http://dx.doi.org/10.1038/239500a0},
  DOI = {10.1038/239500a0},
  number = {5374},
  journal = {Nature},
  publisher = {Springer Science and Business Media LLC},
  author = {Berg,  Howard C. and Brown, Douglas A.},
  year = {1972},
  month = oct,
  pages = {500--504}
}

@article{Budrene1995,
  title = {Dynamics of formation of symmetrical patterns by chemotactic bacteria},
  volume = {376},
  ISSN = {1476-4687},
  url = {http://dx.doi.org/10.1038/376049a0},
  DOI = {10.1038/376049a0},
  number = {6535},
  journal = {Nature},
  publisher = {Springer Science and Business Media LLC},
  author = {Budrene, Elena O. and Berg, Howard C.},
  year = {1995},
  month = jul,
  pages = {49--53}
}

@article {Burenkov02,
    AUTHOR = {Burenkov, V. I. and Davies, E. B.},
     TITLE = {Spectral stability of the {N}eumann {L}aplacian},
   JOURNAL = {J. Differential Equations},
  FJOURNAL = {Journal of Differential Equations},
    VOLUME = {186},
      YEAR = {2002},
    NUMBER = {2},
     PAGES = {485--508},
      ISSN = {0022-0396,1090-2732},
   MRCLASS = {35P15 (35J15 47B25 47F05)},
  MRNUMBER = {1942219},
MRREVIEWER = {B.\ Hellwig},
       DOI = {10.1016/S0022-0396(02)00033-5},
       URL = {https://doi.org/10.1016/S0022-0396(02)00033-5},
}

@article {Bezerra22,
    AUTHOR = {Bezerra, Mario and Cuevas, Claudio and Silva, Clessius and
              Soto, Herme},
     TITLE = {On the fractional doubly parabolic {K}eller-{S}egel system
              modelling chemotaxis},
   JOURNAL = {Sci. China Math.},
  FJOURNAL = {Science China. Mathematics},
    VOLUME = {65},
      YEAR = {2022},
    NUMBER = {9},
     PAGES = {1827--1874},
      ISSN = {1674-7283,1869-1862},
   MRCLASS = {35R11 (35A01 35B40 35K55 92C17)},
  MRNUMBER = {4470356},
       DOI = {10.1007/s11425-020-1846-x},
       URL = {https://doi.org/10.1007/s11425-020-1846-x},
}

@article{chung1996,
  title={Logarithmic Harnack inequalities},
  author={Chung, Fan RK and Yau, Shing-Tung},
  journal={Math. Res. Lett.},
  volume={3},
  number={6},
  pages={793--812},
  year={1996}
}

@article {Costa23,
    AUTHOR = {Costa, Masterson and Cuevas, Claudio and Silva, Clessius and
              Soto, Herme},
     TITLE = {Well-posedness and blow-up of the fractional {K}eller-{S}egel
              model on domains},
   JOURNAL = {Math. Nachr.},
  FJOURNAL = {Mathematische Nachrichten},
    VOLUME = {296},
      YEAR = {2023},
    NUMBER = {12},
     PAGES = {5569--5592},
      ISSN = {0025-584X},
   MRCLASS = {35K51 (35B44 35R11)},
  MRNUMBER = {4694570},
       DOI = {10.1002/mana.202200235},
       URL = {},
}

@article {Chertock08,
    AUTHOR = {Chertock, Alina and Kurganov, Alexander},
     TITLE = {A second-order positivity preserving central-upwind scheme for
              chemotaxis and haptotaxis models},
   JOURNAL = {Numer. Math.},
  FJOURNAL = {Numerische Mathematik},
    VOLUME = {111},
      YEAR = {2008},
    NUMBER = {2},
     PAGES = {169--205},
      ISSN = {0029-599X},
   MRCLASS = {92C17 (65M06 76M12)},
  MRNUMBER = {2456829},
       DOI = {10.1007/s00211-008-0188-0},
}

@article {Chen21,
    AUTHOR = {Chen, Lin and Kong, Fanze and Wang, Qi},
     TITLE = {Global and exponential attractor of the repulsive
              {K}eller-{S}egel model with logarithmic sensitivity},
   JOURNAL = {European J. Appl. Math.},
  FJOURNAL = {European Journal of Applied Mathematics},
    VOLUME = {32},
      YEAR = {2021},
    NUMBER = {4},
     PAGES = {599--617},
      ISSN = {0956-7925,1469-4425},
   MRCLASS = {35B35 (35K51 35K57 92C17)},
  MRNUMBER = {4283031},
       DOI = {10.1017/s0956792520000194},
       URL = {https://doi.org/10.1017/s0956792520000194},
}

@article {Chen20,
    AUTHOR = {Chen, Lin and Kong, Fanze and Wang, Qi},
     TITLE = {Stationary ring and concentric-ring solutions of the
              {K}eller-{S}egel model with quadratic diffusion},
   JOURNAL = {SIAM J. Math. Anal.},
  FJOURNAL = {SIAM Journal on Mathematical Analysis},
    VOLUME = {52},
      YEAR = {2020},
    NUMBER = {5},
     PAGES = {4565--4615},
      ISSN = {0036-1410,1095-7154},
   MRCLASS = {35K51 (35A01 35B40 35K57 35Q92 92C17)},
  MRNUMBER = {4154313},
       DOI = {10.1137/19M1298998},
       URL = {https://doi.org/10.1137/19M1298998},
}

@article{Cremer2019,
  title = {Chemotaxis as a navigation strategy to boost range expansion},
  volume = {575},
  ISSN = {1476-4687},
  url = {http://dx.doi.org/10.1038/s41586-019-1733-y},
  DOI = {10.1038/s41586-019-1733-y},
  number = {7784},
  journal = {Nature},
  publisher = {Springer Science and Business Media LLC},
  author = {Cremer,  Jonas and Honda,  Tomoya and Tang,  Ying and Wong-Ng,  Jerome and Vergassola,  Massimo and Hwa,  Terence},
  year = {2019},
  month = nov,
  pages = {658--663}
}

@article {Deuschel90,
    AUTHOR = {Deuschel, Jean-Dominique and Stroock, Daniel W.},
     TITLE = {Hypercontractivity and spectral gap of symmetric diffusions
              with applications to the stochastic {I}sing models},
   JOURNAL = {J. Funct. Anal.},
  FJOURNAL = {Journal of Functional Analysis},
    VOLUME = {92},
      YEAR = {1990},
    NUMBER = {1},
     PAGES = {30--48},
      ISSN = {0022-1236,1096-0783},
   MRCLASS = {58G32 (58G25 60J60 60K35 82B44)},
  MRNUMBER = {1064685},
MRREVIEWER = {G\'{e}rard\ Ben Arous},
       DOI = {10.1016/0022-1236(90)90066-T},
       URL = {https://doi.org/10.1016/0022-1236(90)90066-T},
}

@article{Dawid2000,
  title = {Biology And Global Distribution Of Myxobacteria In Soils},
  volume = {24},
  ISSN = {1574-6976},
  DOI = {10.1111/j.1574-6976.2000.tb00548.x},
  number = {4},
  journal = {FEMS Microbiol. Rev.},
  publisher = {Oxford University Press},
  author = {Dawid,  Wolfgang},
  year = {2000},
  month = oct,
  pages = {403--427}
}

@article{deAnna2020,
  title = {Chemotaxis under flow disorder shapes microbial dispersion in porous media},
  volume = {17},
  ISSN = {1745-2481},
  url = {},
  DOI = {10.1038/s41567-020-1002-x},
  number = {1},
  journal = {Nat. Phys.},
  publisher = {Springer Science and Business Media LLC},
  author = {de Anna,  Pietro and Pahlavan,  Amir A. and Yawata,  Yutaka and Stocker,  Roman and Juanes,  Ruben},
  year = {2020},
  month = sep,
  pages = {68--73}
}

@article {Dai22,
    AUTHOR = {Dai, Feng and Xiang, Tian},
     TITLE = {Boundedness and asymptotic stabilization in a two-dimensional
              {K}eller-{S}egel-{N}avier-{S}tokes system with sub-logistic source},
   JOURNAL = {Math. Models Methods Appl. Sci.},
  FJOURNAL = {Mathematical Models and Methods in Applied Sciences},
    VOLUME = {32},
      YEAR = {2022},
    NUMBER = {11},
     PAGES = {2237--2294},
      ISSN = {0218-2025},
   MRCLASS = {35A09 (35B40 35K55 35Q35 92C17)},
  MRNUMBER = {4526585},
       DOI = {10.1142/S0218202522500531},
}

@article {Dai23,
    AUTHOR = {Dai, Feng and Liu, Bin},
     TITLE = {How far do indirect signal production mechanisms influence
              regularity in the three-dimensional
              {K}eller-{S}egel-{N}avier-{S}tokes system?},
   JOURNAL = {Math. Models Methods Appl. Sci.},
  FJOURNAL = {Mathematical Models and Methods in Applied Sciences},
    VOLUME = {33},
      YEAR = {2023},
    NUMBER = {14},
     PAGES = {2823--2877},
      ISSN = {0218-2025,1793-6314},
   MRCLASS = {35B65 (35B40 35D30 35K55 35K57 35Q30 92C17)},
  MRNUMBER = {4683275},
       DOI = {10.1142/S0218202523500628},
       URL = {https://doi.org/10.1142/S0218202523500628},
}

@book {Deng20,
    AUTHOR = {Deng, Weihua and Hou, Ru and Wang, Wanli and Xu, Pengbo},
     TITLE = {Modeling Anomalous Diffusion--From Statistics To Mathematics},
 PUBLISHER = {World Scientific Publishing Co. Pte. Ltd., Hackensack, NJ},
      YEAR = {2020},
     PAGES = {xi+254},
      ISBN = {978-981-121-299-4},
   MRCLASS = {60-02 (35R11 44A10 60G51 60K05 60K50 82C31)},
  MRNUMBER = {4316253},
}

@article {Epshteyn19,
    AUTHOR = {Epshteyn, Yekaterina and Xia, Qing},
     TITLE = {Efficient numerical algorithms based on difference potentials
              for chemotaxis systems in 3{D}},
   JOURNAL = {J. Sci. Comput.},
  FJOURNAL = {Journal of Scientific Computing},
    VOLUME = {80},
      YEAR = {2019},
    NUMBER = {1},
     PAGES = {26--59},
      ISSN = {0885-7474},
   MRCLASS = {65M06 (35K45 35K57 65M08 92C17)},
  MRNUMBER = {3954435},
       DOI = {10.1007/s10915-019-00928-z},
}

@article {Epshteyn08,
    AUTHOR = {Epshteyn, Yekaterina and Kurganov, Alexander},
     TITLE = {New interior penalty discontinuous {G}alerkin methods for the
              {K}eller-{S}egel chemotaxis model},
   JOURNAL = {SIAM J. Numer. Anal.},
  FJOURNAL = {SIAM Journal on Numerical Analysis},
    VOLUME = {47},
      YEAR = {2008},
    NUMBER = {1},
     PAGES = {386--408},
      ISSN = {0036-1429},
   MRCLASS = {92C17 (35K57 65M60 92-08)},
  MRNUMBER = {2475945},
MRREVIEWER = {Bisso Saley},
       DOI = {10.1137/07070423X},
}

@article {Estrada18,
    AUTHOR = {Estrada-Rodriguez, Gissell and Gimperlein, Heiko and Painter,
              Kevin J.},
     TITLE = {Fractional {P}atlak-{K}eller-{S}egel equations for chemotactic
              superdiffusion},
   JOURNAL = {SIAM J. Appl. Math.},
  FJOURNAL = {SIAM Journal on Applied Mathematics},
    VOLUME = {78},
      YEAR = {2018},
    NUMBER = {2},
     PAGES = {1155--1173},
      ISSN = {0036-1399,1095-712X},
   MRCLASS = {92C17 (26A33 35K40 35Q92 35R11)},
  MRNUMBER = {3784126},
MRREVIEWER = {Robert\ Willie},
       DOI = {10.1137/17M1142867},
       URL = {https://doi.org/10.1137/17M1142867},
}

@article {Fang23,
    AUTHOR = {Fang, Xing and Qiao, Leijie and Zhang, Fengyang and Sun,
              Fuming},
     TITLE = {Explore deep network for a class of fractional partial
              differential equations},
   JOURNAL = {Chaos Solitons Fractals},
  FJOURNAL = {Chaos, Solitons \& Fractals},
    VOLUME = {172},
      YEAR = {2023},
     PAGES = {113528},
      ISSN = {0960-0779},
   MRCLASS = {65M32 (68T07)},
  MRNUMBER = {4586680},
       DOI = {10.1016/j.chaos.2023.113528},
}

@article {Filbet06,
    AUTHOR = {Filbet, Francis},
     TITLE = {A finite volume scheme for the {P}atlak-{K}eller-{S}egel
              chemotaxis model},
   JOURNAL = {Numer. Math.},
  FJOURNAL = {Numerische Mathematik},
    VOLUME = {104},
      YEAR = {2006},
    NUMBER = {4},
     PAGES = {457--488},
      ISSN = {0029-599X},
   MRCLASS = {92B05 (35B45 65M06 76M12)},
  MRNUMBER = {2249674},
MRREVIEWER = {Gabriela Marinoschi},
       DOI = {10.1007/s00211-006-0024-3},
}

@article {Fujiwara67,
    AUTHOR = {Fujiwara, Daisuke},
     TITLE = {Concrete characterization of the domains of fractional powers
              of some elliptic differential operators of the second order},
   JOURNAL = {Proc. Japan Acad.},
  FJOURNAL = {Proceedings of the Japan Academy},
    VOLUME = {43},
      YEAR = {1967},
     PAGES = {82--86},
      ISSN = {0021-4280},
   MRCLASS = {47.65},
  MRNUMBER = {216336},
MRREVIEWER = {K.\ Gustafson},
       URL = {http://projecteuclid.org/euclid.pja/1195521686},
}

@book {Grisvard11,
    AUTHOR = {Grisvard, Pierre},
     TITLE = {Elliptic Problems In Nonsmooth Domains},
    SERIES = {Classics in Applied Mathematics},
    VOLUME = {69},
      NOTE = {Reprint of the 1985 original [MR0775683],
              With a foreword by Susanne C. Brenner},
 PUBLISHER = {SIAM, Philadelphia, PA},
      YEAR = {2011},
     PAGES = {xx+410},
      ISBN = {978-1-611972-02-3},
   MRCLASS = {35J25 (01A75 35-02)},
  MRNUMBER = {3396210},
       DOI = {10.1137/1.9781611972030.ch1},
       URL = {https://doi.org/10.1137/1.9781611972030.ch1},
}

@article{Gross1975,
  title={Logarithmic {S}obolev inequalities},
  author={Gross, Leonard},
  journal={American Journal of Mathematics},
  volume={97},
  number={4},
  pages={1061--1083},
  year={1975},
  publisher={JSTOR}
}

@article {Gao14,
    AUTHOR = {Gao, Guang-hua and Sun, Zhi-zhong and Zhang, Hong-wei},
     TITLE = {A new fractional numerical differentiation formula to
              approximate the {C}aputo fractional derivative and its
              applications},
   JOURNAL = {J. Comput. Phys.},
  FJOURNAL = {Journal of Computational Physics},
    VOLUME = {259},
      YEAR = {2014},
     PAGES = {33--50},
      ISSN = {0021-9991},
   MRCLASS = {65D25 (26A33 34A08)},
  MRNUMBER = {3148558},
MRREVIEWER = {Vasileios Drakopoulos},
       DOI = {10.1016/j.jcp.2013.11.017},
}

@article {Gyrya11,
    AUTHOR = {Gyrya, Pavel and Saloff-Coste, Laurent},
     TITLE = {Neumann and {D}irichlet heat kernels in inner uniform domains},
   JOURNAL = {Ast\'{e}risque},
  FJOURNAL = {Ast\'{e}risque},
    NUMBER = {336},
      YEAR = {2011},
     PAGES = {viii+144},
      ISSN = {0303-1179,2492-5926},
      ISBN = {978-2-85629-306-5},
   MRCLASS = {35K08 (31C25 35K05 35K20 58J35 60J45 60J60)},
  MRNUMBER = {2807275},
MRREVIEWER = {Yehuda\ Pinchover},
}

@article {Hillen09,
    AUTHOR = {Hillen, T. and Painter, K. J.},
     TITLE = {A user's guide to {PDE} models for chemotaxis},
   JOURNAL = {J. Math. Biol.},
  FJOURNAL = {Journal of Mathematical Biology},
    VOLUME = {58},
      YEAR = {2009},
    NUMBER = {1-2},
     PAGES = {183--217},
      ISSN = {0303-6812,1432-1416},
   MRCLASS = {92C17 (35K57)},
  MRNUMBER = {2448428},
       DOI = {10.1007/s00285-008-0201-3},
       URL = {https://doi.org/10.1007/s00285-008-0201-3},
}

@article {Hillen13,
    AUTHOR = {Hillen, Thomas and Painter, Kevin J. and Winkler, Michael},
     TITLE = {Convergence of a cancer invasion model to a logistic
              chemotaxis model},
   JOURNAL = {Math. Models Methods Appl. Sci.},
  FJOURNAL = {Mathematical Models and Methods in Applied Sciences},
    VOLUME = {23},
      YEAR = {2013},
    NUMBER = {1},
     PAGES = {165--198},
      ISSN = {0218-2025,1793-6314},
   MRCLASS = {35Q92 (35B33 35B45 35K51 35K57 92C17 92C50)},
  MRNUMBER = {2997470},
       DOI = {10.1142/S0218202512500480},
       URL = {https://doi.org/10.1142/S0218202512500480},
}

@article{Karniadakis2021,
  title = {Physics-informed machine learning},
  volume = {3},
  ISSN = {2522-5820},
  DOI = {10.1038/s42254-021-00314-5},
  number = {6},
  journal = {Nat. Rev. Phys.},
  publisher = {Springer Science and Business Media LLC},
  author = {Karniadakis,  George Em and Kevrekidis,  Ioannis G. and Lu,  Lu and Perdikaris,  Paris and Wang,  Sifan and Yang,  Liu},
  year = {2021},
  month = may,
  pages = {422--440}
}

@article {Guo22,
    AUTHOR = {Guo, Ling and Wu, Hao and Yu, Xiaochen and Zhou, Tao},
     TITLE = {Monte {C}arlo f{PINN}s: deep learning method for forward and
              inverse problems involving high dimensional fractional partial
              differential equations},
   JOURNAL = {Comput. Methods Appl. Mech. Engrg.},
  FJOURNAL = {Computer Methods in Applied Mechanics and Engineering},
    VOLUME = {400},
      YEAR = {2022},
     PAGES = {115523},
      ISSN = {0045-7825},
   MRCLASS = {65C05 (68T07)},
  MRNUMBER = {4482903},
MRREVIEWER = {Krzysztof Joachim Bartoszek},
       DOI = {10.1016/j.cma.2022.115523},
}

@article {Guo19,
    AUTHOR = {Guo, Li and Li, Xingjie Helen and Yang, Yang},
     TITLE = {Energy dissipative local discontinuous {G}alerkin methods for
              {K}eller-{S}egel chemotaxis model},
   JOURNAL = {J. Sci. Comput.},
  FJOURNAL = {Journal of Scientific Computing},
    VOLUME = {78},
      YEAR = {2019},
    NUMBER = {3},
     PAGES = {1387--1404},
      ISSN = {0885-7474},
   MRCLASS = {65M60 (92C17)},
  MRNUMBER = {3934671},
MRREVIEWER = {Sebastian Franz},
       DOI = {10.1007/s10915-018-0813-8},
}

@book {Gripenberg90,
    AUTHOR = {Gripenberg, G. and Londen, S.-O. and Staffans, O.},
     TITLE = {Volterra Integral and Functional Equations},
    SERIES = {Encyclopedia Of Mathematics And Its Applications},
    VOLUME = {34},
 PUBLISHER = {Cambridge University Press, Cambridge},
      YEAR = {1990},
     PAGES = {},
      ISBN = {0-521-37289-5},
   MRCLASS = {45D05 (34K05 45-02 47D03 47H20)},
  MRNUMBER = {1050319},
MRREVIEWER = {J.\ \v{S}vejda},
       DOI = {10.1017/CBO9780511662805},
}

@article {Henry2010,
    AUTHOR = {Henry,  B. I. and Langlands,  T. A. M. and Straka,  P.},
     TITLE = {Fractional {F}okker-{P}lanck equations for subdiffusion with space- and time-dependent forces},
   JOURNAL = {Phy. Rev. Lett.},
    VOLUME = {105},
    number = {17},
      YEAR = {2010},
     PAGES = {170602},
      ISSN = {1079-7114},
       DOI = {10.1103/physrevlett.105.170602},
}

@article {Huang23,
    AUTHOR = {Huang, Xueling and Shen, Jie},
     TITLE = {Bound/positivity preserving {SAV} schemes for the
              {P}atlak-{K}eller-{S}egel-{N}avier-{S}tokes system},
   JOURNAL = {J. Comput. Phys.},
  FJOURNAL = {Journal of Computational Physics},
    VOLUME = {480},
      YEAR = {2023},
     PAGES = {112034},
      ISSN = {0021-9991},
   MRCLASS = {76Z05 (35Q92 92C17)},
  MRNUMBER = {4557797},
       DOI = {10.1016/j.jcp.2023.112034},
}

@article {Huang25,
    AUTHOR = {Huang, Xueling and Goubet, Olivier and Shen, Jie},
     TITLE = {Numerical analysis of a semi-implicit {E}uler scheme for the
              {K}eller-{S}egel model},
   JOURNAL = {ESAIM Math. Model. Numer. Anal.},
  FJOURNAL = {ESAIM. Mathematical Modelling and Numerical Analysis},
    VOLUME = {60},
      YEAR = {2026},
    NUMBER = {2},
     PAGES = {517--540},
      ISSN = {2822-7840,2804-7214},
   MRCLASS = {65M12 (35B09 35K20 92C17)},
  MRNUMBER = {5049166},
       DOI = {10.1051/m2an/2026012},
       URL = {https://doi.org/10.1051/m2an/2026012},
}

@article {Hu23,
    AUTHOR = {Hu, Jingwei and Zhang, Xiangxiong},
     TITLE = {Positivity-preserving and energy-dissipative finite difference
              schemes for the {F}okker-{P}lanck and {K}eller-{S}egel
              equations},
   JOURNAL = {IMA J. Numer. Anal.},
  FJOURNAL = {IMA Journal of Numerical Analysis},
    VOLUME = {43},
      YEAR = {2023},
    NUMBER = {3},
     PAGES = {1450--1484},
      ISSN = {0272-4979},
   MRCLASS = {65M06},
  MRNUMBER = {4597190},
       DOI = {10.1093/imanum/drac014},
       URL = {https://doi.org/10.1093/imanum/drac014},
}

@article {Hou18,
    AUTHOR = {Hou, Qianqian and Liu, Cheng-Jie and Wang, Ya-Guang and Wang,
              Zhian},
     TITLE = {Stability of boundary layers for a viscous hyperbolic system
              arising from chemotaxis: {O}ne-dimensional case},
   JOURNAL = {SIAM J. Math. Anal.},
  FJOURNAL = {SIAM Journal on Mathematical Analysis},
    VOLUME = {50},
      YEAR = {2018},
    NUMBER = {3},
     PAGES = {3058--3091},
      ISSN = {0036-1410,1095-7154},
   MRCLASS = {35K45 (35B25 35B35 35B40 35B44 35K57 35Q92 92C17)},
  MRNUMBER = {3814021},
       DOI = {10.1137/17M112748X},
       URL = {https://doi.org/10.1137/17M112748X},
}

@book {Haase06,
    AUTHOR = {Haase, Markus},
     TITLE = {The Functional Calculus for Sectorial Operators},
    SERIES = {Operator Theory: Advances and Applications},
    VOLUME = {169},
 PUBLISHER = {Birkh\"{a}user Verlag, Basel},
      YEAR = {2006},
     PAGES = {xiv+392},
      ISBN = {978-3-7643-7697-0; 3-7643-7697-X},
   MRCLASS = {47A60 (30E05 44A15 46B70 47A55 47D03 47E05 47F05)},
  MRNUMBER = {2244037},
MRREVIEWER = {Christian\ Le Merdy},
       DOI = {10.1007/3-7643-7698-8},
       URL = {https://doi.org/10.1007/3-7643-7698-8},
}

@article {Felix13,
    AUTHOR = {H\"{o}fling, Felix and Franosch, Thomas},
     TITLE = {Anomalous transport in the crowded world of biological cells},
   JOURNAL = {Rep. Progr. Phys.},
  FJOURNAL = {Reports on Progress in Physics},
    VOLUME = {76},
      YEAR = {2013},
    NUMBER = {4},
     PAGES = {046602},
      ISSN = {0034-4885,1361-6633},
   MRCLASS = {92C37},
  MRNUMBER = {3042495},
       DOI = {10.1088/0034-4885/76/4/046602},
       URL = {https://doi.org/10.1088/0034-4885/76/4/046602},
}

@article{Islam2015,
  title = {The mysterious nature of bacterial surface (gliding) motility: A focal adhesion-based mechanism in {M}yxococcus xanthus},
  volume = {46},
  ISSN = {1084-9521},
  url = {http://dx.doi.org/10.1016/j.semcdb.2015.10.033},
  DOI = {10.1016/j.semcdb.2015.10.033},
  journal = {Semin. Cell Dev. Biol.},
  publisher = {Elsevier BV},
  author = {Islam,  Salim T. and Mignot,  T\^am},
  year = {2015},
  pages = {143--154}
}

@book {Jin21book,
    AUTHOR = {Jin, Bangti},
     TITLE = {Fractional Differential Equations: {A}n Approach Via Fractional
              Derivatives},
    SERIES = {Applied Mathematical Sciences},
    VOLUME = {206},
 PUBLISHER = {Springer, Cham},
      YEAR = {2021},
     PAGES = {xiv+368},
      ISBN = {978-3-030-76042-7; 978-3-030-76043-4},
   MRCLASS = {35R11 (26A33 34A08)},
  MRNUMBER = {4290515},
       DOI = {10.1007/978-3-030-76043-4},
       URL = {},
}

@article {Jin16,
    AUTHOR = {Jin, Bangti and Lazarov, Raytcho and Zhou, Zhi},
     TITLE = {An analysis of the {L}1 scheme for the subdiffusion equation
              with nonsmooth data},
   JOURNAL = {IMA J. Numer. Anal.},
  FJOURNAL = {IMA Journal of Numerical Analysis},
    VOLUME = {36},
      YEAR = {2016},
    NUMBER = {1},
     PAGES = {197--221},
      ISSN = {0272-4979},
   MRCLASS = {65M06 (35R11 65M15)},
  MRNUMBER = {3463438},
MRREVIEWER = {Benito M. Chen-Charpentier},
       DOI = {10.1093/imanum/dru063},
}

@article {Jin20,
    AUTHOR = {Jin, Hai-Yang and Wang, Zhi-An},
     TITLE = {Critical mass on the {K}eller-{S}egel system with
              signal-dependent motility},
   JOURNAL = {Proc. Amer. Math. Soc.},
  FJOURNAL = {Proceedings of the American Mathematical Society},
    VOLUME = {148},
      YEAR = {2020},
    NUMBER = {11},
     PAGES = {4855--4873},
      ISSN = {0002-9939},
   MRCLASS = {35K51 (35B44 35K57 35Q92 92C17)},
  MRNUMBER = {4143400},
MRREVIEWER = {Pan Zheng},
       DOI = {10.1090/proc/15124},
}

@article{Jain2025,
author = {Jain, Rikesh and Le, Nguyen-Hung and Bertaux, Lionel and et al.},
  title = {Fatty acid metabolism and the oxidative stress response support bacterial predation},
  journal = {Proc. Natl. Acad. Sci. USA.},
  volume = {122},
  year = {2025},
  number = {5},
   PAGES = {e2420875122},
  ISSN = {1091--6490},
  DOI = {10.1073/pnas.2420875122},
}

@article {Jiang24,
    AUTHOR = {Jiang, Renjin and Lin, Fanghua},
     TITLE = {Riesz transform on exterior {L}ipschitz domains and
              applications},
   JOURNAL = {Adv. Math.},
  FJOURNAL = {Advances in Mathematics},
    VOLUME = {453},
      YEAR = {2024},
     PAGES = {Paper No. 109852, 48},
      ISSN = {0001-8708,1090-2082},
   MRCLASS = {42B37 (35J15 35J25)},
  MRNUMBER = {4778175},
       DOI = {10.1016/j.aim.2024.109852},
       URL = {https://doi.org/10.1016/j.aim.2024.109852},
}

@article {Kato72,
    AUTHOR = {Kato, Tosio},
     TITLE = {Schr\"{o}dinger operators with singular potentials},
   JOURNAL = {Israel J. Math.},
  FJOURNAL = {Israel Journal of Mathematics},
    VOLUME = {13},
      YEAR = {1972},
     PAGES = {135--148 (1973)},
      ISSN = {0021-2172},
   MRCLASS = {47F05 (35J10)},
  MRNUMBER = {333833},
MRREVIEWER = {Teruo\ Ikebe},
       DOI = {10.1007/BF02760233},
       URL = {https://doi.org/10.1007/BF02760233},
}

@article {Keller70,
    AUTHOR = {Keller, Evelyn F. and Segel, Lee A.},
     TITLE = {Initiation of slime mold aggregation viewed as an instability},
   JOURNAL = {J. Theoret. Biol.},
  FJOURNAL = {J. Theor. Biol.},
    VOLUME = {26},
      YEAR = {1970},
    NUMBER = {3},
     PAGES = {399--415},
      ISSN = {0022-5193,1095-8541},
   MRCLASS = {92C17 (35Q92 82C24 92C45)},
  MRNUMBER = {3925816},
       DOI = {10.1016/0022-5193(70)90092-5},
       URL = {https://doi.org/10.1016/0022-5193(70)90092-5},
}

@article {Kavallaris09,
    AUTHOR = {Kavallaris, Nikos I. and Souplet, Philippe},
     TITLE = {Grow-up rate and refined asymptotics for a two-dimensional
              {P}atlak-{K}eller-{S}egel model in a disk},
   JOURNAL = {SIAM J. Math. Anal.},
  FJOURNAL = {SIAM Journal on Mathematical Analysis},
    VOLUME = {40},
      YEAR = {2008/09},
    NUMBER = {5},
     PAGES = {1852--1881},
      ISSN = {0036-1410,1095-7154},
   MRCLASS = {35K57 (35B40 92C17)},
  MRNUMBER = {2471903},
MRREVIEWER = {Gabriela\ Marinoschi},
       DOI = {10.1137/080722229},
       URL = {https://doi.org/10.1137/080722229},
}

@article{Keller1971,
  title = {Traveling bands of chemotactic bacteria: A theoretical analysis},
  volume = {30},
  ISSN = {0022-5193},
  url = {http://dx.doi.org/10.1016/0022-5193(71)90051-8},
  DOI = {10.1016/0022-5193(71)90051-8},
  number = {2},
  journal = {J. Theor. Biol.},
  publisher = {Elsevier BV},
  author = {Keller,  Evelyn F. and Segel,  Lee A.},
  year = {1971},
  month = feb,
  pages = {235--248}
}

@article{Kearns2010,
  title = {A field guide to bacterial swarming motility},
  volume = {8},
  ISSN = {1740-1534},
  url = {http://dx.doi.org/10.1038/nrmicro2405},
  DOI = {10.1038/nrmicro2405},
  number = {9},
  journal = {Nat. Rev. Microbiol.},
  publisher = {Springer Science and Business Media LLC},
  author = {Kearns,  Daniel B.},
  year = {2010},
  pages = {634--644}
}

@book {Kilbas06,
    AUTHOR = {Kilbas, Anatoly A. and Srivastava, Hari M. and Trujillo, Juan
              J.},
     TITLE = {Theory And Applications Of Fractional Differential Equations},
    SERIES = {},
    VOLUME = {204},
 PUBLISHER = {Elsevier Science B.V., Amsterdam},
      YEAR = {2006},
     PAGES = {},
      ISBN = {},
   MRCLASS = {},
  MRNUMBER = {},
MRREVIEWER = {},
}

@book {Kuhlwen68,
    AUTHOR = {K\"{u}hlwen, H. and Reichenbach, H. and Heunert, H.H. and Kuczka, H.},
     TITLE = {Schwarmentwicklung und Morphogenese bei {M}yxobakterien—Archangium, Myxococcus, Chodrococcus, Chondromyces},
    SERIES = {},
    VOLUME = {},
 PUBLISHER = {Encyclopaedia Cinematographica, C893, G. Wolf, ed., Institut f\"{u}r den Wissenschaftlichen Film, G\"{o}ttingen, },
      YEAR = {1968},
     PAGES = {},
      ISBN = {},
   MRCLASS = {},
  MRNUMBER = {},
MRREVIEWER = {},
}

@book {Kuhlwen71,
    AUTHOR = {K\"{u}hlwen, H.},
     TITLE = {Polyangium fuscum (Myxobacteriales). Cystenkeimung und Schwarmentwicklung},
    SERIES = {},
    VOLUME = {},
 PUBLISHER = {Encyclopaedia Cinematographica, E1582, G. Wolf, ed., Institut f\"{u}r den Wissenschaftlichen Film, G\"{o}ttingen, },
      YEAR = {1971},
     PAGES = {},
      ISBN = {},
   MRCLASS = {},
  MRNUMBER = {},
MRREVIEWER = {},
}

@article{Keegstra2022,
  title = {The ecological roles of bacterial chemotaxis},
  volume = {20},
  ISSN = {1740-1534},
  url = {},
  DOI = {10.1038/s41579-022-00709-w},
  number = {8},
  JOURNAL = {Nat. Rev. Microbiol.},
  fjournal = {Nature Reviews Microbiology},
  publisher = {Springer Science and Business Media LLC},
  author = {Keegstra,  Johannes M. and Carrara,  Francesco and Stocker,  Roman},
  year = {2022},
  month = mar,
  pages = {491--504}
}

@article{Kaiser2003,
  title = {Coupling cell movement to multicellular development in myxobacteria},
  volume = {1},
  ISSN = {1740-1534},
  url = {},
  DOI = {10.1038/nrmicro733},
  number = {1},
  journal = {Nat. Rev. Microbiol.},
  publisher = {Springer Science and Business Media LLC},
  author = {Kaiser,  Dale},
  year = {2003},
  month = oct,
  pages = {45--54}
}

@article{Kiskowski2004,
  title = {Role of streams in myxobacteria aggregate formation},
  volume = {1},
  ISSN = {1478-3975},
  url = {},
  DOI = {10.1088/1478-3967/1/3/005},
  number = {3},
  journal = {Phys. Biol.},
  publisher = {IOP Publishing},
  author = {Kiskowski,  Maria A and Jiang,  Yi and Alber,  Mark S},
  year = {2004},
  month = oct,
  pages = {173--183}
}

@book {Klafter15,
    AUTHOR = {Klafter, J. and Sokolov, I. M.},
     TITLE = {First Steps in Random Walks: From Tools to Applications},
 PUBLISHER = {Oxford University Press, Oxford},
      YEAR = {2015},
     PAGES = {vi+152},
      ISBN = {978-0-19-875409-1},
   MRCLASS = {60G50 (60G15 60G35 60J10 82B43 82C41)},
  MRNUMBER = {3444134},
}

@article {Lindemulder26,
    AUTHOR = {Lindemulder, Nick and Lorist, Emiel and Roodenburg, Floris B.
              and Veraar, Mark C.},
     TITLE = {Functional calculus on weighted {S}obolev spaces for the
              {L}aplacian on rough domains},
   JOURNAL = {J. Differential Equations},
  FJOURNAL = {Journal of Differential Equations},
    VOLUME = {454},
      YEAR = {2026},
     PAGES = {Paper No. 113884, 71},
      ISSN = {0022-0396,1090-2732},
   MRCLASS = {47A60 (35K20 46E35)},
  MRNUMBER = {4999633},
       DOI = {10.1016/j.jde.2025.113884},
       URL = {https://doi.org/10.1016/j.jde.2025.113884},
}

@article {Liu18,
    AUTHOR = {Liu, Jian-Guo and Wang, Li and Zhou, Zhennan},
     TITLE = {Positivity-preserving and asymptotic preserving method for
              2{D} {K}eller-{S}egal equations},
   JOURNAL = {Math. Comp.},
  FJOURNAL = {Mathematics of Computation},
    VOLUME = {87},
      YEAR = {2018},
    NUMBER = {311},
     PAGES = {1165--1189},
      ISSN = {0025-5718},
   MRCLASS = {65M06 (35B25 35K45 35Q92 65M12)},
  MRNUMBER = {3766384},
MRREVIEWER = {Isabella Cravero},
       DOI = {10.1090/mcom/3250},
}

@article {Lankeit17,
    AUTHOR = {Lankeit, Johannes and Winkler, Michael},
     TITLE = {A generalized solution concept for the {K}eller-{S}egel system
              with logarithmic sensitivity: {G}lobal solvability for large
              nonradial data},
   JOURNAL = {NoDea-Nonlinear Differ. Equ. Appl.},
  FJOURNAL = {NoDEA. Nonlinear Differential Equations and Applications},
    VOLUME = {24},
      YEAR = {2017},
    NUMBER = {4},
     PAGES = {49},
      ISSN = {1021-9722,1420-9004},
   MRCLASS = {35K51 (35A01 35D30 35K59 92C17)},
  MRNUMBER = {3674184},
       DOI = {10.1007/s00030-017-0472-8},
       URL = {https://doi.org/10.1007/s00030-017-0472-8},
}

@article {Langlands10,
    AUTHOR = {Langlands, T. A. M. and Henry, B. I.},
     TITLE = {Fractional chemotaxis diffusion equations},
   JOURNAL = {Phys. Rev. E},
  FJOURNAL = {Physical Review E. Statistical, Nonlinear, and Soft Matter
              Physics},
    VOLUME = {81},
      YEAR = {2010},
    NUMBER = {5},
     PAGES = {051102, 12},
      ISSN = {1539-3755},
   MRCLASS = {35R11 (35K57 35Q92 92C17)},
  MRNUMBER = {2736241},
       DOI = {10.1103/PhysRevE.81.051102},
       URL = {},
}

@article{Livne2024,
  title = {Collective condensation and auto-aggregation of {E}scherichia coli in uniform acidic environments},
  volume = {7},
  ISSN = {2399-3642},
  url = {http://dx.doi.org/10.1038/s42003-024-06698-1},
  DOI = {10.1038/s42003-024-06698-1},
  number = {1},
  journal = {Commun. Biol.},
  author = {Livne,  Nir and Koler,  Moriah and Vaknin,  Ady},
  year = {2024},
  PAGES = {1028},
}

@article{Livne2025,
  title = {Pattern formation in {E}. coli through negative chemotaxis: Instability, condensation, and merging},
  volume = {7},
  ISSN = {2643-1564},
  url = {http://dx.doi.org/10.1103/PhysRevResearch.7.023095},
  DOI = {10.1103/physrevresearch.7.023095},
  number = {2},
  journal = {Phys. Rev. Res.},
  publisher = {American Physical Society (APS)},
  author = {Livne, Nir and Vaknin, Ady and Agam, Oded},
  year = {2025},
  PAGES = {023095},
}

@article {Li14,
    AUTHOR = {Li, Jingyu and Li, Tong and Wang, Zhi-An},
     TITLE = {Stability of traveling waves of the {K}eller-{S}egel system
              with logarithmic sensitivity},
   JOURNAL = {Math. Models Methods Appl. Sci.},
  FJOURNAL = {Mathematical Models and Methods in Applied Sciences},
    VOLUME = {24},
      YEAR = {2014},
    NUMBER = {14},
     PAGES = {2819--2849},
      ISSN = {0218-2025,1793-6314},
   MRCLASS = {35K40 (35B35 35C07 35K59 46N60 92C17)},
  MRNUMBER = {3269780},
       DOI = {10.1142/S0218202514500389},
       URL = {https://doi.org/10.1142/S0218202514500389},
}

@article {Lankeit16,
    AUTHOR = {Lankeit, Johannes},
     TITLE = {A new approach toward boundedness in a two-dimensional
              parabolic chemotaxis system with singular sensitivity},
   JOURNAL = {Math. Methods Appl. Sci.},
  FJOURNAL = {Mathematical Methods in the Applied Sciences},
    VOLUME = {39},
      YEAR = {2016},
    NUMBER = {3},
     PAGES = {394--404},
      ISSN = {0170-4214,1099-1476},
   MRCLASS = {35K40 (35A01 35Q92 92C17)},
  MRNUMBER = {3454184},
       DOI = {10.1002/mma.3489},
       URL = {https://doi.org/10.1002/mma.3489},
}

@article {McClenny23,
    AUTHOR = {McClenny, Levi D. and Braga-Neto, Ulisses M.},
     TITLE = {Self-adaptive physics-informed neural networks},
   JOURNAL = {J. Comput. Phys.},
  FJOURNAL = {Journal of Computational Physics},
    VOLUME = {474},
      YEAR = {2023},
     PAGES = {111722},
      ISSN = {0021-9991},
   MRCLASS = {68T07 (65N99)},
  MRNUMBER = {4513793},
       DOI = {10.1016/j.jcp.2022.111722},
       URL = {https://doi.org/10.1016/j.jcp.2022.111722},
}

@book {McBride93,
    AUTHOR = {McBride, M. J. and Hartzell, P. and Zusman, D. R.},
     TITLE = {Myxobacteria II},
    SERIES = {},
    VOLUME = {},
   EDITION = {},
 PUBLISHER = {ASM Press, Washington DC},
      YEAR = {1993},
     PAGES = {},
      ISBN = {},
   MRCLASS = {},
  MRNUMBER = {},
}

@article {Mendez01,
    AUTHOR = {Mendez, Osvaldo and Mitrea, Marius},
     TITLE = {Complex powers of the {N}eumann {L}aplacian in {L}ipschitz
              domains},
   JOURNAL = {Math. Nachr.},
  FJOURNAL = {Mathematische Nachrichten},
    VOLUME = {223},
      YEAR = {2001},
     PAGES = {77--88},
      ISSN = {0025-584X,1522-2616},
   MRCLASS = {35J05 (47F05)},
  MRNUMBER = {1817850},
MRREVIEWER = {H.\ Triebel},
       DOI = {10.1002/1522-2616(200103)223:1<77::AID-MANA77>3.3.CO;2-4},
       URL =
              {https://doi.org/10.1002/1522-2616(200103)223:1<77::AID-MANA77>3.3.CO;2-4},
}

@article {Ma23,
    AUTHOR = {Ma, Fugui and Zhao, Lijing and Deng, Weihua and Wang, Yejuan},
     TITLE = {Analyses of the contour integral method for time fractional
              normal-subdiffusion transport equation},
   JOURNAL = {J. Sci. Comput.},
  FJOURNAL = {Journal of Scientific Computing},
    VOLUME = {97},
      YEAR = {2023},
    NUMBER = {2},
     PAGES = {45},
      ISSN = {0885-7474},
   MRCLASS = {65D30 (30E20 33E12)},
  MRNUMBER = {4649445},
MRREVIEWER = {Kai Diethelm},
       DOI = {10.1007/s10915-023-02359-3},
       URL = {},
}

@article {Ma23b,
    AUTHOR = {Ma, Fugui and Zhao, Lijing and Wang, Yejuan and Deng, Weihua},
     TITLE = {The contour integral method for {F}eynman-{K}ac equation with
              two internal states},
   JOURNAL = {Comput. Math. Appl.},
  FJOURNAL = {Computers \& Mathematics with Applications. An International
              Journal},
    VOLUME = {151},
      YEAR = {2023},
     PAGES = {80--100},
      ISSN = {0898-1221,1873-7668},
   MRCLASS = {65M70 (35R60)},
  MRNUMBER = {4649911},
       DOI = {10.1016/j.camwa.2023.09.037},
       URL = {https://doi.org/10.1016/j.camwa.2023.09.037},
}

@article {Ma25,
    AUTHOR = {Ma, Fugui and Tian, Wenyi and Deng, Weihua},
     TITLE = {Mathematical modeling and analysis for the chemotactic
              diffusion in porous media with incompressible
              {N}avier-{S}tokes equations over bounded domain},
   JOURNAL = {J. Differential Equations},
  FJOURNAL = {Journal of Differential Equations},
    VOLUME = {436},
      YEAR = {2025},
     PAGES = {113305},
      ISSN = {0022-0396},
   MRCLASS = {35Q92 (35A01 35A02 35B44 35D30 35Q35 35R11)},
  MRNUMBER = {4891582},
       DOI = {10.1016/j.jde.2025.113305},
       URL = {},
}

@article {Ma25b,
    AUTHOR = {Ma, Fugui and Tian, Wenyi and Deng, Weihua},
     TITLE = {Corrigendum to ``{M}athematical modeling and analysis for the
              chemotactic diffusion in porous media with incompressible
              {N}avier-{S}tokes equations over bounded domain'' [{J}.
              {D}iffer. {E}qu. 436 (2025) 113305]},
   JOURNAL = {J. Differential Equations},
  FJOURNAL = {Journal of Differential Equations},
    VOLUME = {443},
      YEAR = {2025},
     PAGES = {113656},
      ISSN = {0022-0396,1090-2732},
   MRCLASS = {35Q92 (35A01 35A02 35B44 35D30 35Q35 35R11)},
  MRNUMBER = {4941176},
       DOI = {10.1016/j.jde.2025.113656},
       URL = {https://doi.org/10.1016/j.jde.2025.113656},
}

@article {Ma26,
    AUTHOR = {Ma, Fugui},
     TITLE = {Scaling crossover of the generalized {J}effreys-type law},
   JOURNAL = {arXiv:2510.07930},
  FJOURNAL = {},
    VOLUME = {},
      YEAR = {2025},
  archivePrefix={arXiv},
  primaryClass={math.NA},
       DOI = {arXiv:2510.07930},
}

@article {Meerschaert13,
    AUTHOR = {Meerschaert, M. M. and Straka, P.},
     TITLE = {Inverse stable subordinators},
   JOURNAL = {Math. Model. Nat. Phenom.},
  FJOURNAL = {Mathematical Modelling of Natural Phenomena},
    VOLUME = {8},
      YEAR = {2013},
    NUMBER = {2},
     PAGES = {1--16},
      ISSN = {0973-5348},
   MRCLASS = {60E07 (26A33 35R11 60G22 60J60)},
  MRNUMBER = {3049524},
MRREVIEWER = {Marcin Magdziarz},
       DOI = {10.1051/mmnp/20138201},
       URL = {},
}

@article {Magdziarz2008,
    AUTHOR = {Magdziarz,  Marcin and Weron,  Aleksander and Klafter,  Joseph},
     TITLE = {Equivalence of the fractional {F}okker-{P}lanck and subordinated {L}angevin equations:
           the case of a time-dependent force},
   JOURNAL = {Phy. Rev. Lett.},
    VOLUME = {101},
      YEAR = {2008},
     PAGES = {210601},
      ISSN = {1079-7114},
       DOI = {10.1103/physrevlett.101.210601},
}

@article{Mauriello2010,
  title = {Gliding motility revisited: How do the myxobacteria move without flagella?},
  volume = {74},
  ISSN = {1098-5557},
  url = {},
  DOI = {10.1128/mmbr.00043-09},
  number = {2},
  journal = {Microbiol. Mol. Biol. R.},
  publisher = {American Society for Microbiology},
  author = {Mauriello,  Emilia M. F. and Mignot,  T\^{a}m and Yang,  Zhaomin and Zusman,  David R.},
  year = {2010},
  month = jun,
  pages = {229--249}
}

@article{Malla2025,
  author = {Malla,  Tek Narsingh and Aldama,  Luis and Leon,  Viridiana and et al.},
  title = {Observation of early events in the photoactivation of {M}yxobacterial phytochrome using time-resolved serial femtosecond crystallography},
  journal = {Commun. Chem.},
  volume = {8},
  year = {2025},
  number = {1},
  PAGES = {183},
  ISSN = {2399-3669},
  url = {http://dx.doi.org/10.1038/s42004-025-01578-z},
  DOI = {10.1038/s42004-025-01578-z},
}

@article{Meyer2014,
  title = {Active {B}rownian agents with concentration-dependent chemotactic sensitivity},
  volume = {89},
  ISSN = {1550-2376},
  url = {http://dx.doi.org/10.1103/PhysRevE.89.022711},
  DOI = {10.1103/physreve.89.022711},
  number = {2},
  journal = {Phys. Rev. E},
  publisher = {American Physical Society (APS)},
  author = {Meyer,  Marcel and Schimansky-Geier,  Lutz and Romanczuk,  Pawel},
  year = {2014},
  PAGES = {022711},
}

@article{Mittal2003,
  title = {Motility of {E}scherichia coli cells in clusters formed by chemotactic aggregation},
  volume = {100},
  ISSN = {1091-6490},
  url = {http://dx.doi.org/10.1073/pnas.2233626100},
  DOI = {10.1073/pnas.2233626100},
  number = {23},
  journal = {Proc. Natl. Acad. Sci. USA.},
  publisher = {Proceedings of the National Academy of Sciences},
  author = {Mittal,  Nikhil and Budrene,  Elena O. and Brenner,  Michael P. and van Oudenaarden,  Alexander},
  year = {2003},
  month = nov,
  pages = {13259--13263}
}

@article{Magdziarz07,
 author = {Magdziarz,  Marcin and Weron,  Aleksander and Weron,  Karina},
  title = {Fractional {F}okker-{P}lanck dynamics: Stochastic representation and computer simulation},
  journal = {Phys. Rev. E},
  year = {2007},
  volume = {75},
  pages = {016708},
  ISSN = {1550-2376},
  url = {http://dx.doi.org/10.1103/physreve.75.016708},
  DOI = {10.1103/physreve.75.016708},
  number = {1},
}

@article{Nan2011,
  title = {Myxobacteria gliding motility requires cytoskeleton rotation powered by proton motive force},
  volume = {108},
  ISSN = {1091-6490},
  url = {http://dx.doi.org/10.1073/pnas.1018556108},
  DOI = {10.1073/pnas.1018556108},
  number = {6},
  journal = {Proc. Natl. Acad. Sci. USA.},
  publisher = {Proceedings of the National Academy of Sciences},
  author = {Nan,  Beiyan and Chen,  Jing and Neu,  John C. and Berry,  Richard M. and Oster,  George and Zusman,  David R.},
  year = {2011},
  pages = {2498--2503}
}

@article {Othmer97,
    AUTHOR = {Othmer, Hans G. and Stevens, Angela},
     TITLE = {Aggregation, blowup, and collapse: the {ABC}s of taxis in
              reinforced random walks},
   JOURNAL = {SIAM J. Appl. Math.},
  FJOURNAL = {SIAM Journal on Applied Mathematics},
    VOLUME = {57},
      YEAR = {1997},
    NUMBER = {4},
     PAGES = {1044--1081},
      ISSN = {0036-1399},
   MRCLASS = {92B05 (60J15)},
  MRNUMBER = {1462051},
MRREVIEWER = {Istv\'{a}n Ratk\'{o}},
       DOI = {10.1137/S0036139995288976},
}

@book {Podlubny99,
    AUTHOR = {Podlubny, Igor},
     TITLE = {Fractional Differential Equations},
    SERIES = {Mathematics in Science and Engineering},
    VOLUME = {198},
      NOTE = {},
 PUBLISHER = {Academic Press, Inc., San Diego, CA},
      YEAR = {1999},
     PAGES = {xxiv+340},
      ISBN = {0-12-558840-2},
   MRCLASS = {26A33 (34K05)},
  MRNUMBER = {1658022},
MRREVIEWER = {Anatoly Kilbas},
}

@article {Pang21,
    AUTHOR = {Pang, Peter Y. H. and Wang, Yifu and Yin, Jingxue},
     TITLE = {Asymptotic profile of a two-dimensional
              chemotaxis-{N}avier-{S}tokes system with singular sensitivity
              and logistic source},
   JOURNAL = {Math. Models Methods Appl. Sci.},
  FJOURNAL = {Mathematical Models and Methods in Applied Sciences},
    VOLUME = {31},
      YEAR = {2021},
    NUMBER = {3},
     PAGES = {577--618},
      ISSN = {0218-2025},
   MRCLASS = {35K55 (35B40 35Q30 35Q92 76D05 92C17)},
  MRNUMBER = {4260208},
       DOI = {10.1142/S0218202521500135},
}

@article {Qiu21,
    AUTHOR = {Qiu, Changxin and Liu, Qingyuan and Yan, Jue},
     TITLE = {Third order positivity-preserving direct discontinuous
              {G}alerkin method with interface correction for chemotaxis
              {K}eller-{S}egel equations},
   JOURNAL = {J. Comput. Phys.},
  FJOURNAL = {Journal of Computational Physics},
    VOLUME = {433},
      YEAR = {2021},
     PAGES = {110191},
      ISSN = {0021-9991},
   MRCLASS = {65M60 (35Q92 65M12 92C17)},
  MRNUMBER = {4218537},
       DOI = {10.1016/j.jcp.2021.110191},
}

@article {Quan20,
    AUTHOR = {Quan, Chaoyu and Tang, Tao and Yang, Jiang},
     TITLE = {How to define dissipation-preserving energy for
              time-fractional phase-field equations},
   JOURNAL = {CSIAM Trans. Appl. Math.},
  FJOURNAL = {CSIAM Transactions on Applied Mathematics},
    VOLUME = {1},
      YEAR = {2020},
     PAGES = {478-490},
      ISSN = {},
   MRCLASS = {},
  MRNUMBER = {},
       DOI = {10.4208/csiam-am.2020-0024 },
}

@article {Raissi19,
    AUTHOR = {Raissi, M. and Perdikaris, P. and Karniadakis, G. E.},
     TITLE = {Physics-informed neural networks: a deep learning framework
              for solving forward and inverse problems involving nonlinear
              partial differential equations},
   JOURNAL = {J. Comput. Phys.},
  FJOURNAL = {Journal of Computational Physics},
    VOLUME = {378},
      YEAR = {2019},
     PAGES = {686--707},
      ISSN = {0021-9991},
   MRCLASS = {65M70 (68T05)},
  MRNUMBER = {3881695},
       DOI = {10.1016/j.jcp.2018.10.045},
}

@book {Risken89,
    AUTHOR = {Risken, H.},
     TITLE = {The {F}okker-{P}lanck equation},
    SERIES = {Springer Series in Synergetics},
    VOLUME = {18},
   EDITION = {Second},
      NOTE = {Methods of solution and applications},
 PUBLISHER = {Springer-Verlag, Berlin},
      YEAR = {1989},
     PAGES = {xiv+472},
      ISBN = {3-540-50498-2},
   MRCLASS = {82-02 (60J65 82A31)},
  MRNUMBER = {987631},
       DOI = {10.1007/978-3-642-61544-3},
       URL = {https://doi.org/10.1007/978-3-642-61544-3},
}

@book {Reichenbach68,
    AUTHOR = {Reichenbach, H. Heunert and Kuczka, H.},
     TITLE = {Archangium violaceum ({M}yxobacteriales)—{S}chwarmentwicklung und Bildung von Protocysten},
   JOURNAL = {Encyclopaedia Cinematographica, E777, G. Wolf, ed., Institut f\"{u}r den Wissenschaftlichen Film, G\"{o}ttingen},
    VOLUME = {},
      YEAR = {1968},
     PAGES = {},
      ISSN = {},
   MRCLASS = {},
  MRNUMBER = {},
       DOI = {},
}

@book {Rudolf20,
    AUTHOR = {Gorenflo, Rudolf and Kilbas, Anatoly A. and Mainardi,
              Francesco and Rogosin, Sergei},
     TITLE = {Mittag-{L}effler functions, related topics and applications},
    SERIES = {Springer Monographs in Mathematics},
      NOTE = {Second edition [of 3244285]},
 PUBLISHER = {Springer, Berlin},
      YEAR = {[2020] \copyright 2020},
     PAGES = {xvi+540},
      ISBN = {978-3-662-61550-8; 978-3-662-61549-2},
   MRCLASS = {33-02 (26A33 33C60 33E12 34A08 44A20 45K05 60G22)},
  MRNUMBER = {4179587},
       DOI = {10.1007/978-3-662-61550-8},
       URL = {https://doi.org/10.1007/978-3-662-61550-8},
}

@article {Sun06,
    AUTHOR = {Sun, Zhi-zhong and Wu, Xiaonan},
     TITLE = {A fully discrete difference scheme for a diffusion-wave
              system},
   JOURNAL = {Appl. Numer. Math.},
  FJOURNAL = {Applied Numerical Mathematics. An IMACS Journal},
    VOLUME = {56},
      YEAR = {2006},
    NUMBER = {2},
     PAGES = {193--209},
      ISSN = {0168-9274},
   MRCLASS = {65M06},
  MRNUMBER = {2200938},
MRREVIEWER = {Kenneth H. Karlsen},
       DOI = {10.1016/j.apnum.2005.03.003},
}

@article{Saggu2023,
  title = {Myxobacteria: {B}iology and bioactive secondary metabolites},
  volume = {174},
  ISSN = {0923-2508},
  DOI = {10.1016/j.resmic.2023.104079},
  number = {7},
  journal = {Res. Microbiol.},
  publisher = {Elsevier BV},
  author = {Saggu, Sandeep Kaur and Nath,  Amar and Kumar,  Shiv},
  year = {2023},
  month = sep,
  pages = {104079}
}

@article{Sozinova2005,
  title = {A three-dimensional model of myxobacterial aggregation by contact-mediated interactions},
  volume = {102},
  ISSN = {1091-6490},
  url = {},
  DOI = {10.1073/pnas.0504259102},
  number = {32},
  journal = {Proc. Natl. Acad. Sci. U. S. A.},
  publisher = {Proceedings of the National Academy of Sciences},
  author = {Sozinova,  Olga and Jiang,  Yi and Kaiser,  Dale and Alber,  Mark},
  year = {2005},
  month = aug,
  pages = {11308--11312}
}

@article {Stevens00,
    AUTHOR = {Stevens, Angela},
     TITLE = {A stochastic cellular automaton modeling gliding and
              aggregation of myxobacteria},
   JOURNAL = {SIAM J. Appl. Math.},
  FJOURNAL = {SIAM Journal on Applied Mathematics},
    VOLUME = {61},
      YEAR = {2000},
    NUMBER = {1},
     PAGES = {172--182},
      ISSN = {0036-1399},
   MRCLASS = {92C17 (60K99 92D50)},
  MRNUMBER = {1776392},
MRREVIEWER = {John G. Milton},
       DOI = {10.1137/S0036139998342053},
       URL = {},
}

@article {Stevens00b,
    AUTHOR = {Stevens, Angela},
     TITLE = {The derivation of chemotaxis equations as limit dynamics of
              moderately interacting stochastic many-particle systems},
   JOURNAL = {SIAM J. Appl. Math.},
  FJOURNAL = {SIAM Journal on Applied Mathematics},
    VOLUME = {61},
      YEAR = {2000},
    NUMBER = {1},
     PAGES = {183--212},
      ISSN = {0036-1399},
   MRCLASS = {92C17 (60K99 92D50)},
  MRNUMBER = {1776393},
MRREVIEWER = {John G. Milton},
       DOI = {10.1137/S0036139998342065},
}

@article {Saito07,
    AUTHOR = {Saito, Norikazu},
     TITLE = {Conservative upwind finite-element method for a simplified
              {K}eller-{S}egel system modelling chemotaxis},
   JOURNAL = {IMA J. Numer. Anal.},
  FJOURNAL = {IMA Journal of Numerical Analysis},
    VOLUME = {27},
      YEAR = {2007},
    NUMBER = {2},
     PAGES = {332--365},
      ISSN = {0272-4979},
   MRCLASS = {65M15 (35K50 65M60 92C17)},
  MRNUMBER = {2317007},
MRREVIEWER = {Andreas Veeser},
       DOI = {10.1093/imanum/drl018},
}

@article {Sulman19,
    AUTHOR = {Sulman, M. and Nguyen, T.},
     TITLE = {A positivity preserving moving mesh finite element method for
              the {K}eller-{S}egel chemotaxis model},
   JOURNAL = {J. Sci. Comput.},
  FJOURNAL = {Journal of Scientific Computing},
    VOLUME = {80},
      YEAR = {2019},
    NUMBER = {1},
     PAGES = {649--666},
      ISSN = {0885-7474},
   MRCLASS = {65M60 (65M50 92C17)},
  MRNUMBER = {3954456},
MRREVIEWER = {Jason M. Graham},
       DOI = {10.1007/s10915-019-00951-0},
}

@article{Salek2019,
  title = {Bacterial chemotaxis in a microfluidic {T}-maze reveals strong phenotypic heterogeneity in chemotactic sensitivity},
  volume = {10},
  ISSN = {2041-1723},
  url = {http://dx.doi.org/10.1038/s41467-019-09521-2},
  DOI = {10.1038/s41467-019-09521-2},
  number = {1},
  journal = {Nat. Commun.},
  publisher = {Springer Science and Business Media LLC},
  author = {Salek,  M. Mehdi and Carrara,  Francesco and Fernandez,  Vicente and Guasto,  Jeffrey S. and Stocker,  Roman},
  year = {2019},
  PAGES = {1877},
}

@article{Scheidweiler2024,
  title = {Spatial structure, chemotaxis and quorum sensing shape bacterial biomass accumulation in complex porous media},
  volume = {15},
  ISSN = {2041-1723},
  url = {http://dx.doi.org/10.1038/s41467-023-44267-y},
  DOI = {10.1038/s41467-023-44267-y},
  number = {1},
  journal = {Nat. Commun.},
  publisher = {Springer Science and Business Media LLC},
  author = {Scheidweiler,  David and Bordoloi,  Ankur Deep and Jiao, Wenqiao and et al.},
  year = {2024},
  PAGES = {191},
}

@article {Stinner11,
    AUTHOR = {Stinner, Christian and Winkler, Michael},
     TITLE = {Global weak solutions in a chemotaxis system with large
              singular sensitivity},
   JOURNAL = {Nonlinear Anal.-Real World Appl.},
  FJOURNAL = {Nonlinear Analysis. Real World Applications. An International
              Multidisciplinary Journal},
    VOLUME = {12},
      YEAR = {2011},
    NUMBER = {6},
     PAGES = {3727--3740},
      ISSN = {1468-1218,1878-5719},
   MRCLASS = {35K51 (35A20 35K91 92C17)},
  MRNUMBER = {2833007},
       DOI = {10.1016/j.nonrwa.2011.07.006},
       URL = {https://doi.org/10.1016/j.nonrwa.2011.07.006},
}

@book {Taira16,
    AUTHOR = {Taira, Kazuaki},
     TITLE = {Analytic Semigroups And Semilinear Initial Boundary Value
              Problems},
    SERIES = {London Mathematical Society Lecture Note Series},
    VOLUME = {434},
   EDITION = {Second},
 PUBLISHER = {Cambridge University Press, Cambridge},
      YEAR = {2016},
     PAGES = {xvi+331},
      ISBN = {978-1-316-62086-1},
   MRCLASS = {47A05 (35J25 35S15 46N20 47F05 47G30)},
  MRNUMBER = {3444791},
       DOI = {10.1017/CBO9781316729755},
}

@article {Tao13,
    AUTHOR = {Tao, Youshan and Wang, Zhi-An},
     TITLE = {Competing effects of attraction vs. repulsion in chemotaxis},
   JOURNAL = {Math. Models Methods Appl. Sci.},
  FJOURNAL = {Mathematical Models and Methods in Applied Sciences},
    VOLUME = {23},
      YEAR = {2013},
    NUMBER = {1},
     PAGES = {1--36},
      ISSN = {0218-2025,1793-6314},
   MRCLASS = {35K51 (35B35 35B40 35B44 35K57)},
  MRNUMBER = {2997466},
       DOI = {10.1142/S0218202512500443},
       URL = {https://doi.org/10.1142/S0218202512500443},
}

@article{Uar2025,
  title = {Self-generated chemotaxis of mixed cell populations},
  volume = {122},
  ISSN = {1091-6490},
  url = {http://dx.doi.org/10.1073/pnas.2504064122},
  DOI = {10.1073/pnas.2504064122},
  number = {34},
  journal = {Proc. Natl. Acad. Sci. USA.},
  publisher = {Proceedings of the National Academy of Sciences},
  author = {U\c{c}ar,  Mehmet Can and Alsberga,  Zane and Alanko,  Jonna and Sixt,  Michael and Hannezo,  Edouard},
  year = {2025},
  PAGES = {e2504064122},
}

@article{Velicer2003,
  title = {Evolution of novel cooperative swarming in the bacterium {M}yxococcus xanthus},
  volume = {425},
  ISSN = {1476-4687},
  url = {},
  DOI = {10.1038/nature01908},
  number = {6953},
  journal = {Nature},
  publisher = {Springer Science and Business Media LLC},
  author = {Velicer,  Gregory J. and Yu,  Yuen-tsu N.},
  year = {2003},
  month = sep,
  pages = {75--78}
}

@article {Velazquez04a,
    AUTHOR = {Vel\'{a}zquez, J. J. L.},
     TITLE = {Point dynamics in a singular limit of the {K}eller-{S}egel
              model. {I}. {M}otion of the concentration regions},
   JOURNAL = {SIAM J. Appl. Math.},
  FJOURNAL = {SIAM Journal on Applied Mathematics},
    VOLUME = {64},
      YEAR = {2004},
    NUMBER = {4},
     PAGES = {1198--1223},
      ISSN = {0036-1399,1095-712X},
   MRCLASS = {35K45 (35B25 35B40 35Q80 92B05)},
  MRNUMBER = {2068667},
MRREVIEWER = {Thomas\ P.\ Witelski},
       DOI = {10.1137/S0036139903433888},
       URL = {https://doi.org/10.1137/S0036139903433888},
}

@article {Velazquez04b,
    AUTHOR = {Vel\'{a}zquez, J. J. L.},
     TITLE = {Point dynamics in a singular limit of the {K}eller-{S}egel
              model. {II}. {F}ormation of the concentration regions},
   JOURNAL = {SIAM J. Appl. Math.},
  FJOURNAL = {SIAM Journal on Applied Mathematics},
    VOLUME = {64},
      YEAR = {2004},
    NUMBER = {4},
     PAGES = {1224--1248},
      ISSN = {0036-1399,1095-712X},
   MRCLASS = {35K45 (35B25 35B40 35Q80 92B05)},
  MRNUMBER = {2068668},
MRREVIEWER = {Thomas\ P.\ Witelski},
       DOI = {10.1137/S003613990343389X},
       URL = {https://doi.org/10.1137/S003613990343389X},
}

@article {Wang12,
    AUTHOR = {Wang, Rong-Nian and Chen, De-Han and Xiao, Ti-Jun},
     TITLE = {Abstract fractional {C}auchy problems with almost sectorial
              operators},
   JOURNAL = {J. Differential Equations},
  FJOURNAL = {Journal of Differential Equations},
    VOLUME = {252},
      YEAR = {2012},
    NUMBER = {1},
     PAGES = {202--235},
      ISSN = {0022-0396,1090-2732},
   MRCLASS = {34A08 (33E12 34G20 35K20 35K90 35K91)},
  MRNUMBER = {2852204},
MRREVIEWER = {Ralph\ Chill},
       DOI = {10.1016/j.jde.2011.08.048},
       URL = {https://doi.org/10.1016/j.jde.2011.08.048},
}

@article{Wolgemuth2002,
  title = {How Myxobacteria glide},
  volume = {12},
  ISSN = {0960-9822},
  url = {},
  DOI = {10.1016/s0960-9822(02)00716-9},
  number = {5},
  journal = {Curr. Bio.},
  publisher = {Elsevier BV},
  author = {Wolgemuth,  Charles and Hoiczyk,  Egbert and Kaiser,  Dale and Oster,  George},
  year = {2002},
  month = mar,
  pages = {369--377}
}

@article {Winkler10,
    AUTHOR = {Winkler, Michael},
     TITLE = {Absence of collapse in a parabolic chemotaxis system with
              signal-dependent sensitivity},
   JOURNAL = {Math. Nachr.},
  FJOURNAL = {Mathematische Nachrichten},
    VOLUME = {283},
      YEAR = {2010},
    NUMBER = {11},
     PAGES = {1664--1673},
      ISSN = {0025-584X},
   MRCLASS = {35K51 (35A01 35A02 35B45 35Q92 92C17)},
  MRNUMBER = {2759803},
MRREVIEWER = {Narcisa C. Apreutesei},
       DOI = {10.1002/mana.200810838},
}

@article {Winkler11,
    AUTHOR = {Winkler, Michael},
     TITLE = {Global solutions in a fully parabolic chemotaxis system with
              singular sensitivity},
   JOURNAL = {Math. Methods Appl. Sci.},
  FJOURNAL = {Mathematical Methods in the Applied Sciences},
    VOLUME = {34},
      YEAR = {2011},
    NUMBER = {2},
     PAGES = {176--190},
      ISSN = {0170-4214,1099-1476},
   MRCLASS = {35K51 (35A01 35A09 35D30 92C17)},
  MRNUMBER = {2778870},
MRREVIEWER = {Temur\ Jangveladze},
       DOI = {10.1002/mma.1346},
       URL = {https://doi.org/10.1002/mma.1346},
}

@article {Winkler22,
    AUTHOR = {Winkler, Michael},
     TITLE = {Unlimited growth in logarithmic {K}eller-{S}egel systems},
   JOURNAL = {J. Differential Equations},
  FJOURNAL = {Journal of Differential Equations},
    VOLUME = {309},
      YEAR = {2022},
     PAGES = {74--97},
      ISSN = {0022-0396,1090-2732},
   MRCLASS = {35K40 (35B44 35K58 92C17)},
  MRNUMBER = {4344985},
       DOI = {10.1016/j.jde.2021.11.026},
       URL = {https://doi.org/10.1016/j.jde.2021.11.026},
}

@article {Winkler16,
    AUTHOR = {Winkler, Michael},
     TITLE = {The two-dimensional {K}eller-{S}egel system with singular
              sensitivity and signal absorption: {G}lobal large-data solutions
              and their relaxation properties},
   JOURNAL = {Math. Models Methods Appl. Sci.},
  FJOURNAL = {Mathematical Models and Methods in Applied Sciences},
    VOLUME = {26},
      YEAR = {2016},
    NUMBER = {5},
     PAGES = {987--1024},
      ISSN = {0218-2025,1793-6314},
   MRCLASS = {35K40 (35B40 35B65 35D30 35K59 35Q92 92C17)},
  MRNUMBER = {3464427},
MRREVIEWER = {Youshan\ Tao},
       DOI = {10.1142/S0218202516500238},
       URL = {https://doi.org/10.1142/S0218202516500238},
}

@article {Winkler14,
    AUTHOR = {Winkler, Michael},
     TITLE = {Global asymptotic stability of constant equilibria in a fully
              parabolic chemotaxis system with strong logistic dampening},
   JOURNAL = {J. Differential Equations},
  FJOURNAL = {Journal of Differential Equations},
    VOLUME = {257},
      YEAR = {2014},
    NUMBER = {4},
     PAGES = {1056--1077},
      ISSN = {0022-0396,1090-2732},
   MRCLASS = {35K51 (35B35 35B40 35Q92 92C17)},
  MRNUMBER = {3210023},
MRREVIEWER = {Philippe\ Lauren\c{c}ot},
       DOI = {10.1016/j.jde.2014.04.023},
       URL = {https://doi.org/10.1016/j.jde.2014.04.023},
}

@article {Winkler2010,
    AUTHOR = {Winkler, Michael},
     TITLE = {Aggregation vs. global diffusive behavior in the
              higher-dimensional {K}eller-{S}egel model},
   JOURNAL = {J. Differential Equations},
  FJOURNAL = {Journal of Differential Equations},
    VOLUME = {248},
      YEAR = {2010},
    NUMBER = {12},
     PAGES = {2889--2905},
      ISSN = {0022-0396,1090-2732},
   MRCLASS = {35K51 (35B35 35B40 92C17)},
  MRNUMBER = {2644137},
MRREVIEWER = {Bisso\ Saley},
       DOI = {10.1016/j.jde.2010.02.008},
       URL = {https://doi.org/10.1016/j.jde.2010.02.008},
}

@article {Wang25,
    AUTHOR = {Wang, Kun and Liu, Enlong and Feng, Xinlong},
     TITLE = {Optimal error estimate of unconditionally
              positivity-preserving, mass-conserving and energy stable
              method for the {K}eller-{S}egel chemotaxis model},
   JOURNAL = {Math. Comp.},
  FJOURNAL = {Mathematics of Computation},
    VOLUME = {94},
      YEAR = {2025},
    NUMBER = {356},
     PAGES = {2761--2793},
      ISSN = {0025-5718},
   MRCLASS = {65M06 (35Q35 65M15 92C17)},
  MRNUMBER = {4940542},
       DOI = {10.1090/mcom/4041},
}

@article {Zhou17,
    AUTHOR = {Zhou, Guanyu and Saito, Norikazu},
     TITLE = {Finite volume methods for a {K}eller-{S}egel system: {D}iscrete
              energy, error estimates and numerical blow-up analysis},
   JOURNAL = {Numer. Math.},
  FJOURNAL = {Numerische Mathematik},
    VOLUME = {135},
      YEAR = {2017},
    NUMBER = {1},
     PAGES = {265--311},
      ISSN = {0029-599X},
   MRCLASS = {65M08 (35B44 35K59 35M33 65M15 92C17)},
  MRNUMBER = {3592113},
MRREVIEWER = {G\"{u}nter Karl Franz B\"{a}rwolff},
       DOI = {10.1007/s00211-016-0793-2},
       URL = {https://doi.org/10.1007/s00211-016-0793-2},
}

@article {Wang22,
    AUTHOR = {Wang, Shufen and Zhou, Simin and Shi, Shuxun and Chen, Wenbin},
     TITLE = {Fully decoupled and energy stable {BDF} schemes for a class of
              {K}eller-{S}egel equations},
   JOURNAL = {J. Comput. Phys.},
  FJOURNAL = {Journal of Computational Physics},
    VOLUME = {449},
      YEAR = {2022},
     PAGES = {110799},
      ISSN = {0021-9991},
   MRCLASS = {65M60 (65L06 65M12)},
  MRNUMBER = {4339000},
       DOI = {10.1016/j.jcp.2021.110799},
}

@article {Jiang25,
    AUTHOR = {Jiang, Ziwen and Wang, Lizhen},
     TITLE = {Mild solutions to the {C}auchy problem for time-space
              fractional {K}eller-{S}egel-{N}avier-{S}tokes system},
   JOURNAL = {Fract. Calc. Appl. Anal.},
  FJOURNAL = {Fractional Calculus and Applied Analysis. An International
              Journal for Theory and Applications},
    VOLUME = {28},
      YEAR = {2025},
    NUMBER = {3},
     PAGES = {1503--1538},
      ISSN = {1311-0454,1314-2224},
   MRCLASS = {35Q35 (35B30 35Q92 35S16)},
  MRNUMBER = {4914435},
       DOI = {10.1007/s13540-025-00400-w},
       URL = {https://doi.org/10.1007/s13540-025-00400-w},
}
\bibliographystyle{ws-m3as}

\end{document}